\documentclass[reqno,english]{amsart}
\usepackage[utf8]{inputenc}
\usepackage{amsmath,amssymb,amsthm,mathrsfs,color,times,textcomp,yfonts,mathtools}
\usepackage[left=0.8in, right=0.8in, top=0.8in, bottom=0.8in]{geometry}

\allowdisplaybreaks[4]

\usepackage{subcaption}
\usepackage[normalem]{ulem}
\usepackage[export]{adjustbox}
\usepackage{esint}
\usepackage{xcolor}
\usepackage{array}
\usepackage[colorlinks=true]{hyperref}
\hypersetup{urlcolor=blue, citecolor=red, linkcolor=blue}

\hypersetup{
	colorlinks=true,
	linkcolor=red
}

\usepackage[square,numbers]{natbib}

\usepackage{float}
\usepackage{ulem}
\usepackage{syntonly}
\usepackage{mathtools}
\usepackage{bm}
\usepackage{amsfonts,amsmath,latexsym,verbatim,amscd,mathrsfs,color,array}
\usepackage[colorlinks=true]{hyperref}

\usepackage{amsmath,amssymb,amsthm,amsfonts,graphicx,color}
\usepackage{amssymb}
\usepackage{epstopdf}

\newcommand{\1}{\mathbf{1}}

\newcommand{\R}{ {\mathbb{R}}}

\newcommand{\TT}{{\mathcal T}  }
\newcommand{\LLL}{{\mathcal L}  }

\newcommand{\pp}{ {\partial} }

\newcommand{\B}{ \mathcal{B} }

\newcommand{\RR}{{{\mathbb R}}}

\newcommand{\cuad}{{\sqcap\kern-.68em\sqcup}}
\newcommand{\DD}{{\mathcal D}}

\newcommand{\be}{\begin{equation}}
	\newcommand{\ee}{\end{equation}}

\theoremstyle{plain}
\newtheorem{theorem}{Theorem}[section]
\newtheorem{lemma}[theorem]{Lemma}
\newtheorem{prop}[theorem]{Proposition}
\newtheorem{corollary}[theorem]{Corollary}
\newtheorem{remark}{Remark}[theorem]
\newcommand{\bremark}{\begin{remark} \em}
	\newcommand{\eremark}{\end{remark} }

\numberwithin{equation}{section}

\begin{document}
\title[Infinite time blow-up for critical heat equation]{On Fila-King Conjecture in Dimension Four}

\author[J. Wei]{Juncheng Wei}
\address{\noindent
Department of Mathematics,
University of British Columbia, Vancouver, B.C., V6T 1Z2, Canada}
\email{jcwei@math.ubc.ca}

\author[Q. Zhang]{Qidi Zhang}
\address{\noindent
Department of Mathematics,
University of British Columbia, Vancouver, B.C., V6T 1Z2, Canada}
\email{qidi@math.ubc.ca}

\author[Y. Zhou]{Yifu Zhou}
\address{\noindent
Department of Mathematics,
Johns Hopkins University, 3400 N. Charles Street, Baltimore, MD 21218, USA}
\email{yzhou173@jhu.edu}

\begin{abstract}
We consider the following Cauchy problem for the four-dimensional energy critical heat equation
\begin{equation*}
\begin{cases}
u_t=\Delta u+u^{3}  ~&\mbox{ in }~ \R^4 \times (0,\infty),\\
u(x,0)=u_0(x) ~&\mbox{ in }~ \R^4.
\end{cases}
\end{equation*}
 We construct a positive infinite time blow-up solution $u(x,t)$ with the blow-up rate $ \| u(\cdot,t)\|_{L^\infty(\R^4)} \sim \ln t$ as $t\to \infty$ and show the stability of the infinite time blow-up. This gives a rigorous proof of a conjecture by Fila and King   \cite[Conjecture 1.1]{filaking12}.
\end{abstract}
\maketitle

\tableofcontents


\bigskip

\section{Introduction and main results}

\medskip

Since the seminal work of Fujita \cite{Fujita66}, the following nolinear heat equation
\begin{equation}\label{eqn-Fujita}
	\begin{cases}
		u_t=\Delta u+|u|^{p-1}u~&\mbox{ in } ~\R^n\times (0,\infty),\\
		u(x,0)=u_0(x)~&\mbox{ in } ~\R^n,
	\end{cases}
\end{equation}
with $p>1$, $n\geq3$ has been extensively studied. The energy functional for \eqref{eqn-Fujita} is
$$
E(u)=\frac12\int_{\R^n} |\nabla u|^2-\frac1{p+1}\int_{\R^n} |u|^{p+1},
$$
and for classical solution $u(x,t)$ with sufficient spatial decay, one has
$$
\frac{d}{dt}E(u(\cdot,t))=-\int_{\R^n} |u_t|^2.
$$
Many literatures have been devoted to studying problem \eqref{eqn-Fujita} about the singularity formation, especially the blow-up rates, profiles and sets.
We refer the readers to the book of Quittner and Souplet \cite{Souplet07book} for comprehensive survey and also recent developments.

\medskip

For the finite time blow-up, it is said to be of
\medskip
\begin{itemize}
	\item {\em type I~} if
	$${\lim \sup}_{t\to T}(T-t)^{\frac{1}{p-1}}\|u(\cdot,t)\|_{\infty}<\infty;$$
	\item {\em type II~} if
	$${\lim \sup}_{t\to T}(T-t)^{\frac{1}{p-1}}\|u(\cdot,t)\|_{\infty}=\infty.$$
\end{itemize}
\medskip
Type I blow-up is more ``generic'', while type II blow-up is much more difficult to detect. In particular, two different types of blow-up phenomena in problem \eqref{eqn-Fujita} depend sensitively on the value of the exponent $p$. In this setting, the critical Sobolev exponent
\begin{equation*}
    p_s=\begin{cases}
        \frac{n+2}{n-2}~&\mbox{ for }~n\geq 3\\
        \infty ~&\mbox{ for }~n=1,2\\
    \end{cases}
\end{equation*}
is special in various ways. Giga, Matsui and Sasayama \cite{Giga04Indiana,Giga04-MathMethod} proved that for $p<p_s$, only type I blow-up can occur in the case that $\Omega$ is $\mathbb{R}^n$ or a convex domain. For the energy critical case $p= p_s$,
in the positive radial and monotonically decreasing class, Filippas, Herrero and Vel\'azquez \cite{FHV00} excluded the possibility of type II blow-up for $n\ge 3$, and Matano and Merle \cite[Theorem 1.7]{MatanoMerle04} removed the monotone assumption and obtained the same result.
Wang and Wei \cite{wang2021refined}
proved the same result to the non-radial positive class in higher dimensions $n\geq 7$.  For $p<p_s$, finite time type I blow-up solution was found and its stability was studied in \cite{Merle97Duke}. For the critical case $p=p_s$ in $\mathbb{R}^n$ with $n\geq 7$, classification results were proved near the ground state of the energy critical heat equation  in \cite{Collot17CMP}. On the other hand, sign-changing type II blow-up solutions to the energy critical heat equation in dimensions $n=3,4,5,6$ were first conjectured to exist by \cite{FHV00} and have been rigorously constructed recently in \cite{Schweyer12JFA,type25D,ni4D,harada2019higher,harada6D,type23D,li2022slow}. In the supercritical case, classification of type I and type II solutions in radially symmetric class have been studied in \cite{Merle04CPAM,Merle09JFA,Merle11JFA} and the references therein, and  the construction of Type II blow-up was first established in the radial case by Herrero and Vel\'azquez \cite{HV1993} and in the non-radial case (under some restrictions of the exponent $p$) by Collot \cite{Collot2017}.

\medskip

Concerning infinite time blow-up for $p=p_s$, Galaktionov and King \cite{King03JDE} investigated positive, radially symmetric, global unbounded solutions for problem \eqref{eqn-Fujita} in the case of unit ball with Dirichlet boundary condition in dimensions $n\geq 3$. See also \cite[Theorem 1.4]{Suzuki08Indiana} for the case that the domain is symmetric and convex. In the non-radial setting, positive infinite time blow-up solution for problem \eqref{eqn-Fujita} with Dirichlet boundary condition and $n\geq 5$ was constructed by Cortazar, del Pino and Musso in \cite{Green16JEMS}. The solution constructed in \cite{Green16JEMS} takes the profile of sharply scaled Aubin-Talenti bubbles
\begin{equation*}
	U_{\mu,\xi}(x)=\mu^{-\frac{n-2}{2}}U\left(\frac{x-\xi}{\mu}\right)= (n(n-2))^{\frac{n-2}{4}} \left(\frac{\mu}{\mu^2+|x-\xi|^2}\right)^{\frac{n-2}{2}},
\end{equation*}
which solve the Yamabe problem
$$\Delta U+ U^{\frac{n+2}{n-2}}=0~\mbox{ in }~\R^n.$$
Moreover, the blow-up location for the solution is determined by the Green's function of $-\Delta$ in $\Omega$, while for elliptic problems, the role of the Green's function in bubbling phenomena  has been known for a long time since the works \cite{BC88CPAM} and \cite{Bahri95CVPDE}. In \cite{del2018sign}, non-radial and sign-changing solution which blows up at infinite time has been constructed. Bubble towers at infinite time and backward time infinity have been constructed in \cite{tower7D} and \cite{sun2021bubble}, respectively.

In a very interesting paper \cite{filaking12}, Fila and King studied problem \eqref{eqn-Fujita} in the whole space $\RR^n$ with the critical exponent $p=p_s$ and gave insight on the infinite time blow-up in the case of a
radially symmetric, positive initial condition with an exact power decay rate. By  formal matched
asymptotic analysis, they demonstrated that the blow-up rate is  determined by the power decay in a precise
manner. Intriguingly enough, their analysis leads them to conjecture that infinite time blow-up should
only happen in low dimensions 3 and 4, see Conjecture 1.1 in \cite{filaking12}. Recently, this has been confirmed and rigorously proved by del Pino, Musso and the first author  in \cite{173D} for $n=3$, where the leading part of the scaling parameter is achieved by asymptotic analysis. For the case $n=4$, Fila and King conjectured that infinite time blow-up only exists when $\ell>2$ for radial solutions, where
$$
\lim_{|x|\to \infty}|x|^{\ell}u_0(|x|)=A
$$
for some $A>0$.

\medskip

In other contexts, for instance,  Liouville-type theorems for Fujita equation, in parallel with the seminal work of Gidas and Spruck \cite{GS81CPAM} in the elliptic setting, and long-time behaviors for the solutions to Fujita equation with supercritical exponent have been studied in \cite{Polacik03MathAnn,Fila06JDE,PQS07Indiana,Fila07ADE,Polacik14MathAnn,Quittner21Duke,PQ21DCDS} and the references therein.

\medskip

In this paper, we are concerned with the following Cauchy problem for the Fujita equation with critical exponent in dimension $n=4$
\begin{equation}\label{eqn}
	\begin{cases}
		u_t=\Delta u+u^{3}~&\mbox{ in }~ \R^4\times (t_0, \infty),\\
		u(x,t_0)=u_0(x)~&\mbox{ in }~ \R^4.
	\end{cases}
\end{equation}
The aim of this paper is to construct infinite time blow-up solution, confirming the conjecture by \cite[Conjecture 1.1]{filaking12}, and further investigate the stability of the infinite time blow-up.
Throughout this paper, $\eta$ is a smooth cut-off function which satisfies that  $\eta(s) =1$ for $s \le  1$ and $\eta(s) =0$ for $s\ge \frac 32$. Our main results are stated as follows.

\medskip

\begin{theorem}\label{thm1}
For $t_0$ sufficiently large, there exists initial value $u_0>0$ with exponential decay such that the positive solution $u(x,t)$ to \eqref{eqn} blows up at infinite time. More precisely, the solution takes the form of the sharply scaled bubble
	$$
	u(x,t) =  \eta\left(\frac{x-\xi}{\sqrt{t}}\right)\mu^{-1}(t)w\left(\frac{x-\xi(t)}{\mu(t)}\right) + O\big( (\ln t)^{-1} \min\{ t^{-1}, |x|^{-2}\} \big)
	$$
where $w(y) = 2^{\frac 3 2} \frac 1{1+|y|^2}$. The blow-up rate and location are given by
	$$\mu(t)= \frac{1}{\ln t} \Big(1+ O\big(\frac{\ln\ln t}{\ln t}\big)\Big),\quad \xi(t) = O(t^{-1}).$$
\end{theorem}
More precisely, the positive initial value of the solution constructed above is
\begin{equation*}
	\begin{aligned}
	u(x,t_0)= \ &
		\mu^{-1}(t_0) w\big(\frac{x-\xi(t_0)  }{ \mu(t_0)  }\big) \eta\big(\frac{ x-\xi(t_0)   }{\sqrt {t_0} }\big)
		+
		2^{\frac 32} \mu(t_0) |x-\xi(t_0) |^{-2} \Big(e^{-\frac{|x-\xi(t_0)  |^2}{4t_0 }} -\eta\big(\frac{ x-\xi(t_0)  }{\sqrt {t_0} }\big)\Big)
		\\
		&
		+
		\bar{\mu}_0^{-1}(t_0) \Phi_0\big(\frac{ x-\xi(t_0)  }{\bar{\mu}_0 (t_0)} , t_0\big)\eta\big(\frac{4(x-\xi(t_0) ) }{\sqrt {t_0}}\big)
		+
		\eta\big(\frac{x-\xi(t_0) }{ \mu_0(t_0) R(t_0)}\big)
		e_0  \mu^{-1}(t_0) Z_{0}\big(\frac{x-\xi(t_0) }{ \mu(t_0) }\big),
	\end{aligned}
\end{equation*}
where $\mu_0$, $\bar{\mu}_{0}$ are the leading order of $\mu $, and $\bar{\mu}_{0}\sim \mu_0 = (\ln t)^{-1}$; $\Phi_{0}$ is a global correction function given in Section \ref{SolElliptic-sec};  $e_0$ is a constant and $Z_{0}$ is the eigenfunction with respect to the first eigenvalue for the linearized operator, which has exponential decay, see \eqref{def-Z0Z0}.

\medskip

We further investigate the stability of the blow-up solution constructed in Theorem \ref{thm1} and obtain the stability in the following sense.
\begin{theorem}\label{stability-th}
For any $g_0$, not necessarily radially symmetric, satisfying  $|g_0(x)|\le C_{g} t_{0}^{-\frac{ \min{ \{\ell,4} \} }{2}} \langle x\rangle^{-\ell}$, $\ell>3$, and for $t_0$ sufficiently large, there exists a solution $u[g_0](x,t)$ to \eqref{eqn} blowing up at infinite time
with the rate
	$$\mu[g_0](t)= \frac{1}{\ln t} \Big(1+ O\big(\frac{\ln\ln t}{\ln t}\big)\Big),\quad \xi[g_0](t) = O(t^{-1}).$$
The initial value is given by
\begin{equation*}
	\begin{aligned}
	u(x,t_0)= \ &
		\big( \mu[g_0] (t_0)\big)^{-1} w\big(\frac{x-\xi[g_0](t_0) }{ \mu[g_0](t_0) }\big) \eta\big(\frac{ x-\xi[g_0](t_0)  }{\sqrt {t_0} }\big)
		\\
		&
		+
		2^{\frac 32}  \mu[g_0](t_0) \big|x-\xi[g_0](t_0) \big|^{-2} \Big(e^{-\frac{|x-\xi[g_0](t_0) |^2}{4t_0 }} -\eta\big(\frac{ x-\xi[g_0](t_0) }{\sqrt {t_0} }\big)\Big)
		\\
		&
		+
		\big(\bar{\mu}_0(t_0)\big)^{-1} \Phi_0\big(\frac{ x-\xi[g_0](t_0) }{\bar{\mu}_0(t_0) } , t_0\big)\eta\big(\frac{4(x-\xi[g_0](t_0)) }{\sqrt {t_0}}\big)
		\\
		&
		+
		\eta\big(\frac{x-\xi[g_0](t_0)}{ \mu_0(t_0) R(t_0) }\big)
		e_0[g_0]  \big( \mu[g_0](t_0)\big)^{-1} Z_{0}\big(\frac{x-\xi[g_0](t_0) }{ \mu[g_0](t_0) }\big) + g_0
	\end{aligned}
\end{equation*}
where $\mu[g_0] \rightarrow \mu$, $\xi[g_0] \rightarrow \xi$, $e_0[g_0] \rightarrow e_0$ in some topology as $C_g \rightarrow 0$. In the radial setting, the same conclusion holds for $\ell >2$ with $\xi[g_0]\equiv 0$ and  $\bar{\mu}_0[g_0]\rightarrow \bar{\mu}_0$ as $C_g \rightarrow 0$ additionally.

\end{theorem}

\begin{remark}
\noindent
\begin{itemize}
\item Indeed, the initial value of the infinite time blow-up solution in Theorem \ref{thm1} has exponential decay at space infinity. By Theorem \ref{stability-th}, we can add suitable perturbation for the initial value to achieve that  $$\lim\limits_{|x|\rightarrow\infty} |x|^{\ell} u(x,t_0) =A$$ for any $|A|$ small enough, recovering the assumption on the initial value in the conjecture by Fila and King \cite{filaking12}.
	
\item It is very possible to generalize the stability result for all $\ell >2$ in the non-radial setting, see Remark \ref{rmk10.0.1}.

\item We do not know if the solution we construct is threshold solution or not.

\end{itemize}
\end{remark}

\medskip

Our construction is based on the {\em inner--outer gluing method} developed recently in \cite{Green16JEMS,17HMF} for parabolic problems, and the gluing method has been a powerful tool to investigate the singularity formation for various nonlinear PDEs such as parabolic equations and systems, fluid equations, geometric flows and others. See \cite{173D,18Euler,17halfHMF} and the references therein. The parabolic gluing method is much more different from the asymptotic analysis given in \cite{filaking12}. Some essential new features and difficulties in this paper are listed below.

\medskip

One key feature and difficulty is the non-local dynamics for the scaling parameter $\mu(t)$. It turns out that the dynamics for $\mu(t)$ is governed by an integro-differential operator, which is a natural consequence of the fact that the linear generator of dilations of the Aubin-Talenti bubble is of slow decay in lower dimensions. This non-local phenomenon has also been observed in \cite{17HMF,173D,ni4D,gluingKS} for lower dimensional problems. In our case here,  neither the usual Laplace transform nor Riemann-Louville type method is applicable since the integro-differential equation is not in the class of Abel-type integral equations. The non-local operator here is the threshold/endpoint case in certain sense, and one needs to carry out much more delicate analysis to investigate its solvability.

\medskip

Our strategy is to decompose the non-local equation for $\mu(t)$ into two parts: the dominating term and the remainder term. The dominating term will be solved by contraction mapping theorem, while the remainder term will leave a much smaller error. To be more precise, the desired blow-up rate is determined at leading order. However, due to the way that we handle the non-local operator, the time decay is not fast enough for the remainder in the gluing procedure, and we will iterate this process finitely many times to make the remainder term have faster time decay than the one provided by the outer problem. This smaller remainder will be handled when solving the next order of $\mu(t)$.

\medskip

After getting the leading order of $\mu(t)$, we need to solve the corresponding linearized elliptic equation to improve the time decay of the error term, which is essential for finding suitable weighted topologies ensuring the implementation of the gluing procedure. When solving the next order $\mu_1(t)$, we still need to decompose the non-local equation into two parts. The main difference is that the involved outer problem in the equation of $\mu_1(t)$ only has H\"older continuity in $t$ variable. The derivative of $\mu_1(t)$ will inherit H\"older continuity from the outer problem, which will be used to control the remainder term.

\medskip

On the other hand, the rather slow logarithmic blow-up rate produces following difficulties.
 There are several slow decaying linear terms which involves the inner part cannot be controlled as the right hand side of the inner or outer problem. Instead, we regard these slow decaying terms as part of the linearization of the inner problem and develop a new linear theory. See Remark \ref{f1f2-remark}. The dealing of these terms is in a similar spirit as in \cite{gluingKS}, where the logarithmic blow-up speed also appears.


\medskip

Thanks to the generality for the gluing method, we are able to study the stability for the solution constructed in Theorem \ref{thm1} with both radial and non-radial perturbations, and non-radial infinite time blow-up solutions are easily found by suitable perturbation for the initial value.

\medskip

Before carrying out the construction, we list several commonly used notations throughout the paper as follows.

\noindent
\textbf{Notations:}
\begin{itemize}
\item We write $a\lesssim b$ ($a \gtrsim b$) if there exists a constant $C > 0$ such that $a \le  Cb$ ($a \ge  Cb$) where $C$ is independent of $t$, $t_0$. Set $a \sim b$ if $b \lesssim a \lesssim b$.

\item In general, the letter $C(a,b,\dots)$
stands for a positive constant depending on parameters $a, b\dots$ that
might change its actual value at each occurrence.

\item The symbol $f[g_1,g_2,\dots]$ means that the function $f$ depends on some functions $g_1$, $g_2$, \dots.

\item $f\approx g$ means that $|f-g|\rightarrow 0$ as $t\rightarrow \infty$.

\item The symbol $O(f(x))$ is used to denote a real-valued function that satisfies $|O(f(x))| \lesssim |f(x)|$ in a domain of $x$ that is either specified explicitly or
follows from the context.

\item For any fixed real number $x$, the symbol $x-$ denotes a number which is less than $x$ and can  be chosen close to $x$ arbitrarily.

\item Denote $\langle y \rangle = \sqrt{1+|y|^2}$ for any $y\in \RR^n$.

\item Denote $\1_{\{x\in \Omega\}}$ as the characteristic function with $\1_{\{x\in \Omega\}}=1$ if $x\in \Omega$ and $\1_{\{x\in \Omega\}}=0$ if $x\not\in \Omega$.
\end{itemize}

\medskip

\section{Approximate solution and improvement}
	
\subsection{First approximate solution}

We consider the energy critical heat equation in dimension $4$
\begin{equation}\label{nonlinear-heateq}
\begin{cases}
u_t=\Delta u+u^{3}~&\mbox{ in }~ \R^4 \times (0,\infty),\\
u(x,0)=u_0(x)~&\mbox{ in }~ \R^4 .
\end{cases}
\end{equation}

Since changing the initial time will not change the structure of the nonlinear heat equation, we assume the initial time is $t = t_0$ and $t_0$ is sufficiently large.

We use the steady state solution
$$
w(y) = 2^{\frac 3 2} \frac 1{1+|y|^2}
$$
 as the building block of construction. It is known that all the bounded kernels of the corresponding linearized operator $\Delta+3w^2$ are given by
\[
Z_{i}(y) = \pp_{y_{i}} w= -2^{\frac 52} \frac{y_i}{(1+|y|^2)^2}, \mbox{ \ for \ } i=1,\dots,4,  \quad
Z_5 (y) = w + y \cdot \nabla w = 2^{\frac 3 2}\frac{1-|y|^2 }{(1+|y|^2 )^2 } .
\]

We take the leading profile of the infinite time blow-up solution as
\begin{equation*}
	u_1(x,t) = \mu^{-1}(t) w\left(\frac {x-\xi(t)}{\mu(t)} \right) \eta\left(\frac{x-\xi(t)}{\sqrt t}\right)
\end{equation*}
where $\mu(t)$, $\xi(t)\in C^1(t_0,\infty)$. Throughout this paper, we make the following ansatz
\begin{equation}\label{ansatz-mu}
\frac{1}{C_\mu \ln t} \le |\mu| + t\ln t |\mu_t| \le \frac{C_\mu }{\ln t} ,
\end{equation}
\begin{equation*}
 \xi(t)\rightarrow 0 \mbox{ \ as \ } t\rightarrow\infty
\end{equation*}
where $C_\mu \ge 1$ is a large constant. Later we shall rigorously justify the above ansatz about the asymptotics for the scaling and translation parameters.

Denote the error function as
\[
S[g ] := -\pp_t g + \Delta g + g^3 .
\]
Then the error produced by the first approximate solution $u_1$ is given by
\begin{equation*}
\begin{aligned}
	S[u_1 ]
= \ &
\mu^{-2} \mu_t Z_5\left(\frac{x-\xi}{\mu}\right) \eta\left(\frac{x-\xi}{\sqrt t}\right)
+ E[\mu]
+
\mu^{-2} \xi_t \cdot \nabla w\left(\frac{x-\xi}{\mu}\right)
\eta\left(\frac{x-\xi}{\sqrt t}\right)\\
\ &
+ \mu^{-1} t^{-\frac 12} \xi_t  \cdot \nabla \eta\left(\frac{x-\xi}{\sqrt t} \right)
  w\left(\frac{x-\xi}{\mu}\right)
\end{aligned}
\end{equation*}
where
\begin{equation*}
\begin{aligned}
E[\mu] := \ &
2^{-1} \mu^{-1} t^{-1} w\left(\frac{x-\xi}{\mu}\right) \nabla \eta\left(\frac{x-\xi}{\sqrt t} \right)\cdot \frac{x-\xi}{ \sqrt t}
+ 2 \mu^{-2} t^{-\frac 12} \nabla w\left(\frac{x-\xi}{\mu}\right) \cdot \nabla \eta\left(\frac{x-\xi}{\sqrt t}\right)
\\
&
+ \mu^{-1} t^{-1} w\left(\frac{x-\xi}{\mu}\right) \Delta \eta\left(\frac{x-\xi}{\sqrt t}\right)
 +
\mu^{-3} w^3\left(\frac {x-\xi}{\mu} \right) \left[ \eta^3\left(\frac{x-\xi}{\sqrt t}\right)
-
\eta\left(\frac{x-\xi}{\sqrt t}\right)\right].
\end{aligned}
\end{equation*}
In next section, we shall add two global corrections to improve the slow decaying error.

\medskip

\subsection{Transferring slow decaying terms by heat equations}\label{sec-cut-slowdecay}
For some admissible function $f(x,t)$, denote
\begin{equation}\label{T-out-def}
	\mathcal{T}^{out}_{n}[f](x,t) :=
	\int_{t_{0}}^{t} \int_{\RR^n }
	(4\pi(t-s))^{-\frac n2}
	e^{-\frac{|x-z|^2}{4(t-s)}}
	f(z,s) d z d s.
\end{equation}
In the rest of the paper, we will use Lemma \ref{annular-forward} and Lemma \ref{far-forward} in the appendix to estimate
$	\mathcal{T}^{out}_{n}$  frequently and sometimes will not state repeatedly.

Set $y=\frac{x-\xi}{\mu}$. A term is said to be of slow decay if its spatial decay is equal to or slower than $\langle y \rangle^{-2}$. Otherwise, it is of fast decay. Fast decay is necessary for the gluing procedure.
For this reason we will transfer the slow decaying terms in $S[u_1]$ by heat equations. We now introduce the correction function $\varphi$ to improve the error. For
\begin{equation*}
S[u_1 + \varphi] =
 - \pp_t \varphi  + \Delta \varphi + S[u_1]
 + (u_1 + \varphi)^3 - u_1^3,
\end{equation*}
we set $\bar{x}= x-\xi$ and choose $\varphi(\bar{x},t) = \varphi_{1}(\bar{x},t) +\varphi_{2}(\bar{x},t)$ such that
\begin{equation*}
\pp_t \varphi_{1} =  \Delta_{\bar{x}} \varphi_{1}
+ E[\mu],\quad	\pp_{t} \varphi_{2} = \Delta_{\bar{x}} \varphi_{2} + \mu^{-2} \mu_t Z_5(\frac{\bar{x}}{\mu}) \eta(\frac{\bar{x}}{\sqrt t}) .
\end{equation*}
The properties of $\varphi_1$ and $\varphi_2$ are given in the following two lemmas.

\begin{lemma}\label{varphi1-lem}
	Assume that $\mu$ satisfies \eqref{ansatz-mu} and $\mu_1$ satisfies $|\mu_{1}|\le \frac{\mu}{2}$.
	Consider
	\begin{equation}\label{varphi1-eq}
		\pp_t \varphi_{1} =  \Delta_{\bar{x}} \varphi_{1}
		+ E[\mu] .
	\end{equation}
	There exists a solution $\varphi_{1} = \varphi_{1}[\mu]$ satisfying the following pointwise estimates
\begin{equation*}
	| \varphi_1[\mu] |
	\lesssim
(t\ln t)^{-1} \1_{\{ |\bar{x}| \le 2t^{\frac 12} \}}
	+
	t^{2} (\ln t)^{-1} |\bar{x}|^{-6}
	\1_{\{ |\bar{x}| > 2t^{\frac 12} \}},
\end{equation*}	
	\begin{equation*}
		|\nabla_{\bar{x}} \varphi_1|
		\lesssim
		t^{-\frac 32} (\ln t)^{-1} \1_{\{ |\bar{x}| \le 2t^{\frac 12} \}}
		+
		t^{\frac 32} (\ln t)^{-1} |\bar{x}|^{-6}
		\1_{\{ |\bar{x}| > 2t^{\frac 12} \}}
		.
	\end{equation*}	
	More precisely,
	\begin{align*}
		\varphi_1[\mu] = \ &
		\left[
		-2^{-\frac 12} \mu  t^{-1}  + O( \mu t^{-2} |\bar{x}|^2 )
		 + O \Big(
		t^{-2} \int_{t_0/2}^{t} (  s^{-1 } \mu^{3}(s) + 	s |\mu_t(s)|  ) ds
		\Big)
		\right] \1_{\{ |\bar{x}|  \le 2 t^{\frac 12} \}}
\\
& +
O\left( \mu |\bar{x}|^{-2} e^{-\frac{| \bar  x|^2}{16 t}}
+
|\bar{x}|^{-6}
 \int_{t_0/2}^{t} s^2 |\mu_t(s)|  ds
+
t^{-2} e^{-\frac{|\bar{x}|^2}{16 t} }\int_{t_0/2}^{t/2} (s^{-1 } \mu^{3}(s)  + 	s  |\mu_t(s)| ) ds
\right) \1_{\{ |\bar{x}| > 2t^{\frac 12}  \}}
		,
	\end{align*}
	\begin{align*}
		&
		\varphi_1[\mu+\mu_1] - \varphi_1[\mu]
		=
		\bigg[
		-2^{-\frac 12} \mu_1  t^{-1}   + O( |\mu_1| t^{-2} |\bar{x}|^2 )
		\\
		&
		+
		O \Big( t^{-2}  \mu^2  \sup\limits_{t_1\in [t/2,t] }|\mu_1(t_1)|
		+
		\sup\limits_{t_1\in [t/2,t] }|\mu_{1t}(t_1)|  +
		t^{-2} \int_{t_0/2}^{t/2}
		\big(
		s^{-1} |\mu_1(s)| \mu^{2}(s)  + s  |\mu_{1t}(s)| \big) ds \Big)
		\bigg]  \1_{\{ |\bar{x}|  \le 2t^{\frac 12} \}}
		\\
		&
		+
		O \bigg(  \sup\limits_{t_{1}\in [t/2,t] }	|\mu_1(t_1)| |\bar{x}|^{-2} e^{-\frac{| \bar  x|^2}{16 t}}
		+
		|\bar{x}|^{-6}
		\Big(
		t^3 \sup\limits_{t_{1}\in [t/2,t] }	|\mu_{1t}(t_1)| +  \int_{t_0/2}^{t/2} s^2 |\mu_{1t}(s)|  ds \Big)
		\\
		&
		+ 	t^{-2} e^{-\frac{|\bar{x}|^2}{16 t} }\int_{t_0/2}^{t/2}
		\Big(
		s^{-1} |\mu_1(s)| \mu^{2}(s) + s |\mu_{1t}(s)|  \Big) ds \bigg) \1_{\{ |\bar{x}|  > 2t^{\frac 12} \}}
		.
	\end{align*}

\end{lemma}

\begin{proof}
The support of $E$ is in $\{t^{\frac 12} \le |\bar{x}| \le 2 t^{\frac 12} \}$. In this region, by \eqref{ansatz-mu}, $\mu^{-1} |\bar x| \gg 1$, which implies
\begin{equation*}
w(\frac{|\bar x|}{\mu}) = 2^{\frac 32} \mu^{2} |\bar x|^{-2} + O(\mu^{4} |\bar x|^{-4}),\quad
\frac{|\bar x|}{\mu}  w'(\frac{|\bar x|}{\mu})
= -2^{\frac 52} \mu^{2} |\bar x|^{-2} +  O(\mu^{4} |\bar x|^{-4})
.
\end{equation*}
Then the leading term of $E$ denoted by $\tilde{E}$ is given by
	\[
	\tilde{E} =
	2^{\frac 32} \mu t^{-2}
	\left(
	2^{-1}
	\zeta^{-1} \eta'(\zeta )
	-
	\zeta^{-3} \eta'( \zeta )  +  \zeta^{-2} \eta''(\zeta)
	\right) ,
	\quad \zeta = \frac{|\bar x|}{\sqrt t}.
	\]
	Take $\tilde{\varphi}_{1}$ as the approximate solution to \eqref{varphi1-eq}. Set $\tilde{\varphi}_1 = \mu \hat{\varphi}_1$, $
	\tilde{E} = \mu \hat{E}$ and $\hat{\varphi}_{1}$ satisfies
	\[
	\pp_t \hat{\varphi}_1 = \Delta_{\bar x} \hat{\varphi}_1 + \hat{E} .
	\]
We take
	$
	\hat{\varphi}_1 = t^{-1} A\left(\frac{|\bar x|}{\sqrt t}\right)
	$ in the self-similar form. Then
	\begin{equation}\label{A-eq}
		A'' + \left(\frac 3{\zeta} + \frac{\zeta}{2} \right) A' + A
		+
		h(\zeta) = 0 ,
	\end{equation}
	where
	\begin{equation*}
		h(\zeta) = 2^{\frac{3}{2} } \zeta^{-2}
		\left(
		\eta''(\zeta) - \frac 1{\zeta} \eta'(\zeta) + \frac{\zeta}{2} \eta'(\zeta)
		\right) .
	\end{equation*}
	
	Observe that $\zeta^{-2}$, $\zeta^{-2} (1 - e^{-\frac{\zeta^2}{4}})$ are linearly independent kernels to the homogeneous part of \eqref{A-eq}. And \eqref{A-eq} has a particular solution
	\begin{equation*}
		A_{p}(\zeta)=-\zeta^{-2} \int_0^{\zeta} a e^{-\frac{a^2}{4}}\int_0^a h(b)b e^{\frac{b^2}{4}} d b d a
		=
		-\zeta^{-2} \int_0^{\zeta}  2^{\frac 32}  \eta'(a) d a
		=
		2^{\frac 32}  \zeta^{-2}   (1- \eta(\zeta) ) ,
	\end{equation*}
	where we have used $h(b) b e^{\frac{b^2}{4}} =
	2^{\frac 32} (b^{-1} e^{\frac{b^2}{4}} \eta'(b))' $.

	In order to find a solution with fast spatial decay, we take
	\begin{equation*}
		A(\zeta)=A_{p}(\zeta) - 2^{\frac 32} \zeta^{-2} (1 - e^{-\frac{\zeta^2}{4}}) = 2^{\frac 32} \zeta^{-2} ( e^{-\frac{\zeta^2}{4}} - \eta(\zeta))
	\end{equation*}
which implies that
	\begin{equation}\label{tilde-varphi1}
	\hat{\varphi}_{1}( \bar{x} ,t)= 2^{\frac 32 } |\bar{x}|^{-2}
\Big(
e^{-\frac{| \bar{x} |^2}{4t}} -\eta(\frac{| \bar{x} |}{\sqrt t})\Big) ,
\quad
		\tilde{\varphi}_1[\mu]
		= 2^{\frac 32 } \mu | \bar  x|^{-2}
	\Big(
		e^{-\frac{| \bar  x|^2}{4t}} -\eta(\frac{|\bar  x|}{\sqrt t})
\Big) .
	\end{equation}
It is straightforward to see $\tilde{\varphi}_1[\mu](0,t)=-2^{-\frac 12} \mu t^{-1}$,
\begin{equation*}
	\tilde{\varphi}_1[\mu]
=
\left(-2^{-\frac 12} \mu  t^{-1} + O( \mu t^{-2} |\bar{x}|^2 ) \right) \1_{\{ |\bar{x}|  \le 2 t^{\frac 12} \}}
+
O\left(\mu |\bar{x}|^{-2} e^{-\frac{| \bar  x|^2}{4t}} \right)
\1_{\{ |\bar{x}|  > 2 t^{\frac 12} \}}
 ,
\end{equation*}
\begin{equation*}
\tilde{\varphi}_1[\mu+\mu_1] - \tilde{\varphi}_1[\mu]
=
\left(-2^{-\frac 12} \mu_1  t^{-1} + O( |\mu_1| t^{-2} |\bar{x}|^2 ) \right) \1_{\{ |\bar{x}|  \le 2 t^{\frac 12} \}}
+
O\left( |\mu_1| |\bar{x}|^{-2} e^{-\frac{| \bar  x|^2}{4t}} \right)
\1_{\{ |\bar{x}|  >2  t^{\frac 12} \}},
\end{equation*}
and
\begin{equation*}
\begin{aligned}
	\pp_{|\bar{x}|}\hat{\varphi}_{1} =&~ -2^{\frac 52 }  |\bar{x}|^{-3}
	\Big(
	e^{-\frac{| \bar{x} |^2}{4t}} -\eta(\frac{| \bar{x} |}{\sqrt t})
	\Big)
	-
	2^{\frac 32 } |\bar{x}|^{-2}
	\Big(
	e^{-\frac{| \bar{x} |^2}{4t}} \frac{|\bar{x}|}{2t} +\eta'(\frac{| \bar{x} |}{\sqrt t} ) t^{-\frac 12}
	\Big) \\
=&~
O\Big( |\bar{x}| t^{-2} \1_{\{ |\bar{x}|\le t^{\frac 12} \} }
+
|\bar{x}|^{-1} t^{-1}
	e^{-\frac{| \bar{x} |^2}{4t}}
\1_{\{ |\bar{x}|> t^{\frac 12} \} }
\Big)
	,
	\end{aligned}
\end{equation*}
\begin{equation}\label{z1}
|\nabla_{\bar{x}} \tilde{\varphi}_1 [\mu] | = 	|\mu \nabla_{\bar{x}} \hat{\varphi}_1   |
	\lesssim
	|\bar{x} | |\mu| t^{-2}  \1_{\{ |\bar{x}| \le 2t^{\frac 12} \}}
	+
	| \bar{x} |^{-1}  |\mu|  t^{-1}
	e^{-\frac{| \bar{x} |^2}{4t}}
	\1_{\{ |\bar{x}| > 2t^{\frac 12} \}}
	.
\end{equation}
	
	Take $\varphi_{1} = \tilde{\varphi}_{1} + \tilde{\varphi}_{1b}$. Then $\tilde{\varphi}_{1b}$ satisfies
	\begin{equation*}
		\pp_{t} \tilde{\varphi}_{1b} = \Delta_{\bar{x}} \tilde{\varphi}_{1b} -\mu_t \hat{\varphi}_{1} + E - \tilde{E},
	\end{equation*}
	where $\tilde{\varphi}_{1b}$ is given by
	\begin{equation}\label{til-varphi1b}
		\tilde{\varphi}_{1b}[\mu](\bar x,t) = \TT_4^{out}[ -\mu_t \hat{\varphi}_{1} + E - \tilde{E} ] (\bar{x},t)
	\end{equation}
with
	\begin{equation}\label{ZE-1}
		\begin{aligned}
		&	E - \tilde{E} =
			\frac 12 \mu^{-1} t^{-1}
			\Big(
			w(\frac{|\bar{x}|}{\mu}) - 2^{\frac 32} \mu^2 |\bar{x}|^{-2}
			\Big) \eta'(\frac{|\bar{x}|}{\sqrt t} ) \frac{|\bar{x}|}{\sqrt t }
		 + 2 \mu^{-1} t^{-1}
		\Big( \frac{|\bar{x}|}{\mu}  w'(\frac{|\bar{x}|}{\mu})
			+ 2^{\frac 52} \mu^2 |\bar{x}|^{-2}
		\Big)
			\frac{1}{|\bar{x}| t^{-\frac 12}} \eta'(\frac{|\bar{x}|}{\sqrt t})
			\\
			& + \mu^{-1} t^{-1}
			\Big(
			w(\frac{|\bar{x}|}{\mu}) - 2^{\frac 32} \mu^2 |\bar{x}|^{-2}
			\Big)
			\Delta \eta(\frac{|\bar{x}|}{\sqrt t})  +
			\mu^{-3} w^3(\frac {|\bar{x}|}{\mu} ) \Big( \eta^3(\frac{|\bar{x}|}{\sqrt t})
			-\eta(\frac{|\bar{x}|}{\sqrt t}) \Big)
=
O( \mu^3 t^{-3}
\1_{\{ \sqrt t \le |\bar{x}| \le 2\sqrt t \} } )
			.
		\end{aligned}
	\end{equation}
Similarly, we evaluate
\begin{equation*}
(E - \tilde{E} )[\mu+\mu_1] - (E - \tilde{E} )[\mu ]
=
O \big(
 |\mu_1| \mu^2 t^{-3}
\1_{\{ \sqrt t \le |\bar{x}| \le 2\sqrt t \} }  \big) .
\end{equation*}
By Lemma \ref{annular-forward}, one has
\begin{equation*}
\TT_4^{out} \left[\mu^3 t^{-3}
\1_{\{ \sqrt t \le |\bar{x}| \le 2\sqrt t \} }\right]
\lesssim
t^{-2} e^{-\frac{|\bar{x}|^2}{16 t} }\int_{t_0/2}^{t/2} \mu^{3}(s) s^{-1} ds
+
\begin{cases}
\mu^3  t^{-2} &
\mbox{ \ if \ } |\bar{x}|\le t^{\frac 12}
\\
\mu^3 t^{-1} |\bar x|^{-2}  e^{-\frac{|\bar x|^2}{16 t}}
&
\mbox{ \ if \ }   |\bar{x}| >  t^{\frac 12}
\end{cases}
,
\end{equation*}
\begin{equation*}
\begin{aligned}
	\TT_4^{out} \left[|\mu_1| \mu^2 t^{-3}
	\1_{\{ \sqrt t \le |\bar{x}| \le 2\sqrt t \} }\right]
	\lesssim \ &
	t^{-2} e^{-\frac{|\bar{x}|^2}{16 t} }\int_{t_0/2}^{t/2} |\mu_1(s)| \mu^{2}(s) s^{-1} ds
	\\
	&
	+
	\begin{cases}
	\sup\limits_{t_1\in[t/2,t]} |\mu_1(t_1)|	\mu^2 t^{-2} &
		\mbox{ \ if \ } |\bar{x}|\le t^{\frac 12}
		\\
	\sup\limits_{t_1\in[t/2,t] } |\mu_1(t_1)|	\mu^2 t^{-1} |\bar x|^{-2}  e^{-\frac{|\bar x|^2}{16 t}}
		&
		\mbox{ \ if \ }   |\bar{x}| >  t^{\frac 12}
	\end{cases}
	.
\end{aligned}
\end{equation*}
Notice that
	\begin{equation}\label{ZE-2}
		| \mu_t \hat{\varphi}_{1} |
			\lesssim
			 |\mu_t| t^{-1} \1_{\{ |\bar x|\le t^{\frac 12} \} } +
		 |\mu_t| |\bar x|^{-2}
			e^{-\frac{|\bar x|^2}{4t}}
			\1_{ \{ |\bar x| >  t^{\frac 12} \} } .
	\end{equation}
Therefore, by Lemma \ref{annular-forward}, we obtain
\begin{equation*}
\TT_4^{out} \left[|\mu_t| t^{-1} \1_{\{ |\bar x|\le t^{\frac 12} \} }\right]
\lesssim
t^{-2} e^{-\frac{|\bar{x}|^2}{16 t} }\int_{t_0/2}^{t/2}
s |\mu_t(s)|  ds +
\begin{cases}
|\mu_t| &
\mbox{ \ if \ } |\bar{x}|\le t^{\frac 12}
\\
|\mu_t| t |x|^{-2} e^{-\frac{|\bar{x}|^2}{16 t}}
&
\mbox{ \ if \ } |\bar{x}| > t^{\frac 12}
\end{cases}
.
\end{equation*}
By Lemma \ref{far-forward}, we have
\begin{equation*}
	\begin{aligned}
		&
		\TT_4^{out} \left[ |\mu_t| |\bar x|^{-2}
		e^{-\frac{|\bar x|^2}{4t}}
		\1_{ \{ |\bar x| >  t^{\frac 12} \} } \right]
		\lesssim
		\TT_4^{out} \left[ |\mu_t| t^2 |\bar x|^{-6}
		\1_{ \{ |\bar x| >  t^{\frac 12} \} } \right]
		\\
		\lesssim \ &
		t^{-2} e^{-\frac{|\bar{x}|^2}{16 t} }\int_{t_0/2}^{t/2} s |\mu_t(s)|  ds +
		\begin{cases}
			|\mu_t|
			&
			\mbox{ \ if \ }  |\bar{x}|\le t^{\frac 12}
			\\
			|\bar{x}|^{-6}
(t^3 |\mu_t|
+
\int_{t_0/2}^{t/2} |\mu_t(s)| s^2 ds
)
			&
			\mbox{ \ if \ }  |\bar{x}| >  t^{\frac 12}
		\end{cases}
		,
	\end{aligned}
\end{equation*}
and thus
\begin{equation*}
| \TT_4^{out}[ \mu_t \hat{\varphi}_{1} ] |
	\lesssim
	t^{-2} e^{-\frac{|\bar{x}|^2}{16 t} }\int_{t_0/2}^{t/2}
	|\mu_t(s)| s ds +
	\begin{cases}
		|\mu_t| &
		\mbox{ \ if \ } |\bar{x}|\le t^{\frac 12}
		\\
	|\bar{x}|^{-6} \int_{t_0/2}^t |\mu_t(s)| s^2 ds
		&
		\mbox{ \ if \ } |\bar{x}| > t^{\frac 12}
	\end{cases}.
\end{equation*}
It then follows that
\begin{equation*}
\begin{aligned}
|\tilde{\varphi}_{1b}[\mu]|
	\lesssim \ &
	t^{-2} e^{-\frac{|\bar{x}|^2}{16 t} }\int_{t_0/2}^{t/2} ( s^{-1 } \mu^{3}(s)  + 	s  |\mu_t(s)| ) ds
	\\
	&
	+
	\begin{cases}
	\mu^3 t^{-2} + 	|\mu_t| &
		\mbox{ \ if \ } |\bar{x}|\le t^{\frac 12}
		\\
	\mu^3 t^{-1} |\bar x|^{-2}  e^{-\frac{|\bar x|^2}{16 t}}
		+
		|\bar{x}|^{-6}
 \int_{t_0/2}^{t}  s^2  |\mu_t(s)|ds
		&
		\mbox{ \ if \ }   |\bar{x}| >  t^{\frac 12}
	\end{cases}
	,
\end{aligned}
\end{equation*}
\begin{equation}\label{varphi1b-1}
\begin{aligned}
& \big|\tilde{\varphi}_{1b}[\mu +\mu_1] - \tilde{\varphi}_{1b}[\mu]\big|
	\lesssim
	t^{-2} e^{-\frac{|\bar{x}|^2}{16 t} }\int_{t_0/2}^{t/2}
	\big(
	s^{-1} |\mu_1(s)| \mu^{2}(s)  + s |\mu_{1t}(s)|  \big) ds
\\
&	+
	\begin{cases}
t^{-2} 	\mu^2 \sup\limits_{t_{1}\in [t/2,t ] }|\mu_1(t_1)| + \sup\limits_{t_{1}\in [t/2,t ] } |\mu_{1t}(t_1)| &
		\mbox{ \ if \ } |\bar{x}|\le t^{\frac 12}
		\\
	\sup\limits_{t_{1}\in [t/2,t ] }	|\mu_1(t_1)|	\mu^2 t^{-1} |\bar x|^{-2}  e^{-\frac{|\bar x|^2}{16 t}}
		+
		|\bar{x}|^{-6}
\Big(
t^3 \sup\limits_{t_{1}\in [t/2,t ] }	|\mu_{1t}(t_1)| +  \int_{t_0/2}^{t/2} s^2 |\mu_{1t}(s)|  ds \Big)
		&
		\mbox{ \ if \ }   |\bar{x}| >  t^{\frac 12}
	\end{cases}
.
\end{aligned}
\end{equation}
In particular, for $|\mu|\lesssim (\ln t)^{-1}$, $|\mu_{t}| \lesssim t^{-1} (\ln t)^{-2}$, one has
\begin{equation}\label{til1b-upp}
	|\tilde{\varphi}_{1b}[\mu]|
	\lesssim
t^{-1} (\ln t)^{-2}
\1_{\{ |\bar{x}|\le t^{\frac 12} \}}
+
t^{2} (\ln t)^{-2} |\bar{x}|^{-6}
\1_{\{ |\bar{x}| >  t^{\frac 12}  \}}
,
\end{equation}
\begin{equation*}
	|	E - \tilde{E} | \lesssim
 (t \ln t)^{-3}
	\1_{\{ \sqrt t \le |\bar{x}| \le 2\sqrt t \} }
,\quad
	| \mu_t \hat{\varphi}_{1} |
	\lesssim
 (t \ln t)^{-2} \1_{\{ |\bar x|\le t^{\frac 12} \} } +
	t^{-1} (\ln t)^{-2} |\bar x|^{-2}
	e^{-\frac{|\bar x|^2}{4t}}
	\1_{ \{ |\bar x| >  t^{\frac 12} \} } .
\end{equation*}
Then by scaling argument, we have
\begin{equation*}
	|\nabla_{\bar{x}} \tilde{\varphi}_{1b}[\mu]|
	\lesssim
t^{-\frac{3}{2}} (\ln t)^{-2}
\1_{\{ |\bar{x}|\le t^{\frac{1}{2}} \} }
+
t^{\frac 32} (\ln t)^{-2} |\bar{x}|^{-6}
\1_{\{ |\bar{x}| > t^{\frac{1}{2}} \} } .
\end{equation*}
Combining above estimates with \eqref{z1}, we have
\begin{equation*}
	|\nabla_{\bar{x}} \varphi_1|
	\lesssim
t^{-\frac 32} (\ln t)^{-1} \1_{\{ |\bar{x}| \le 2t^{\frac 12} \}}
	+
t^{\frac 32} (\ln t)^{-1} |\bar{x}|^{-6}
	\1_{\{ |\bar{x}| > 2t^{\frac 12} \}}
	.
\end{equation*}

\end{proof}

\begin{lemma}\label{varphi2-lem}
Assume that $\mu$ satisfies \eqref{ansatz-mu} and $\mu_1$ satisfies $|\mu_{1}|\le \frac{\mu}{2}$, $|\mu_{1t}| \le \frac{|\mu_t| }{2}$. Consider
\begin{equation*}
\pp_{t} \varphi_{2} = \Delta_{\bar{x}} \varphi_{2} + \mu^{-2} \mu_t Z_5(\frac{\bar{x}}{\mu}) \eta(\frac{\bar{x} }{\sqrt t}) ,
\end{equation*}
where $\varphi_{2}$ is given by $\varphi_{2} = \varphi_{2}[\mu]
=
\TT^{out}_4
\left[\mu^{-2} \mu_t Z_5(\frac{\bar{x}}{\mu}) \eta(\frac{\bar{x} }{\sqrt t}) \right]
$.
Then the following estimates hold
\begin{equation}\label{varphi2-s-up}
\begin{aligned}
	|\varphi_2[\mu] |
\lesssim \ &
t^{-2} e^{-\frac{|\bar{x}|^2}{16 t }}\int_{t_0/2}^{t/2} s |\mu_t(s)| ds
+
\begin{cases}
|\mu_t| (\ln (\mu^{-1} t^{\frac 12} ) + 1)
&
\mbox{ \ if \ } |\bar x|\le \mu
\\
|\mu_t|  (\ln(|\bar x|^{-1} t^{\frac 12}) + 1)
&
\mbox{ \ if \ }  \mu < |\bar x|\le t^{\frac 12}
\\
|\mu_t|  t  |\bar x|^{-2} e^{-\frac{| \bar x |^2}{16 t}}
&
\mbox{ \ if \ }   |\bar x| > t^{\frac 12}
\end{cases}
\\
\lesssim \ &
		\begin{cases}
			(t\ln t)^{-1}
			&
			\mbox{ \ if \ } |\bar x|\le \mu
			\\
			t^{-1} (\ln t)^{-2} (\ln(|\bar x|^{-1} t^{\frac 12}) + 1)
			&
			\mbox{ \ if \ }  \mu < |\bar x|\le t^{\frac 12}
			\\
			t^{-1} (\ln t)^{-2} e^{-\frac{| \bar x |^2}{16 t}}
			&
			\mbox{ \ if \ }   |\bar x| > t^{\frac 12}
		\end{cases} ,
\end{aligned}
\end{equation}
\begin{equation}\label{nab-varphi2-s-up}
	|\nabla_{\bar{x}}\varphi_2[\mu] |
	\lesssim
	\begin{cases}
		t^{-1}
		&
		\mbox{ \ if \ } |\bar x|\le \mu
		\\
		t^{-1} (\ln t)^{-2} (\ln(|\bar x|^{-1} t^{\frac 12}) + 1)
		|\bar{x}|^{-1}
		&
		\mbox{ \ if \ }  \mu < |\bar x|\le t^{\frac 12}
		\\
		t^{-\frac 32}	(\ln t)^{-2}
		 e^{-\frac{| \bar x |^2}{30 t}}
		&
		\mbox{ \ if \ }   |\bar x| > t^{\frac 12}
	\end{cases} ,
\end{equation}
\begin{equation}\label{varphi2-mu-mu1}
	\begin{aligned}
		\big|\varphi_2[\mu+\mu_1] - \varphi_2[\mu]\big|
		\lesssim \ &
		t^{-2} e^{-\frac{|\bar{x}|^2}{16 t }}\int_{t_0/2}^{t/2} s |\mu_t(s)| \Big(\frac{|\mu_1(s)|}{ \mu(s) } + \frac{|\mu_{1t}(s)|}{|\mu_t(s)|}  \Big) ds
		\\
		&
		+
		|\mu_t| \sup\limits_{t_1\in[t/2,t]} \Big(\frac{|\mu_1(t_1)|}{ \mu(t) } + \frac{|\mu_{1t}(t_1)|}{|\mu_t(t)|}  \Big)
		\begin{cases}
			\ln (\mu^{-1} t^{\frac 12} ) + 1
			&
			\mbox{ \ if \ } |\bar x|\le \mu
			\\
			\ln(|\bar x|^{-1} t^{\frac 12}) + 1
			&
			\mbox{ \ if \ }  \mu < |\bar x|\le t^{\frac 12}
			\\
			t  |\bar x|^{-2} e^{-\frac{| \bar x |^2}{16 t}}
			&
			\mbox{ \ if \ }   |\bar x| > t^{\frac 12}
		\end{cases}
		.
	\end{aligned}
\end{equation}
More precisely,
\begin{equation*}
	\begin{aligned}
		\varphi_2[\mu] = \ &
		\bigg[
		-2^{-\frac 12} \int_{t/2}^{t-\mu_{0}^2 } \frac{\mu_t(s)}{t-s} d s
		+ 	 O\Big( t^{-2}
		\int_{t_0/2}^{t}
		s |\mu_t(s)|
		d s
		+
		\min\{ \mu^{-1} |\mu_t | |\bar x| ,  \ln t |\mu_t | \}
		\Big)
		\bigg]
		\1_{\{ |\bar{x}| \le 2 t^{\frac 12} \}}
		\\
		&
		+
		O \Big( t^{-2} \int_{t_0/2}^{t} s |\mu_t(s)| ds
		\Big) e^{-\frac{|\bar{x}|^2}{16 t }}
		\1_{\{ |\bar{x}| > 2 t^{\frac 12} \}} ,
	\end{aligned}
\end{equation*}
\begin{equation*}
	\begin{aligned}
		&
		\varphi_2[\mu + \mu_1 ] - 	\varphi_2[\mu]
		\\
		= \ &
		\Bigg[
		-2^{-\frac 12} \int_{t/2}^{t-\mu_{0}^2 } \frac{\mu_{1t} (s)}{t-s} d s
		+ O\bigg( |\mu_t| \sup\limits_{t_{1}\in[t/2,t]}\Big(\frac{|\mu_1(t_1)|}{ \mu(t) } + \frac{|\mu_{1t}(t_1)|}{|\mu_t(t)|}  \Big)+  t^{-2}
		\int_{t_0/2}^{t/2}
		s |\mu_t(s)| \Big(\frac{|\mu_1(s)|}{ \mu(s) } + \frac{|\mu_{1t}(s) | }{|\mu_t(s) | }\Big)
		d s
		\\
		&
		+
		|\mu_{t} | \ln t \sup\limits_{t_1\in[t/2,t]} \Big(\frac{|\mu_1(t_1)|}{ \mu(t) } + \frac{|\mu_{1t}(t_1)|}{|\mu_t(t)|}\Big)^2
		+
		\mu^{-1} |\mu_t|  \sup\limits_{t_1\in[t/2,t]}\Big(\frac{|\mu_1(t_1)|}{\mu(t)} + \frac{|\mu_{1t}(t_1)|}{|\mu_t(t)|}\Big) |\bar x|
		\bigg) \Bigg] \1_{\{ |\bar{x}| \le 2 t^{\frac 12} \}}
		\\
		& +  	O\bigg( |\mu_t| \sup\limits_{t_1\in [t/2,t]}  \Big(\frac{|\mu_1(t_1)|}{ \mu(t) } + \frac{|\mu_{1t}(t_1)|}{|\mu_t(t)|}  \Big) +  t^{-2} \int_{t_0/2}^{t/2} s |\mu_t(s)| \Big(\frac{|\mu_1(s)|}{\mu(s)} + \frac{|\mu_{1t}(s)|}{|\mu_t(s)|}  \Big) ds\bigg) e^{-\frac{| \bar x |^2}{16 t}}
		\1_{\{ |\bar{x}| > 2 t^{\frac 12} \}}
		.
	\end{aligned}
\end{equation*}

\end{lemma}
\begin{proof}
Since
\begin{equation}\label{ZA-3-2}
| \mu^{-2} \mu_t Z_5(\frac{\bar{x}}{\mu}) \eta(\frac{\bar{x} }{\sqrt t}) | \lesssim  \mu^{-2} \mu_t \1_{\{ |\bar{x}| \le \mu \}} +  \mu_t |\bar x|^{-2} \1_{\{  \mu < |\bar{x}| \le 2 t^{\frac 12} \}},
\end{equation}
 by Lemma \ref{annular-forward} and \eqref{ansatz-mu}, we conclude the validity of  \eqref{varphi2-s-up}.
By scaling argument, \eqref{nab-varphi2-s-up} follows.

For $\mu_1$ satisfying $|\mu_1| \le \frac{\mu}{2}$ and $|\mu_{1t}| \le \frac{|\mu_t| }{2}$, we have
\begin{equation}\label{Z5-mu-mu1}
\begin{aligned}
&
(\mu +\mu_1)^{-2} (\mu_t  +\mu_{1t}  )Z_5(\frac{\bar{x}}{ \mu +\mu_1 })  -
\mu^{-2} \mu_t Z_5(\frac{\bar{x}}{\mu})
\\
= \ &
\mu^{-2}   \mu_{1t}   Z_5(\frac{\bar{x}}{\mu  } )  -
\mu^{-3 } \mu_1 \mu_t
\Big( 2Z_5(\frac{\bar {x}}{\mu  }) +
\frac{\bar{x}}{\mu  } \cdot  \nabla Z_5 (\frac{\bar {x}}{\mu  } )  \Big)
 +
\big(  \mu^{-3}  \mu_1 \mu_t  + \mu^{-2}   \mu_{1t} \big)  \langle \frac{\bar{x}}{\mu} \rangle^{-2} O\Big(\frac{|\mu_1| }{\mu} + \frac{|\mu_{1t}|}{|\mu_t|}\Big)
\\
= \ &
O\bigg( \mu^{-2} |\mu_t| \Big(\frac{|\mu_1|}{\mu} + \frac{|\mu_{1t}|}{|\mu_t|}\Big) \langle \frac{\bar x}{\mu} \rangle^{-2} \bigg) .
\end{aligned}
\end{equation}
Then by Lemma \ref{annular-forward}, one gets \eqref{varphi2-mu-mu1}.

In order to extract the dominating part of $\varphi_2$ for the preparation of solving the orthogonal equation, we split $\varphi_{2}$ into several parts to estimate. Set $\mu_{0}(t) = (\ln t)^{-1}$ and consider
\begin{equation*}
\begin{aligned}
	\varphi_{2}=&~  \left(
	\int_{t_0}^{t/2}
	+
	\int_{t/2}^{t-\mu_{0}^2(t)}
	+
	\int_{t-\mu_{0}^2(t)}^{t}
	\right)
	\int_{\RR^4} (4\pi (t-s))^{-2} e^{-\frac{|x-z|^2}{4(t-s)}}
	\mu^{-2}(s) \mu_t(s) Z_5( \frac{|z|}{\mu(s) } ) \eta(\frac{|z|}{\sqrt s})   d z d s \\
:= &~
I_{1} + I_{2} + I_{3} .
\end{aligned}
\end{equation*}
For $I_{1}$, by rearrangement inequality, we have
\begin{equation*}
	\begin{aligned}
& |I_{1}| \lesssim
\int_{t_0/2}^{t/2}
\int_{\RR^4}  (t-s)^{-2} e^{-\frac{|x-z|^2}{4(t-s)}}
\mu^{-2}(s) |\mu_t(s)|
\langle \frac{|z|}{\mu(s) } \rangle^{-2} \eta(\frac{|z|}{\sqrt s})   d z d s
		\\
 \lesssim \ &
		t^{-2}
		\int_{t_0/2}^{t/2}
		\int_{\RR^4}  e^{-\frac{|z|^2}{4 t}}
		\mu^{-2}(s) |\mu_t(s)|
		\langle \frac{|z|}{\mu(s) } \rangle^{-2} \eta(\frac{|z|}{\sqrt s})   d z d s
\lesssim
t^{-2}
\int_{t_0/2}^{t/2}
\mu^{-2}(s) |\mu_t(s)|
\int_{0}^{2\sqrt s}
\langle \frac{ r }{\mu(s) } \rangle^{-2} r^3 d r d s
\\
\lesssim \ &
t^{-2}
\int_{t_0/2}^{t/2}
s |\mu_t(s)|
 d s
	\end{aligned}
\end{equation*}
since $
\int_{0}^{\mu(s)}
\langle \frac{ r }{\mu(s) } \rangle^{-2} r^3 d r
\sim  \mu^4(s),
\quad
\int_{\mu(s)}^{2\sqrt s}
\langle \frac{ r }{\mu(s) } \rangle^{-2} r^3 d r
\lesssim   s \mu^2(s) $.

Using \eqref{Z5-mu-mu1} and similar calculations above, one has
\begin{equation*}
\big|I_1[\mu+\mu_1] - I_1[\mu] \big|
\lesssim
t^{-2}
\int_{t_0/2}^{t/2}
s |\mu_t(s)| \Big(\frac{|\mu_1(s)|}{ \mu(s) } + \frac{|\mu_{1t}(s) | }{|\mu_t(s) | }\Big)
d s.
\end{equation*}

For $I_{3}$, we have
\begin{equation*}
\begin{aligned}
|I_{3}|
\lesssim \ &
\int_{t - \mu_{0}^2(t) }^{t}
\int_{\RR^4}  (t-s)^{-2} e^{-\frac{|z|^2}{4(t-s)}}
\mu^{-2}(s) |\mu_t(s)|
\langle \frac{|z|}{\mu(s) } \rangle^{-2} \eta(\frac{|z|}{\sqrt s})   d z d s
\\
\lesssim  \ &
\mu^{-2}  |\mu_t|
\int_{t -\mu_{0}^2(t) }^{t}
(t-s)^{-2}
\int_{0}^{2\sqrt t}   e^{-\frac{r^2}{4(t-s)}}
\langle \frac{r}{\mu(t) } \rangle ^{-2}  r^3 d r d s
\lesssim
|\mu_t|
\end{aligned}
\end{equation*}
since for $s\in (t-\mu_{0}^2(t),t)$,
\begin{equation*}
\begin{aligned}
	&
\int_{0}^{\mu(t)}   e^{-\frac{r^2}{4(t-s)}}
\langle \frac{r}{\mu(t) } \rangle ^{-2}  r^3 d r
\sim
\int_{0}^{\mu(t)}   e^{-\frac{r^2}{4(t-s)}}
  r^3 d r
\sim
(t-s)^2 \int_{0}^{\frac{\mu^2(t)}{4(t-s)}} e^{-z} z d z \sim (t-s)^2 ,
\\
	&
\int_{\mu(t)}^{2\sqrt t}   e^{-\frac{r^2}{4(t-s)}}
\langle \frac{r}{\mu(t) } \rangle ^{-2}  r^3 d r
\sim
(t-s) \mu^2(t)
\int_{\frac{\mu^2(t)}{4(t-s)}}^{\frac{t}{t-s}}
e^{-z} d z
\lesssim
(t-s) \mu^2(t) e^{-\frac{\mu^2(t)}{4(t-s)}}
\lesssim
(t-s)^2 .
\end{aligned}
\end{equation*}
Similarly, using \eqref{Z5-mu-mu1}, one has
\begin{equation*}
\big|I_3[\mu+\mu_1] - I_3[\mu ]  \big|
\lesssim
|\mu_t| \sup\limits_{t_{1}\in[t/2,t]}\Big(\frac{|\mu_1(t_1)|}{\mu(t)} + \frac{|\mu_{1t}(t_1)|}{|\mu_t(t)|}  \Big) .
\end{equation*}

For $I_{2}$, more delicate calculations are needed to single out the leading term. Set
\begin{equation*}
I_{2} = I_{02} +  (I_{2} - I_{02}),
\end{equation*}
where
\begin{equation*}
		I_{02} =
		\int_{t/2}^{t-\mu_{0}^2(t)}
		\int_{\RR^4} (4\pi (t-s))^{-2} e^{-\frac{|z|^2}{4(t-s)}}
		\mu^{-2}(s) \mu_t(s) Z_5( \frac{|z|}{\mu(s) } ) \eta(\frac{|z|}{\sqrt s})   d z d s
:=
I_{*} + I_{021} +  I_{022}
\end{equation*}
and
\begin{equation*}
\begin{aligned}
I_{*} = \ & -2^{\frac 32} \int_{ t/2}^{t-\mu_{0}^2(t)}
\int_{\RR^4} (4\pi (t-s))^{-2} e^{-\frac{|z|^2}{4(t-s)}}
\mu^{-2}(s) \mu_t(s)  \frac{\mu^2(s)}{|z|^2}  d z d s ,
\\
I_{021} = \ &
2^{\frac 32} \int_{t/ 2}^{t-\mu_{0}^2(t)}
\int_{\RR^4} (4\pi (t-s))^{-2} e^{-\frac{|z|^2}{4(t-s)}}
\mu^{-2}(s) \mu_t(s)  \frac{\mu^2(s)}{|z|^2} \Big(1- \eta(\frac{|z|}{\sqrt s}) \Big)  d z d s,
\\
I_{022} = \ &
\int_{t/2}^{t-\mu_{0}^2(t) }
\int_{\RR^4} (4\pi (t-s))^{-2} e^{-\frac{|z|^2}{4(t-s)}}
\mu^{-2}(s) \mu_t(s)
\Big(
 Z_5( \frac{|z|}{\mu(s) } )
 + 2^{\frac 32} \frac{\mu^2(s)}{|z|^2}
\Big) \eta(\frac{|z|}{\sqrt s})   d z d s .
\end{aligned}
\end{equation*}
For $I_*$, we evaluate
\begin{equation*}
\begin{aligned}
I_{*} =
-2^{\frac 32}|S^3| \int_{t/2}^{t-\mu_{0}^2(t) }
\mu_t(s)
\int_{0}^{\infty} (4\pi (t-s))^{-2} e^{-\frac{ r^2}{4(t-s)}}
   r  d r d s
=
-2^{-\frac 12} \int_{t/2}^{t-\mu_{0}^2(t) } \frac{\mu_t(s)}{t-s} d s .
\end{aligned}
\end{equation*}
In the same way, one has
\begin{equation*}
I_*[\mu+\mu_1] - I_*[\mu ]
=
-2^{-\frac 12} \int_{t/2}^{t-\mu_{0}^2(t) } \frac{\mu_{1t}(s)}{t-s} d s .
\end{equation*}
For $I_{021} $, we get
\begin{equation*}
|I_{021} |
\lesssim
|\mu_t | \int_{\frac t2}^{t-\mu_{0}^2(t) }
\int_{\frac{\sqrt t}{2}}^{\infty}  (t-s)^{-2} e^{-\frac{r^2}{4(t-s)}}
  r    d r d s
\lesssim
|\mu_t | ,
\quad
	\big|I_{021}[\mu+\mu_1] - I_{021}[\mu ] \big|
	\lesssim
\sup\limits_{t_{1}\in[t/2,t]}	|\mu_{1t}(t_1) | .
\end{equation*}
For $I_{022} $, we have
	\begin{align*}
&	|I_{022} | \lesssim
		\int_{t/2}^{t-\mu_0^2}
		\int_{\RR^4} (t-s)^{-2} e^{-\frac{|z|^2}{4(t-s)}}
		\mu^{-2}(s) |\mu_t(s) |
 (\frac{|z| }{\mu(s) } )^{-2}
		\langle \frac{|z|}{\mu(s) } \rangle^{-2} \eta(\frac{|z|}{\sqrt s})   d z d s
		\\
\lesssim \ &
|\mu_t|
\int_{t/2}^{t-\mu_0^2}
\int_{\RR^4}  (t-s)^{-2} e^{-\frac{|z|^2}{4(t-s)}}
\left| z \right|^{-2}
\langle \frac{|z|}{\mu(t) } \rangle^{-2} \eta(\frac{|z|}{2 \sqrt t})   d z d s
\\
\lesssim \ &
|\mu_t |
\int_{t/2}^{t-\mu_0^2}
(t-s)^{-2}
\int_{0}^{4\sqrt t}   e^{-\frac{ r^2}{4(t-s)}}
\langle \frac{ r }{\mu(t) } \rangle^{-2} r  d r d s
\lesssim
|\mu_t | \mu^{2}
\int_{t/2}^{t-\mu_0^2}
 (t-s)^{-2} \Big(1- \ln (\frac{\mu^2(t)}{4(t-s)})  \Big) d s
\lesssim
|\mu_t |
	\end{align*}
since for $\frac t2 \le s \le t-\mu^2(t)$,
\begin{equation*}
\int_{0}^{\mu(t)}   e^{-\frac{ r^2}{4(t-s)}}
\langle \frac{ r }{\mu(t) } \rangle^{-2} r  d r
\lesssim \mu^{2}(t) ,
\end{equation*}
\begin{equation*}
	\begin{aligned}
		&
		\int_{\mu(t)}^{4\sqrt t}   e^{-\frac{ r^2}{4(t-s)}}
		\langle \frac{ r }{\mu(t) } \rangle^{-2} r  d r
\sim
\mu^{2}(t)
\int_{\mu(t)}^{4\sqrt t}   e^{-\frac{ r^2}{4(t-s)}}
 r^{-1}  d r
\sim
\mu^{2} (t)
\int_{\frac{\mu^2(t) }{4(t-s)}}^{\frac{4t}{t-s}} e^{-z} z^{-1} d z
\lesssim
\mu^{2}(t)  \Big(1- \ln (\frac{\mu^2(t)}{4(t-s)})  \Big) .
	\end{aligned}
\end{equation*}

Next, we estimate $I_{022}[\mu+\mu_1] - I_{022}[\mu]$. By \eqref{Z5-mu-mu1}, we have
\begin{equation*}
	\begin{aligned}
		&
		(\mu+\mu_1)^{-2} (\mu_t +\mu_{1t} )
		\left(
		Z_5( \frac{|z|}{\mu +\mu_1 } )
		+ 2^{\frac 32} \frac{(\mu+\mu_1)^2}{|z|^2}
		\right) -
		\mu^{-2} \mu_t
		\left(
		Z_5( \frac{|z|}{\mu } )
		+ 2^{\frac 32} \frac{\mu^2 }{|z|^2}
		\right)
		\\
		= \ &
		\mu^{-2}   \mu_{1t}   Z_5(\frac{z}{\mu  } )  + 2^{\frac{3}{2}} \frac{\mu_{1t}}{|z|^2} -
		\mu^{-3 } \mu_1 \mu_t
		\Big( 2Z_5(\frac{z}{\mu  }) +
		\frac{z}{\mu  } \cdot  \nabla Z_5 (\frac{z}{\mu  } )  \Big)
	 +
		\big(  \mu^{-3}  \mu_1 \mu_t  + \mu^{-2}   \mu_{1t} \big)  \langle \frac{z}{\mu} \rangle^{-2} O\Big(\frac{|\mu_1|}{\mu} + \frac{|\mu_{1t}|}{|\mu_t|}\Big)
		\\
		= \ &
		O\Big(\mu^{-2} |\mu_{1t} | \frac{|z|^{-2}}{\mu^{-2}} \langle \frac{z}{\mu} \rangle^{-2} \Big) -
		\mu^{-3 } \mu_1 \mu_t
		\Big( 2Z_5(\frac{z}{\mu  }) +
		\frac{z}{\mu  } \cdot  \nabla Z_5 (\frac{z}{\mu  } )  \Big)
		+
		\big(  \mu^{-3}  \mu_1 \mu_t  + \mu^{-2}   \mu_{1t} \big)  \langle \frac{z}{\mu} \rangle^{-2} O\Big(\frac{|\mu_1|}{\mu} + \frac{|\mu_{1t}|}{|\mu_t|}\Big)
		\\
		= \ &
		O\bigg(\mu^{-2} |\mu_{t} | \Big(\frac{|\mu_1|}{\mu} + \frac{|\mu_{1t}|}{|\mu_t|}\Big)  \frac{|z|^{-2}}{\mu^{-2}} \langle \frac{ z }{\mu} \rangle^{-2} \bigg)
		+
		O\bigg(\mu^{-2} |\mu_{t} | \Big(\frac{|\mu_1|}{\mu} + \frac{|\mu_{1t}|}{|\mu_t|}\Big)^2   \langle \frac{ z }{\mu} \rangle^{-2} \bigg)
		.
	\end{aligned}
\end{equation*}
Similar to the estimates of $I_{022}$, we then have
\begin{equation*}
	\begin{aligned}
		&
		\int_{t/2}^{t-\mu_0^2}
		\int_{\RR^4} (4\pi (t-s))^{-2} e^{-\frac{|z|^2}{4(t-s)}}
		\mu^{-2}(s) |\mu_{t}(s) | \Big(\frac{|\mu_1(s)|}{\mu(s)} + \frac{|\mu_{1t}(s)|}{|\mu_t(s)|}\Big)  \frac{|z|^{-2}}{\mu^{-2}(s)} \langle \frac{ z }{\mu(s) } \rangle^{-2}  \eta(\frac{|z|}{\sqrt s})   d z d s
		\\
		\lesssim \ &
		|\mu_{t} | \sup\limits_{t_1\in[t/2,t]}\Big(\frac{|\mu_1(t_1)|}{\mu(t) } + \frac{|\mu_{1t} (t_1)|}{|\mu_t(t) |}\Big) ,
	\end{aligned}
\end{equation*}

\begin{equation*}
	\begin{aligned}
		&
		\int_{t/2}^{t-\mu_0^2}
		\int_{\RR^4} (4\pi (t-s))^{-2} e^{-\frac{|z|^2}{4(t-s)}}
		\mu^{-2}(s) |\mu_{t}(s) | \Big(\frac{|\mu_1(s)|}{\mu(s)} + \frac{|\mu_{1t}(s)|}{|\mu_t(s)|}\Big)^2   \langle \frac{ z }{\mu(s)} \rangle^{-2}  \eta(\frac{|z|}{\sqrt s})   d z d s
		\\
		\lesssim \ &
		\mu^{-2}  |\mu_{t}  | \Big(\frac{|\mu_1 |}{\mu } + \frac{|\mu_{1t} |}{|\mu_t |}\Big)^2
		\int_{t/2}^{t-\mu_0^2}
		(t-s)^{-2}
		\int_{0}^{2t^{\frac 12}}
		e^{-\frac{ r^2}{4(t-s)}}
		\langle \frac{ r }{\mu(t)} \rangle^{-2}  r^3  d r d s
		\\
		\lesssim \ &
		|\mu_{t} | \ln t
		\sup\limits_{t_1\in[t/2,t]} \Big(\frac{|\mu_1(t_1)|}{\mu(t)} + \frac{|\mu_{1t}(t_1)|}{|\mu_t(t)|}\Big)^2
	\end{aligned}
\end{equation*}
since
\begin{equation}\label{z5}
	\int_{0}^{ \mu(t) }
	e^{-\frac{ r^2}{4(t-s)}}
	\langle \frac{ r }{\mu(t)} \rangle^{-2}  r^3  d r
	\lesssim \mu^4 ,\quad
	\int_{\mu(t)}^{2t^{\frac 12}}
	e^{-\frac{ r^2}{4(t-s)}}
	\langle \frac{ r }{\mu(t)} \rangle^{-2}  r^3  d r
	\sim
	\mu^2
	\int_{\mu(t)}^{2t^{\frac 12}}
	e^{-\frac{ r^2}{4(t-s)}}
	r  d r
	\lesssim
	\mu^2 (t-s) .
\end{equation}
Therefore, one has
\begin{equation*}
	\big|I_{022}[\mu+\mu_1] - I_{022}[\mu]\big| \lesssim
	|\mu_{t} | \sup\limits_{t_1\in[t/2,t]}\Big(\frac{|\mu_1(t_1)|}{\mu(t) } + \frac{|\mu_{1t}(t_1) |}{|\mu_t(t) |}\Big)
	+
	|\mu_{t} | \ln t \sup\limits_{t_1\in[t/2,t]} \Big(\frac{|\mu_1(t_1)|}{\mu(t) } + \frac{|\mu_{1t}(t_1)|}{|\mu_t(t)|}\Big)^2  .
\end{equation*}

Let us now estimate $I_{2} -I_{02}$
\begin{equation*}
	\begin{aligned}
		&
		|I_{2} -I_{02} |
		=
		\bigg| \int_{t/2}^{t-\mu_{0}^2(t)}
		\int_{\RR^4}
		\int_{0}^{1} (4\pi (t-s))^{-2} e^{-\frac{|\theta x-z|^2}{4(t-s)}} \frac{\theta x -z}{2(t-s)} \cdot x
		\mu^{-2}(s) \mu_t(s) Z_5( \frac{|z|}{\mu(s) } ) \eta(\frac{|z|}{\sqrt s})   d \theta d z d s \bigg|
		\\
		\lesssim \ &
		\mu^{-2} |\mu_t| |x|
		\int_{t/2}^{t-\mu_{0}^2(t)}
		\int_{\RR^4}
		\int_{0}^{1}  (t-s)^{-\frac 52} e^{-\frac{|\theta x-z|^2}{8(t-s)}}
		\langle \frac{|z|}{\mu(s) } \rangle^{-2} \eta(\frac{|z|}{\sqrt s})   d \theta d z d s
		\\
		\lesssim \ &
		\mu^{-2}  |\mu_t| |x|
		\int_{t/2}^{t-\mu_{0}^2(t)}
		\int_{\RR^4}
		(t-s)^{-\frac 52} e^{-\frac{|z|^2}{8(t-s)}}
		\langle \frac{|z|}{\mu(s) } \rangle^{-2} \eta(\frac{|z|}{\sqrt s})   d z d s
		\\
		\lesssim \ &
		\mu^{-2}  |\mu_t| |x|
		\int_{t/2}^{t-\mu_{0}^2(t)}
		(t-s)^{-\frac 52}
		\int_{0}^{2
			\sqrt t}
		 e^{-\frac{ r^2}{8(t-s)}}
		\langle \frac{r }{\mu(t) } \rangle^{-2} r^{3} d r d s
		\lesssim
		\mu^{-1} |\mu_t| |\bar x|
	\end{aligned}
\end{equation*}
since $\int_{0}^{2
	\sqrt t}
e^{-\frac{ r^2}{8(t-s)}}
\langle \frac{r }{\mu(t) } \rangle^{-2} r^{3} d r \lesssim \mu^2(t-s)$ by similar estimate in \eqref{z5}.

Using rearrangement inequality, one has another upper bound for $|I_{2} -I_{02} |$,
\begin{equation*}
	\begin{aligned}
	&	|I_{2} -I_{02} | \lesssim
		\int_{t/2}^{t-\mu_{0}^2(t)}
		\int_{\RR^4} (t-s)^{-2} e^{-\frac{|z|^2}{4(t-s)}}
		\mu^{-2}(s) |\mu_t(s) |
		\langle \frac{|z|}{\mu(s) } \rangle^{-2} \eta(\frac{|z|}{\sqrt s})   d z d s
		\\
		\lesssim \ &
		\mu^{-2} |\mu_t |
		\int_{t/2}^{t-\mu_{0}^2(t)}
		(t-s)^{-2}
		\int_{0}^{2t^{\frac 12 } }  e^{-\frac{ r^2}{4(t-s)}}
		\langle \frac{ r }{\mu(t) } \rangle^{-2} r^3  d r d s
		\lesssim
		|\mu_t(t) |
		\int_{t/2}^{t-\mu_{0}^2(t)}
		(t-s)^{-1} d s \lesssim  \ln t |\mu_t |
		.
	\end{aligned}
\end{equation*}
Thus
\begin{equation*}
	|I_{2} -I_{02} | \lesssim
	\min\big\{ \mu^{-1} |\mu_t | |\bar x| ,  \ln t |\mu_t | \big\} .
\end{equation*}
Using \eqref{Z5-mu-mu1} and similar calculations, one has
\begin{equation*}
	\big| (I_{2} -I_{02})[\mu+\mu_1] - (I_{2} -I_{02})[\mu]\big| \lesssim
	\mu^{-1} |\mu_t|  \sup\limits_{t_1\in[t/2,t]}\Big(\frac{|\mu_1(t_1)|}{\mu(t)} + \frac{|\mu_{1t}(t_1)|}{|\mu_t(t)|}\Big) |\bar x|  .
\end{equation*}
Combining all the estimates above, we conclude the validity of Lemma \ref{varphi2-lem}.

\end{proof}

Recalling $\varphi[\mu] = \varphi_1[\mu] + \varphi_2[\mu]$ and combining Lemma \ref{varphi1-lem} and Lemma \ref{varphi2-lem}, one has
\begin{corollary}\label{varphi-coro}
	Assume that $\mu$ satisfies \eqref{ansatz-mu} and $\mu_1$ satisfies $|\mu_{1}|\le \frac{\mu}{2}$, $|\mu_{1t}| \le \frac{|\mu_t| }{2}$. We have
\begin{equation*}
	\begin{aligned}
		|\varphi[\mu] |
		\lesssim \ &
		(\mu  t^{-1}
		+ 	g[\mu]  ) \1_{\{ |\bar{x}|\le 2t^{\frac 12}  \}}
		+
		\begin{cases}
			|\mu_t| (\ln (\mu^{-1} t^{\frac 12} ) + 1)
			&
			\mbox{ \ if \ } |\bar x|\le \mu
			\\
			|\mu_t| (\ln(|\bar x|^{-1} t^{\frac 12}) + 1)
			&
			\mbox{ \ if \ }  \mu < |\bar x|\le t^{\frac 12}
			\\
			t |\mu_t| |\bar x|^{-2} e^{-\frac{| \bar x |^2}{16 t}}
			&
			\mbox{ \ if \ }   |\bar x| > t^{\frac 12}
		\end{cases}
		\\
		& +
		O\Big( \mu |\bar{x}|^{-2} e^{-\frac{| \bar  x|^2}{16 t}}
		+
		|\bar{x}|^{-6} \int_{t_0/2}^t s^2  |\mu_t(s)| ds
		+
		g[\mu] e^{-\frac{|\bar{x}|^2}{16 t} }
		\Big) \1_{\{ |\bar{x}| > 2t^{\frac 12}  \}}
		\\
		\lesssim \ &
		(t\ln t)^{-1} \1_{\{ |\bar{x}| \le 2t^{\frac 12} \}}
		+
		O\big( t^{2} (\ln t)^{-1}
		|\bar{x}|^{-6}
		\big) \1_{\{ |\bar{x}| > 2t^{\frac 12}  \}}
	\end{aligned}
\end{equation*}
where
\begin{equation*}
	g[\mu] = O \Big( 	
	t^{-2} \int_{t_0/2}^{t} ( s^{-1 } \mu^{3}(s)  + 	s |\mu_t(s)| )  ds
	\Big) .
\end{equation*}
\begin{equation*}
	|\nabla_{\bar{x}}\varphi[\mu] |
	\lesssim
	\begin{cases}
		t^{-1}
		&
		\mbox{ \ if \ } |\bar x|\le \mu
		\\
		t^{-1} (\ln t)^{-2} (\ln(|\bar x|^{-1} t^{\frac 12}) + 1)
		|\bar{x}|^{-1}
		+
		t^{-\frac 32} (\ln t)^{-1}
		&
		\mbox{ \ if \ }  \mu < |\bar x|\le t^{\frac 12}
		\\
		t^{\frac 32} (\ln t)^{-1} |\bar{x}|^{-6}
		&
		\mbox{ \ if \ }   |\bar x| > t^{\frac 12}
	\end{cases}
	.
\end{equation*}

\begin{equation*}
	\begin{aligned}
		&
		\big|\varphi [\mu+\mu_1] - \varphi [\mu]\big|
		\lesssim
		\left( O ( |\mu_1|  t^{-1} )
		+
		\tilde{g}[\mu,\mu_1] \right) \1_{\{ |\bar{x}|  \le 2t^{\frac 12} \}}
		\\
		&
		+
		\sup\limits_{t_1\in[t/2,t]} \Big(\frac{|\mu_1(t_1)|}{ \mu(t) } + \frac{|\mu_{1t}(t_1)|}{|\mu_t(t)|}  \Big)
		\begin{cases}
			|\mu_t| (\ln (\mu^{-1} t^{\frac 12} ) + 1)
			&
			\mbox{ \ if \ } |\bar x|\le \mu
			\\
			|\mu_t| (\ln(|\bar x|^{-1} t^{\frac 12}) + 1)
			&
			\mbox{ \ if \ }  \mu < |\bar x|\le t^{\frac 12}
			\\
			t |\mu_t| |\bar x|^{-2} e^{-\frac{| \bar x |^2}{16 t}}
			&
			\mbox{ \ if \ }   |\bar x| > t^{\frac 12}
		\end{cases}
		\\
		&
		+
		O \bigg(  \sup\limits_{t_1\in[t/2,t]}|\mu_1(t_1)| |\bar{x}|^{-2} e^{-\frac{| \bar  x|^2}{16 t}}
		+
		|\bar{x}|^{-6} \Big( t^3 \sup\limits_{t_{1}\in [t/2,t] }	|\mu_{1t}(t_1)| +  \int_{t_0/2}^t s^2 |\mu_{1t}(s)|  ds \Big)
		+ \tilde{g}[\mu,\mu_1] e^{-\frac{|\bar{x}|^2}{16 t} } \bigg) \1_{\{ |\bar{x}|  > 2t^{\frac 12} \}}
	\end{aligned}
\end{equation*}
where
\begin{equation*}
	\begin{aligned}
		&
		\tilde{g}[\mu,\mu_1] =
		O\Big(|\mu_{t} | \ln t
		\sup\limits_{t_{1}\in[t/2,t]} \Big(\frac{|\mu_1(t_1)|}{\mu(t)}
		+ \frac{|\mu_{1t}(t_1)|}{|\mu_t(t)|}\Big)^2\Big)
		\\
		&
		+ 	O \bigg(	  |\mu_t| \sup\limits_{t_{1}\in[t/2,t]}\Big(\frac{|\mu_1(t_1)|}{ \mu(t) } + \frac{|\mu_{1t}(t_1)|}{|\mu_t(t)|}  \Big)
		+	 t^{-2} \int_{t_0/2}^{t}
		\Big( s^{-1} |\mu_1(s)| \mu^{2}(s) + s |\mu_t(s)| \Big(\frac{|\mu_1(s)|}{ \mu(s) } + \frac{|\mu_{1t}(s)|}{|\mu_t(s)|}  \Big) \Big)  ds \bigg) .
	\end{aligned}
\end{equation*}
More precisely,
\begin{equation*}
	\begin{aligned}
		\varphi[\mu] = \ & \bigg[ -2^{-\frac 12} \Big( \mu  t^{-1}
		+ \int_{t/2}^{t-\mu_{0}^2 } \frac{\mu_t(s)}{t-s} d s
		\Big)
		+ O\big( \mu t^{-2} |\bar{x}|^2  +
		|\mu_t|
		\min\{\frac{ |\bar x|}{\mu}  ,\ln t \} \big)
		+ g[\mu]
		\bigg]
		\1_{\{ |\bar{x}|\le 2t^{\frac 12}  \}}
		\\
		& +
		O\Big( \mu  |\bar{x}|^{-2} e^{-\frac{| \bar  x|^2}{16 t}}
		+
		|\bar{x}|^{-6} \int_{t_0/2}^t s^2 |\mu_t(s)|  ds
		+
		g[\mu] e^{-\frac{|\bar{x}|^2}{16 t} }
		\Big)
		\1_{\{ |\bar{x}| > 2 t^{\frac 12} \}} ,
	\end{aligned}
\end{equation*}
\begin{equation*}
	\begin{aligned}
		&
		\varphi[\mu + \mu_1 ] - 	\varphi[\mu]
		=
		\bigg[
		-2^{-\frac 12} \Big( \mu_1  t^{-1}
		+
		\int_{t/2}^{t-\mu_{0}^2 } \frac{\mu_{1t} (s)}{t-s} d s
		\Big)
		\\
		&
		+ O\Big( |\mu_1| t^{-2} |\bar{x}|^2 +  |\mu_t| \sup\limits_{t_{1}\in[t/2,t]}\Big(\frac{|\mu_1(t_1)|}{ \mu(t)} + \frac{|\mu_{1t}(t_1)|}{|\mu_t(t)|}  \Big)  \frac{ |\bar x |}{\mu} \Big)
		+
		\tilde{g}[\mu,\mu_1]
		\bigg]
		\1_{\{ |\bar{x}| \le 2 t^{\frac 12} \}}
\\
		&
		+
		O \Big(   \sup\limits_{t_{1}\in [t/2,t] }	|\mu_1(t_1)|  |\bar{x}|^{-2} e^{-\frac{| \bar  x|^2}{16 t}}
		+
		|\bar{x}|^{-6}
		\Big(
		t^3 \sup\limits_{t_{1}\in [t/2,t] }	|\mu_{1t}(t_1)| +  \int_{t_0/2}^{t/2} s^2 |\mu_{1t}(s)|  ds \Big) +   e^{-\frac{| \bar x |^2}{16 t}} 	\tilde{g}[\mu,\mu_1]
		\Big)
		\1_{\{ |\bar{x}| > 2 t^{\frac 12} \}}
		.
	\end{aligned}
\end{equation*}

\end{corollary}

\medskip

In order to extract the leading term, we will use the precise version of $\varphi[\mu]$ and $\varphi[\mu+\mu_1] - \varphi[\mu]$ when calculating the orthogonal equation. In other cases, we are inclined to adopt the rougher upper bound.

With introduction of the correction term $\varphi$, the new error is given by
\begin{equation*}
	\begin{aligned}
		S[u_1 + \varphi[\mu]  ]
		= \ &
		3 u_1^{2}  \varphi[\mu] + 3u_1 \varphi^2[\mu] + \varphi^3[\mu] +  \xi_t\cdot \nabla_{\bar{x}} \varphi[\mu](x-\xi,t)
		\\
		&
		+ \mu^{-2} \xi_t \cdot \nabla w(\frac{x-\xi}{\mu})
		\eta(\frac{x-\xi}{\sqrt t})
		+ \mu^{-1} t^{-\frac 12}
		\xi_t 	\cdot \nabla \eta(\frac{x-\xi}{\sqrt t} ) w(\frac{x-\xi}{\mu})  .
	\end{aligned}
\end{equation*}

\medskip

\subsection{Further improvement by solving an elliptic equation}\label{SolElliptic-sec}

In order to find suitable parameters to design the topology for solving inner-outer gluing system and the orthogonal equation, we will use the corresponding linearized elliptic equation to cut off the error term so that the time decay rate  will be improved.

Set the correction term as
\begin{equation*}
\bar{\mu}_0^{-1} \Phi_0(\frac{x-\xi}{\bar{\mu}_0},t)
\end{equation*}
where $\bar{\mu}_0$ is the leading order of $\mu$ to be determined later. Formally speaking,  $\Phi_0$ will be chosen to satisfy the following equation
\begin{equation}\label{Phi0-eq}
\begin{aligned}
	\Delta_y \Phi_0 + 3w^2(y) \Phi_0
	\approx \ &
-	\mu^3 \Big(3 u_1^{2}  \varphi[\mu]  (\bar{x},t) + 3 u_1 \varphi^2[\mu](\bar{x},t) \Big)
\\
	= \ &
-	3\mu\Big(w^2(y) \eta^2(\frac{  \mu y }{\sqrt  t}) \varphi[\mu]( \mu y,t) + \mu w(y) \eta(\frac{ \mu y }{\sqrt  t}) \varphi^2[\mu]( \mu y ,t)  \Big)
	.
\end{aligned}
\end{equation}
Set
\begin{equation*}
	\begin{aligned}
	\mathcal{M}[\mu] := \ &
		\int_{\RR^4 }
		\Big(w^2(y) \eta^2(\frac{\mu y}{\sqrt  t}) \varphi[\mu](\mu y,t) + \mu w(y) \eta(\frac{\mu y}{\sqrt  t}) \varphi^2[\mu](\mu y,t)  \Big)   Z_{5}(y) d y
		\\
		= \ &
		\mu^{-4}
		\int_{\RR^4 }
		\Big(w^2(\frac{\bar{x}}{\mu}) Z_{5}(\frac{\bar{x}}{\mu}) \eta^2(\frac{ \bar{x} }{\sqrt  t}) \varphi[\mu]( \bar{x},t) + \mu w(\frac{\bar{x}}{\mu}) Z_{5}(\frac{\bar{x}}{\mu}) \eta(\frac{ \bar{x} }{\sqrt  t}) \varphi^2[\mu](\bar{x},t)  \Big)    d \bar{x} .
	\end{aligned}
\end{equation*}
In order to find $\Phi_0$ with fast spatial decay, we aim to find $\bar{\mu}_0$ as the leading order of $\mu$ such that $\mathcal{M}[\mu] \approx 0$. In other words, above orthogonality condition is satisfied at leading order for careful choice of $\bar \mu_0$, which will be adjusted and corrected several times in order to further improve the time decay, and we shall see that
$$
\bar\mu_0\sim (\ln t)^{-1}.
$$

The iteration of finding proper $\bar\mu_0$ consists of three steps:
\begin{itemize}
\item the first step is to single out the leading part in above orthogonal equation, and this results in the blow-up rate predicted in \cite{filaking12},
\item the second step is to add next-order correction of the scaling parameter,
\item the last step is to iterate the second step finitely many times such that the new error has sufficiently fast time decay.
\end{itemize}

We now start the iteration.

\noindent {\bf Step 1. Finding the leading part $\mu_0$.}

Using the precise expression of $\varphi[\mu]$ in Corollary \ref{varphi-coro}, one has
\begin{equation*}
	\begin{aligned}
		&  \int_{\RR^4 }
		w^2(y)  Z_{5}(y) \eta^2(\frac{\mu y}{\sqrt  t}) \varphi(\mu y,t)    d y
		\\
		= \ &
		-2^{-\frac 12}
		\Big(
		\mu  t^{-1}
		+ \int_{t/2}^{t-\mu_{0}^2 } \frac{\mu_t(s)}{t-s} d s
		\Big)
		\int_{\RR^4 }
		w^2(y)  Z_{5}(y) \eta^2(\frac{\mu y}{\sqrt  t}) d y 	
		+
		O\big( \mu^3 t^{-2} \ln(\mu^{-1} t^{\frac 12}) \big)
		+ O(  |\mu_t| )
		+
		g[\mu] ,
	\end{aligned}
\end{equation*}
and
\begin{equation*}
	\begin{aligned}
		\mu \int_{\RR^4}  w(y) \eta(\frac{\mu y}{\sqrt t}) \varphi^2(\mu y, t) Z_5(y) dy
		= \ &
		\mu \int_{\RR^4}  w(y) \eta(\frac{\mu y}{\sqrt t}) Z_5(y)
		O\big(  	\mu^2  t^{-2} +
		g^2[\mu]
		+
		|\mu_t|^2 (\ln (\mu^{-1} t^{\frac 12} ) )^2 \big)
		dy
		\\
		= \ &
		\mu \ln(\mu^{-1} t^{\frac 12})
		O\big(  	\mu^2  t^{-2}
		 +
		g^2[\mu]
		+
		|\mu_t|^2 (\ln (\mu^{-1} t^{\frac 12} ) )^2 \big)  .
	\end{aligned}
\end{equation*}
Therefore, we obtain
\begin{equation*}
	\begin{aligned}
	\mathcal{M}[\mu]
		= \ &
		-2^{-\frac 12}
		\int_{\RR^4 }
		w^2(y)  Z_{5}(y) \eta^2(\frac{\mu y}{\sqrt  t}) d y
		\bigg(	\mu  t^{-1}
		+ \int_{t/2}^{t-\mu_{0}^2 } \frac{\mu_t(s)}{t-s} d s
		\\
		&
		+ O(  |\mu_t| )
		+
		g[\mu]
		+ \mu \ln(\mu^{-1} t^{\frac 12})
		O\big(  	\mu^2  t^{-2}
		+
		g^2[\mu]
		+
		|\mu_t|^2 (\ln (\mu^{-1} t^{\frac 12} ) )^2  \big) \bigg)
	\end{aligned}
\end{equation*}
where $\int_{\RR^4 }
w^2(y)  Z_{5}(y) \eta^2(\frac{\mu y}{\sqrt  t}) d y <0$ when $t$ is large.
Balancing the following two leading terms
\begin{equation*}
	\mu  t^{-1}
	+ \int_{t/2}^{t-\mu_{0}^2 } \frac{\mu_t(s)}{t-s} d s
	\sim
	\mu  t^{-1}
	+ \mu_t(t) \int_{t/2}^{t-\mu_{0}^2 } \frac{1}{t-s} d s \sim
	\mu  t^{-1}
	+ \mu_t \ln t =0,
\end{equation*}
one gets $\mu_0 = (\ln t)^{-1}$ as the leading order of $\mu$.
Notice that
\begin{equation*}
	\begin{aligned}
		&
		\mu_0  t^{-1}
		+ \int_{t/2}^{t-\mu_{0}^2 } \frac{\mu_{0t}(s)}{t-s} d s
		=
		(t \ln t)^{-1}
		-
		t^{-1}
		\int_{1/2}^{1-t^{-1} (\ln t)^{-2}}
		\frac{(\ln t + \ln z)^{-2} }{z(1-z)} dz
		\\
		= \ &
		(t \ln t)^{-1}
		-
		t^{-1} (\ln t)^{-2} (1+ O( (\ln t) ^{-1}))
		\int_{1/2}^{1-t^{-1} (\ln t)^{-2}}
		\frac{1}{z(1-z)} dz
		\\
		= \ &
		(t \ln t)^{-1}
		-
		t^{-1} (\ln t)^{-2} (1+ O( (\ln t) ^{-1}))
		\ln(t(\ln t)^2 -1)
		=
		O (t^{-1} (\ln t)^{-2}  \ln\ln t),
	\end{aligned}
\end{equation*}
\begin{equation*}
	O(  |\mu_{0t}| )
	+
	g[\mu_0]
	+ \mu_0 \ln(\mu_0^{-1} t^{\frac 12})
	O\big(  	\mu_0^2  t^{-2}
	+
	|\mu_{0t}|^2 (\ln (\mu_0^{-1} t^{\frac 12} ) )^2 +
	g^2[\mu_0] \big)  = O(t^{-1} (\ln t)^{-2}),
\end{equation*}
 and thus
\begin{equation*}
\mathcal{M}[\mu_0] = O (t^{-1} (\ln t)^{-2}  \ln\ln t) .
\end{equation*}

\noindent {\bf Step 2. Finding the corrected term $\mu_{01}$.}

In order to improve the time decay of the error, we introduce the next order term $\mu_{01}$ and make the ansatz $|\mu_{01}|\ll  \mu_0 $, $|\mu_{01t}|\ll |\mu_{0t}|$.

Then
by Corollary \ref{varphi-coro}, we estimate
\begin{align*}
&
	\mathcal{M}[\mu_0+\mu_{01}]
	\\
= \ &
\int_{\RR^4 }
\bigg( (\mu_0 +\mu_{01} )^{-4} w^2(\frac{\bar{x}}{ \mu_0 +\mu_{01} }) Z_{5}(\frac{\bar{x}}{ \mu_0 +\mu_{01} }) \eta^2(\frac{ \bar{x} }{\sqrt  t}) \varphi[ \mu_0 +\mu_{01} ]( \bar{x},t)
\\
&
+
( \mu_0 +\mu_{01} )^{-3} w(\frac{\bar{x}}{ \mu_0 +\mu_{01} }) Z_{5}(\frac{\bar{x}}{ \mu_0 +\mu_{01} }) \eta(\frac{ \bar{x} }{\sqrt  t}) \varphi^2[ \mu_0 +\mu_{01} ](\bar{x},t)  \bigg)    d \bar{x}
\\
= \ &
\int_{\RR^4 }
\Bigg\{ \Big[ (\mu_0 +\mu_{01} )^{-4} w^2(\frac{\bar{x}}{ \mu_0 +\mu_{01} }) Z_{5}(\frac{\bar{x}}{ \mu_0 +\mu_{01} })
-
\mu_0^{-4} w^2(\frac{\bar{x}}{ \mu_0  }) Z_{5}(\frac{\bar{x}}{ \mu_0 })
\Big]
 \eta^2(\frac{ \bar{x} }{\sqrt  t}) \varphi[ \mu_0 +\mu_{01} ]( \bar{x},t)
\\
&
+
\mu_0^{-4} w^2(\frac{\bar{x}}{ \mu_0 }) Z_{5}(\frac{\bar{x}}{ \mu_0 }) \eta^2(\frac{ \bar{x} }{\sqrt  t}) ( \varphi[ \mu_0 +\mu_{01} ] - \varphi[ \mu_0  ] )( \bar{x},t)
+
\mu_0^{-4} w^2(\frac{\bar{x}}{ \mu_0 }) Z_{5}(\frac{\bar{x}}{ \mu_0 }) \eta^2(\frac{ \bar{x} }{\sqrt  t})  \varphi[ \mu_0  ]( \bar{x},t)
\\
&
+
\Big[ ( \mu_0 +\mu_{01} )^{-3} w(\frac{\bar{x}}{ \mu_0 +\mu_{01} }) Z_{5}(\frac{\bar{x}}{ \mu_0 +\mu_{01} })
-
 \mu_0^{-3} w(\frac{\bar{x}}{ \mu_0  }) Z_{5}(\frac{\bar{x}}{ \mu_0  })
\Big]
 \eta(\frac{ \bar{x} }{\sqrt  t}) \varphi^2[ \mu_0 +\mu_{01} ](\bar{x},t)
 \\
& +
 \mu_0^{-3} w(\frac{\bar{x}}{ \mu_0  }) Z_{5}(\frac{\bar{x}}{ \mu_0  })
\eta(\frac{ \bar{x} }{\sqrt  t}) ( \varphi^2[ \mu_0 +\mu_{01} ] - \varphi^2[ \mu_0  ] ) (\bar{x},t)
+
\mu_0^{-3} w(\frac{\bar{x}}{ \mu_0  }) Z_{5}(\frac{\bar{x}}{ \mu_0  })
\eta(\frac{ \bar{x} }{\sqrt  t}) \varphi^2[ \mu_0  ]  (\bar{x},t)
\Bigg\}    d \bar{x}
\\
=  \ &
\int_{\RR^4 }
\Bigg\{
O\Big(\frac{|\mu_{01}|}{\mu_0}
\mu_0^{-4} \langle \frac{\bar{x}}{ \mu_0 }\rangle^{-6} \Big)
\eta^2(\frac{ \bar{x} }{\sqrt  t})
\bigg[ -2^{-\frac 12} \Big( (\mu_0 +\mu_{01} )  t^{-1}
+ \int_{t/2}^{t-\mu_{0}^2 } \frac{\mu_{0t}(s) + \mu_{01t}(s) }{t-s} d s
\Big)
\\
&
+ O\Big( (\mu_0 +\mu_{01} )  t^{-2} |\bar{x}|^2  +
| \mu_{0t} +\mu_{01t}  | \frac{ |\bar x|}{\mu_0 + \mu_{01}}  \Big)
 + g[\mu_0 + \mu_{01}]  \bigg]
\\
&
+
\mu_0^{-4} w^2(\frac{\bar{x}}{ \mu_0 }) Z_{5}(\frac{\bar{x}}{ \mu_0 }) \eta^2(\frac{ \bar{x} }{\sqrt  t}) \bigg[
-2^{-\frac 12} \Big( \mu_{01}  t^{-1}
+
\int_{t/2}^{t-\mu_{0}^2 } \frac{\mu_{01t} (s)}{t-s} d s
\Big)
\\
&
+ O\Big( |\mu_{01}| t^{-2} |\bar{x}|^2 +  |\mu_{0t}| \sup\limits_{t_1\in[t/2,t]}\Big(\frac{|\mu_{01}(t_1)|}{\mu_0}
 + \frac{|\mu_{01t}(t_1)|}{|\mu_{0t}|}  \Big)  \frac{ |\bar x |}{\mu_0} \Big)   +
 \tilde{g}[\mu_0,\mu_{01}]
\bigg]\\
& +
\frac{|\mu_{01}|}{\mu_0}
\mu_0^{-3} \langle \frac{\bar{x}}{ \mu_0  } \rangle^{-4}
\eta(\frac{ \bar{x} }{\sqrt  t})
O((  t \ln t)^{-2} )
+
\mu_0^{-3} w(\frac{\bar{x}}{ \mu_0  }) Z_{5}(\frac{\bar{x}}{ \mu_0  })
\eta(\frac{ \bar{x} }{\sqrt  t})
O((t\ln t)^{-2} )
\Bigg\}    d \bar{x}
+
\mathcal{M}[\mu_0]
\\
=  \ &
\int_{\RR^4 }
\Bigg\{
O\Big(\frac{|\mu_{01}|}{\mu_0}
 \langle \frac{\bar{x}}{ \mu_0 }\rangle^{-6} \Big)
\eta^2(\frac{ \bar{x} }{\sqrt  t})
\bigg[ -2^{-\frac 12} \Big( \mu_{01} t^{-1}
+ \int_{t/2}^{t-\mu_{0}^2 } \frac{ \mu_{01t}(s) }{t-s} d s \Big) +
O (t^{-1} (\ln t)^{-2}  \ln\ln t)
\\
&
+ O\Big( \mu_0^3  t^{-2} \frac{|\bar{x}|^2}{\mu_0^2}  +
| \mu_{0t} | \frac{ |\bar x|}{\mu_0}  \Big)
  \bigg]
+
 w^2(\frac{\bar{x}}{ \mu_0 }) Z_{5}(\frac{\bar{x}}{ \mu_0 }) \eta^2(\frac{ \bar{x} }{\sqrt  t}) \bigg[
-2^{-\frac 12} \Big( \mu_{01}  t^{-1}
+
\int_{t/2}^{t-\mu_{0}^2 } \frac{\mu_{01t} (s)}{t-s} d s
\Big)
\\
&
+ O\Big( \mu_0^2 |\mu_{01}| t^{-2} \frac{|\bar{x}|^2}{\mu_0^2} +  |\mu_{0t}| \sup\limits_{t_1\in [t/2,t] }\Big(\frac{|\mu_{01}(t_1)|}{\mu_0}
+ \frac{|\mu_{01t}(t_1)|}{|\mu_{0t}|}  \Big)  \frac{ |\bar x |}{\mu_0} \Big)
 + \tilde{g}[\mu_0,\mu_{01}]
\bigg]
\\
&
+
\frac{|\mu_{01}|}{\mu_0^2}
 \langle \frac{\bar{x}}{ \mu_0  } \rangle^{-4}
\eta(\frac{ \bar{x} }{\sqrt  t})
O((  t \ln t)^{-2} )
+
\mu_0 \langle \frac{\bar{x}}{ \mu_0  } \rangle^{-4}
\eta(\frac{ \bar{x} }{\sqrt  t})
O((  t \ln t)^{-2} )
\Bigg\}    d (\frac{\bar{x}  }{\mu_0} )
+ \mathcal{M}[\mu_0]
\\
= \ &
 O( 	\frac{|\mu_{01}|}{\mu_0} )
		 \Big( \mu_{01} t^{-1}
		+ \int_{t/2}^{t-\mu_{0}^2 } \frac{ \mu_{01t}(s) }{t-s} d s
		\Big)
				-2^{-\frac 12}  \int_{\RR^4 }  w^2(y) Z_{5}(y) \eta^2(\frac{ \mu_0 y }{\sqrt  t}) dy
	\Big( \mu_{01}  t^{-1}
		+
		\int_{t/2}^{t-\mu_{0}^2 } \frac{\mu_{01t} (s)}{t-s} d s
		\Big)
		\\
		& +
		O (t^{-1} (\ln t)^{-2}  \ln\ln t) \frac{|\mu_{01}|}{\mu_0}
		+ O\Big( \mu_0^2 |\mu_{01}| t^{-2} \ln t +  |\mu_{0t}| \sup\limits_{t_1\in[t/2,t]} \Big(\frac{|\mu_{01}(t_1)|}{\mu_0}
		+ \frac{|\mu_{01t}(t_1) |}{|\mu_{0t}|}  \Big)  \Big)
		\\
		&
		 + \tilde{g}[\mu_0,\mu_{01}]
		 +
		 \frac{|\mu_{01}|}{\mu_0^2}
		 \ln t
		 O((  t \ln t)^{-2} )
		+
		\mu_0 \ln t
		O((t\ln t)^{-2} )
		+ \mathcal{M}[\mu_0]\\
=  \ &
\Big( O(\frac{|\mu_{01}|}{\mu_0} )
-2^{-\frac 12}  \int_{\RR^4 }  w^2(y) Z_{5}(y) \eta^2(\frac{ \mu_0 y }{\sqrt  t}) dy  \Big)
\Big( \mu_{01}  t^{-1}
+
\int_{t/2}^{t-\mu_{0}^2 } \frac{\mu_{01t} (s)}{t-s} d s
\Big)
\\
&
+ O\big( (t\ln t)^{-1} \sup\limits_{t_1\in[t/2,t]}  |\mu_{01}(t_1)|  + \sup\limits_{t_1\in[t/2,t]}  |\mu_{01t}(t_1)|  \big)
\\
&
+
O (t^{-1} (\ln t)^{-1}  \ln\ln t) |\mu_{01}|
+ \tilde{g}[\mu_0,\mu_{01}]
+
O((t\ln t)^{-2} )
+ \mathcal{M}[\mu_0]
\\
=  \ &
\Big( O( \frac{|\mu_{01}|}{\mu_0} )
-2^{-\frac 12}  \int_{\RR^4 }  w^2(y) Z_{5}(y) \eta^2(\frac{ \mu_0 y }{\sqrt  t}) dy  \Big)
\bigg[
\mu_{01}  t^{-1} (1+ O((\ln t)^{-\frac 12} ) )
+ \int_{t/2}^{t-t^{1-\nu_1} } \frac{\mu_{01t}(s)}{t-s} d s
\\
&
+ \int_{ t-t^{1-\nu_1}  }^{t-\mu_{0}^2(t) } \frac{\mu_{01t}(t)}{t-s} d s
+
\mathcal{E}_{\nu_1}[\mu_{01}]
+ O\big( (t\ln t)^{-1} \sup\limits_{t_1\in[t/2,t]}  |\mu_{01}(t_1)|  + \sup\limits_{t_1\in[t/2,t]}  |\mu_{01t}(t_1)|  \big)
\\
&
+ \tilde{g}[\mu_{01},\mu_0]
+
O((t\ln t)^{-2} )
+ \mathcal{M}[\mu_0]
\bigg]
\\
=  \ &
\Big(  O( \frac{|\mu_{01}|}{\mu_0} )
-2^{-\frac 12}  \int_{\RR^4} w^2(y) Z_{5}(y) \eta^2(\frac{ \mu_0 y }{\sqrt  t})
dy
\Big) \bigg[ \mu_{01}  t^{-1} \big(1+ O((\ln t)^{-\frac 12})\big)
+ \int_{t/2}^{t-t^{1-\nu_1} } \frac{\mu_{01t}(s)}{t-s} d s
\\
&
+
\mu_{01t}
((1-\nu_1)\ln t + 2\ln\ln t )
+
O\big( (t\ln t)^{-1} \sup\limits_{t_1\in[t/2,t]}  |\mu_{01}(t_1)|  + \sup\limits_{t_1\in[t/2,t]}  |\mu_{01t}(t_1)|  \big)
\\
&
+
\mathcal{E}_{\nu_1}[\mu_{01}]
+ \tilde{g}[\mu_{01},\mu_0]
+ O((t\ln t)^{-2} )
+ \mathcal{M}[\mu_0]
\bigg],
\end{align*}
where
\begin{equation*}
	\mathcal{E}_{\nu_1}[\mu_{01}] = \int_{ t-t^{1-\nu_1}  }^{t-\mu_{0}^2(t) } \frac{\mu_{01t}(s) - \mu_{01t}(t)}{t-s} d s.
\end{equation*}

Since it is too difficult to solve the nonlocal equation about $\mu_{01}$ thoroughly, we put $\mathcal{E}_{\nu_1}[\mu_{01}] $ aside as the new error term and consider the following equation
\begin{equation}\label{mu01-eq}
\mu_{01t} + \beta_{\nu_1}(t)
    \mu_{01}
= f_{\nu_1}[\mu_{01}],
\end{equation}
where
\begin{equation*}
	\beta_{\nu_1}(t) = t^{-1} (1+ O((\ln t)^{-\frac 12})) [ (1-\nu_1)\ln t + 2\ln\ln t  ] ^{-1} ,
\end{equation*}
\begin{equation*}
\begin{aligned}
f_{\nu_1}[\mu_{01}] = \ &
\chi(t)
\big[(1-\nu_1)\ln t + 2\ln\ln t  \big]^{-1}
\bigg(- \int_{t/2}^{t-t^{1-\nu_1} } \frac{\mu_{01t}(s)}{t-s} d s
- \tilde{g}[\mu_{01},\mu_0]
\\
&
+
O\big( (t\ln t)^{-1} \sup\limits_{t_1\in[t/2,t]}  |\mu_{01}(t_1)|  + \sup\limits_{t_1\in[t/2,t]}  |\mu_{01t}(t_1)|  \big)
+ O((t\ln t)^{-2} )
- \mathcal{M}[\mu_0] \bigg),
\end{aligned}
\end{equation*}
$\chi(t)$ is a smooth cut-off function such that $\chi(t)= 0$ for $t< \frac{3}{4} t_0$ and $\chi(t) = 1$ for $t\ge t_0$. Since $\mu_{01}(t)$ will be defined in $(\frac{t_0}{4},\infty)$, the introduction of $\chi(t)$ is used to avoid the occurrence of $\mu_{01}(t)$ for $t$ beyond $(\frac{t_0}{4},\infty)$ in the terms like
$\int_{t/2}^{t-t^{1-\nu_1} } \frac{\mu_{01t}(s)}{t-s} d s$. After all, the original orthogonal equation is only required to hold in $(t_0,\infty)$. For technical reasons, we extend  the domain of $\mu_{01}$ to $(\frac{t_0}{4},\infty)$.

It then suffices to consider the following fixed point problem:
\begin{equation}\label{mu01-fixed point}
	\begin{aligned}
		A_{\nu_1}[\mu_{01}](t) = \ &
		 -\int_{t}^{\infty} \pp_{t}A_{\nu_1}[\mu_{01}](s) ds
		=
		-e^{-\int^t \beta_{\nu_1}(u) du}
		\int_{t}^\infty
		e^{\int^s \beta_{\nu_1}(u) du} f_{\nu_1}[\mu_{01}](s) ds ,
		\\
		\pp_{t} A_{\nu_1}[\mu_{01}](t) = \ &
		\beta_{\nu_1}(t) e^{-\int^t \beta_{\nu_1}(u) du}
		\int_{t}^\infty
		e^{\int^s \beta_{\nu_1}(u) du} f_{\nu_1}[\mu_{01}](s) ds
		+
		f_{\nu_1}[\mu_{01}](t),
	\end{aligned}
\end{equation}
where $\nu_1\in(0,\frac 12)$ will be determined later.

Since
\begin{equation*}
	O((t\ln t)^{-2} ) + |\mathcal{M}[ \mu_0 ]  |
	\le
	C_0  t^{-1} (\ln t)^{-2} \ln \ln t
\end{equation*}
where $C_{0}\ge 1$ is a large constant independent of $t_0$, we have
\begin{equation*}
	\big|	[(1-\nu_1)\ln t + 2\ln\ln t  ]^{-1} ( O((t\ln t)^{-2} ) + |\mathcal{M}[ \mu_0 ]  |  ) \big|
	\le
	(1-\nu_1)^{-1} C_0  t^{-1} (\ln t)^{- 3} \ln \ln t  .
\end{equation*}

By L'H\^opital's rule,
\begin{equation*}
	\begin{aligned}
		&
		\lim\limits_{t\rightarrow \infty}
		\frac{
			-\int_{t}^{ \infty }
			e^{\int^s \beta_{\nu_1}(u) du}
			s^{-1}
			(\ln s)^{- 3} \ln \ln s
			ds}
		{e^{\int^t \beta_{\nu_1}(u) du}  (\ln t)^{- 2 } \ln \ln t}
		\\
		= \ &
		\lim\limits_{t\rightarrow \infty}
		\frac{
			e^{\int^t \beta_{\nu_1}(u) du}
			t^{-1}
			(\ln t)^{- 3} \ln \ln t }
		{\beta_{\nu_1}(t) e^{\int^t \beta_{\nu_1}(u) du}  (\ln t)^{- 2 } \ln \ln t +
			e^{\int^t \beta_{\nu_1}(u) du}  (-2)t^{-1}(\ln t)^{- 3 } \ln \ln t +
			e^{\int^t \beta_{\nu_1}(u) du}  t^{-1}(\ln t)^{- 3 } }
		\\
		= \ &
		[(1-\nu_1)^{-1} -2]^{-1} .
	\end{aligned}
\end{equation*}
Notice $\nu_1<\frac {1}{2}$ implies $e^{\int^t \beta_{\nu_1}(u) du} \ll (\ln t)^{2}$ so that $
\int_{t}^{ \infty }
e^{\int^s \beta_{\nu_1}(u) du}
s^{-1}
(\ln s)^{- 3} \ln \ln s
ds$ is well defined.
Thus we have
\begin{equation*}
	\frac{
		-\int_{t}^{\infty }
		e^{\int^s \beta_{\nu_1}(u) du}
		s^{-1}
		(\ln s)^{- 3} \ln \ln s
		ds}
	{e^{\int^t \beta_{\nu_1}(u) du}  (\ln t)^{- 2 } \ln \ln t}
	=
	[(1-\nu_1)^{-1} -2]^{-1} + o(1)
\end{equation*}
where $o(1) \rightarrow 0$ as $t_0 \rightarrow \infty$. Then
\begin{equation*}
	\begin{aligned}
		&
		\bigg|e^{-\int^t \beta_{\nu_1}(u) du}
		\int_{\infty}^{t}
		e^{\int^s \beta_{\nu_1}(u) du}
		[ (1-\nu_1) \ln s + 2\ln\ln s
		 ]^{-1}  s^{-1} (\ln s)^{-2} \ln \ln s    ds\bigg|
		\\
		\le \ &
		(1-\nu_1)^{-1}
		\bigg|e^{-\int^t \beta_{\nu_1}(u) du}
		\int_{\infty}^{t}
		e^{\int^s \beta_{\nu_1}(u) du}
		s^{-1} (\ln s)^{-3} \ln \ln s ds\bigg|
		=
		|( 2 \nu_1 -1)^{-1} + o(1)  |
		(\ln t)^{-2} \ln \ln t ,
	\end{aligned}
\end{equation*}
\begin{equation*}
	\begin{aligned}
		&
		\bigg|\beta_{\nu_1}(t) e^{-\int^t \beta_{\nu_1}(u) du}
		\int_{ \infty }^{t}
		e^{\int^s \beta_{\nu_1}(u) du}  [ (1-\nu_1) \ln s + 2\ln\ln s
	 ]^{-1} s^{-1} (\ln s)^{-2} \ln \ln s ds \bigg|
		\\
		&
		+
		\Big|[ (1-\nu_1) \ln t + 2\ln\ln t
		  ]^{-1} t^{-1} (\ln t)^{-2} \ln \ln t  \Big|
		\\
		\le \ &
		t^{-1} (1-\nu_1)^{-1} (\ln t)^{-1}  |(2\nu_1 -1)^{-1} + o(1)  |
		(\ln t)^{-2} \ln \ln t
		+
		(1-\nu_1)^{-1} t^{-1} (\ln t)^{-3}\ln \ln t
		\\
		= \ &
		(1-\nu_1)^{-1} (1+ |(2\nu_1 -1)^{-1} + o(1)  | ) t^{-1} (\ln t)^{-3 } \ln \ln t .
	\end{aligned}
\end{equation*}
From the  estimates above, for $\mu_{01} \in C^{1}(t_{0}/4,\infty)$ and $\mu_{01}(t)\rightarrow 0$ as $t\rightarrow \infty$, we set the norm as
\begin{equation*}
	\| \mu_{01}\|_{01} =
	\sup\limits_{t\ge t_{0}/ 4}
	t(\ln t)^{3} (\ln \ln t)^{-1}   |\mu_{01t}(t)|
\end{equation*}
and will solve the fixed point problem \eqref{mu01-fixed point} in the space
\begin{equation*}
	B_{01}  =
	\left\{ g(t) \in C^{1}(t_{0}/4,\infty),\quad g(t)\rightarrow 0 \mbox{ \ as \ } t\rightarrow \infty \ : \
	\| g \|_{01} \le 2 C_{0} C(\nu_1)  \right\}
\end{equation*}
where $C(\nu_1) = (1-\nu_1)^{-1} (1+ |(2\nu_1 -1)^{-1} + o(1)  | ) $. We take $\nu_1<\frac 12$ and $t_0$ large enough to guarantee  $C(\nu_1)<\infty$. Let us estimate other terms for  $\pp_{t}A_{\mu_1}[\mu_{01}]$ in \eqref{mu01-fixed point}.

For any $\bar{\mu}_{01}\in B_{01}$,
\begin{equation}\label{typical}
	\begin{aligned}
		&
		\chi(t)\Big|\int_{t/2}^{t-t^{1-\nu_1} } \frac{\bar{\mu}_{01t}(s)}{t-s} d s \Big|
		\le
		\| \bar{\mu}_{01}\|_{01}
		\int_{t/2}^{t-t^{1-\nu_1} } \frac{ s^{-1} (\ln s)^{-3 } \ln \ln s }{t-s} d s
		\\
		= \ &
		\| \bar{\mu}_{01}\|_{01}
		t^{-1} \int_{1/2}^{1-t^{-\nu_1}}
		\frac{(\ln t +\ln z)^{-3} \ln (\ln t +\ln z) }{z(1-z)} dz
		\\
		\le \ &
		\| \bar{\mu}_{01}\|_{01}
		(1+O((\ln t)^{-\frac 12}) ) t^{-1}  (\ln t )^{-3}\ln \ln t
		\int_{1/2}^{1-t^{-\nu_1}}
		\frac{1 }{z(1-z)} dz
		\\
		= \ &
		\| \bar{\mu}_{01}\|_{01}
		(1+O((\ln t)^{-\frac 12}) ) t^{-1}
		(\ln t )^{-3}\ln \ln t
		\ln(t^{\nu_1} -1)
		\\
		\le \ &
		\| \bar{\mu}_{01}\|_{01}
		\nu_1 (1+O((\ln t)^{-\frac 12}) )
		t^{-1}  (\ln t )^{-2}\ln \ln t
		\\
		\le \ &
		2 C_{0} C(\nu_1)
		\nu_1 (1+O((\ln t)^{-\frac 12}) )
		t^{-1}  (\ln t )^{-2}\ln \ln t
	\end{aligned}
\end{equation}
which implies
\begin{equation*}
	\bigg| \big[(1-\nu_1)\ln t + 2\ln\ln t + O(1)\big]^{-1} \chi(t) \int_{t/2}^{t-t^{1-\nu_1} } \frac{\bar{\mu}_{01t}(s)}{t-s} d s  \bigg|
	\le 2 C_{0} C(\nu_1)
	\nu_1 (1-\nu_1)^{-1}  (1+O((\ln t)^{-\frac 12}) )
	t^{-1}  (\ln t )^{-3}\ln \ln t .
\end{equation*}

We take $\nu_1\in (0,\frac{1}{4})$ to make
$ \nu_1 (1-\nu_1)^{-1} (1+ |(2 \nu_1 -1)^{-1}  | ) <1  $.

Since $\bar{\mu}_{01}\in B_{01}$, one has
$|\bar{\mu}_{01}(t)| = |\int_{t}^{\infty} \bar{\mu}_{01t}(s) ds| \lesssim \|\bar{\mu}_{01}\|_{01} (\ln t)^{-2} \ln\ln t $.
Then
\begin{equation*}
\begin{aligned}
\chi(t)	\tilde{g}[\bar{\mu}_{01},\mu_0] = \ &
	O \Big(		 t^{-2}
	\int_{t_0/2}^{t }
	( s^{-1}  \|\bar{\mu}_{01}\|_{01} (\ln s)^{-4} \ln\ln s
	+ (\ln s)^{-2} \|\bar{\mu}_{01}\|_{01} (\ln s)^{-1} \ln\ln s
	)
	d s
	\\
	&
	+
(t\ln t)^{-1} \|\bar{\mu}_{01} \|_{01}^2
((\ln t)^{-1} \ln\ln t)^2
	\Big)
\\
= \ &
(\ln t_0)^{-\frac 12}
O \big(		  (\|\bar{\mu}_{01}\|_{01} + \|\bar{\mu}_{01} \|_{01}^2 ) t^{-1} (\ln t)^{-2} \ln\ln t
\big)  ,
\end{aligned}
\end{equation*}
\begin{equation*}
\chi(t)O\big( (t\ln t)^{-1} \sup\limits_{t_1\in[t/2,t]}  |\mu_{01}(t_1)|  + \sup\limits_{t_1\in[t/2,t]}  |\mu_{01t}(t_1)|  \big)
\lesssim t^{-1} (\ln t)^{-3} \ln\ln t
\lesssim
(\ln t_0)^{-1}
t^{-1} (\ln t)^{-2} \ln\ln t.
\end{equation*}
Then for any fixed $\nu_1\in (0,\frac{1}{4})$ and $t_0$ large enough, one sees that $A_{\nu_1}[\bar{\mu}_{01}] \in B_{01}$.

The contraction property can be derived similarly. Indeed, for any $\mu_{01a}, \mu_{01b} \in B_{01} $, similar to \eqref{typical}, we have
\begin{equation*}
	\begin{aligned}
		&
		\chi(t)\bigg|\int_{t/2}^{t-t^{1-\nu_1} } \frac{\pp_t\bar{\mu}_{01a}(s)}{t-s} d s - \int_{t/2}^{t-t^{1-\nu_1} } \frac{\pp_t\bar{\mu}_{01b}(s)}{t-s} d s  \bigg|
		\le
		\| \bar{\mu}_{01a} - \bar{\mu}_{01b} \|_{01}
		\int_{t/2}^{t-t^{1-\nu_1} } \frac{ s^{-1} (\ln s)^{-3 } \ln \ln s }{t-s} d s
		\\
		\le \ &
		\| \bar{\mu}_{01a} - \bar{\mu}_{01b} \|_{01}
		\nu_1 (1+O((\ln t)^{-\frac 12}) )
		t^{-1}  (\ln t )^{-2}\ln \ln t ,
	\end{aligned}
\end{equation*}
\begin{equation*}
		\chi(t) |	\tilde{g}[ \mu_{01a},\mu_0] -\tilde{g}[ \mu_{01b},\mu_0] |
		=
		(\ln t_0)^{-\frac 12}
		O (		  C_{0} C(\nu_1)  t^{-1} (\ln t)^{-2} \ln\ln t
		) \| \bar{\mu}_{01a} - \bar{\mu}_{01b} \|_{01}  ,
\end{equation*}
\begin{equation}\label{mu1-mu2-est}
\begin{aligned}
&
\chi(t) \Big|	O( (t\ln t)^{-1} \sup\limits_{t_1\in[t/2,t]}  |\mu_{01a}(t_1)|  + \sup\limits_{t_1\in[t/2,t]}  |\pp_{t}\mu_{01a}(t_1)|  )
\\
&\quad
	-
O( (t\ln t)^{-1} \sup\limits_{t_1\in[t/2,t]}  |\mu_{01b}(t_1)|  + \sup\limits_{t_1\in[t/2,t]}  |\pp_{t}\mu_{01b}(t_1)|  ) \Big|
\\
\lesssim \ &
\chi(t) \big|	O( (t\ln t)^{-1} \sup\limits_{t_1\in[t/2,t]}  |\mu_{01a}(t_1) - \mu_{01b}(t_1)|  + \sup\limits_{t_1\in[t/2,t]}  |\pp_{t}\mu_{01a}(t_1) - \pp_{t}\mu_{01b}(t_1) |  )\big|
\\
\lesssim \ &
\| \bar{\mu}_{01a} - \bar{\mu}_{01b} \|_{01}
t^{-1} (\ln t)^{-3} \ln\ln t
\lesssim
	(\ln t_0)^{-1} \| \bar{\mu}_{01a} - \bar{\mu}_{01b} \|_{01}
	t^{-1} (\ln t)^{-2} \ln\ln t
\end{aligned}
\end{equation}
by the estimate of $\varphi[\mu_0 +\mu_{01a}] - \varphi[\mu_0 +\mu_{01b}]$ in Corollary \ref{varphi-coro}.

Due to the choice of $\nu_1$ and $t_0$ above, the contraction property is achieved. By contraction mapping theorem, there exists a unique solution $\mu_{01}\in B_{01}$ for \eqref{mu01-fixed point}.

From now on, $\nu_1$ will be regarded as a general constant unless otherwise stated. For notational simplicity, $\pp_{t}$  is denoted by ``$~'~$''.
Once we have solved $\mu_{01}$, the regularity of $\mu_{01}$ can be improved by the equation of $\mu_{01}$ and $\mu_{01}''$ decays to $0$ as $t\rightarrow\infty$. For the purpose of finding a better decay estimate of $\mu_{01}''$, we take derivative on both sides of \eqref{mu01-eq}. Then
\begin{equation*}
	\mu_{01}''
	+
	(\beta_{\nu_1}(t) )' \mu_{01}
	+ 	\beta_{\nu_1}(t) \mu_{01}'
	=
	(f_{\nu_1}[\mu_{01}] )',
\end{equation*}
where we can evaluate
\begin{equation*}
	|(\beta_{\nu_1}(t) )' \mu_{01}
	+ 	\beta_{\nu_1}(t) \mu_{01}' | \lesssim   t^{-2} (\ln t)^{-3} \ln\ln t,
\end{equation*}
\begin{equation*}
	\begin{aligned}
		&
	\big( \chi(t)	[(1-\nu_1)\ln t + 2\ln\ln t  ]^{-1} 	 \big)'
		\bigg(- \int_{t/2}^{t-t^{1-\nu_1} } \frac{\mu_{01t}(s)}{t-s} d s
		- \tilde{g}[\mu_{01},\mu_0]
		\\
		&
		+
		O( (t\ln t)^{-1} \sup\limits_{t_1\in[t/2,t]}  |\mu_{01}(t_1)|  + \sup\limits_{t_1\in[t/2,t]}  |\mu_{01t}(t_1)|  )
		+ O((t\ln t)^{-2} )
		- \mathcal{M}[\mu_0] \bigg)
		\lesssim
		t^{-2} (\ln t)^{-3} \ln\ln t  ,
	\end{aligned}
\end{equation*}
\begin{equation*}
\chi(t)
	\big[(1-\nu_1)\ln t + 2\ln\ln t  \big]^{-1}
	\big(
	- \tilde{g}[\mu_{01},\mu_0]
	+ O((t\ln t)^{-2} )
	- \mathcal{M}[\mu_0] \big)'
	 \lesssim
	t^{-2} (\ln t)^{-3} \ln\ln t ,
\end{equation*}
where we used similar calculation in \eqref{pptu-eq} for $(\tilde{g}[\mu_{01},\mu_0])'$.
\begin{equation*}
\begin{aligned}
&
	\chi(t) \big[(1-\nu_1)\ln t + 2\ln\ln t \big]^{-1}
	\Big(- \int_{t/2}^{t-t^{1-\nu_1} } \frac{\mu_{01t}(s)}{t-s} d s \Big)'
\\
= \ &
\chi(t) \big[(1-\nu_1)\ln t + 2\ln\ln t \big]^{-1}
\Big[ -
		\frac{\mu_{01}'(t-t^{1-\nu_1})}{t^{1-\nu_1}} (1-(1-\nu_1)t^{-\nu_1})
		 +
		  \frac{\mu_{01}'(\frac t2)}{t}
		+
		\int_{t/2}^{t-t^{1-\nu_1} } \frac{\mu_{01}'(s)}{(t-s)^2} d s )
		\Big]
\\
= \ &
\chi(t) \big[(1-\nu_1)\ln t + 2\ln\ln t \big]^{-1}
\Big[
(1-\nu_1)
		\frac{\mu_{01}'(t-t^{1-\nu_1})}{t}
		-
		 \frac{\mu_{01}'(\frac t2)}{t}
		-
		\int_{t/2}^{t-t^{1-\nu_1} } \frac{\mu_{01}''(s)}{t-s} d s
\Big]
\\
= \ &
-\chi(t) \big[(1-\nu_1)\ln t + 2\ln\ln t \big]^{-1}
		\int_{t/2}^{t-t^{1-\nu_1} } \frac{\mu_{01}''(s)}{t-s} d s
		+O( t^{-2} (\ln t)^{-4} \ln \ln t  ) .
\end{aligned}
\end{equation*}
Revisiting the process of proving Corollary \ref{varphi-coro}, we have
\begin{equation}\label{sup-cal}
\begin{aligned}
 & \chi(t) \big[(1-\nu_1)\ln t + 2\ln\ln t \big]^{-1} \big(O( (t\ln t)^{-1} \sup\limits_{t_1\in[t/2,t]}  |\mu_{01}(t_1)|  + \sup\limits_{t_1\in[t/2,t]}  |\mu_{01t}(t_1)|  ) \big)'
 \\
 = \ &
  \chi(t)  \big[(1-\nu_1)\ln t + 2\ln\ln t \big]^{-1}
 O\big(\sup\limits_{t_1\in[t/2,t]}  |\mu_{01}''(t_1)| \big)
 +
 O(t^{-2} (\ln t)^{-4} \ln\ln t) .
\end{aligned}
\end{equation}
Using the estimates above, we have
\begin{equation*}
	\mu_{01}''=  \chi(t)	\big[(1-\nu_1)\ln t + 2\ln\ln t\big]^{-1}
	\Big( -
	\int_{t/2}^{t-t^{1-\nu_1} } \frac{\mu_{01}''(s)}{t-s} d s
	+  O(\sup\limits_{t_1\in[t/2,t]}  |\mu_{01}''(t_1)| )\Big)
	+ \tilde{f}(t)
\end{equation*}
where $|\tilde{f}(t)| \le C_1 t^{-2} (\ln t)^{-3} \ln \ln t$.
For this reason, we solve $\mu_{01}''$ in the following space
\begin{equation*}
	B_{2,2} = \{ g\in C(t_0/4,\infty), g(t) \rightarrow 0 \mbox{ \ as \ } t\rightarrow \infty \ : \ \| g\|_{2,2} \le 2C_1 \} ,
\end{equation*}
where for $g\in C(t_0/4,\infty)$, we define
\begin{equation*}
	\| g\|_{a,b} := \sup\limits_{t\ge t_0/4}  t^a (\ln t)^b |g(t)| .
\end{equation*}
For any $g_1, g_2 \in B_{2,2}$, similar to \eqref{typical}, we have
\begin{equation*}
	\begin{aligned}
		&\chi(t) \big[(1-\nu_1)\ln t + 2\ln\ln t \big]^{-1}\bigg|
		\int_{t/2}^{t-t^{1-\nu_1} } \frac{ g_1(s)}{t-s} d s
		-
		\int_{t/2}^{t-t^{1-\nu_1} } \frac{ g_2(s)}{t-s} d s \bigg|
		\\
		\le \ &
		(1-\nu_1)^{-1} (\ln t)^{-1}
		\| g_1-g_2\|_{2,2}
		\int_{t/2}^{t-t^{1-\nu_1}} \frac{s^{-2} (\ln s)^{-2} }{t-s} ds
		\\
		\le \ &
		\nu_1 (1-\nu_1)^{-1}
		\| g_1-g_2\|_{2,2}
		t^{-2} (\ln t )^{-2} (1+(\ln t)^{-1} )
		( 1 + (\nu_1 \ln t )^{-1} )
		.
	\end{aligned}
\end{equation*}
Similar to \eqref{sup-cal}, one has
\begin{equation*}
\begin{aligned}
&
\chi(t)	\big[(1-\nu_1)\ln t + 2\ln\ln t\big]^{-1}
\Big|  O(\sup\limits_{t_1\in[t/2,t]}  | g_1(t_1)|  ) - O(\sup\limits_{t_1\in[t/2,t]}  | g_2(t_1)|  ) \Big|
\\
\lesssim \ &
\chi(t)	\big[(1-\nu_1)\ln t + 2\ln\ln t\big]^{-1}
  O\big(\sup\limits_{t_1\in[t/2,t]}  | g_1(t_1) - g_2(t_1) |     \big)
\lesssim
t^{-2} ( \ln t)^{-3} \| g_1 -g_2\|_{2,2} .
\end{aligned}
\end{equation*}
For any $g\in B_{2,2}$, we have
\begin{equation*}
	\begin{aligned}
		&
		\Big| \chi(t)	[(1-\nu_1)\ln t + 2\ln\ln t ]^{-1}
		\int_{t/2}^{t-t^{1-\nu_1} } \frac{g(s)}{t-s} d s  \Big|
		\le
		(1-\nu_1)^{-1} (\ln t)^{-1} \| g\|_{2,2}
		\int_{t/2}^{t-t^{1-\nu_1} } \frac{  s^{-2} (\ln s)^{-2} }{t-s} d s
\\
\le \ &
\nu_1 (1-\nu_1)^{-1}
\| g \|_{2,2}
t^{-2} (\ln t )^{-2} (1+(\ln t)^{-1} )
( 1 + (\nu_1 \ln t )^{-1} ) .
	\end{aligned}
\end{equation*}
Since $\nu_1\in(0,\frac 14)$, when $t_0$ is large enough, the contraction property follows and then $\mu_{01}''\in B_{2,2}$. Thus the improved error is given by
\begin{equation*}
	\mathcal{M}[ \mu_{0}+\mu_{01}]
	= \mathcal{E}_{\nu_1}[\mu_{01}] = \int_{ t-t^{1-\nu_1}  }^{t-\mu_{0}^2(t) } \frac{\mu_{01}'(s) - \mu_{01}'(t)}{t-s} d s = O(t^{-1-\nu_1} (\ln t)^{-2} ) .
\end{equation*}

\noindent {\bf Step 3. Further improvement by iteration.}

Repeating Step 2 finitely many times, we can find $\mu_{0i}$, $i=1,\dots,k_0$ such that
\begin{equation}\label{bar-mu0-est}
	M\Big[\mu_0 + \sum\limits_{i=1}^{k_0} \mu_{0i}\Big] = O(t^{-2}) .
\end{equation}
Denote $$\bar{\mu}_0 = \mu_0 + \sum\limits_{i=1}^{k_0} \mu_{0i}.$$ From the construction above, we see that $\bar{\mu}_0 \sim \mu_0=(\ln t)^{-1}$, $\bar{\mu}_{0t} \sim \mu_{0t}$.

\medskip

Since $\bar{\mu}_0$ is determined, we are now able to describe $\Phi_0$ rigorously. Set $\bar{y} = \frac{\bar{x}}{\bar{\mu}_0}$ and consider
\begin{equation*}
	\Delta_{\bar{y}} \Phi_0 + 3w^2(\bar{y}) \Phi_0 = \tilde{H}(|\bar{y}| ,t),
\end{equation*}
where
\begin{equation*}
	\tilde{H}(|\bar{y}| ,t)   =
	- 3 \bar{\mu}_0 \left(
	w^2(\bar{y})
	\eta^2(\frac{\bar{\mu}_0 \bar{y} }{\sqrt  t})  \varphi[\bar{\mu}_0](\bar{\mu}_0 \bar{y},t)
	+
	\bar{\mu}_0 w(\bar{y}) \eta(\frac{\bar{\mu}_0 \bar{y}}{\sqrt  t}) \varphi^2[\bar{\mu}_0](\bar{\mu}_0 \bar{y},t) \right)
+ 3\bar{\mu}_0 \mathcal{M}[\bar{\mu}_0]
	\frac{\eta(\bar{y}) Z_{5}(\bar{y}) }{\int_{B_{2} } \eta(z) Z_{5}^2(z) d z}.
\end{equation*}
Then $\Phi_0(\bar{y},t)$ is given by
\[
\Phi_0(\bar{y},t) =  \tilde{Z}_{5}(\bar{y}) \int_0^{|\bar{y}|} \tilde{H} (s,t)Z_{ 5 }(s)s^{3} d s
-
 Z_{ 5 }(\bar{y}) \int_{0}^{| \bar{y} |} \tilde{H} (s,t) \tilde{Z}_{ 5 }(s)s^{3} d s,
\]
where $\tilde{Z}_{5}(r)$ is the other linearly independent kernel of the homogeneous equation, which satisfies that the Wronskian $W[Z_{ 5 },\tilde{Z}_{ 5 }]=r^{-3}$, so $\tilde{Z}_{ 5 }(r) \sim r^{-2}$ if $r\rightarrow 0$ and $\tilde{Z}_{5}(r) \sim  1$ if $r\rightarrow \infty$.

By the definition of $\mathcal{M}[\bar{\mu}_0]$, it is easy to have
\begin{equation}\label{tilde-H-orth}
\int_{\RR^4} \tilde{H}(z, t)Z_5(z) dz =0.
\end{equation}

By Corollary \ref{varphi-coro}, $|\tilde{H}| \lesssim  t^{-1} (\ln t)^{-2} \langle \bar{y} \rangle^{-4}$.
Due to the special choice of $\bar{\mu}_0$, one can get better time decay for $\tilde{H}$. Indeed, we have
\begin{equation*}
\begin{aligned}
&
	w^2(\bar{y})
\eta^2(\frac{\bar{\mu}_0 \bar{y} }{\sqrt  t})  \varphi[\bar{\mu}_0](\bar{\mu}_0 \bar{y},t)
\\
= \ &
	w^2(\bar{y})
\eta^2(\frac{\bar{\mu}_0 \bar{y} }{\sqrt  t})
\bigg[
-2^{-\frac 12} \Big( \bar{\mu}_{0}  t^{-1}
+ \int_{t/2}^{t-\mu_{0}^2 } \frac{\bar{\mu}_{0t}(s)}{t-s} d s
\Big)
+ O( \bar{\mu}_{0}^3 t^{-2} |\bar{y}|^2  +
| \bar{\mu}_{0t} | |\bar y| )
+ O ( t^{-1} (\ln t)^{-2} )
\bigg]
\\
= \ &
\Big[
O(  t^{-1} (\ln t)^{-2} |\bar y| )
+ O ( t^{-1} (\ln t)^{-2} \ln \ln t)
\Big] \langle \bar{y} \rangle^{-4}
\1_{\{ |\bar{y}| \le 2\bar{\mu}_0^{-1}
	t^{\frac 12} \}}
\\
= \ &  O ( t^{-1} (\ln t)^{-2} \ln \ln t)
\langle \bar{y} \rangle^{-3}
\1_{\{ |\bar{y}| \le 2\bar{\mu}_0^{-1}
	t^{\frac 12} \}} ,
\end{aligned}
\end{equation*}
\begin{equation*}
	\bar{\mu}_0 w(\bar{y}) \eta(\frac{\bar{\mu}_0 \bar{y}}{\sqrt  t}) \varphi^2[\bar{\mu}_0](\bar{\mu}_0 \bar{y},t)
=  O(t^{-2} (\ln t)^{-3} ) \langle \bar{y} \rangle^{-2} \1_{\{ |\bar{y}| \le 2\bar{\mu}_0^{-1}
	t^{\frac 12} \}}
=
O ( t^{-1} (\ln t)^{-2} \ln \ln t)
\langle \bar{y} \rangle^{-3}
\1_{\{ |\bar{y}| \le 2\bar{\mu}_0^{-1}
	t^{\frac 12} \}},
\end{equation*}
which implies $|\tilde{H}| \lesssim  t^{-1} (\ln t)^{-3} \ln \ln t
\langle \bar{y} \rangle^{-3} $. As a result, one has
\begin{equation}\label{tilde-H-upp}
|\tilde{H}| \lesssim
\min\big\{
t^{-1} (\ln t)^{-2} \langle \bar{y} \rangle^{-4},
t^{-1} (\ln t)^{-3} \ln \ln t
\langle \bar{y} \rangle^{-3}
\big\} .
\end{equation}
Claim:
\begin{equation}\label{Phi0-est}
\begin{aligned}
|\Phi_0(\bar{y},t) | \lesssim \ &
\min\big\{
t^{-1} (\ln t)^{-2} \langle \bar{y} \rangle^{-2}
\ln (2+ |\bar{y}|)
,
t^{-1} (\ln t)^{-3} \ln \ln t
\langle \bar{y} \rangle^{-1} \big\}
,
\\
|\nabla_{\bar{y}}\Phi_0(\bar{y},t) | \lesssim \ &
\min\big\{
t^{-1} (\ln t)^{-2}
\langle \bar{y} \rangle^{-3}
\ln (2+ |\bar{y}|),
t^{-1} (\ln t)^{-3} \ln \ln t
\langle \bar{y} \rangle^{-2} \big\}
,
\\
&|\pp_t \Phi_0(\bar{y},t) | \lesssim  t^{-2} (\ln t)^{-1}  \langle \bar{y} \rangle^{-2}
\ln (2+ |\bar{y}|) .
\end{aligned}
\end{equation}
Indeed, the estimate about $\Phi_0(\bar{y},t)  $ is derived from \eqref{tilde-H-orth} \eqref{tilde-H-upp}. The upper bound of $\nabla_{\bar{y}}\Phi_0(\bar{y},t) $ follows by scaling argument. In order to estimate $\pp_t \Phi_0(\bar{y},t) $, we need to take a closer look at $\pp_{t}\tilde{H}(z, t)$. By the definition of $\tilde{H}$, it is straightforward to have
\begin{equation*}
	\int_{\RR^4} \pp_{t}\tilde{H}(z, t)Z_5(z) dz =0.
\end{equation*}
\begin{equation*}
\begin{aligned}
	&
	\pp_{t} \tilde{H}(|\bar{y}| ,t)
	=
	- 3 \bar{\mu}_{0t}\bigg(
	w^2(\bar{y})
	\eta^2(\frac{\bar{\mu}_0 \bar{y} }{\sqrt  t})  \varphi[\bar{\mu}_0](
	\bar{\mu}_0 \bar{y},t)
	+
	\bar{\mu}_0 w(\bar{y}) \eta(\frac{\bar{\mu}_0 \bar{y}}{\sqrt  t}) \varphi^2[\bar{\mu}_0](\bar{\mu}_0 \bar{y},t) -\mathcal{M}[\bar{\mu}_0]
	\frac{\eta(\bar{y}) Z_{5}(\bar{y}) }{\int_{B_{2} } \eta Z_{5}^2 d z} \bigg)
\\
&
- 3 \bar{\mu}_0
\Bigg[
w^2(\bar{y})
2\eta(\frac{\bar{\mu}_0 \bar{y} }{\sqrt  t})
\nabla \eta (\frac{\bar{\mu}_0 \bar{y} }{\sqrt  t}) \cdot \frac{\bar{\mu}_0 \bar{y} }{\sqrt  t}
(\frac{\bar{\mu}_0 }{\sqrt  t})^{-1} \pp_{t} (\frac{\bar{\mu}_0 }{\sqrt  t})
 \varphi[\bar{\mu}_0](\bar{\mu}_0 \bar{y},t)
 \\
 & +
 w^2(\bar{y})
 \eta^2(\frac{\bar{\mu}_0 \bar{y} }{\sqrt  t})
 \big( \nabla_{\bar{x}} \varphi[\bar{\mu}_0](
 \bar{\mu}_0 \bar{y},t)  \cdot \bar{\mu}_{0t} \bar{y} + \pp_{t} \varphi[\bar{\mu}_0](
 \bar{\mu}_0 \bar{y},t)
 \big)
 \\
 &
+
\bar{\mu}_{0t} w(\bar{y}) \eta(\frac{\bar{\mu}_0 \bar{y}}{\sqrt  t}) \varphi^2[\bar{\mu}_0](\bar{\mu}_0 \bar{y},t)
+
\bar{\mu}_0 w(\bar{y})
\nabla \eta (\frac{\bar{\mu}_0 \bar{y} }{\sqrt  t}) \cdot \frac{\bar{\mu}_0 \bar{y} }{\sqrt  t}
(\frac{\bar{\mu}_0 }{\sqrt  t})^{-1} \pp_{t} (\frac{\bar{\mu}_0 }{\sqrt  t})
 \varphi^2[\bar{\mu}_0](\bar{\mu}_0 \bar{y},t)
\\
&
+
\bar{\mu}_0 w(\bar{y}) \eta(\frac{\bar{\mu}_0 \bar{y}}{\sqrt  t}) 2\varphi[\bar{\mu}_0](\bar{\mu}_0 \bar{y},t)
\big( \nabla_{\bar{x}} \varphi[\bar{\mu}_0](
\bar{\mu}_0 \bar{y},t)  \cdot \bar{\mu}_{0t} \bar{y} + \pp_{t} \varphi[\bar{\mu}_0](
\bar{\mu}_0 \bar{y},t)
\big)
- \pp_{t} (\mathcal{M}[\bar{\mu}_0] )
\frac{\eta(\bar{y}) Z_{5}(\bar{y}) }{\int_{B_{2} } \eta Z_{5}^2 d z} \Bigg]
.
\end{aligned}
\end{equation*}
Using \eqref{varphi-bar-mu0} in Appendix, one has
\begin{equation*}
	\begin{aligned}
		&
		\Big| \nabla_{\bar{x}} \varphi[\bar{\mu}_0](
		\bar{\mu}_0 \bar{y},t)  \cdot \bar{\mu}_{0t} \bar{y} + \pp_{t} \varphi[\bar{\mu}_0](
		\bar{\mu}_0 \bar{y},t)
		\Big|  \1_{\{ \bar{\mu}_0 |\bar{y}| \le 2t^{\frac 12}\}}
		\\
		\lesssim \ &
		\Big|
		\Big[ (t\ln t)^{-1} \1_{\{ |\bar{x}| \le \mu_0 \}} + (t^{-1} (\ln t)^{-2} |\bar{x}|^{-1}
		+
		t^{-\frac 32} (\ln t)^{-1})\1_{\{ |\bar{x}| > \mu_0 \}}  \Big]   \bar{\mu}_{0}^{-1}|\bar{\mu}_{0t}| |\bar{x}| +
		t^{-2}
		\Big|  \1_{\{  |\bar{x}| \le 2t^{\frac 12}\}}
		\lesssim  t^{-2} .
	\end{aligned}
\end{equation*}
Thus
\begin{equation*}
\begin{aligned}
&
|	\pp_{t} \tilde{H}(|\bar{y}| ,t) |
\lesssim
 t^{-2} (\ln t)^{-3} \langle \bar{y} \rangle^{-4}
+
(\ln t)^{-1}
\Big[ t^{-2} (\ln t)^{-1} \langle \bar{y} \rangle^{-4}  + t^{-2}  \langle \bar{y} \rangle^{-4} + t^{-3} (\ln t)^{-4} \langle \bar{y} \rangle^{-2}
\\
 &
+
 t^{-3} (\ln t)^{-3} \langle \bar{y} \rangle^{-2}
 +
 t^{-3} (\ln t)^{-2} \langle \bar{y} \rangle^{-2}  + t^{-3}
 \1_{ \{ |\bar{y}| \le 3 \} }
 \Big] \1_{\{ \bar{\mu}_0 |\bar{y}| \le 2t^{\frac 12}\}}
 \\
 \lesssim \ &
 t^{-2} (\ln t)^{-3} \langle \bar{y} \rangle^{-4}
 +
 (\ln t)^{-1}
 \Big(  t^{-2}  \langle \bar{y} \rangle^{-4}
 +
 t^{-3} (\ln t)^{-2} \langle \bar{y} \rangle^{-2}
 \Big) \1_{\{ \bar{\mu}_0 |\bar{y}| \le 2t^{\frac 12}\}}
 \lesssim
t^{-2} (\ln t)^{-1} \langle \bar{y}  \rangle^{-4} .
\end{aligned}
\end{equation*}
Therefore, we have the estimate about $\pp_t \Phi_0(\bar{y},t) $ in \eqref{Phi0-est}.

In order to avoid the influence in the remote region $|\bar{x}|\gtrsim t^{\frac 12}$, we add cut-off function and set $\bar{\mu}_0^{-1} \Phi_0(\frac{x-\xi}{\bar{\mu}_0} , t)\eta(\frac{4(x-\xi) }{\sqrt t})$ as the correction term. It is easy to check
\begin{equation*}
	\pp_{t} \Big( \bar{\mu}_0^{-1} \Phi_0(\frac{x-\xi}{\bar{\mu}_0} , t) \Big)
	=
	-\bar{\mu}_0^{-2} \bar{\mu}_{0t}
	\Phi_0(\frac{x-\xi}{\bar{\mu}_0} , t)
	- \bar{\mu}_0^{-2} \nabla_{\bar{y}} \Phi_0(\frac{x-\xi}{\bar{\mu}_0} , t) \cdot \big(\bar{\mu}_{0t} \frac{x-\xi}{\bar{\mu}_0} + \xi_{t} \big)
	+
	\bar{\mu}_0^{-1} \pp_{t} \Phi_0(\frac{x-\xi}{\bar{\mu}_0} , t) .
\end{equation*}

Set $\mu=\bar{\mu}_0 +\mu_1$ where $|\mu_1| \le \frac{\bar{\mu}_0}{2}$, $|\mu_{1t}| \le \frac{ | \bar{\mu}_{0t} |}{2}$. Let us estimate the new error
	\begin{align}\label{S-1}
		&
		S\Big[u_1 + \varphi[\mu] +  \bar{\mu}_0^{-1} \Phi_0(\frac{x-\xi}{\bar{\mu}_0} , t) \eta(\frac{4(x-\xi) }{\sqrt t})
		\Big]
		\\\notag
= \ & - \pp_t \Big(\bar{\mu}_0^{-1} \Phi_0(\frac{x-\xi}{\bar{\mu}_0} , t) \eta(\frac{ 4(x-\xi) }{\sqrt t}) \Big) +
\Delta_x \Big(\bar{\mu}_0^{-1} \Phi_0(\frac{x-\xi}{\bar{\mu}_0} , t) \eta(\frac{ 4(x-\xi) }{\sqrt t})\Big) + S[u_1 +\varphi[\mu] ]
\\\notag
&
+ \Big(u_1 + \varphi[\mu] + \bar{\mu}_0^{-1} \Phi_0(\frac{x-\xi}{\bar{\mu}_0} , t) \eta(\frac{ 4 (x-\xi) }{\sqrt t}) \Big)^3 - (u_1+\varphi[\mu] )^3
\\\notag
= \ &  \bar{\mu}_0^{-1} \Phi_0(\frac{x-\xi}{\bar{\mu}_0} , t) \nabla \eta(\frac{4(x-\xi)}{\sqrt t}) \cdot
\big(4 t^{-\frac 12} \xi_t + 2 t^{-\frac{3}{2}}(x-\xi)\big)
\\\notag
&
-  \eta(\frac{4(x-\xi) }{\sqrt t})
\Big(-\bar{\mu}_0^{-2} \bar{\mu}_{0t}
\Phi_0(\frac{x-\xi}{\bar{\mu}_0} , t)
- \bar{\mu}_0^{-2} \nabla_{\bar{y}} \Phi_0(\frac{x-\xi}{\bar{\mu}_0} , t) \cdot (\bar{\mu}_{0t} \frac{x-\xi}{\bar{\mu}_0} + \xi_{t} )
+
\bar{\mu}_0^{-1} \pp_{t} \Phi_0(\frac{x-\xi}{\bar{\mu}_0} , t)\Big)
\\\notag
&
 +
\bar{\mu}_0^{-3}
\Delta_{\bar{y}} \Phi_0(\frac{x-\xi}{\bar{\mu}_0} , t) \eta(\frac{4(x-\xi)}{\sqrt t})
+
8 t^{-\frac 12} \bar{\mu}_0^{-2} \nabla_{\bar{y}}  \Phi_0(\frac{x-\xi}{\bar{\mu}_0} , t)\cdot \nabla\eta (\frac{4(x-\xi)}{ \sqrt t})
\\\notag
&
 + 16 t^{-1} \bar{\mu}_0^{-1} \Phi_0(\frac{x-\xi}{\bar{\mu}_0} , t) \Delta\eta (\frac{4(x-\xi) }{\sqrt t})
 + S[u_1 +\varphi[\mu] ]  + 3 \bar{\mu}_0^{-3} w^2(\frac{x-\xi}{\bar{\mu}_0 })  \Phi_0(\frac{x-\xi}{\bar{\mu}_0} , t) \eta(\frac{4(x-\xi) }{\sqrt t})
\\\notag
&
+ \Big(u_1 + \varphi[\mu] + \bar{\mu}_0^{-1} \Phi_0(\frac{x-\xi}{\bar{\mu}_0} , t) \eta(\frac{4(x-\xi)}{ \sqrt t}) \Big)^3 - (u_1+\varphi[\mu] )^3
-
3 \bar{\mu}_0^{-3} w^2(\frac{x-\xi}{\bar{\mu}_0 })  \Phi_0(\frac{x-\xi}{\bar{\mu}_0} , t) \eta(\frac{4(x-\xi)}{\sqrt t})
\\\notag
= \ &  \bar{\mu}_0^{-1} \Phi_0(\frac{\bar{x}}{\bar{\mu}_0} , t) \nabla \eta (\frac{4\bar{x}}{ \sqrt t}) \cdot
\big( 4 t^{-\frac 12} \xi_t +  2  t^{-\frac{3}{2}} \bar{x} \big)
\\\notag
&
+
8t^{-\frac 12} \bar{\mu}_0^{-2} \nabla_{\bar{y}}  \Phi_0(\frac{\bar{x}}{\bar{\mu}_0} , t)\cdot \nabla\eta (\frac{4\bar{x}}{\sqrt t})
+
16 t^{-1} \bar{\mu}_0^{-1} \Phi_0(\frac{\bar{x}}{\bar{\mu}_0} , t) \Delta\eta (\frac{4\bar{x}}{ \sqrt t})
\\\notag
&
-  \eta(\frac{4\bar{x}}{ \sqrt t})
\Big[-\bar{\mu}_0^{-2} \bar{\mu}_{0t}
\Phi_0(\frac{\bar{x}}{\bar{\mu}_0} , t)
- \bar{\mu}_0^{-2} \nabla_{\bar{y}} \Phi_0(\frac{\bar{x}}{\bar{\mu}_0} , t) \cdot (\bar{\mu}_{0t} \frac{\bar{x}}{\bar{\mu}_0} + \xi_{t} )
+
\bar{\mu}_0^{-1} \pp_{t} \Phi_0(\frac{\bar{x}}{\bar{\mu}_0} , t)\Big]
\\\notag
&
+
3\Big(\eta^2(\frac{\bar{x}}{\sqrt t}) - \eta(\frac{4\bar{x}}{\sqrt t}) \Big)
\mu^{-2} w^2(\frac{\bar{x}}{\mu})  \varphi[\mu]( \bar{x},t)
+ 3 \Big(\eta(\frac{\bar{x}}{\sqrt t})
- \eta(\frac{4\bar{x}}{\sqrt t})
\Big)
  \mu^{-1} w(\frac{\bar{x}}{\mu})  \varphi^2[\mu]( \bar{x},t)
\\\notag
	&
+
3 \eta(\frac{4\bar{x}}{\sqrt t})
\Big(
\mu^{-2} w^2(\frac{\bar{x}}{\mu})  \varphi[\mu]( \bar{x},t)   -    \bar{\mu}_0^{-2} w^2( \frac{\bar{x}}{\bar{\mu}_0 }) \varphi[\bar{\mu}_0]( \bar{x},t)
\Big)
\\\notag
&
+ 3\eta(\frac{4\bar{x}}{\sqrt t}) \Big(  \mu^{-1} w(\frac{\bar{x}}{\mu})  \varphi^2[\mu]( \bar{x},t)
-
\bar{\mu}_0^{-1}  w(  \frac{\bar{x}}{\bar{\mu}_0 } )  \varphi^2[\bar{\mu}_0] ( \bar{x},t)
\Big)
\\\notag
 &
 +  \xi_t\cdot \nabla_{\bar{x}} \varphi[\mu](\bar{x},t)
+ \mu^{-2} \xi_t \cdot \nabla w (\frac{\bar{x}}{\mu})
\eta(\frac{\bar{x}}{\sqrt t})
+ \mu^{-1} t^{-\frac 12}  w(\frac{\bar{x}}{\mu})
\xi_t \cdot
\nabla \eta (\frac{\bar{x}}{\sqrt t} )
\\\notag
& + \varphi^3[\mu]
+ 3\bar{\mu}_0^{-2} \mathcal{M}[\bar{\mu}_0]
\big(\int_{B_{2} } \eta(z) Z_{5}^2(z) d z\big)^{-1}
\eta(\frac{\bar{x}}{\bar{\mu}_0}) Z_{5}(\frac{\bar{x}}{\bar{\mu}_0})
\\\notag
&
+ \Big(u_1 + \varphi[\mu] + \bar{\mu}_0^{-1} \Phi_0(\frac{\bar{x}}{\bar{\mu}_0} , t) \eta(\frac{4\bar{x}}{\sqrt t}) \Big)^3 - (u_1+\varphi[\mu] )^3
-
3 \bar{\mu}_0^{-3} w^2(\frac{\bar{x}}{\bar{\mu}_0 })  \Phi_0(\frac{\bar{x}}{\bar{\mu}_0} , t) \eta(\frac{4\bar{x}}{\sqrt t}) .
	\end{align}
Claim:
\begin{equation}\label{S2-est}
	\begin{aligned}
		&
		\bigg|	S\Big[u_1 + \varphi[\mu] +  \bar{\mu}_0^{-1} \Phi_0(\frac{x-\xi}{\bar{\mu}_0} , t) \eta(\frac{4(x-\xi) }{\sqrt t})
		\Big] \bigg|
		\lesssim
		t^{-2}   \langle \bar{y} \rangle^{-2}
		\ln (2+ |\bar{y}|) 	\1_{\{ |x|\le 9t^{\frac 12} \}}
		\\
		&
		+
		\bigg[ t^{-1}(\ln t)^{2} |\mu_1|
		    +  (\ln t)^2
		    \Big(|\tilde{g}[\bar{\mu}_0,\mu_1] | +
		    |\bar{\mu}_{0t}| \ln t
		    \sup\limits_{t_1\in[t/2,t]} \Big(\frac{|\mu_1(t_1)|}{\bar{\mu}_{0}(t) } + \frac{|\mu_{1t}(t_1)|}{|\bar{\mu}_{0t}(t)|}  \Big)
		    \Big)
		\bigg] \langle y\rangle^{-4}  \1_{\{ |x|\le 9t^{\frac 12} \}}
		\\
		& +   |\xi_t|
		(\ln t)^{2} \langle y \rangle^{-3} \1_{\{ |x| \le 2t^{\frac 12}\}}  +
		\Big[
		|\xi_t| t^{\frac 32} (\ln t)^{-1} |x|^{-6}
		+ ( t^{2} (\ln t)^{-1}
		| x |^{-6}
		)^3
		\Big] \1_{\{ | x | > 2t^{\frac 12}  \}}
		+ t^{-2} (\ln t)^2 \eta(\bar{y}) .
	\end{aligned}
\end{equation}
We need to estimate term by term. Indeed, by \eqref{Phi0-est}, one has
\begin{equation*}
	\begin{aligned}
		&
		\Big| \bar{\mu}_0^{-1} \Phi_0(\frac{\bar{x}}{\bar{\mu}_0} , t)  \nabla \eta (\frac{4\bar{x}}{ \sqrt t}) \cdot
		( 4 t^{-\frac 12} \xi_t +  2  t^{-\frac{3}{2}} \bar{x} ) \Big|
		\lesssim
	 (t \ln t)^{-1} \langle \bar{y} \rangle^{-2}
		\ln (2+ |\bar{y}|)
		( t^{- \frac 12} |\xi_t| + t^{-1})
		\1_{\{ 9^{-1} t^{\frac 12} \le |x| \le 9 t^{\frac 12}\}}
		\\
			\sim \ &
	\left[t^{- \frac 52} (\ln t)^{-2} |\xi_t| + t^{-3 } (\ln t)^{-2}\right]
		\1_{\{ 9^{-1} t^{\frac 12} \le |x| \le 9 t^{\frac 12}\}}
		 .
	\end{aligned}
\end{equation*}
Also, we have
\begin{equation*}
	\begin{aligned}
		&
		\Big|
8t^{-\frac 12} \bar{\mu}_0^{-2} \nabla_{\bar{y}}  \Phi_0(\frac{\bar{x}}{\bar{\mu}_0} , t)\cdot  \nabla\eta (\frac{4\bar{x}}{\sqrt t})
+
16 t^{-1} \bar{\mu}_0^{-1} \Phi_0(\frac{\bar{x}}{\bar{\mu}_0} , t) \Delta\eta (\frac{4\bar{x}}{ \sqrt t})
		\Big|
		\\
		\lesssim \ &
		\Big|
		t^{-\frac{3}{2}}
		\langle \bar{y} \rangle^{-3}
		\ln (2+ |\bar{y}|)
		+
		t^{-2} (\ln t)^{-1} \langle \bar{y} \rangle^{-2}
		\ln (2+ |\bar{y}|) \Big|
		\1_{\{ 9^{-1} t^{\frac 12} \le | x | \le 9 t^{\frac 12} \}}
		\sim
		t^{-3} (\ln t)^{-2}
			\1_{\{ 9^{-1} t^{\frac 12} \le | x | \le 9 t^{\frac 12} \}} ,
	\end{aligned}
\end{equation*}
\begin{equation*}
	\begin{aligned}
		&
		\Big|\bar{\mu}_0^{-2}  \bar{\mu}_{0t}
		\Phi_0(\frac{\bar{x}}{\bar{\mu}_0} , t)
		+
		\bar{\mu}_0^{-2} \nabla_{\bar{y}} \Phi_0(\frac{\bar{x}}{\bar{\mu}_0} , t) \cdot (  \bar{\mu}_{0t} \frac{\bar{x}}{\bar{\mu}_0} +  \xi_t )
		-
		\bar{\mu}_0^{-1} \pp_{t} \Phi_0(\frac{\bar{x}}{\bar{\mu}_0} , t) \Big| \eta(\frac{4\bar{x}}{ \sqrt t})
		\\
		\lesssim \ &
	\Big[
		(t\ln t)^{-2} \langle \bar{y} \rangle^{-2}
		\ln (2+ |\bar{y}|)
		+
		t^{-1} |  \xi_{t} |
		\langle \bar{y} \rangle^{-3}
		\ln (2+ |\bar{y}|)
		+
		t^{-2}   \langle \bar{y} \rangle^{-2}
		\ln (2+ |\bar{y}|) \Big]
\eta(\frac{4\bar{x}}{ \sqrt t})
\\
\sim \ &
\Big( t^{-1} |  \xi_{t} |
\langle \bar{y} \rangle^{-3}
\ln (2+ |\bar{y}|)
+
t^{-2}   \langle \bar{y} \rangle^{-2}
\ln (2+ |\bar{y}|) \Big)
\eta(\frac{4\bar{x}}{ \sqrt t}) .
	\end{aligned}
\end{equation*}

By Corollary \ref{varphi-coro}, we have the following estimates
\begin{equation*}
\begin{aligned}
&
\bigg|
3\big(\eta^2(\frac{\bar{x}}{\sqrt t}) - \eta(\frac{4\bar{x}}{\sqrt t}) \big)
\mu^{-2} w^2(\frac{\bar{x}}{\mu})  \varphi[\mu]( \bar{x},t)
+ 3 \big(\eta(\frac{\bar{x}}{\sqrt t})
- \eta(\frac{4\bar{x}}{\sqrt t})
\big)
\mu^{-1} w(\frac{\bar{x}}{\mu})  \varphi^2[\mu]( \bar{x},t)   \bigg|
\\
\lesssim \ &
\Big[(\ln t)^2
\langle y\rangle^{-4} (t\ln t)^{-1}
+ \ln t \langle y\rangle^{-2} (t\ln t)^{-2} \Big] \1_{\{ 9^{-1} t^{\frac 12} \le |\bar{x}| \le 9 t^{\frac 12} \}}
\sim
(t \ln t)^{-3} \1_{\{ 9^{-1} t^{\frac 12} \le |\bar{x}| \le 9 t^{\frac 12} \}} ,
\end{aligned}
\end{equation*}
\begin{equation*}
\begin{aligned}
&
\bigg| \eta(\frac{4\bar{x}}{ \sqrt t})
(
\mu^{-2} w^2(\frac{\bar{x}}{\mu})  \varphi[\mu]( \bar{x},t)   -    \bar{\mu}_0^{-2} w^2( \frac{\bar{x}}{\bar{\mu}_0 }) \varphi[\bar{\mu}_0]( \bar{x},t)
) \bigg|
\\
= \ &
\bigg| \eta(\frac{4\bar{x}}{ \sqrt t})
\Big[
\big(\mu^{-2} w^2(\frac{\bar{x}}{\mu}) -
\bar{\mu}_0^{-2} w^2( \frac{\bar{x}}{\bar{\mu}_0 }) \big) \varphi[\mu]   +    \bar{\mu}_0^{-2} w^2( \frac{\bar{x}}{\bar{\mu}_0 }) ( \varphi[\mu] - \varphi[\bar{\mu}_0] )
\Big] \bigg|
\\
\lesssim \ &
\eta(\frac{4\bar{x}}{ \sqrt t})
\bigg[
 t^{-1}(\ln t)^{2} |\mu_1|
\langle y\rangle^{-4}     +    (\ln t)^2      \langle y\rangle^{-4}
\Big(|\tilde{g}[\bar{\mu}_0,\mu_1] | +
(t \ln t)^{-1}
\sup\limits_{t_1\in[t/2,t]} \Big(\frac{|\mu_1(t_1)|}{\bar{\mu}_{0}(t) } + \frac{|\mu_{1t}(t_1)|}{|\bar{\mu}_{0t}(t)|}  \Big)
 \Big) \bigg],
\end{aligned}
\end{equation*}
 \begin{align*}
 &
\eta(\frac{4\bar{x}}{ \sqrt t})  \Big| \big(  \mu^{-1} w(\frac{\bar{x}}{\mu})  \varphi^2[\mu]( \bar{x},t)
 -
 \bar{\mu}_0^{-1}  w(  \frac{\bar{x}}{\bar{\mu}_0 } )  \varphi^2[\bar{\mu}_0] ( \bar{x},t)
 \big)  \Big|
 \\
= \ &
\eta(\frac{4\bar{x}}{ \sqrt t})
\Big|    \big( \mu^{-1} w(\frac{\bar{x}}{\mu})
-
\bar{\mu}_0^{-1}  w(  \frac{\bar{x}}{\bar{\mu}_0 } )
\big) \varphi^2[\mu]
+
\bar{\mu}_0^{-1}  w(  \frac{\bar{x}}{\bar{\mu}_0 } )
( \varphi[\mu]  - \varphi[\bar{\mu}_0]
) ( \varphi[\mu]  + \varphi[\bar{\mu}_0]
)
 \Big|
\\
\lesssim \ &
\eta(\frac{4\bar{x}}{ \sqrt t})
\bigg[
 \bar{\mu}_0^{-2} |\mu_1| \langle y \rangle^{-2} (t\ln t)^{-2}
+
\bar{\mu}_0^{-1}  \langle y \rangle^{-2}
(t\ln t)^{-1} \Big(
|\mu_1|t^{-1} + |\tilde{g}[\bar{\mu}_0,\mu_1]|
\\
&
+ |\mu_{t}| \ln t
\sup\limits_{t_1\in[t/2,t]} \Big(\frac{|\mu_1(t_1)|}{ \mu(t) } + \frac{|\mu_{1t}(t_1)|}{|\mu_t(t)|}  \Big)\Big)
\bigg]
\\
\sim \ &
\eta(\frac{4\bar{x}}{ \sqrt t})
\bigg[
 t^{-2} |\mu_1| \langle y \rangle^{-2}
+
t^{-1}    \langle y \rangle^{-2}
\Big(|\tilde{g}[\bar{\mu}_0,\mu_1] | + (t \ln t)^{-1}
\sup\limits_{t_1\in[t/2,t]} \Big(\frac{|\mu_1(t_1)|}{ \mu(t) } + \frac{|\mu_{1t}(t_1)|}{|\mu_t(t)|}  \Big) \Big)
\bigg]
\\
\lesssim \ &
\eta(\frac{4\bar{x}}{ \sqrt t})
\bigg[
t^{-1}(\ln t)^{2} |\mu_1|
\langle y\rangle^{-4}     +    (\ln t)^2      \langle y\rangle^{-4} \Big(|\tilde{g}[\bar{\mu}_0,\mu_1] |
+ (t \ln t)^{-1}
\sup\limits_{t_1\in[t/2,t]} \Big(\frac{|\mu_1(t_1)|}{ \mu(t) } + \frac{|\mu_{1t}(t_1)|}{|\mu_t(t)|}  \Big) \Big)
\bigg].
 \end{align*}

Using Corollary \ref{varphi-coro}, one has
\begin{equation*}
\begin{aligned}
&
\Big| \xi_t\cdot \nabla_{\bar{x}} \varphi[\mu](\bar{x},t)
+ \mu^{-2} \xi_t \cdot \nabla w(\frac{\bar{x}}{\mu})
\eta(\frac{\bar{x}}{\sqrt t})
+ \mu^{-1} t^{-\frac 12}  w(\frac{\bar{x}}{\mu})
\xi_t \cdot
\nabla \eta(\frac{\bar{x}}{\sqrt t} )
\Big|\\
\lesssim \ &
|\xi_t| \Big(t^{-1} \langle y \rangle^{-1} \1_{\{ |\bar{x}| \le 2t^{\frac 12}\}} +
t^{\frac 32} (\ln t)^{-1} |\bar{x}|^{-6}
\1_{\{ |\bar{x}| > 2t^{\frac 12}\}}
+
(\ln t)^{2} \langle y \rangle^{-3} \1_{\{ |\bar{x}| \le 2t^{\frac 12}\}}
 \Big)
 \\
\sim \ &
|\xi_t| \Big(
(\ln t)^{2} \langle y \rangle^{-3} \1_{\{ |x| \le 2t^{\frac 12}\}}  +
t^{\frac 32} (\ln t)^{-1} |x|^{-6}
\1_{\{ |x| > 2t^{\frac 12}\}}
\Big) .
\end{aligned}
\end{equation*}

\begin{equation*}
|\varphi^3[\mu] | \lesssim
(t\ln t)^{-3} \1_{\{ |\bar{x}| \le 2t^{\frac 12} \}}
+
( t^{2} (\ln t)^{-1}
|\bar{x}|^{-6}
)^3 \1_{\{ |\bar{x}| > 2t^{\frac 12}  \}}  ,\
\Big|\bar{\mu}_0^{-2} \mathcal{M}[\bar{\mu}_0]
\frac{\eta(\bar{y}) Z_{5}(\bar{y}) }{\int_{B_{2} } \eta(z) Z_{5}^2(z) d z} \Big| \lesssim
  t^{-2} (\ln t)^2 \eta(\bar{y}) .
\end{equation*}

\begin{equation*}
\begin{aligned}
&
\bigg| \Big(u_1 + \varphi[\mu] + \bar{\mu}_0^{-1} \Phi_0(\frac{\bar{x}}{\bar{\mu}_0} , t) \eta(\frac{4\bar{x}}{\sqrt t}) \Big)^3 - (u_1+\varphi[\mu] )^3
-
3 \bar{\mu}_0^{-3} w^2(\frac{\bar{x}}{\bar{\mu}_0 })  \Phi_0(\frac{\bar{x}}{\bar{\mu}_0} , t) \eta(\frac{4\bar{x}}{\sqrt t})  \bigg|
\\
= \ &
3\bigg|  \Big( \mu^{-1} w(\frac{\bar{x}}{ \mu }) -
\bar{\mu}_0^{-1} w(\frac{\bar{x}}{\bar{\mu}_0 })   + \varphi[\mu] +
\theta \bar{\mu}_0^{-1} \Phi_0(\frac{\bar{x}}{\bar{\mu}_0} , t) \eta(\frac{4\bar{x}}{\sqrt t}) \Big)
\\
& \times
\Big(u_1 + \varphi[\mu] +
\theta \bar{\mu}_0^{-1} \Phi_0(\frac{\bar{x}}{\bar{\mu}_0} , t) \eta(\frac{4\bar{x}}{\sqrt t}) +
\bar{\mu}_0^{-1} w(\frac{\bar{x}}{\bar{\mu}_0 }) \Big)
\bar{\mu}_0^{-1} \Phi_0(\frac{\bar{x}}{\bar{\mu}_0} , t)
\eta(\frac{4\bar{x}}{\sqrt t}) \bigg|
\\
\lesssim \ &
 \Big(  |\mu_1| \mu_0^{-2} \langle y \rangle^{-2}  + (t\ln t)^{-1} +
t^{-1} (\ln t)^{-1} \langle \bar{y} \rangle^{-2}
\ln (2+ |\bar{y}|)
\Big)
\\
& \times
\big((t\ln t)^{-1}  +
\ln t \langle y\rangle^{-2}  \big)
 (t \ln t)^{-1} \langle \bar{y} \rangle^{-2}
\ln (2+ |\bar{y}|)
\eta(\frac{4\bar{x}}{\sqrt t})
\\
\lesssim \ &
\big(  |\mu_1| (\ln t)^2 \langle y \rangle^{-2}  + (t\ln t)^{-1}
\big)
t^{-1}  \langle y\rangle^{-4}
\ln (2+ |y|)
\eta(\frac{4\bar{x}}{\sqrt t})
\\
= \ &
|\mu_1|
t^{-1} (\ln t)^2    \langle y\rangle^{-6}
\ln (2+ |y|)
\eta(\frac{4\bar{x}}{\sqrt t})
+
t^{-2}  (\ln t)^{-1}   \langle y\rangle^{-4}
\ln (2+ |y|)
\eta(\frac{4\bar{x}}{\sqrt t})
.
\end{aligned}
\end{equation*}
We have completed the proof of claim \eqref{S2-est}.

\medskip

\section{Gluing system and solving the outer problem}

In this section, we formulate the inner--outer gluing system such that an infinite time blow-up solution
to \eqref{nonlinear-heateq} with desired asymptotics can be found. We look for solution of the form
$$
u(x,t)=u_1 + \varphi[\mu] +  \bar{\mu}_0^{-1} \Phi_0(\frac{x-\xi}{\bar{\mu}_0} , t) \eta(\frac{4(x-\xi) }{\sqrt t}) + \psi(x,t) + \eta_R  \mu^{-1} \phi(\frac{x-\xi}{\mu},t)
$$
with
\[
\eta_R (x,t)= \eta(\frac{x-\xi}{\mu_{0} R(t)}),
\quad
R(t) = t^{\gamma},
\quad
 0<\gamma<\frac 12,
\]
where $\psi$, $\phi$ are perturbations in the outer region and inner region, respectively. In order for the following to hold
\begin{align*}
& 0 =
S\Big[ u_1 + \varphi[\mu] +  \bar{\mu}_0^{-1} \Phi_0(\frac{x-\xi}{\bar{\mu}_0} , t) \eta(\frac{4(x-\xi) }{\sqrt t}) + \psi + \eta_R  \mu^{-1} \phi(\frac{x-\xi}{\mu},t)  \Big]
\\
= \ &
- \pp_t \psi  - \pp_t \eta_R \mu^{-1}\phi(\frac {x-\xi}{\mu},t)
+ \eta_R \mu^{-2}\mu_{t} \big( \phi (\frac {x-\xi}{\mu},t) +\frac{x-\xi}{\mu}\cdot\nabla_y \phi(\frac {x-\xi}{\mu},t) \big)
\\
&
+
\eta_R \mu^{-2} \xi_{t} \cdot \nabla_y \phi(\frac{x-\xi}{\mu},t)
-
\eta_R \mu^{-1} \pp_t \phi(\frac {x-\xi}{\mu},t)
\\
&
+
\Delta_x \psi + \Delta_x \eta_R \mu^{-1}\phi(\frac {x-\xi}{\mu},t)
+ 2 \nabla_x \eta_R \cdot \mu^{-2} \nabla_y \phi(\frac {x-\xi}{\mu},t)
+ \eta_R \mu^{-3} \Delta_y \phi(\frac {x-\xi}{\mu},t)
\\
& + 3 \big(\mu^{-1} w(\frac{x-\xi}{\mu})\big)^2 \big( \psi + \eta_R  \mu^{-1} \phi(\frac{x-\xi}{\mu},t)  \big)
+
S\big[u_1 + \varphi[\mu] +  \bar{\mu}_0^{-1} \Phi_0(\frac{x-\xi}{\bar{\mu}_0} , t)
\eta(\frac{4(x-\xi) }{\sqrt t}) \big]
\\
&
+
\Big(u_1 + \varphi[\mu] +  \bar{\mu}_0^{-1} \Phi_0(\frac{x-\xi}{\bar{\mu}_0} , t) \eta(\frac{4(x-\xi) }{\sqrt t}) + \psi + \eta_R  \mu^{-1} \phi(\frac{x-\xi}{\mu},t) \Big)^3
\\
&
 -
\Big(u_1 + \varphi[\mu] +  \bar{\mu}_0^{-1} \Phi_0(\frac{x-\xi}{\bar{\mu}_0} , t) \eta(\frac{4(x-\xi) }{\sqrt t})  \Big)^3
- 3 \big(\mu^{-1} w(\frac{x-\xi}{\mu})\big)^2 \big( \psi + \eta_R  \mu^{-1} \phi(\frac{x-\xi}{\mu},t)  \big) ,
\end{align*}
it suffices to solving the following inner-outer gluing system for $(\psi,\phi)$.

The outer problem for $\psi$:
\begin{equation*}
\pp_t \psi   =
 \Delta_x \psi + \mathcal{G} [\psi,\phi,\mu_1,\xi]
 \mbox{ \ \ in \ } \RR^4 \times (t_0,\infty),
\end{equation*}
where
\begin{equation}\label{G}
\begin{aligned}
& \quad \mathcal{G} [\psi,\phi,\mu_1,\xi] :=
 3 \mu^{-2} w^2(\frac{x-\xi}{\mu})  \psi (1-\eta_R)
+
\eta_R \mu^{-2} \xi_{t} \cdot \nabla_y \phi(\frac{x-\xi}{\mu},t)
\\
&
+ \Delta_x \eta_R \mu^{-1}\phi(\frac {x-\xi}{\mu},t)
+ 2 \nabla_x \eta_R \cdot \mu^{-2} \nabla_y \phi(\frac {x-\xi}{\mu},t)
- \pp_t \eta_R \mu^{-1}\phi(\frac {x-\xi}{\mu},t)
\\
& + (1-\eta_R) S\big[u_1 + \varphi[\mu] +  \bar{\mu}_0^{-1} \Phi_0(\frac{x-\xi}{\bar{\mu}_0} , t)
\eta(\frac{4(x-\xi) }{\sqrt t}) \big]
\\
&
+
\Big(u_1 + \varphi[\mu] +  \bar{\mu}_0^{-1} \Phi_0(\frac{x-\xi}{\bar{\mu}_0} , t) \eta(\frac{4(x-\xi) }{\sqrt t}) + \psi + \eta_R  \mu^{-1} \phi(\frac{x-\xi}{\mu},t) \Big)^3
\\
&
-
\Big(u_1 + \varphi[\mu] +  \bar{\mu}_0^{-1} \Phi_0(\frac{x-\xi}{\bar{\mu}_0} , t) \eta(\frac{4(x-\xi) }{\sqrt t}) \Big)^3
- 3 \big(\mu^{-1} w(\frac{x-\xi}{\mu})\big)^2 \big( \psi + \eta_R  \mu^{-1} \phi(\frac{x-\xi}{\mu},t)  \big)
\\
& -
\bigg[ 3 \big(u_1 + \varphi[\mu] -\mu^{-1} w(\frac{x-\xi}{\mu}) \big)
\big(u_1 + \varphi[\mu] +\mu^{-1} w(\frac{x-\xi}{\mu}) \big)
\\
& +
6 (u_1 + \varphi[\mu] )   \bar{\mu}_0^{-1} \Phi_0(\frac{x-\xi}{\bar{\mu}_0} , t) \eta(\frac{4(x-\xi)}{\sqrt t})
\bigg] \eta_R  \mu^{-1} \phi(\frac{x-\xi}{\mu},t)
.
\end{aligned}
\end{equation}

The inner problem for $\phi$:
\begin{equation}\label{phi-eq}
	\begin{aligned}
		&
	 \mu^{2} \pp_t \phi(y,t)
		=    \Delta_y \phi(y,t)
		+ 3 w^2(y)   \phi(y,t)
		+ f_1(y,t) \phi (y,t) + f_2(t) y\cdot\nabla_y \phi(y,t)
		+
		\mathcal{H}[\psi,\mu_1,\xi](y,t)
		\mbox{ \ \  in \ } \DD_{4R}
	\end{aligned}
\end{equation}
where $y= \frac{x-\xi}{\mu}$, $\mathcal{D}_{4R} = \{
(y,t) \ : \
y\in B_{4R(t)}, \ t \in (t_0,\infty)
\}$, and
\begin{equation*}
\begin{aligned}
f_{1}(y,t) = \mu \mu_{t}  + \mu^{2} \Big[ 3 \big(u_1 + \varphi[\mu] -\mu^{-1} w(y) \big)
\big(u_1 + \varphi[\mu] +\mu^{-1} w( y ) \big)
+
6 (u_1 + \varphi[\mu] )   \bar{\mu}_0^{-1} \Phi_0(\frac{ \mu y }{\bar{\mu}_0} , t)
\Big] , \
f_{2}(t) =  \mu \mu_{t},
\end{aligned}
\end{equation*}
\begin{equation}\label{H-fun}
\mathcal{H}[\psi,\mu_1,\xi](y,t) := \mu^3 \Big( 3 \mu^{-2} w^2(y)  \psi(\mu y +\xi,t)   +
 S\big[u_1 + \varphi[\mu] +  \bar{\mu}_0^{-1} \Phi_0(\frac{\mu y }{\bar{\mu}_0} , t)
\eta(\frac{4\mu y  }{\sqrt t}) \big] \Big) .
\end{equation}
By \eqref{S2-est}, one has
\begin{equation}\label{H-upp}
	\begin{aligned}
		&\big|\mathcal{H}[\psi,\mu_1,\xi](y,t)\big| \\
		\lesssim \ &
		(\ln t)^{-1} \langle y\rangle^{-4}
		| \psi(\mu y +\xi,t)  |
		+
		t^{-2} (\ln t)^{-3}  \langle \bar{y} \rangle^{-2}
		\ln (2+ |\bar{y}|) 	
		\\
		&
		+
		(t\ln t)^{-1} |\mu_1|
		\langle y\rangle^{-4}     +    (\ln t)^{-1}     \Big(|\tilde{g}[\bar{\mu}_0,\mu_1] | +
		|\bar{\mu}_{0t}| \ln t
		\sup\limits_{t_1\in[t/2,t]} \Big(\frac{|\mu_1(t_1)|}{\bar{\mu}_{0}(t) } + \frac{|\mu_{1t}(t_1)|}{|\bar{\mu}_{0t}(t)|}  \Big) \Big) \langle y\rangle^{-4}
		\\
		& +   |\xi_t|
		(\ln t)^{-1} \langle y \rangle^{-3}
		+ t^{-2} (\ln t)^{-1} \eta(\bar{y})
		\\
		\lesssim \ &
		(\ln t)^{-1} \langle y\rangle^{-4}
		| \psi(\mu y +\xi,t)  |
		+
		t^{a_{1}\gamma-2} (\ln t)^{-2}  \langle y \rangle^{-2-a_{1}}
		\\
		&
		+
		(t\ln t)^{-1} |\mu_1|
		\langle y\rangle^{-4}     +    (\ln t)^{-1}     \Big(|\tilde{g}[\bar{\mu}_0,\mu_1] | +
		|\bar{\mu}_{0t}| \ln t
		\sup\limits_{t_1\in[t/2,t]} \Big(\frac{|\mu_1(t_1)|}{\bar{\mu}_{0}(t) } + \frac{|\mu_{1t}(t_1)|}{|\bar{\mu}_{0t}(t)|}  \Big) \Big) \langle y\rangle^{-4}
		\\
		& +   |\xi_t|
		(\ln t)^{-1} \langle y \rangle^{-3}
	\end{aligned}
\end{equation}
where we have used $|y|\le 4R= 4t^{\gamma}$, and for later purpose, we require that
\begin{equation}\label{inner-para}
a_{1}\gamma -2 <5\delta -\kappa-a\gamma,\
0<a_1\le 1.
\end{equation}
Here above constants are those which measure the weighted topology for the inner problem (see \eqref{norm-phi}).
Notice in $\DD_{4R}$, we have
\begin{equation*}
\begin{aligned}
|f_{1}(y,t) | + |f_{2}(t) |
\lesssim \ &
t^{-1} (\ln t)^{-3} +
	\Big[ (t\ln t)^{-1}
\ln t \langle y\rangle^{-2}
+
\ln t \langle y\rangle^{-2}
t^{-1} (\ln t)^{-1} \langle \bar{y} \rangle^{-2}
\ln (2+ |\bar{y}|)
\Big] (\ln t)^{-2}
\\
\sim \ &
t^{-1} (\ln t)^{-3} +
t^{-1}
(\ln t)^{-2} \langle y\rangle^{-2} .
\end{aligned}
\end{equation*}

\begin{remark}\label{f1f2-remark}
Due to the time decay rate of $f_1(y,t)$, $f_2(t)$, we are forced to put $f_1(y,t) \phi (y,t) + f_2(t) y\cdot\nabla_y \phi(y,t)   $ in the linear part of the inner problem. We can not put this term in the right hand side of the outer problem since this will influence the H\"older continuity of $\psi$ about $t$ variable. Besides, we can not use the inner linear theory in \cite{Green16JEMS} since $f_1(y,t) \phi (y,t) + f_2(t) y\cdot\nabla_y \phi(y,t)   $ will influence the H\"older about $\mu_{1t}$ through the orthogonal equation, which will result in failure to choose suitable topology for solving the inner--outer gluing system. Instead, we rebuild a new inner linear theory in Section \ref{sec-inner} to avoid including $f_1(y,t) \phi (y,t) + f_2(t) y\cdot\nabla_y \phi(y,t)   $  in the orthogonal equation about $\mu_1$.
\end{remark}

We decompose the inner problem \eqref{phi-eq} into two parts. Set $\phi = \phi_{1} + \phi_{2}$, then it suffices to consider
\begin{equation}\label{phi1-eq}
	\begin{aligned}
	& \mu^{2} \pp_t \phi_{1}(y,t)
		=  \Delta_y \phi_{1}(y,t)
		+ 3 w^2(y)   \phi_{1}(y,t)
		+ f_1(y,t) \phi_{1} (y,t) + f_2(t) y\cdot\nabla_y \phi_{1}(y,t)
		+
		\mathcal{H}[\psi,\mu_1,\xi](y,t)
\\
&
\qquad	+
\big(\int_{B_{2} } \eta(z) Z_{5}^2(z) dz\big)^{-1}  \Big(-2^{-\frac 12} 3 \int_{B_{2R_{0}}}   w^2(z) Z_5(z)  dz + O((t\ln t)^{-1} ) \Big) \mu \mathcal{E}_{\nu}[\mu_{1}]
\eta(y) Z_{5}(y)
\quad
		\mbox{ \ in \ } \DD_{4R},
	\end{aligned}
\end{equation}
\begin{equation}\label{phi2-eq}
	\begin{aligned}
	& \mu^{2} \pp_t \phi_{2}(y,t)
		=   \Delta_y \phi_{2}(y,t)
		+ 3 w^2(y)   \phi_{2}(y,t)
		+ f_1(y,t) \phi_{2} (y,t) + f_2(t) y\cdot\nabla_y \phi_{2}(y,t)
		\\
		& \qquad
		-\big(\int_{B_{2} } \eta(z) Z_{5}^2(z) dz\big)^{-1}  \Big(-2^{-\frac 12} 3 \int_{B_{2R_{0}}}   w^2(z) Z_5(z)  dz + O((t\ln t)^{-1} ) \Big) \mu \mathcal{E}_{\nu}[\mu_{1}]
		\eta(y) Z_{5}(y)	\quad
		\mbox{ \ in \ } \DD_{4R},
	\end{aligned}
\end{equation}
where  $R_{0}(\tau) = \tau^{\delta}$ with  $\delta>0$ very small and
\begin{equation*}
	\mathcal{E}_{\nu}[\mu_{1}] = \int_{ t-t^{1-\nu}  }^{t-\mu_{0}^2(t) } \frac{\mu_{1t}(s) - \mu_{1t}(t)}{t-s} d s.
\end{equation*}
Set
\[
\tau(t) = \int_{t_0}^t  \mu^{-2}(s) d s + t_{0} (\ln t_0)^2,\quad \tau_0 = t_{0} (\ln t_0)^2 .
\]
Then $\tau(t) \sim t(\ln t)^2 $ for all $t\ge t_{0}$. In $\tau$ variable,  $\mathcal{D}_{4R} = \{
(y,\tau) \ : \
y\in B_{4R(t(\tau))}, \ t \in (\tau_0,\infty)
\}$.
It is easy to rewrite \eqref{phi1-eq} and \eqref{phi2-eq} in the form as in Proposition \ref{R0-linear} and Lemma \ref{mode0-nonorth}, respectively.

The reason for decomposing the inner problem into above two parts is that the orthogonal equation involving $\mu_1$ is too difficult to solve. More detailed explanations will be given in Section \ref{mu1-xi-subsec}.

\medskip

Before stating the solvability of the outer problem, let us first fix the inner solution $\phi$ to the inner problem, the next order of scaling parameter $\mu_1$  and translating parameter $\xi$ in the spaces with the following norms
\begin{equation}\label{norm-phi}
	\| \phi \|_{i,\kappa -5\delta,a} :=
\sup\limits_{(y,\tau)\in \DD_{4R}}
\tau^{\kappa-5 \delta }
\langle y\rangle^{a}
\big( \langle y\rangle |\nabla \phi(y,t(\tau) )| + |\phi(y,t(\tau) )| \big)
\end{equation}
where  $\kappa$, $a$ are some positive constants to be determined later.

For $ \mu_1(t)\in C^{1}(\frac{t_0}{4},\infty)$, $\mu_{1}(t) \rightarrow 0$ as $t\rightarrow\infty$, denote
\begin{equation}\label{norm-mu}
\| \mu_{1} \|_{*1}:= \sup\limits_{t\ge t_0/4}
\big[ \ln t	(t (\ln t)^{2})^{5\delta-\kappa }
R(t)^{- a}\big]^{-1} |\mu_{1t}|.
\end{equation}
For $\xi(t) = (\xi_1(t),\dots,\xi_4(t)) \in C^1(t_0,\infty)$, $\xi(t)\rightarrow 0$ as $t\rightarrow\infty$, denote
\begin{equation}\label{norm-xi}
\| \xi \|_{*2}:= \max\limits_{1\le j\le 4}\sup\limits_{t\ge t_0}
\big[(\ln t)^2 (t (\ln t)^{2})^{5\delta-\kappa }
R(t)^{- a}\big]^{-1} |\xi_{jt}| .
\end{equation}

The outer problem is solved in the following Proposition.
\begin{prop}\label{psi-prop}
	Consider
	\begin{equation}\label{o-eq}
		\pp_t \psi(x,t)
		=
		\Delta \psi(x,t) + \mathcal{G} [\psi,\phi,\mu_1,\xi]
		\mbox{ \ in \ } \RR^4 \times (t_{0},\infty),\quad \psi(x,t_0)=0 \mbox{ \ in \ } \RR^4
	\end{equation}
	where $\mathcal{G} [\psi,\phi,\mu_1,\xi] $ is given in \eqref{G}. Assume $\phi$, $\mu_1$, $\xi$ satisfy $ \| \phi\|_{i,\kappa-5\delta,a}, \| \mu_{1} \|_{*1}, \| \xi \|_{*2} < \Lambda_1$ where $\Lambda_1>1$ is a constant and the parameters satisfy
	\begin{equation}\label{outer-para}
		5\delta -\kappa -a\gamma >-2,\  5\delta-\kappa < -1 , \ 0<a<2,
		\  0<\gamma<\frac 12 ,
	\end{equation}
	then
	there exists a solution $\psi=\psi[\phi,\mu_1,\xi]$ with the following estimates:
\begin{equation*}
\begin{aligned}
|\psi(x,t)| \le \ &  C(\Lambda_1)
\ln t (t (\ln t)^{2})^{5\delta-\kappa }
R^{- a}  \left(\1_{\{|x|\le t^{\frac 12} \}} + t |x|^{-2} \1_{\{|x| > t^{\frac 12} \}}\right),
\\
|\nabla \psi(x,t)|
\le \ &
 C(\Lambda_1)
\ln t (t (\ln t)^{2})^{5\delta-\kappa }
R^{- a},
\end{aligned}
\end{equation*}
\begin{equation*}
	\begin{aligned}
		&
		\sup\limits_{s_1, s_2\in ( t -\frac{ \lambda^2(t)}{4} ,t)}
		\frac{|\psi(x , s_1) -\psi(x ,  s_2) |}{| s_1 -  s_2|^{\alpha}}
		\le
		C(\Lambda_1,\alpha)
		\bigg\{
		\lambda^{-2\alpha}(t) \ln t (t (\ln t)^{2})^{5\delta-\kappa }
		R^{- a}
		\\
		& \qquad
		+
		\lambda^{2-2\alpha}(t)
		\left[ (\mu_{0} R)^{-2}
		\ln t (t (\ln t)^{2})^{5\delta-\kappa }
		R^{-a}
		+
		(\ln t)^3 (t (\ln t)^{2})^{10\delta-2\kappa }
		\right] \bigg\}
	\end{aligned}
\end{equation*}
	where $0<\lambda(t) \le t^{\frac 12}$.
\end{prop}
The proof is postponed to Section \ref{sec-outer-proof}.

\medskip

\section{Orthogonal equations for $\mu_1$, $\xi$}
\subsection{Solving $\mu_1$ and $\xi$}\label{mu1-xi-subsec}
In order to utilize Proposition \ref{R0-linear} with $R_0=\tau^{\delta}\sim (t(\ln t)^2)^{\delta}$ where $\delta>0$ is small,
one needs to adjust $\mu_1$, $\xi$ such that $c_i[\mathcal{H}]= 0$, $i=1,\dots,5$ in Proposition \ref{R0-linear} with $\mathcal{H}$  given in \eqref{H-fun}.

However, for $i=5$, it is too difficult to solve $c_5[\mathcal{H}]= 0$ thoroughly. We are only able to make
$c_5[\mathcal{H}]  \approx 0$ and  leave smaller remainder to be solved by the non-orthogonal linear theory of the inner problem.

In this section, we only care about the estimate in $|y|\le 4R$ since this is served for the inner problem. Set
\begin{equation*}
\begin{aligned}
	\mathcal{M}_{i}[\psi,\mu_1,\xi]  = \ & \int_{B_{2R_{0}}} \mathcal{H}[\psi,\mu_1,\xi](y,t) Z_i(y) dy,\quad  i =1,\dots,5 ,
	\\
\mathcal{H}_{5}(|y|,t) = \ &
\int_{S^{3}} \mathcal{H}[\psi,\mu_1,\xi](|y|\theta,t) \Upsilon_{0}(\theta) d\theta,
\
\mathcal{H}_{i} (|y|,t) = \int_{S^{3}} \mathcal{H}[\psi,\mu_1,\xi](|y|\theta,t) \Upsilon_{i}(\theta) d\theta,
\quad i=1,\dots,4
\end{aligned}
\end{equation*}
where $\Upsilon_{i}$ are spherical harmonic functions, which are given in Section \ref{sec-inner}.

Using \eqref{S-1}, for $i=5$, since $Z_{5}$ is radially symmetric, one has
	\begin{align*}
		&\quad  \mathcal{M}_{5}[\psi,\mu_1,\xi] =
		\int_{B_{2R_0}} 	3 \mu w^2(y)  \psi(\mu y +\xi,t) Z_5(y) dy
		\\
		&
		+ \int_{B_{ 2R_0 }}
		\mu^3 \Big(
		\bar{\mu}_0^{-2}  \bar{\mu}_{0t}
		\Phi_0(\frac{ \mu y }{\bar{\mu}_0} , t)
		+
		\bar{\mu}_0^{-2} \nabla_{\bar{y}} \Phi_0(\frac{ \mu y }{\bar{\mu}_0} , t) \cdot   \bar{\mu}_{0t} \frac{ \mu y }{\bar{\mu}_0}
		-
		\bar{\mu}_0^{-1} \pp_{t} \Phi_0(\frac{ \mu y }{\bar{\mu}_0} , t)
		\Big)  Z_5(y) dy
		\\
		&
		+ \int_{B_{ 2R_0 }}  3\mu^3 \Big(  \mu^{-2} w^2(y)  \varphi[\mu]( \mu y,t)   -    \bar{\mu}_0^{-2} w^2( \frac{ \mu y }{\bar{\mu}_0 }) \varphi[\bar{\mu}_0]( \mu y ,t)
		\Big)  Z_5(y) dy
		\\
		&
		+ \int_{B_{ 2R_0 }}  3 \mu^3 \Big(  \mu^{-1} w(y)  \varphi^2[\mu]( \mu y,t)
		-
		\bar{\mu}_0^{-1}  w(  \frac{ \mu y }{\bar{\mu}_0 } )  \varphi^2[\bar{\mu}_0] ( \mu y,t)
		\Big) Z_5(y) dy
		\\
		&
		+  \int_{B_{ 2R_0 }}  \mu^3 \varphi^3[\mu]( \mu y,t)  Z_{5}(y) dy
		+  \int_{B_{4R}}  3\mu^3 \bar{\mu}_0^{-2} \mathcal{M}[\bar{\mu}_0]
		\frac{\eta( \bar{\mu}_0^{-1} \mu y ) Z_{5}( \bar{\mu}_0^{-1} \mu y ) }{\int_{B_{2} } \eta(z) Z_{5}^2(z) d z}
		Z_5(y) dy
		\\
		&
		+ \int_{B_{ 2R_0 }}   \mu^3 \bigg[
		 \big(u_1 + \varphi[\mu] + \bar{\mu}_0^{-1} \Phi_0(\frac{ \mu y }{\bar{\mu}_0} , t) \eta(\frac{4 \mu y }{\sqrt t}) \big)^3 - (u_1+\varphi[\mu] )^3
		-
		3 \bar{\mu}_0^{-3} w^2(\frac{ \mu y }{\bar{\mu}_0 })  \Phi_0(\frac{ \mu y }{\bar{\mu}_0} , t) \eta(\frac{4 \mu y }{\sqrt t}) \bigg]  Z_5(y) dy ,
	\end{align*}
and
\begin{equation*}
	\begin{aligned}
		& \quad
		\mathcal{H}_{5}(|y|,t)
		=
		\int_{S^{3}}	3 \mu  w^2(|y|\theta)  \psi(\mu |y|\theta +\xi,t) \Upsilon_{0}(\theta) d\theta
		\\
		&
		+
		\int_{S^{3}}
		\mu^3 \Big(
		\bar{\mu}_0^{-2} \bar{\mu}_{0t}
		\Phi_0(\frac{ \mu |y|\theta }{\bar{\mu}_0} , t)
		+
		\bar{\mu}_0^{-2} \nabla_{\bar{y}} \Phi_0(\frac{ \mu |y|\theta }{\bar{\mu}_0} , t) \cdot  \bar{\mu}_{0t} \frac{ \mu |y|\theta }{\bar{\mu}_0}
		-
		\bar{\mu}_0^{-1} \pp_{t} \Phi_0(\frac{ \mu |y|\theta }{\bar{\mu}_0} , t)
	\Big)  \Upsilon_{0}(\theta) d\theta
		\\
		& +
		\int_{S^{3}}
		3 \mu^3
		\Big(
		\mu^{-2} w^2(|y|\theta)  \varphi[\mu]( \mu |y|\theta ,t)   -    \bar{\mu}_0^{-2} w^2( \frac{ \mu |y|\theta }{\bar{\mu}_0 }) \varphi[\bar{\mu}_0]( \mu |y|\theta,t)
		\Big) \Upsilon_{0}(\theta) d\theta
		\\
		&
		+
		\int_{S^{3}} 3 \mu^3 \Big(  \mu^{-1} w(|y|\theta)  \varphi^2[\mu]( \mu |y|\theta,t)
		-
		\bar{\mu}_0^{-1}  w(  \frac{\mu |y|\theta}{\bar{\mu}_0 } )  \varphi^2[\bar{\mu}_0] ( \mu |y|\theta,t)
		\Big) \Upsilon_{0}(\theta) d\theta
		\\
		&
		+
		\int_{S^{3}}  \mu^3\varphi^3[\mu]( \mu |y|\theta,t)  \Upsilon_{0}(\theta) d\theta
		+
		 \int_{S^{3}} 3 \mu^3
		\bar{\mu}_0^{-2} \mathcal{M}[\bar{\mu}_0]
		\frac{\eta(  \bar{\mu}_0^{-1} \mu |y|\theta ) Z_{5}(   \bar{\mu}_0^{-1} \mu |y|\theta ) }{\int_{B_{2} } \eta(z) Z_{5}^2(z) d z}
		\Upsilon_{0}(\theta) d\theta
		\\
		&
		+ \int_{S^{3}} \mu^3 \bigg[\big(u_1 + \varphi[\mu] + \bar{\mu}_0^{-1} \Phi_0(\frac{\mu |y|\theta}{\bar{\mu}_0} , t) \eta(\frac{4 \mu |y|\theta}{\sqrt t}) \big)^3 - (u_1+\varphi[\mu] )^3
		\\
		&
		-
		3 \bar{\mu}_0^{-3} w^2(\frac{ \mu |y|\theta }{\bar{\mu}_0 })  \Phi_0(\frac{ \mu |y|\theta }{\bar{\mu}_0} , t) \eta(\frac{4 \mu |y|\theta }{\sqrt t})  \bigg] \Upsilon_{0}(\theta) d\theta
		.
	\end{aligned}
\end{equation*}
For $i=1,\dots,4$, we have
	\begin{align*}
		\mathcal{M}_{i}[\psi,\mu_1,\xi]
		= \ &
		\int_{B_{2R_0}}  3 \mu  w^2(y)  \psi(\mu y +\xi,t) Z_{i}(y) dy
		+
		\xi_{it} \int_{B_{2R_0}}  \mu^3 		\bar{\mu}_0^{-2}  \pp_{\bar{y}_i} \Phi_0(\frac{ \mu y }{\bar{\mu}_0} , t)    Z_{i}(y) dy
		\\
		&
		+  \xi_{it}  \int_{B_{2R_0}}
		\Big( \mu^3   \pp_{\bar{x}_i} \varphi(\mu y,t) Z_{i}(y)
		+ \mu    Z_{i}^2(y)   \Big)  dy
		\\
		= \ &
		\int_{B_{2R_0}}  3 \mu  w^2(y)  \psi(\mu y +\xi,t) Z_{i}(y) dy
		+ \mu  \xi_{it}\Big( \int_{B_{2R_0}}   Z_{i}^2(y)     dy + O(t^{-\frac 12}) \Big)
	\end{align*}
by Corollary \ref{varphi-coro} and \eqref{Phi0-est}. Also,
\begin{equation*}
	\begin{aligned} \mathcal{H}_{i}(|y|,t)
		= \ &
		\int_{S^{3}}	3 \mu  w^2(|y|\theta)  \psi(\mu |y|\theta +\xi,t) \Upsilon_{i}(\theta) d\theta
		+
		\xi_{it}
		\int_{S^{3}}
		\mu^3
		\bar{\mu}_0^{-2} \pp_{\bar{y}_i} \Phi_0(\frac{ \mu |y|\theta }{\bar{\mu}_0} , t)
		\Upsilon_{i}(\theta) d\theta
		\\
		&
		+
		\xi_{it}
		\int_{S^{3}} \mu^3 \Big(  \pp_{\bar{x}_i} \varphi(\mu |y|\theta,t)
		+
		\mu^{-2}  \pp_{z_i} w( |y|\theta )
		\eta(\frac{ \mu |y|\theta}{\sqrt t})
		\Big)
		\Upsilon_{i}(\theta) d\theta .
	\end{aligned}
\end{equation*}
Using similar calculations as in \eqref{S2-est}, one has
\begin{equation*}
	\begin{aligned}
		&
		|	\mathcal{H}_{5}(|y|,t) |
		\lesssim
		(\ln t)^{-1} \langle y\rangle^{-3} \sup\limits_{z\in B_{4 R(t)}}
	\langle z \rangle^{-1}
	|\psi(\mu z +\xi,t) |
		+
		(\ln t)^{-3}
		\bigg\{t^{-2}   \langle \bar{y} \rangle^{-2}
		\ln (2+ |\bar{y}|) 	
		\\
		&
		+
		\Big[ t^{-1}(\ln t)^{2} |\mu_1|
		+    (\ln t)^2
		\Big(|\tilde{g}[\bar{\mu}_0,\mu_1] | +
		|\bar{\mu}_{0t}| \ln t
		\sup\limits_{t_1\in[t/2,t]} \Big(\frac{|\mu_1(t_1)|}{\bar{\mu}_{0}(t) } + \frac{|\mu_{1t}(t_1)|}{|\bar{\mu}_{0t}(t)|}  \Big) \Big)  \Big] \langle y\rangle^{-4}
		+ t^{-2} (\ln t)^2 \1_{\{ |y|\le 4 \}}
		\bigg\}
\\
		\lesssim \ &
(\ln t)^{-1}
\langle y\rangle^{-3} \sup\limits_{z\in B_{4 R(t)}}
	\langle z \rangle^{-1}
	|\psi(\mu z +\xi,t) |
+
t^{a_1\gamma-2}  (\ln t)^{-2} \langle y \rangle^{-2-a_1}
\\
&
+
\Big[ (t \ln t)^{-1} |\mu_1|
+    (\ln t)^{-1}
\Big(|\tilde{g}[\bar{\mu}_0,\mu_1] | +
|\bar{\mu}_{0t}| \ln t
\sup\limits_{t_1\in[t/2,t]} \Big(\frac{|\mu_1(t_1)|}{\bar{\mu}_{0}(t) } + \frac{|\mu_{1t}(t_1)|}{|\bar{\mu}_{0t}(t)|}  \Big) \Big) \Big] \langle y\rangle^{-4}
	\end{aligned}
\end{equation*}
where $a_1>0$ is chosen such that $
  a_1\gamma-2 <5\delta -\kappa -a\gamma $. It then follows that
\begin{equation*}
	| \mathcal{H}_{i}(|y|,t)  |\lesssim
	(\ln t)^{-1} \langle y\rangle^{-3} \sup\limits_{z\in B_{4 R(t)}}
	\langle z \rangle^{-1}
	|\psi(\mu z +\xi,t) |
	+
	|\xi_{it}|
	(\ln t)^{2} \langle y \rangle^{-3} ,\quad i=1,\dots 4.
\end{equation*}

By Proposition \ref{R0-linear}, the orthogonal equation  $c_{i}[\mathcal{H}][\tau] =0$ $(i=1,\dots,4)$ is equivalent to solving
\begin{equation*}
	\begin{aligned}
		&
		\mathcal{M}_i[\psi,\mu_1,\xi] + (t (\ln t)^2)^{-\delta\epsilon_0} O\big(
		\sup\limits_{y\in B_{4R(t)}} \langle y \rangle^{ 3 } |\mathcal{H}_{i}(y,t) | \big)
=
	\int_{B_{2R_0}}  3 \mu  w^2(y)  \psi(\mu y +\xi,t) Z_{i}(y) dy
	\\
&
+ \mu  \xi_{it}\big( \int_{B_{2R_0}}   Z_{i}^2(y)     dy + O(t^{-\frac 12}) \big)
 + (t (\ln t)^2)^{-\delta\epsilon_0} O\Big(
 	(\ln t)^{-1}  \sup\limits_{z\in B_{4 R(t)}}
 \langle z \rangle^{-1}|\psi(\mu z +\xi,t) |
 +
 |\xi_{it}|
 (\ln t)^{2}
  \Big)
= 0
		\end{aligned}
\end{equation*}
where $\epsilon_0>0$ is given in Proposition \ref{R0-linear}. One can write above equation as
\begin{equation}\label{xi-eq}
		\xi_{it}
		= \Pi_{i}[\mu_1,\xi]
\end{equation}
where
\begin{equation*}
\begin{aligned}
\Pi_{i}[\mu_1,\xi] = \ &
	\big( \int_{B_{2R_0}}   Z_{i}^2(y)     dy + O(t (\ln t)^2)^{-\frac{\delta\epsilon_0}{2}}\big)^{-1}
\bigg[- 	\int_{B_{2R_0}}  3   w^2(y)  \psi(\mu y +\xi,t) Z_{i}(y) dy
\\
&
- (t (\ln t)^2)^{-\delta\epsilon_0} O(
\sup\limits_{z\in B_{4 R(t)}}
\langle z \rangle^{-1} |\psi(\mu z +\xi,t) |
) \bigg] .
\end{aligned}
\end{equation*}

Let us estimate $\mathcal{M}_{5}$ term by term. By \eqref{Phi0-est}, one has
\begin{equation*}
	\int_{B_{2R_0}}
	\mu^3 \Big[
	\bar{\mu}_0^{-2}  \bar{\mu}_{0t}
	\Phi_0(\frac{ \mu y }{\bar{\mu}_0} , t)
	+
	\bar{\mu}_0^{-2} \nabla_{\bar{y}} \Phi_0(\frac{ \mu y }{\bar{\mu}_0} , t) \cdot \bar{\mu}_{0t} \frac{ \mu y }{\bar{\mu}_0}
	-
	\bar{\mu}_0^{-1} \pp_{t} \Phi_0(\frac{ \mu y }{\bar{\mu}_0} , t)
	\Big]  Z_5(y) dy =  O( t^{-2} (\ln t)^{-1} ).
\end{equation*}
By Corollary \ref{varphi-coro} and the special choice of $\bar{\mu}_0$, we have for  $|\bar{x}|\le 2t^{\frac 12}$
\begin{equation*}
	\begin{aligned}
		\varphi[ \bar{\mu}_0 ] = \ & -2^{-\frac 12} \Big( \bar{\mu}_0  t^{-1}
		+ \int_{t/2}^{t-\mu_{0}^2 } \frac{\bar{\mu}_{0t}(s)}{t-s} d s
		\Big)
		+ O\Big(\bar{\mu}_0  t^{-2} |\bar{x}|^2  +
		|\bar{\mu}_{0t}| \frac{ |\bar x|}{\bar{\mu}_0 }  \Big)    + O (t^{-1} (\ln t)^{-2} )
		\\
		= \ &
		O (t^{-1} (\ln t)^{-2}  \ln\ln t)  + O\Big(\bar{\mu}_0  t^{-2} |\bar{x}|^2  +
		|\bar{\mu}_{0t}| \frac{ |\bar x|}{\bar{\mu}_0 }  \Big)  .
	\end{aligned}
\end{equation*}
Notice that
\begin{equation*}
	\begin{aligned}
		\mu^{-2} w^2(y)
		-
		\bar{\mu}_0^{-2}
		w^2( \frac{ \mu y }{\bar{\mu}_0 })
		= \ &
		-2 \mu_1 (\theta \mu + (1-\theta) \bar{\mu}_0 )^{-3}
		(w^2(y_\theta) + w(y_\theta)\nabla w(y_\theta) \cdot y_\theta)\Big|_{y_{\theta }  = \frac{x-\xi}{\theta \mu + (1-\theta)\bar{\mu}_0 } }
		\\
		= \ &
		-2 \mu_1 \bar{\mu}_0^{-3}
		(w^2(y ) + w(y )\nabla w(y ) \cdot y )
		+ O(\mu_1^2 \bar{\mu}_0^{-4}  \langle y \rangle^{-4}  )
		.
	\end{aligned}
\end{equation*}
Then by Corollary \ref{varphi-coro}, it follows that
	\begin{align*}
		&
	\int_{B_{ 2R_0 }}  3\mu^3 \Big(  \mu^{-2} w^2(y)  \varphi[\mu]( \mu y,t)   -    \bar{\mu}_0^{-2} w^2( \frac{ \mu y }{\bar{\mu}_0 }) \varphi[\bar{\mu}_0]( \mu y ,t)
	\Big)  Z_5(y) dy
		\\
		= \ &
		3 \mu^3 \int_{B_{2R_{0}}}  \Big[ \mu^{-2} w^2(y)  (\varphi[\mu] - \varphi[\bar{\mu}_0 ] )
		+
		\big( \mu^{-2} w^2(y)
		-   \bar{\mu}_0^{-2} w^2( \frac{ \mu y }{\bar{\mu}_0 }) \big) \varphi[\bar{\mu}_0]
		\Big]   Z_5(y) dy
		\\
		= \ &
		3 \mu^3 \int_{B_{2R_{0}}}  \Bigg\{ \mu^{-2} w^2(y) Z_5(y)   \bigg[
		-2^{-\frac 12} \Big( \mu_1  t^{-1}
		+
		\int_{t/2}^{t-\mu_{0}^2 } \frac{\mu_{1t} (s)}{t-s} d s
		\Big)	
			\\
		&
		+ O\Big( |\mu_1| \bar{\mu}_0^2 t^{-2} \frac{|\bar{x}|^2 }{\bar{\mu}_0^2}
		+  |\bar{\mu}_{0t}| \sup\limits_{t_1\in [t/2,t]}\Big(\frac{|\mu_1(t_1)|}{\bar{\mu}_0 } + \frac{|\mu_{1t}(t_1)|}{|\bar{\mu}_{0t}| }  \Big)  \frac{ |\bar x |}{\bar{\mu}_{0}} \Big)
	 + \tilde{g}[\bar{\mu}_0,\mu_1]
		\bigg]
			\\
		&
		+
		\Big[ -2 \mu_1 \bar{\mu}_0^{-3}
		(w^2(y ) + w(y )\nabla w(y ) \cdot y )
		+ O(\mu_1^2 \bar{\mu}_0^{-4}  \langle y \rangle^{-4}  )  \Big]  Z_5(y)
		\\
		& \times \Big(
		O (t^{-1} (\ln t)^{-2}  \ln\ln t)  + O(\bar{\mu}_0  t^{-2} |\bar{x}|^2  +
		|\bar{\mu}_{0t}| \frac{ |\bar x|}{\bar{\mu}_0 }  )
		\Big)
		\Bigg\}   dy
		\\
		= \ &
		\mu \bigg[-2^{-\frac 12} 3 \int_{B_{2R_{0}}}   w^2(y) Z_5(y)  dy \Big( \mu_1  t^{-1}
		+
		\int_{t/2}^{t-\mu_{0}^2 } \frac{\mu_{1t} (s)}{t-s} d s
		\Big)
		\\
		&
		+O\Big(    (t\ln t)^{-1} \sup\limits_{t_1\in [t/2,t]}|\mu_1(t_1)| + \sup\limits_{t_1\in [t/2,t]}|\mu_{1t}(t_1)|    \Big)   + \tilde{g}[\bar{\mu}_0,\mu_1]
		+
		O (|\mu_1 | t^{-1} (\ln t)^{-1}  \ln\ln t) \bigg] .
	\end{align*}

By Corollary \ref{varphi-coro}, we have
	\begin{align*}
		&
		\int_{B_{2R_{0}}}  3 \mu^3 \Big(  \mu^{-1} w(y)  \varphi^2[\mu]( \mu y ,t)
		-
		\bar{\mu}_0^{-1}  w(  \frac{ \mu y }{\bar{\mu}_0 } )  \varphi^2[\bar{\mu}_0] ( \mu y ,t)
		\Big) Z_5(y) dy
		\\
		= \ &
		3 \mu^3 \int_{B_{2R_{0}}}   \bigg[  \mu^{-1} w(y)  (\varphi[\mu]  -
		\varphi[\bar{\mu}_0]  ) (\varphi[\mu]  +
		\varphi[\bar{\mu}_0]  )
		+
		\big( \mu^{-1} w(y) -
		\bar{\mu}_0^{-1}  w(  \frac{\mu y }{\bar{\mu}_0 } ) \big)  \varphi^2[\bar{\mu}_0]
		\bigg] Z_5(y) dy
		\\
		= \ &
		3 \mu^3 \int_{B_{2R_{0}}}   \Bigg\{  \mu^{-1} w(y)  Z_5(y)  \bigg[
		-2^{-\frac 12} \Big( \mu_1  t^{-1}
		+
		\int_{t/2}^{t-\mu_{0}^2 } \frac{\mu_{1t} (s)}{t-s} d s
		\Big)
		\\
		&
		+ O\Big( |\mu_1| \bar{\mu}_0^2 t^{-2} \frac{|\bar{x}|^2}{ \bar{\mu}_0^2 } +  |\bar{\mu}_{0t}| \sup\limits_{t_1\in[t/2,t]}\Big(\frac{|\mu_1(t_1)|}{\bar{\mu}_0}
		+ \frac{|\mu_{1t}(t_1)|}{|\bar{\mu}_{0t}|}  \Big)  \frac{ |\bar x |}{\bar{\mu}_0 } \Big)
		+
		\tilde{g}[\bar{\mu}_0,\mu_1]\bigg]
		(t\ln t)^{-1}
			\\ &
		+
		(
		- \mu_1 \bar{\mu}_0^{-2}
		(w(y)
		+
		y \cdot \nabla w(y)
		) +
		O(\mu_1^2 \bar{\mu}_0^{-3}) \langle y \rangle^{-2}
		)   Z_5(y) (t \ln t)^{-2}
		\Bigg\} dy
		\\
		= \ &
		\mu^3     \Bigg\{  \bigg[
	\Big( \mu_1  t^{-1}
		+
		\int_{t/2}^{t-\mu_{0}^2 } \frac{\mu_{1t} (s)}{t-s} d s
		\Big) O(\ln t )
		+ O\Big(
		|\mu_1| \bar{\mu}_0^2 t^{-2} R_0^2
		+  |\bar{\mu}_{0t}| \sup\limits_{t_1\in[t/2,t]}\Big(\frac{|\mu_1(t_1)|}{\bar{\mu}_0}
		+ \frac{|\mu_{1t}(t_1)|}{|\bar{\mu}_{0t}|}  \Big)  R_0 \Big)
		\\ &
		+ O(\ln t)
		\tilde{g}[\bar{\mu}_0,\mu_1]\bigg]
		O(t^{-1} )
		+
		\mu_1 \bar{\mu}_0^{-2} O(t^{-2} (\ln t)^{-1} )
		\Bigg\}
		\\
		= \ &
		\mu
		\Bigg[
		\Big( \mu_1  t^{-1}
		+
		\int_{t/2}^{t-\mu_{0}^2 } \frac{\mu_{1t} (s)}{t-s} d s
		\Big) O((t\ln t)^{-1} )
		\\
		&
		+ O\Big(
	t^{-\frac 32} \sup\limits_{t_1\in[t/2,t]} |\mu_1(t_1)|
		+  t^{-\frac 12} \sup\limits_{t_1\in[t/2,t]}|\mu_{1t} (t_1)|   \Big)
		+
	O(  (t\ln t)^{-1} ) \tilde{g}[\bar{\mu}_0,\mu_1]
		\Bigg]
	\end{align*}
since $\delta>0$ is very small and
$ \mu^{-1} w(y) -
		\bar{\mu}_0^{-1}  w(  \frac{ \mu y }{\bar{\mu}_0 } )
		=
		- \mu_1 \bar{\mu}_0^{-2}
		(w( y )
		+
		y
		\cdot
		\nabla w(y)
		) +
		O(\mu_1^2 \bar{\mu}_0^{-3}) \langle y \rangle^{-2} $. Similarly, the following estimates hold
\begin{equation*}
	\int_{B_{ 2R_0 }}  \mu^3 \varphi^3[\mu]( \mu y,t)  Z_{5}(y) dy
	= O(\mu^3 (t\ln t)^{-3} R_0^2 ) = O(t^{-\frac 52})
\end{equation*}
when $\delta$ is small enough.
\begin{equation*}
	\bigg|  \int_{B_{2R_{0}}}  \mu^3\bar{\mu}_0^{-2} \mathcal{M}[\bar{\mu}_0]
	\frac{\eta(\bar{y}) Z_{5}(\bar{y}) }{\int_{B_{2} } \eta(z) Z_{5}^2(z) d z}
	Z_5(y) dy\bigg| \lesssim O( t^{-2}(\ln t)^{-1} ) ,
\end{equation*}
\begin{equation*}
	\begin{aligned}
		&
		\int_{B_{ 2R_0 }}   \mu^3 \bigg[
		\big(u_1 + \varphi[\mu] + \bar{\mu}_0^{-1} \Phi_0(\frac{ \mu y }{\bar{\mu}_0} , t) \eta(\frac{4 \mu y }{\sqrt t}) \big)^3 - (u_1+\varphi[\mu] )^3
		-
		3 \bar{\mu}_0^{-3} w^2(\frac{ \mu y }{\bar{\mu}_0 })  \Phi_0(\frac{ \mu y }{\bar{\mu}_0} , t) \eta(\frac{4 \mu y }{\sqrt t}) \bigg]  Z_5(y) dy
		\\
		= \ &
\int_{B_{ 2R_0 }}   \mu^3 \bigg[
\big(u_1 + \varphi[\mu] + \bar{\mu}_0^{-1} \Phi_0(\frac{ \mu y }{\bar{\mu}_0} , t)  \big)^3 - (u_1+\varphi[\mu] )^3
-
3 \mu^{-2} w^2(y) \bar{\mu}_0^{-1} \Phi_0(\frac{ \mu y }{\bar{\mu}_0} , t)
\\
&
+
3 \big( \mu^{-2} w^2(y) - \bar{\mu}_0^{-2} w^2(\frac{ \mu y }{\bar{\mu}_0 })
\big)
\bar{\mu}_0^{-1} \Phi_0(\frac{ \mu y }{\bar{\mu}_0} , t)  \bigg]  Z_5(y) dy
=  \mu
O( |\mu_1|  (t \ln t)^{-1} \ln \ln t)
+ O(t^{-2} (\ln t)^{-4})
	\end{aligned}
\end{equation*}
since by \eqref{Phi0-est},
	\begin{align*}
		&
		\int_{B_{2R_{0}}}  3 \mu^3 \Big[ (\mu^{-1} w(y))^2
		-(\bar{\mu}_0^{-1} w(\frac{ \mu y }{\bar{\mu}_0 }))^2 \Big] \bar{\mu}_0^{-1} \Phi_0(\frac{ \mu y }{\bar{\mu}_0} , t)
		Z_5(y) dy
		\\
		= \ &
		3 \mu^3 \bar{\mu}_0^{-1} \int_{B_{2R_{0}}}  \Big[
		-2 \mu_1 \bar{\mu}_0^{-3}
		(w^2(y ) + w(y )\nabla w(y ) \cdot y )
		+ O(\mu_1^2 \bar{\mu}_0^{-4}  \langle y \rangle^{-4}  )
		\Big]
		O(t^{-1} (\ln t)^{-3} \ln \ln t
		\langle y \rangle^{-1} )
		Z_5(y) dy
		\\
		= \ & \mu
		O( |\mu_1|  (t \ln t)^{-1} \ln \ln t) ,
	\end{align*}
	\begin{align*}
		&
		\Bigg|
		\int_{B_{2R_{0}}}   \mu^3 \bigg[
		\big(u_1 + \varphi[\mu] + \bar{\mu}_0^{-1} \Phi_0(\frac{ \mu y }{\bar{\mu}_0} , t)  \big)^3 - (u_1+\varphi[\mu] )^3
		-
		3 \mu^{-2} w^2(y) \bar{\mu}_0^{-1} \Phi_0(\frac{ \mu y }{\bar{\mu}_0} , t) \bigg]  Z_5(y) dy \Bigg|
		\\
		= \ & \Bigg|
		\mu^3  \int_{B_{2R_{0}}}   \Bigg[  3(u_1 +\varphi[\mu]) \big(\bar{\mu}_0^{-1} \Phi_0(\frac{ \mu y }{\bar{\mu}_0} , t) \big)^2
		+
		\big(\bar{\mu}_0^{-1} \Phi_0(\frac{\mu y }{\bar{\mu}_0} , t) \big)^3
		\\
		&
		+  3 (u_1 +\varphi[\mu]
		-  \mu^{-1} w(y) ) (u_1 +\varphi[\mu]
		+ \mu^{-1} w(y) )
		\bar{\mu}_0^{-1} \Phi_0(\frac{ \mu y }{\bar{\mu}_0} , t)
		\Bigg]  Z_5(y) dy \Bigg|
		\\
		\lesssim \ &
		\mu (\ln t)^{-2} \int_{B_{2R_{0}}}   \Bigg[  ( \ln t \langle y \rangle^{-2} + (t\ln t)^{-1} )
		( (t \ln t)^{-1} \langle \bar{y} \rangle^{-2}
		\ln (2+ |\bar{y}|))^2
		+
		(  (t \ln t)^{-1} \langle \bar{y} \rangle^{-2}
		\ln (2+ |\bar{y}|) )^3
		\\
		&
		+  \big( \ln t \langle y\rangle^{-2} \1_{\{ |\bar{x}|\ge t^{\frac 12} \} }+ (t\ln t)^{-1}
		\big) (  \ln t \langle y\rangle^{-2} + (t\ln t)^{-1} )
		(t \ln t)^{-1} \langle \bar{y} \rangle^{-2}
		\ln (2+ |\bar{y}|)
		\Bigg]  \langle y\rangle^{-2} dy
		\\
		\lesssim \ &
		\mu (\ln t)^{-2} \int_{B_{2R_{0}}}   \Bigg[   \ln t \langle y \rangle^{-2}
		( (t \ln t)^{-1} \langle \bar{y} \rangle^{-2}
		\ln (2+ |\bar{y}|))^2
		+
		(  (t \ln t)^{-1} \langle \bar{y} \rangle^{-2}
		\ln (2+ |\bar{y}|) )^3
		\\
		&
		+   (t\ln t)^{-1}  \ln t \langle y\rangle^{-2}
		(t \ln t)^{-1} \langle \bar{y} \rangle^{-2}
		\ln (2+ |\bar{y}|)
		\Bigg]  \langle y\rangle^{-2} dy
		\lesssim  t^{-2} (\ln t)^{-4}.
	\end{align*}
Finally, we get
	\begin{align*}
		&\quad  \mathcal{M}_{5}[\psi,\mu_1,\xi] =
		\int_{B_{2R_0}} 	3 \mu w^2(y)  \psi(\mu y +\xi,t) Z_5(y) dy
		+	O( t^{-2} (\ln t)^{-1} )
		\\
		&
		+
	\mu \Bigg[-2^{-\frac 12} 3 \int_{B_{2R_{0}}}   w^2(y) Z_5(y)  dy \bigg( \mu_1  t^{-1}
	+
	\int_{t/2}^{t-\mu_{0}^2 } \frac{\mu_{1t} (s)}{t-s} d s
	\bigg)
	\\
	&
	+O\big(    (t\ln t)^{-1} \sup\limits_{t_1\in [t/2,t]}|\mu_1(t_1)| + \sup\limits_{t_1\in [t/2,t]}|\mu_{1t}(t_1)|    \big)   + \tilde{g}[\bar{\mu}_0,\mu_1]
	+
	O (|\mu_1 | t^{-1} (\ln t)^{-1}  \ln\ln t) \Bigg]
		\\
		&
		+
		\mu
	\Bigg[
	\bigg( \mu_1  t^{-1}
	+
	\int_{t/2}^{t-\mu_{0}^2 } \frac{\mu_{1t} (s)}{t-s} d s
	\bigg) O((t\ln t)^{-1} )
	\\
	&
	+ O\big(
	t^{-\frac 32} \sup\limits_{t_1\in[t/2,t]} |\mu_1(t_1)|
	+  t^{-\frac 12} \sup\limits_{t_1\in[t/2,t]}|\mu_{1t} (t_1)|   \big)
	+
	O(  (t\ln t)^{-1} ) \tilde{g}[\bar{\mu}_0,\mu_1]
	\Bigg]
		\\
		&
		+  \mu
		O( |\mu_1|  (t \ln t)^{-1} \ln \ln t)
\\
= \ &
\mu \Bigg\{
\int_{B_{2R_0}} 	3  w^2(y)  \psi(\mu y +\xi,t) Z_5(y) dy
+	O( t^{-2} )
\\
&
+ O\big(    (t\ln t)^{-1} \sup\limits_{t_1\in [t/2,t]}|\mu_1(t_1)| + \sup\limits_{t_1\in [t/2,t]}|\mu_{1t}(t_1)|    \big)   + O(\tilde{g}[\bar{\mu}_0,\mu_1] )
+
O (|\mu_1 |  (t \ln t)^{-1}  \ln\ln t)
\\
&
+ \bigg(-2^{-\frac 12} 3 \int_{B_{2R_{0}}}   w^2(y) Z_5(y)  dy + O((t\ln t)^{-1} ) \bigg) \bigg( \mu_1  t^{-1}
+
\int_{t/2}^{t-\mu_{0}^2 } \frac{\mu_{1t} (s)}{t-s} d s
\bigg)
\Bigg\}
\\
= \ &
\mu \big(-2^{-\frac 12} 3 \int_{B_{2R_{0}}}   w^2(y) Z_5(y)  dy + O((t\ln t)^{-1} ) \big)
\\
& \times
\Bigg\{
\big(-2^{-\frac 12} 3 \int_{B_{2R_{0}}}   w^2(y) Z_5(y)  dy + O((t\ln t)^{-1} ) \big)^{-1} \int_{B_{2R_0}} 	3  w^2(y)  \psi(\mu y +\xi,t) Z_5(y) dy
\\
&
+	O( t^{-2} )
 + O\big(    (t\ln t)^{-1} \sup\limits_{t_1\in [t/2,t]}|\mu_1(t_1)| + \sup\limits_{t_1\in [t/2,t]}|\mu_{1t}(t_1)|    \big)   + O(\tilde{g}[\bar{\mu}_0,\mu_1] )
\\
&
+   \mu_1  t^{-1} \big(1+ O (( \ln t)^{-1}  \ln\ln t ) \big)
+
\int_{t/2}^{t-t^{1-\nu} } \frac{\mu_{1t}(s)}{t-s} d s
+
\mu_{1t}
\big((1-\nu)\ln t + 2\ln\ln t  \big)
+
\mathcal{E}_{\nu}[\mu_{1}]
\Bigg\}
	\end{align*}
where
\begin{equation*}
	\mathcal{E}_{\nu}[\mu_{1}] = \int_{ t-t^{1-\nu}  }^{t-\mu_{0}^2(t) } \frac{\mu_{1t}(s) - \mu_{1t}(t)}{t-s} d s.
\end{equation*}

By Proposition \ref{R0-linear}, $c_{5}[\mathcal{H}] =0$ is equivalent to
 \begin{align*}
 &
 	\mathcal{M}_5[\psi,\mu_1,\xi] + (t (\ln t)^2)^{-\delta\epsilon_0} O\big(
 	\sup\limits_{y\in B_{4R(t)}} \langle y \rangle^{2+a_{1}} |\mathcal{H}_{5}(y,t) | \big)
\\
= \ &
\mu \big(-2^{-\frac 12} 3 \int_{B_{2R_{0}}}   w^2(y) Z_5(y)  dy + O((t\ln t)^{-1} ) \big)
\\
& \times
\Bigg\{
\big(-2^{-\frac 12} 3 \int_{B_{2R_{0}}}   w^2(y) Z_5(y)  dy + O((t\ln t)^{-1} ) \big)^{-1} \int_{B_{2R_0}} 	3  w^2(y)  \psi(\mu y +\xi,t) Z_5(y) dy
\\
&
+	O( t^{-2} )
+ O(    (t\ln t)^{-1} \sup\limits_{t_1\in [t/2,t]}|\mu_1(t_1)| + \sup\limits_{t_1\in [t/2,t]}|\mu_{1t}(t_1)|    )   + O(\tilde{g}[\bar{\mu}_0,\mu_1] )
\\
&
+   \mu_1  t^{-1} (1+ O (( \ln t)^{-1}  \ln\ln t ) )
+
\int_{t/2}^{t-t^{1-\nu} } \frac{\mu_{1t}(s)}{t-s} d s
+
\mu_{1t}
((1-\nu)\ln t + 2\ln\ln t  )
+
\mathcal{E}_{\nu}[\mu_{1}]
\Bigg\}
\\
& + (t (\ln t)^2)^{-\delta\epsilon_0} O\Bigg(
(\ln t)^{-1}  \sup\limits_{y\in B_{4R(t)}} \langle y \rangle^{-1} |\psi(\mu y +\xi,t)|
+
t^{a_1\gamma-2}  (\ln t)^{-2}
\\
&
+
(t \ln t)^{-1} |\mu_1|
+    (\ln t)^{-1}
\left(|\tilde{g}[\bar{\mu}_0,\mu_1] | +
|\bar{\mu}_{0t}| \ln t
\sup\limits_{t_1\in[t/2,t]} \bigg(\frac{|\mu_1(t_1)|}{\bar{\mu}_{0}(t) } + \frac{|\mu_{1t}(t_1)|}{|\bar{\mu}_{0t}(t)|}  \bigg) \right)
\Bigg)
 \\
 = \ &
 \mu \big(-2^{-\frac 12} 3 \int_{B_{2R_{0}}}   w^2(y) Z_5(y)  dy + O((t\ln t)^{-1} ) \big)
 \\
 & \times
 \Bigg\{
 \big(-2^{-\frac 12} 3 \int_{B_{2R_{0}}}   w^2(y) Z_5(y)  dy + O((t\ln t)^{-1} ) \big)^{-1} \int_{B_{2R_0}} 	3  w^2(y)  \psi(\mu y +\xi,t) Z_5(y) dy
 \\
 &
 +	O( t^{-2} )
 + O(    (t\ln t)^{-1} \sup\limits_{t_1\in [t/2,t]}|\mu_1(t_1)| + \sup\limits_{t_1\in [t/2,t]}|\mu_{1t}(t_1)|    )   + O(\tilde{g}[\bar{\mu}_0,\mu_1] )
 \\
 &
 +   \mu_1  t^{-1} (1+ O (( \ln t)^{-1}  \ln\ln t ) )
 +
 \int_{t/2}^{t-t^{1-\nu} } \frac{\mu_{1t}(s)}{t-s} d s
 +
 \mu_{1t}
 ((1-\nu)\ln t + 2\ln\ln t  )
 +
 \mathcal{E}_{\nu}[\mu_{1}]
\\
& + (t (\ln t)^2)^{-\delta\epsilon_0} O(
  \sup\limits_{y\in B_{4R(t)}} \langle y \rangle^{-1} |\psi(\mu y +\xi,t)|
+
t^{a_{1}\gamma-2}  (\ln t)^{-1}
)
 \Bigg\}
 	= 0
 \end{align*}
 where we have used similar calculations as in \eqref{H-upp}, and
 \begin{equation*}
 \begin{aligned}
 	&
 	\sup\limits_{y\in B_{4R(t)}} \langle y \rangle^{2+a_1} |\mathcal{H}_{5}(y,t) |  \lesssim
 	(\ln t)^{-1}  \sup\limits_{y\in B_{4R(t)}} \langle y \rangle^{-1} |\psi(\mu y +\xi,t)|
 	+
 	t^{a_1\gamma-2}  (\ln t)^{-2}
 	\\
 	&
 	+
 	(t \ln t)^{-1} |\mu_1|
 	+    (\ln t)^{-1}
 	\left(|\tilde{g}[\bar{\mu}_0,\mu_1] | +
 	|\bar{\mu}_{0t}| \ln t
 	\sup\limits_{t_1\in[t/2,t]} \bigg(\frac{|\mu_1(t_1)|}{\bar{\mu}_{0}(t) } + \frac{|\mu_{1t}(t_1)|}{|\bar{\mu}_{0t}(t)|}  \bigg) \right) .
\end{aligned}
 \end{equation*}

Similar to the methodology in Section \ref{SolElliptic-sec}, we leave $\mathcal{E}_{\nu}[\mu_{1}]$ as the remainder term and consider the following equation about $\mu_1$.
\begin{equation*}
\begin{aligned}
&
\big(-2^{-\frac 12} 3 \int_{B_{2R_{0}}}   w^2(y) Z_5(y)  dy + O((t\ln t)^{-1} ) \big)^{-1} \int_{B_{2R_0}} 	3  w^2(y)  \psi(\mu y +\xi,t) Z_5(y) dy
\\
& +
O\big(    (t\ln t)^{-1} \sup\limits_{t_1\in [t/2,t]}|\mu_1(t_1)| + \sup\limits_{t_1\in [t/2,t]}|\mu_{1t}(t_1)|    \big)
+ O(\tilde{g}[\bar{\mu}_0,\mu_1] )
\\
&
+   \mu_1  t^{-1} (1+ O (( \ln t)^{-1}  \ln\ln t ) )
+
\int_{t/2}^{t-t^{1-\nu} } \frac{\mu_{1t}(s)}{t-s} d s
+
\mu_{1t}
((1-\nu)\ln t + 2\ln\ln t  )
\\
& + (t (\ln t)^2)^{-\delta\epsilon_0} O(
\sup\limits_{y\in B_{4R(t)}} \langle y \rangle^{-1} |\psi(\mu y +\xi,t)|
) +
 (t (\ln t)^2)^{-\delta\epsilon_0} O(
t^{a_{1}\gamma-2}  (\ln t)^{-1}
)
= 0
\end{aligned}
\end{equation*}
when $a_1\gamma > \delta\epsilon_0$.
That is,
\begin{equation}\label{mu1-eq}
\mu_{1t}  + \beta_{v}(t) \mu_1
=  \Pi_{5}[\mu_1,\xi]
\end{equation}
where
\begin{equation*}
\beta_{\nu}(t) = t^{-1} \left[(1-\nu)\ln t + 2\ln\ln t   \right]^{-1}  (1+ O (( \ln t)^{-1}  \ln\ln t ) ),
\end{equation*}
\begin{equation}\label{Pi5}
\begin{aligned}
& \Pi_{5}[\mu_1,\xi] =
\chi(t)
((1-\nu)\ln t + 2\ln\ln t   )^{-1}
\Bigg[
-
\int_{t/2}^{t-t^{1-\nu} } \frac{\mu_{1t}(s)}{t-s} d s
-
O(    (t\ln t)^{-1} \sup\limits_{t_1\in [t/2,t]}|\mu_1(t_1)| + \sup\limits_{t_1\in [t/2,t]}|\mu_{1t}(t_1)|    )
\\
&
- O(\tilde{g}[\bar{\mu}_0,\mu_1] )
+
\big(2^{-\frac 12} 3 \int_{B_{2R_{0}}}   w^2(y) Z_5(y)  dy + O((t\ln t)^{-1} ) \big)^{-1} \int_{B_{2R_0}} 	3  w^2(y)  \psi(\mu y +\xi,t) Z_5(y) dy
\\
&
- (t (\ln t)^2)^{-\delta\epsilon_0} O(
\sup\limits_{y\in B_{4R(t)}} \langle y \rangle^{-1} |\psi(\mu y +\xi,t)|
)
-
(t (\ln t)^2)^{-\delta\epsilon_0} O(
t^{a_{1}\gamma-2}  (\ln t)^{-1}
) \Bigg]
.
\end{aligned}
\end{equation}

Similar to \eqref{mu01-fixed point}, in order to solve \eqref{mu1-eq} and \eqref{xi-eq}, it is sufficient to consider the following fixed point problem:
\begin{equation}\label{mu1-xi-fixed point}
\begin{aligned}
	\mathcal{S}_5[\mu_1,\xi](t) = \ & \Pi_{5}[\mu_1,\xi](t) +
	\beta_{\nu}(t) e^{ - \int^{t} \beta_{\nu}(u) du} \int_{t}^{\infty}
	e^{\int^s \beta_{\nu}(u) du}  \Pi_{5}[\mu_1,\xi](s) ds ,
\\
\mathcal{S}_{i}[\mu_1,\xi](t) = \ & \Pi_{i}[\mu_1,\xi],\quad i=1,\dots 4  .
\end{aligned}
\end{equation}
Notice that $|\psi| \lesssim 	\ln t (t (\ln t)^{2})^{5\delta-\kappa }
R^{- a} $ and
recall the norms \eqref{norm-mu}, \eqref{norm-xi} for $\mu_1$, $\xi$. We will solve \eqref{mu1-xi-fixed point} in the following spaces
\begin{equation}\label{muxi-topo}
	B_{\mu_1} = \{ \mu_1 \in C^{1}(t_0/4,\infty) \ : \ \| \mu_{1} \|_{*1} \le 2 \},
	\quad  B_{\xi} =  \{ \xi \in C^1(t_0,\infty) \ : \ \| \xi \|_{*2} \le 2 \}.
\end{equation}

For any $\mu_{1a}, \mu_{1b}\in B_{\mu_1} $ and $\xi_{a},\xi_{b} \in B_{\xi}$, similar to \eqref{typical}, one has
\begin{equation*}
\begin{aligned}
&
\chi(t) \bigg|\int_{t/2}^{t-t^{1-\nu} } \frac{\mu_{1at}(s) - \mu_{1bt}(s)}{t-s} d s \bigg|
\le
\| \mu_{1a} - \mu_{1b} \|_{*1}
\int_{t/2}^{t-t^{1-\nu} } \frac{ \ln s (s (\ln s)^{2})^{5\delta-\kappa }
	R^{- a}(s) }{t-s} d s
\\
= \ &
\| \mu_{1a} - \mu_{1b} \|_{*1}  (1+ O((\ln t)^{-1} ) )
\nu (\ln t)^2  (t(\ln t)^2)^{ 5\delta-\kappa } R^{-a} .
\end{aligned}
\end{equation*}
By gradient estimate in Proposition \ref{psi-prop}, we have
\begin{equation*}
\begin{aligned}
&
\bigg|\int_{B_{2R_0}}  w^2(y)  (\psi(\mu_{1a} y +\xi_{a},t) -\psi(\mu_{1b} y +\xi_{b},t) ) Z_5(y) dy \bigg|
\le
C	\ln t (t (\ln t)^{2})^{5\delta-\kappa }
R^{- a}
(|\mu_{1a} -\mu_{1b}| + |\xi_{a} -\xi_{b}|)
\\
\le \ &
C t \ln t	[\ln t (t (\ln t)^{2})^{5\delta-\kappa }
R^{- a} ]^2
(\|\mu_{1a} -\mu_{1b}\|_{*1} + \|\xi_{a} -\xi_{b}\|_{*2}) .
\end{aligned}
\end{equation*}
The estimate for $\int_{B_{2R_0}}  3   w^2(y)  \psi(\mu y +\xi,t) Z_{i}(y) dy$ is the same.
\begin{equation*}
\begin{aligned}
& (t (\ln t)^2)^{-\delta\epsilon_0} \bigg| O(
\sup\limits_{y\in B_{4R(t)}} \langle y \rangle^{-1} |\psi(\mu_{1a} y +\xi_{a},t)|
) - O(
\sup\limits_{y\in B_{4R(t)}} \langle y \rangle^{-1} |\psi(\mu_{1b} y +\xi_{b},t)|
) \bigg|
\\
\le \ &
(t (\ln t)^2)^{-\delta\epsilon_0} \bigg| O(
\sup\limits_{y\in B_{4R(t)}} \langle y \rangle^{-1} |\psi(\mu_{1a} y +\xi_{a},t) - \psi(\mu_{1b} y +\xi_{b},t)|
) \bigg|
\\
\le \ &
C (t (\ln t)^2)^{-\delta\epsilon_0}
t \ln t	[\ln t (t (\ln t)^{2})^{5\delta-\kappa }
R^{- a} ]^2
(\|\mu_{1a} -\mu_{1b}\|_{*1} + \|\xi_{a} -\xi_{b}\|_{*2})
\end{aligned}
\end{equation*}
since $O(
\sup\limits_{y\in B_{4R(t)}} \langle y \rangle^{-1}|\psi(\mu y +\xi,t)|
) $ depends on $\psi$ linearly.

From the same calculations as in \eqref{tilde-g-mu1-est}, one has
\begin{equation*}
\chi(t)|\tilde{g}[\bar{\mu}_0, \mu_{1a}] - \tilde{g}[\bar{\mu}_0, \mu_{1b}] | \le C \ln t (t(\ln t)^2)^{5\delta-\kappa} R^{-a} \|\mu_{1a} -\mu_{1b}\|_{*1}
\end{equation*}
when $5\delta-\kappa -a\gamma>-2$. Similar to \eqref{mu1-mu2-est}, one has
\begin{equation*}
\begin{aligned}
&
\chi(t) \Big|O(    (t\ln t)^{-1} \sup\limits_{t_1\in [t/2,t]}|\mu_{1a}(t_1)| + \sup\limits_{t_1\in [t/2,t]}|\mu_{1at}(t_1)|    )
- O(    (t\ln t)^{-1} \sup\limits_{t_1\in [t/2,t]}|\mu_{1b}(t_1)| + \sup\limits_{t_1\in [t/2,t]}|\mu_{1bt}(t_1)|    )\Big|
\\
\lesssim \ &
\chi(t)
O\big(    (t\ln t)^{-1} \sup\limits_{t_1\in [t/2,t]}|\mu_{1a}(t_1) - \mu_{1b}(t_1)| + \sup\limits_{t_1\in [t/2,t]}|\mu_{1at}(t_1) - \mu_{1bt}(t_1)|    \big)
\\
\lesssim \ &
\ln t (t(\ln t)^2)^{5\delta-\kappa} R^{-a} \|\mu_{1a} -\mu_{1b}\|_{*1}  .
\end{aligned}
\end{equation*}

In conclusion, under the following restrictions
\begin{equation}\label{ortho-para}
a_1>0, \
a_1\gamma-2 <5\delta -\kappa -a\gamma, \
a_1\gamma > \delta\epsilon_0, \
5\delta-\kappa -a\gamma>-2, \
0<\nu<\frac 12,
\end{equation}
for $t_0$ is sufficiently large, $(\mathcal{S}_{5},\mathcal{S}_{i})$ is a contraction mapping in $B_{\mu_{1}}\times B_{\xi}$.

Similarly, for $(\mu_1,\xi)\in B_{\mu_{1}}\times B_{\xi}$, we have
\begin{equation*}
		\chi(t) \Big|\int_{t/2}^{t-t^{1-\nu} } \frac{\mu_{1t}(s) }{t-s} d s \Big|
		\le
	\| \mu_{1} \|_{*1} (1+ O((\ln t)^{-1} ) )
v (\ln t)^2  (t(\ln t)^2)^{5\delta-\kappa} R^{-a} ,
\end{equation*}
\begin{equation*}
	|\tilde{g}[\bar{\mu}_0, \mu_{1}]  | \le C \ln t (t(\ln t)^2)^{5\delta-\kappa} R^{-a} \|\mu_{1}\|_{*1}   ,
\end{equation*}
\begin{equation*}
\chi(t)\Big|O(    (t\ln t)^{-1} \sup\limits_{t_1\in [t/2,t]}|\mu_1(t_1)| + \sup\limits_{t_1\in [t/2,t]}|\mu_{1t}(t_1)|    )\Big|
\lesssim \ln t (t(\ln t)^2)^{5\delta-\kappa} R^{-a} \|\mu_{1}\|_{*1}  .
\end{equation*}
Then
\begin{equation*}
(\mathcal{S}_{5},\mathcal{S}_{i}) :
B_{\mu_{1}}\times B_{\xi}
\rightarrow B_{\mu_{1}}\times B_{\xi}.
\end{equation*}
Consequently, by the contraction mapping theorem, we  find a unique solution $(\mu_1,\xi)$ in $B_{\mu_{1}}\times B_{\xi}$.

\medskip

\subsection{H\"older continuity of $\mu_{1t}$ and estimate for $\mu \mathcal{E}_{\nu}[\mu_{1}]$}
In order to estimate the left error
\begin{equation*}
	\mathcal{E}_{\nu}[\mu_{1}] = \int_{ t-t^{1-\nu}  }^{t-\mu_{0}^2(t) } \frac{\mu_{1t}(s) - \mu_{1t}(t)}{t-s} d s,
\end{equation*}
we need H\"older estimate of $\mu_{1t}$, which satisfies
\begin{equation*}
	\mu_{1t} =  \Pi_{5}[\mu_1,\xi](t) +
	\beta_{\nu}(t) e^{ - \int^{t} \beta_{\nu}(u) du} \int_{t}^{\infty}
	e^{\int^s \beta_{\nu}(u) du}  \Pi_{5}[\mu_1,\xi](s) ds .
\end{equation*}

Assume $\frac{3t}{4} \le t_2 < t_1 \le t$, $\frac 89 < A < 1$. $A$ will be chosen to be close to $1$ later depending on $\nu$ and independent of $t_0$. We revisit \eqref{Pi5} term by term.

Notice that $\psi$ only has H\"older continuity in $t$ variable, which  restricts the regularity for $\mu_{1t}$. Using Proposition \ref{psi-prop} with $\lambda^2(t) =t^{\frac 12}$, one has
	\begin{align*}
		&
		|\psi(\mu(t_1) y +\xi(t_1),t_1) - \psi(\mu(t_2) y +\xi(t_2) ,t_2) |
		\\
		\le \ &
		|\psi(\mu(t_1) y +\xi(t_1),t_1) - \psi(\mu(t_2) y + \xi(t_2) ,t_1) |
		+
		|\psi(\mu(t_2) y + \xi(t_2) ,t_1) - \psi(\mu(t_2) y + \xi(t_2) ,t_2) |
		\\
		\lesssim \ &
		\big[ \ln t
		(t (\ln t)^{2})^{5\delta-\kappa }
		R^{- a}(t) |y|  +
(\ln t )^2
(t (\ln t)^{2})^{5\delta-\kappa }
R^{- a}(t)
		\big]\ln t (t (\ln t)^{2})^{5\delta-\kappa }
		R^{- a}(t) |t_1-t_2|
		\\
	&
+
C(\alpha) \Bigg\{
\lambda^{-2\alpha}(t) \ln t (t (\ln t)^{2})^{5\delta-\kappa }
R^{- a}(t)
\\
&
+
\lambda^{2-2\alpha}(t)
[ (\mu_{0} R)^{-2}(t)
\ln t (t (\ln t)^{2})^{5\delta-\kappa }
R^{-a}(t)
+
(\ln t)^3 (t (\ln t)^{2})^{10\delta-2\kappa }
] \Bigg\} |t_1 -t_2|^{\alpha}
	\end{align*}
which implies
\begin{equation*}
	\begin{aligned}
		&
		( \ln t_1)^{-1}
		\Big|\int_{B_{2R_{0}}} 	  w^2(y) Z_5(y)  (\psi(\mu(t_1) y +\xi(t_1), t_1 )
		-
		\psi(\mu(t_2) y +\xi(t_2), t_2 )
		) dy \Big|
		\\
\lesssim \ &
[\ln t (t (\ln t)^{2})^{5\delta-\kappa }
R^{- a}(t)]^2 |t_1-t_2|
+
C(\alpha)
\Bigg\{
\lambda^{-2\alpha}(t)  (t (\ln t)^{2})^{5\delta-\kappa }
R^{- a}(t)
\\
&
+
\lambda^{2-2\alpha}(t)
[ (\mu_{0} R)^{-2}(t)
 (t (\ln t)^{2})^{5\delta-\kappa }
R^{-a}(t)
+
(\ln t)^2 (t (\ln t)^{2})^{10\delta-2\kappa }
] \Bigg\} |t_1 -t_2|^{\alpha} .
	\end{aligned}
\end{equation*}
Similarly, $(t (\ln t)^2)^{-\delta\epsilon_0} O(
\sup\limits_{y\in B_{4R(t)}} \langle y \rangle^{-1} |\psi(\mu y +\xi,t)|
) $ provides the same H\"older estimate as above.

Reviewing the analysis details in solving \eqref{mu1-xi-fixed point}, one has
\begin{equation*}
	|\Pi_{5}[\mu_1,\xi] |\lesssim
	\ln t (t(\ln t)^2)^{5\delta-\kappa} R^{-a}  \chi(t)
	.
\end{equation*}
Then
\begin{equation*}
	\begin{aligned}
		&
		\Big|(\beta_{\nu}(t) e^{ - \int^{t} \beta_{\nu}(u) du} \int_{t}^{\infty}
		e^{\int^s \beta_{\nu}(u) du}  \Pi_{5}[\mu_1,\xi](s) ds)' \Big|
		=
		\Bigg|\beta_{\nu}'(t) e^{ - \int^{t} \beta_{\nu}(u) du} \int_{t}^{\infty}
		e^{\int^s \beta_{\nu}(u) du}  \Pi_{5}[\mu_1,\xi](s) ds
		\\
		&
		-
		\beta_{\nu}^2(t) e^{ - \int^{t} \beta_{\nu}(u) du} \int_{t}^{\infty}
		e^{\int^s \beta_{\nu}(u) du}  \Pi_{5}[\mu_1,\xi](s) ds
		-
		\beta_{\nu}(t)
		\Pi_{5}[\mu_1,\xi]
		\Bigg|
		\\
		\lesssim \ &
	 t^{-2} (\ln t)^{-1} e^{ - \int^{t} \beta_{\nu}(u) du} \int_{t}^{\infty}
		e^{\int^s \beta_{\nu}(u) du}  	\ln s (s(\ln s)^2)^{5\delta-\kappa} R^{-a}(s) ds
		\\
		&
		+
		(t\ln t)^{-2} e^{ - \int^{t} \beta_{\nu}(u) du} \int_{t}^{\infty}
		e^{\int^s \beta_{\nu}(u) du}  \ln s (s(\ln s)^2)^{5\delta-\kappa} R^{-a}(s) ds
		+ (t\ln t)^{-1} \ln t (t(\ln t)^2)^{5\delta-\kappa} R^{-a}
		\\
		\lesssim \ &
		(t\ln t)^{-1} \ln t (t(\ln t)^2)^{5\delta-\kappa} R^{-a}
	\end{aligned}
\end{equation*}
where in the last inequality, we have used $e^{ - \int^{t} \beta_{\nu}(u) du} \int_{t}^{\infty}
e^{\int^s \beta_{\nu}(u) du}  	\ln s (s(\ln s)^2)^{5\delta-\kappa} R^{-a}(s) ds  \lesssim t 	\ln t (t(\ln t)^2)^{5\delta-\kappa} R^{-a}$ when $t_{0}$ is sufficiently large. Also
\begin{equation*}
	\left| [((1-\nu)\ln t + 2\ln\ln t  )^{-1}  \chi(t) ]'
	\big[ ((1-\nu)\ln t + 2\ln\ln t ) \Pi_{5}[\mu_1,\xi] \big]
	\right| \lesssim
	t^{-1}	\ln t (t(\ln t)^2)^{5\delta-\kappa} R^{-a}  1_{\{t\ge \frac{3t_0}{4} \}} ,
\end{equation*}
\begin{equation*}
	\begin{aligned}
		&
		\Big|((1-\nu)\ln t + 2\ln\ln t   )^{-1} [
		-
		O(    (t\ln t)^{-1} \sup\limits_{\tau_1\in [t/2,t]}|\mu_1(\tau_1)|    )
		- O(\tilde{g}[\bar{\mu}_0,\mu_1] )
		-
		(t (\ln t)^2)^{-\delta\epsilon_0} O(
		t^{a_{1}\gamma-2}  (\ln t)^{-1}
		)  ]' \Big|
		\\
		\lesssim \ &
		(\ln t)^{-1}
		(t^{-1} \ln t (t(\ln t)^2)^{5\delta-\kappa} R^{-a} + t^{-1} 	(t (\ln t)^2)^{-\delta\epsilon_0} O(
		t^{a_{1}\gamma-2}  (\ln t)^{-1}
		))
		\lesssim
		(t \ln t)^{-1}
		\ln t (t(\ln t)^2)^{5\delta-\kappa} R^{-a} .
	\end{aligned}
\end{equation*}

In order to get the  estimate
\begin{equation*}
\begin{aligned}
&
		( (1-\nu)\ln t + 2\ln\ln t   )^{-1} \Big|
		O(     \sup\limits_{\tau_1\in [t_1/2,t_1]}|\mu_{1t}(\tau_1)|    ) - O(     \sup\limits_{\tau_1\in [t_2/2,t_2] }|\mu_{1t}(\tau_1)|    )  \Big|
		\\
\lesssim \ &
(\ln t)^{-1}
( C(A)  t^{-1} \ln t (t(\ln t)^2)^{5\delta -\kappa} R^{-a} |t_1-t_2| + [ \mu_{1t} ]_{C^{\alpha}(\frac{3At}{4},t ) } |t_1-t_2|^{\alpha} ),
\end{aligned}
\end{equation*}
rigorously speaking, we need to estimate all the terms that appeared in the proof of Lemma \ref{varphi1-lem} and Lemma \ref{varphi2-lem} except  the leading term. For simplicity, we take $\tilde{\varphi}_{1b}[\mu+\mu_1] - \tilde{\varphi}_{1b}[\mu ]$ as an example to illustrate the key idea. We decompose $\tilde{\varphi}_{1b}[\mu+\mu_1] - \tilde{\varphi}_{1b}[\mu ]$ into two parts to estimate.
\begin{equation*}
\begin{aligned}
&
	( \tilde{\varphi}_{1b}[\mu +\mu_{1}] - \tilde{\varphi}_{1b}[\mu  ] )(\bar x,t) = \TT_4^{out}[ -\mu_{1t} \hat{\varphi}_{1} + (E - \tilde{E})[\mu+\mu_1] - (E - \tilde{E})[\mu ] ] (\bar{x},t)
\\
= \ & \left(\int_{t_0}^{A t} + \int_{At}^{t}\right)
\int_{\RR^4 }
(4\pi(t-s))^{- 2}
e^{-\frac{|\bar{x}-z|^2}{4(t-s)}}
(  -\mu_{1t} \hat{\varphi}_{1} + (E - \tilde{E})[\mu+\mu_1] - (E - \tilde{E})[\mu ] )(z,s) dzds .
\end{aligned}
\end{equation*}
Here $\bar{x}$ is regarded to be independent of $t$. Then
\begin{align*}
&
\left| \pp_{t} ( \int_{t_0}^{At}
\int_{\RR^4 }
(4\pi(t-s))^{- 2}
e^{-\frac{|\bar{x}-z|^2}{4(t-s)}}
(  -\mu_{1t} \hat{\varphi}_{1} + (E - \tilde{E})[\mu+\mu_1] - (E - \tilde{E})[\mu ] )(z,s) dzds ) \right|
\\
= \ &
\Bigg| A
\int_{\RR^4 }
(4\pi(t-At ))^{- 2}
e^{-\frac{|\bar{x}-z|^2}{4(t-At)}}
(  -\mu_{1t} \hat{\varphi}_{1} + (E - \tilde{E})[\mu+\mu_1] - (E - \tilde{E})[\mu ] )(z,At) dz
\\
& +
\int_{t_0}^{At}
\int_{\RR^4 }
\pp_{t} (
(4\pi(t-s))^{- 2}
e^{-\frac{|\bar{x}-z|^2}{4(t-s)}} )
(  -\mu_{1t} \hat{\varphi}_{1} + (E - \tilde{E})[\mu+\mu_1] - (E - \tilde{E})[\mu ] )(z,s) dzds
\Bigg|
\\
\lesssim \ &
C(A)
t^{-2}
\int_{\RR^4 }
e^{-\frac{|\bar{x}-z|^2}{4(t-At)}}
\Bigg(   |\mu_{1t}(At)| t^{-1} \1_{\{ | z |\le t^{\frac 12} \} } +
| \mu_{1t} (At)| | z |^{-2}
e^{-\frac{| z |^2}{4t}}
\1_{ \{ | z | >  t^{\frac 12} \} }
\\
&
+  |\mu_1(At)| \mu^2(At) t^{-3}
\1_{\{ 2^{-1}\sqrt t \le |z| \le 4 \sqrt t \} }
 \Bigg)  dz
\\
& +
\int_{t_0}^{At}
\int_{\RR^4 }
(t-s)^{- 3}
e^{-\frac{|\bar{x}-z|^2}{8(t-s)}}
\left|  -\mu_{1t} \hat{\varphi}_{1} + (E - \tilde{E})[\mu+\mu_1] - (E - \tilde{E})[\mu ] )(z,s) \right| dzds
\\
\lesssim \ &
C(A)
t^{-2}
\int_{\RR^4 }
e^{-\frac{|z|^2}{4(t-At)}}
\Bigg(   |\mu_{1t}(At)| t^{-1} \1_{\{ | z |\le t^{\frac 12} \} } +
| \mu_{1t} (At)| | z |^{-2}
e^{-\frac{| z |^2}{4t}}
\1_{ \{ | z | >  t^{\frac 12} \} }
\\
&
+  |\mu_1(At)| \mu^2(At) t^{-3}
\1_{\{ |z| \le 4 \sqrt t \} }
\Bigg)  dz
\\
& +
C(A)
t^{-1}
\int_{t_0}^{At}
\int_{\RR^4 }
(t-s)^{- 2}
e^{-\frac{|\bar{x}-z|^2}{8(t-s)}}
\bigg|
(  -\mu_{1t} \hat{\varphi}_{1} + (E - \tilde{E})[\mu+\mu_1] - (E - \tilde{E})[\mu ] )(z,s) \bigg| dzds
\\
\lesssim \ &
  C(A) \Bigg[ t^{-1}  |\mu_{1t}(At)|
+  t^{-3} |\mu_1(At)| \mu^2(At)
\\
& +
t^{-1}
\int_{t_0}^{At}
\int_{\RR^4 }
(t-s)^{- 2}
e^{-\frac{|\bar{x}-z|^2}{8(t-s)}}
\bigg| (  -\mu_{1t} \hat{\varphi}_{1} + (E - \tilde{E})[\mu+\mu_1] - (E - \tilde{E})[\mu ] )(z,s) \bigg| dz ds
\Bigg]
\\
\lesssim \ &
C(A) t^{-1}
\ln t (t (\ln t)^2)^{5\delta-\kappa} R^{-a}
\end{align*}
where the last inequality follows from the same calculations as in \eqref{varphi1b-1}.

For the other part, we have
\begin{equation*}
\begin{aligned}
&
 \int_{At}^{t}
\int_{\RR^4 }
(4\pi(t-s))^{- 2}
e^{-\frac{|\bar{x}-z|^2}{4(t-s)}}
\bigg(  -\mu_{1t} \hat{\varphi}_{1} + (E - \tilde{E})[\mu+\mu_1] - (E - \tilde{E})[\mu ] \bigg)(z,s) dzds
\\
= \ &
\int_{A}^{1}
\int_{\RR^4 }
t
(t-ta)^{- 2}
e^{-\frac{|\bar{x}-z|^2}{4(t-ta)}}
\bigg(  -\mu_{1t} \hat{\varphi}_{1} + (E - \tilde{E})[\mu+\mu_1] - (E - \tilde{E})[\mu ] \bigg)(z,ta) dzda
.
\end{aligned}
\end{equation*}

The terms independent of $\mu_{1t}$ are $C^{1}$ in time variable $t$.
We only need to focus on the terms including $\mu_{1t}$. By similar calculations in \eqref{varphi1b-1}, we have
\begin{equation*}
\begin{aligned}
	&
\left| \int_{A}^{1}
\int_{\RR^4 }
\pp_{t} (t
(t-ta)^{- 2}
e^{-\frac{|\bar{x}-z|^2}{4(t-ta)}} )
(\mu_{1t} \hat{\varphi}_{1} ) (z,ta) dz da \right|
\lesssim
\int_{A}^{1}
\int_{\RR^4 }
(t-ta)^{- 2}
e^{-\frac{|\bar{x}-z|^2}{8(t-ta)}} )
|\mu_{1t} \hat{\varphi}_{1} (z,ta) |  dzda
\\
\lesssim \ &
t^{-1} \ln t (t (\ln t)^2)^{5\delta-\kappa} R^{-a},
\end{aligned}
\end{equation*}
\begin{equation*}
\begin{aligned}
	&
\left| \int_{A}^{1}
\int_{\RR^4 }
t
(t-ta)^{- 2}
e^{-\frac{|\bar{x}-z|^2}{4(t-ta)}}
 ( \mu_{1t}(t_1a) -  \mu_{1t}(t_2 a) )  \hat{\varphi}_{1}(z,ta)   dzda \right|
 \\
 \lesssim \ &
 \int_{A}^{1}
 \int_{\RR^4 }
 t
 (t-ta)^{- 2}
 e^{-\frac{|\bar{x}-z|^2}{4(t-ta)}}
[ \mu_{1t}]_{C^{\alpha} (\frac{3At}{4},t )  }
 |t_1 -  t_2 |^{\alpha} a^{\alpha} | \hat{\varphi}_{1}(z,ta)  |  dzda
 \lesssim
 [ \mu_{1t}]_{C^{\alpha} (\frac{3At}{4},t )  }
 |t_1 -  t_2 |^{\alpha} .
\end{aligned}
\end{equation*}

Next, for $\int_{t/2}^{t-t^{1-\nu} } \frac{\mu_{1t}(s)}{t-s} d s = (\int_{t/2}^{ A t } + \int_{ A t }^{t-t^{1-\nu} } ) \frac{\mu_{1t}(s)}{t-s} d s$, we have
\begin{equation*}
	\left|( \int_{t/2}^{ A t }   \frac{\mu_{1t}(s)}{t-s} d s) ' \right|
	=
	\left| \frac{A \mu_{1t}( A t)}{t- A t}  - \frac{\mu_{1t}(\frac{t}{2})}{t} -  \int_{t/2}^{ A t }   \frac{\mu_{1t}(s)}{(t-s)^2} d s  \right|
	\lesssim C(A)
	t^{-1}  \ln t	(t (\ln t)^{2})^{5\delta-\kappa }
	R^{- a} ,
\end{equation*}
and
	\begin{align*}
		&
		\left|\int_{A t_1  }^{t_1-t_1^{1-\nu} } \frac{\mu_{1t}(s)}{t_1-s} d s - \int_{A t_2 }^{t_2-t_2^{1-\nu} } \frac{\mu_{1t}(s)}{t_2-s} d s\right|
		=
		\left|\int_{ A }^{1-t_1^{-\nu}} \frac{\mu_{1t}(t_1z)}{1-z} dz - \int_{A }^{1-t_2^{-\nu}} \frac{\mu_{1t}(t_2z)}{1-z} dz\right|
		\\
		= \ &
		\left|\int_{ A }^{1-t_1^{-\nu}} \frac{\mu_{1t}(t_1z) - \mu_{1t}(t_2z)}{1-z} dz + \int_{1-t_2^{-\nu}}^{1-t_1^{-\nu}} \frac{\mu_{1t}(t_2z)}{1-z} dz\right|
		\\
		\le \ &
		|t_1-t_2|^{\alpha} [\mu_{1t}]_{C^{\alpha}
			( \frac{3At}{4},t)}\int_{ A }^{1-t_1^{-\nu}} \frac{ z^\alpha }{1-z} dz
		+
		C
		\int_{1-t_2^{-\nu}}^{1-t_1^{-\nu}}
		\frac{ (t_2z)^{5\delta-\kappa-a\gamma} (\ln (t_2 z))^{ 1+2(5\delta -\kappa)} }{1-z} dz
		\\
		= \ &
		|t_1-t_2|^{\alpha} [\mu_{1t}]_{ C^{\alpha}( \frac{3At}{4},t) }
		\nu\ln t_1 (1+ O(|\ln (1-A)|(\ln t_1)^{-1}) )
		\\
		&
		+
		C (1+ O((\ln t_2)^{-1}) )  t_2^{5\delta-\kappa-a\gamma} (\ln t_2 )^{ 1+2(5\delta -\kappa)}
		\int_{1-t_2^{-\nu}}^{1-t_1^{-\nu}}
		\frac{ z^{5\delta-\kappa-a\gamma} }{1-z} dz
		\\
		\le \ &
		|t_1-t_2|^{\alpha} [\mu_{1t}]_{  C^{\alpha}( \frac{3At}{4},t)  }
		\nu\ln t_1 (1+ O( |\ln(1-A)|(\ln t_1)^{-1}) )
		\\
		&
		+
		C (1+ O((\ln t_2)^{-1}) )  t_2^{5\delta-\kappa-a\gamma} (\ln t_2 )^{ 1+2(5\delta -\kappa)}
		(1+O(t^{-1})) \nu |\ln t_1 -\ln t_2|
		\\
		\le \ &
		|t_1-t_2|^{\alpha} [\mu_{1t}]_{  C^{\alpha}( \frac{3At}{4},t)  }
		\nu\ln t_1 (1+ O( |\ln(1-A)|(\ln t_1)^{-1}) )
		\\
		&
		+
		C (1+ O((\ln t_2)^{-1}) )  t_2^{5\delta-\kappa-a\gamma} (\ln t_2 )^{ 1+2(5\delta -\kappa)}
		(1+O(t^{-1})) \nu t_2^{-1}|t_1 - t_2|
		.
	\end{align*}

Combining the estimates above, one gets that
\begin{equation*}
	\begin{aligned}
		&
		|\mu_{1t}(t_1) - \mu_{1t}(t_2)|
		\\ \le \ &
		C(A) [(t\ln t)^{-1} \ln t (t(\ln t)^2)^{5\delta-\kappa} R^{-a}(t) +
		t^{-1}	\ln t (t(\ln t)^2)^{5\delta-\kappa} R^{-a}(t)  1_{\{t \ge \frac{3t_0}{4} \}}
		]|t_1 -t_2|
		\\
		&
		+
		C(\alpha) \Big\{
		\lambda^{-2\alpha}(t)  (t (\ln t)^{2})^{5\delta-\kappa }
		R^{- a}(t)
		\\
		&
		+
		\lambda^{2-2\alpha}(t)
		[ (\mu_{0} R)^{-2}(t)
		(t (\ln t)^{2})^{5\delta-\kappa }
		R^{-a}(t)
		+
		(\ln t)^2 (t (\ln t)^{2})^{10\delta-2\kappa }
		] \Big\} |t_1 -t_2|^{\alpha}
\\
& +
|t_1-t_2|^{\alpha} [\mu_{1t}]_{  C^{\alpha}(\frac{3At}{4},t)  }
[\nu (1-\nu)^{-1} (1+ O( |\ln(1-A)|(\ln t)^{-1}) + O((\ln t)^{-1} ) )]
	\end{aligned}
\end{equation*}
where we have used $\frac{3t}{4}\le t_2 < t_1 \le t$. Thus one has
\begin{equation*}
 [\mu_{1t}]_{  C^{\alpha}(\frac{3t}{4},t)  } \le
C(A,\alpha) \rho(t) +
[\mu_{1t}]_{  C^{\alpha}( \frac{3At}{4},t)  }
[
\nu (1-\nu)^{-1}  (1+ O( |\ln(1-A)|(\ln t_0)^{-1})
+ O((\ln t)^{-1} ) ]
)\1_{\{ t\ge \frac{3t_0}{4} \}}
\end{equation*}
where
\begin{equation*}
\rho(t) =
 t^{-\alpha} \ln t (t (\ln t)^{2})^{5\delta-\kappa }
R^{- a}
+
t^{1-\alpha}
[ (\mu_{0} R)^{-2}
(t (\ln t)^{2})^{5\delta-\kappa }
R^{-a}
+
(\ln t)^2 (t (\ln t)^{2})^{10\delta-2\kappa }
] .
\end{equation*}
Thus
\begin{equation*}
\begin{aligned}
&
\sup\limits_{ \frac{t_0}{2} \le t \le T }\rho^{-1}(t)	[\mu_{1t}]_{  C^{\alpha}(\frac{3t}{4},t)  }
\\
\le \ &
	C(A,\alpha)  +
	[
\nu (1-\nu)^{-1}  (1+ O( |\ln(1-A)|(\ln t_0)^{-1}) )
+ O((\ln t_0)^{-1})
]
	\sup\limits_{ \frac{t_0}{2} \le t \le T } \rho^{-1}(t)
	[\mu_{1t}]_{  C^{\alpha}( \frac{3At}{4} ,t)  }
	\1_{\{ t \ge \frac{3t_0}{4} \}}
\\
= \ &
C(A,\alpha)  +
[
\nu (1-\nu)^{-1}  (1+ O( |\ln(1-A)|(\ln t_0)^{-1}) )
+ O( (\ln t_0)^{-1} )
]
\sup\limits_{ \frac{3t_0}{4} \le t\le T } \rho^{-1}(t)
[\mu_{1t}]_{  C^{\alpha} ( \frac{3At}{4} ,t)  }.
\end{aligned}
\end{equation*}
Notice
\begin{align*}
	&
\rho^{-1}(t) [\mu_{1t}]_{  C^{\alpha} ( \frac{3At}{4} ,t)  }
= \rho^{-1}(t) \sup\limits_{s_1,s_2\in ( \frac{3At}{4} ,t)  } \frac{|\mu_{1t}(s_1) - \mu_{1t}(s_2)|}{|s_1 -s_2|^{\alpha}}
\\
= \ &
\max\Bigg\{
\rho^{-1}(t) \sup\limits_{s_1,s_2\in ( \frac{3t}{4} ,t)  } \frac{|\mu_{1t}(s_1) - \mu_{1t}(s_2)|}{|s_1 -s_2|^{\alpha}}
 ,
\rho^{-1}(t) \sup\limits_{s_1,s_2\in ( \frac{3At}{4} , At )  } \frac{|\mu_{1t}(s_1) - \mu_{1t}(s_2)|}{|s_1 -s_2|^{\alpha}},
\\
&
\rho^{-1}(t) \sup\limits_{s_2\in ( \frac{3At}{4} , \frac{3t}{4}  ), s_1\in (At,t)  } \frac{|\mu_{1t}(s_1) - \mu_{1t}(s_2)|}{|s_1 -s_2|^{\alpha}}
 \Bigg\}
\\
\le \ &
\max\Bigg\{
\rho^{-1}(t) \sup\limits_{s_1,s_2\in ( \frac{3t}{4} ,t)  } \frac{|\mu_{1t}(s_1) - \mu_{1t}(s_2)|}{|s_1 -s_2|^{\alpha}}
,
\rho^{-1}(t) \rho(At) \rho^{-1}(At) \sup\limits_{s_1,s_2\in ( \frac{3At}{4} , At )  } \frac{|\mu_{1t}(s_1) - \mu_{1t}(s_2)|}{|s_1 -s_2|^{\alpha}}
\Bigg\} +C,
\end{align*}
then one has
\begin{equation*}
\sup\limits_{ \frac{3t_0}{4} \le t \le T } \rho^{-1}(t)
[\mu_{1t}]_{  C^{\alpha} ( \frac{3At}{4} ,t)  } \le A^{C(\alpha,\delta,\kappa,a)} (1+ O( (\ln t_0)^{-1} ) )  \sup\limits_{ \frac{t_0}{2} \le t \le T }\rho^{-1}(t)	[\mu_{1t}]_{  C^{\alpha}(\frac{3t}{4},t)  }  + C(A,\alpha) .
\end{equation*}
Thus, when $\nu<\frac 12$, taking $A$ close to $1$ sufficiently, which depends on $\nu$, and then making $t_0$ large enough, one has $\sup\limits_{ \frac{t_0}{2} \le t \le T }\rho^{-1}(t)	[\mu_{1t}]_{  C^{\alpha}(\frac{3t}{4},t)  }
\le
C(\nu,\alpha)$. Making $T\rightarrow\infty$, one finally gets
\begin{equation}
\sup\limits_{t \ge \frac{t_0}{2} }\rho^{-1}(t)	[\mu_{1t}]_{  C^{\alpha}(\frac{3t}{4},t)  }
\le
C(\nu,\alpha) .
\end{equation}

Finally, we estimate $\mu \mathcal{E}_{\nu}[\mu_{1}]$ as follows
\begin{equation}\label{muE-est}
	\begin{aligned}
&	|	\mu \mathcal{E}_{\nu}[\mu_{1}] | \lesssim
C(\nu,\alpha) \{
(\ln t)^{-1}
t^{ -\nu\alpha}
 \ln t (t (\ln t)^{2})^{5\delta-\kappa }
R^{- a}
\\
&
+ (\ln t)^{-1}
t^{1-\nu\alpha}
 (\mu_{0} R)^{-2}
(t (\ln t)^{2})^{5\delta-\kappa }
R^{-a}
+ (\ln t)^{-1}
t^{1-\nu\alpha}
(\ln t)^2 (t (\ln t)^{2})^{10\delta-2\kappa } \} .
	\end{aligned}
\end{equation}
Although $C_{\nu,\alpha}$ goes to $\infty$ as $\nu\rightarrow \frac 12$ and $\alpha \rightarrow 1$, the smallness is given by $t_{0}^{-\epsilon}$ where $\epsilon>0$ when solving \eqref{phi2-eq}. Once $\nu$ and $\alpha$ are fixed, we take $t_0$ large enough.

\medskip

\section{Solving the inner problem}
Recalling \eqref{norm-phi}, for any fixed $\phi\in B_{i}$ with
\begin{equation}\label{phi-topo}
	B_i = \{ \phi  \ : \ \| \phi \|_{i,\kappa - 5\delta,a} \le 2 C_i  \}
\end{equation}
where $C_i >1$ is a constant,
we have found $\psi[\phi]\in B_{o}$, $\mu_{1}[\phi]\in B_{\mu_1}$ and $\xi[\phi]\in B_{\xi}$. We abbreviate  $\mathcal{H}[\phi] = \mathcal{H}[\psi[\phi],\mu_1[\phi],\xi[\phi]] $.
By \eqref{H-upp}, \eqref{inner-para} and \eqref{muxi-topo}, we obtain $|\mathcal{H}[\phi](y,t(\tau)) |  \lesssim \tau^{5\delta -\kappa} (\ln \tau)^4 R^{-a}(t(\tau))  \langle y \rangle^{-2-a_1} $.
The orthogonal equations of $\mu_1$ and $\xi$ have been solved in Section \ref{mu1-xi-subsec}, then by Proposition \ref{R0-linear}, one finds a solution for \eqref{phi1-eq} satisfying
\begin{equation*}
\langle y\rangle|\nabla \phi_1(y,\tau)| + 	|\phi_1(y,\tau)| \lesssim
	\tau^{5\delta -\kappa} (\ln \tau)^4 R^{-a}(t(\tau))  R_0^{5} \langle y \rangle^{-a_1}
	\lesssim
	\tau_0^{-\epsilon} 	\tau^{5\delta -\kappa}  \langle y \rangle^{-a}
\end{equation*}
with $\epsilon>0$ sufficiently small provided
\begin{equation}\label{inn-para-1}
\gamma\min\{a,a_1\}>5\delta .
\end{equation}

Combining \eqref{muE-est} and Lemma \ref{mode0-nonorth}, one can find a solution for \eqref{phi2-eq} with the estimate
\begin{equation*}
\langle y\rangle |\nabla\phi_2(y,\tau)|+|\phi_2(y,\tau)|\lesssim \tau_{0}^{-\epsilon} \tau^{5\delta -\kappa}  \langle y \rangle^{-a}
\end{equation*}
if
\begin{equation}\label{inn-para2}
		-\nu \alpha + (2-a)\gamma<0,
\
		1-\nu\alpha -a\gamma <0,
\
		1-\nu\alpha +5\delta -\kappa+2\gamma<0, \ 0<a<2.
\end{equation}

Combining \eqref{inner-para}, \eqref{outer-para}, \eqref{ortho-para}, \eqref{inn-para-1}, \eqref{inn-para2} and the assumption about parameters in Proposition \ref{R0-linear} and Lemma \ref{mode0-nonorth}, one needs to choose parameters such that all the inequalities below hold
\begin{equation}\label{para-eq}
\begin{aligned}
&	5\delta -\kappa -a\gamma >-2,\  5\delta-\kappa < -1 , \ 0<a<2,
	\  0<\gamma<\frac 12 ,\ 0<\alpha<1, \ 0<\nu<\frac 12,
\\
& a_1\gamma -2 <5\delta -\kappa-a\gamma , \  0<a_1\le 1,
\ \gamma\min\{a,a_1\}>5\delta, \ 6\delta<1,
\\
& -\nu \alpha + (2-a)\gamma<0,
\
1-\nu\alpha -a\gamma <0,
\
1-\nu\alpha +5\delta -\kappa+2\gamma<0.
\end{aligned}
\end{equation}
There exists solution given by
\begin{equation}\label{para-require}
\begin{aligned}
&
1<\kappa \le \frac{5}{4}, \
\frac{2-2\alpha \nu}{-1+\kappa+\nu} <a<2, \
\frac{1-\alpha\nu}{a} <\gamma <\frac{-1+\kappa +\nu}{2}, \
\\
&
\frac{2-\kappa}{2} <\alpha\nu <\frac{1}{2}, \ 0<\alpha<1, \ 0<\nu<\frac 12, \ 6\delta<1,
\\
&
0<5\delta<\kappa-1, \
 a_1\gamma < 5\delta + 2 -\kappa-a\gamma, \ 5\delta <\gamma\min\{a,a_1\}, \   0<a_1\le 1 .
\end{aligned}
\end{equation}
Indeed, one may take for example
 $
\kappa= \frac{9}{8}, \nu =\frac{49}{100}, \alpha=\frac{375}{392},
 a=\frac{36}{19}, \gamma = \frac{9}{32},  5\delta=\frac{1}{64}, a_1 =\frac{1}{9} $.

Thanks to \eqref{para-eq}, the desired $\phi_1$, $\phi_2$ can really be found and then $\phi_1 +\phi_2 \in B_{i}$ when $\tau_0$ is large enough. The compactness is a consequence of parabolic estimates, so we can find a solution for the inner problem \eqref{phi-eq}. Making more efforts to calculate the Lipschitz continuity of $\mathcal{H}[\phi]$  about $\phi$, one can prove the existence for the inner problem \eqref{phi-eq} by the contraction mapping theorem.

Collecting the estimates in Proposition \ref{psi-prop}, Corollary \ref{varphi-coro}, \eqref{Phi0-est} and \eqref{phi-topo}, one gets
\begin{equation*}
	\begin{aligned}
	&
\left| \varphi[\mu] +  \bar{\mu}_0^{-1} \Phi_0(\frac{x-\xi}{\bar{\mu}_0} , t) \eta(\frac{4(x-\xi) }{\sqrt t}) + \psi + \eta_R  \mu^{-1} \phi(\frac{x-\xi}{\mu},t) \right|
\\
\lesssim \ &
		(t\ln t)^{-1} \1_{\{ |\bar{x}| \le 2t^{\frac 12} \}}
		+
		O( t^{2} (\ln t)^{-1}
		|\bar{x}|^{-6}
		) \1_{\{ |\bar{x}| > 2t^{\frac 12}  \}}
		+
		( \ln t )	t^{-1} (\ln t)^{-2} \langle \bar{y} \rangle^{-2}
		\ln (2+ |\bar{y}|) \1_{ \{ |\bar{x}| \le t^{\frac 12} \} }
		\\
		& +
		\ln t (t (\ln t)^{2})^{5\delta-\kappa }
		R^{- a}  \left(\1_{\{|x|\le t^{\frac 12} \}} + t |x|^{-2} \1_{\{|x| > t^{\frac 12} \}}\right)
		+ (t(\ln t)^2)^{-\kappa +5\delta} \langle y\rangle^{-a} \1_{ \{|y|\le 4R \} }
		\\
		\lesssim \ &
		(t\ln t)^{-1} \1_{\{ |\bar{x}| \le 2t^{\frac 12} \}}
		+
		O(  (\ln t)^{-1}
		|\bar{x}|^{-2}
		) \1_{\{ |\bar{x}| > 2t^{\frac 12}  \}}
		\lesssim
(\ln t)^{-1} \min\{ t^{-1}, |x|^{-2}\} .
	\end{aligned}
\end{equation*}

\medskip

\textbf{Positivity of the solution $u$.}
We will demonstrate that the initial value $u(x,t_0)$ that we take in the construction is positive. For simplicity, we abuse the symbols $\mu=\mu(t_0), \bar{\mu}_0 =\bar{\mu}_0(t_0)$ in the remainder of this section. Indeed, recalling \eqref{tilde-varphi1}, \eqref{Phi0-est} and \eqref{phi-topo}, we have
\begin{equation*}
	\begin{aligned}
&	u(x,t_0) =
		\mu^{-1}w\left(\frac{\bar{x} }{\mu}\right) \eta\left(\frac{\bar{x} }{\sqrt {t_0} }\right)
		+
		\tilde{\varphi}_{1}(\bar{x},t_0)
		+
		\bar{\mu}_0^{-1} \Phi_0\left(\frac{\bar{x}}{\bar{\mu}_0} , t_0\right)\eta\left(\frac{4\bar{x}}{\sqrt {t_0}}\right)
		+
		\eta_{R} \mu^{-1} \phi\left(\frac{\bar{x}}{\mu},t_0\right)
		\\
		= \ &
		2^{\frac 32} \mu |\bar{x}|^{-2}
		\Bigg\{
		\left[
		e^{-\frac{| \bar{x} |^2}{4t_0 }}
		- \left( 1+ \left(\frac{| \bar{x} |}{\mu}\right)^2 \right)^{-1}
		+
		2^{-\frac{3}{2}} \mu^{-1} |\bar{x}|^2
		\bar{\mu}_0^{-1} \Phi_0\left(\frac{\bar{x}}{\bar{\mu}_0} , t_0 \right)\eta\left(\frac{4\bar{x}}{\sqrt {t_0} }\right)
		\right]
		\eta\left(\frac{| \bar{x} |}{\sqrt {t_0} }\right)\\
		\ &+
		e^{-\frac{| \bar{x} |^2}{4t_0 }} \left(1- \eta\left(\frac{| \bar{x} |}{\sqrt {t_0}}\right) \right)
		\Bigg\}
		+
		\eta_{R} \mu^{-1} \phi\left(\frac{\bar{x}}{\mu},t_0\right)
		\\
		\ge \ &
		2^{\frac 32} \mu |\bar{x}|^{-2}
		\left[
		1 -\frac{| \bar{x} |^2}{4t_0 }
		- \left( 1+ \left(\frac{| \bar{x} |}{\mu}\right)^2 \right)^{-1}
		-
		C |\bar{x}|^2 	(t_0 \ln t_0)^{-1} \ln \ln t_0
		\eta\left(\frac{4\bar{x}}{\sqrt {t_0} }\right)
		\right] \eta\left(\frac{| \bar{x} |}{\sqrt {t_0}}\right)
		+
		\eta_{R} \mu^{-1} \phi\left(\frac{\bar{x}}{\mu},t_0\right)
		\\
		= \ & 2^{\frac 32} \mu |\bar{x}|^{-2}
		| \bar{x} |^2 \left[ (\mu^2 + | \bar{x} |^2)^{-1} -(4t_0)^{-1}  -
		C	(t_0 \ln t_0)^{-1} \ln \ln t_0
		\eta\left(\frac{4\bar{x}}{\sqrt {t_0}}\right) \right] \eta\left(\frac{| \bar{x} |}{\sqrt {t_0}}\right)
		+
		\eta_{R} \mu^{-1} \phi\left(\frac{\bar{x}}{\mu},t_0\right)
		\\
		\ge \ &
		2^{-\frac 12 } \mu^{-1} ( 1 + | y |^2)^{-1}  \eta\left(\frac{| \bar{x} |}{\sqrt {t_0} }\right)
		-
		C_1 \mu^{-1} (t_0(\ln t_0)^2)^{ 5\delta -\kappa  }
		\langle y\rangle^{-a}  \eta_{R} >0
	\end{aligned}
\end{equation*}
where we have used  $\eta(s) =0$ for $s\ge \frac 32$ to make $ [ \frac{5}{8} (\mu^2 + | \bar{x} |^2)^{-1} -(4t_0)^{-1}  -
		C	(t_0 \ln t_0)^{-1} \ln \ln t_0
		\eta(\frac{4\bar{x}}{\sqrt {t_0}}) ] \eta(\frac{|\bar{x}|}{\sqrt{t_{0}}} )\ge 0$ when $t_{0}$ is large, and $5\delta-\kappa + \gamma(2-a)<0$ is used in the last inequality. Therefore, the solution $u(x,t)$ is positive by maximum principle.

\section{Stability of blow-up: proof of Theorem \ref{stability-th}}\label{stability-sec}
In this section, we will  analyze the stability of the blow-up solution constructed in Theorem \ref{thm1}.
\begin{proof}[Proof of Theorem \ref{stability-th}]
 Consider any perturbation $g_0(x)$ satisfying $|g_0(x)|\lesssim t_{0}^{-\frac{\min{ \{\ell,4 \} } }{2}} \langle x\rangle^{-\ell}$, $\ell>2$. Set
\begin{equation*}
	\psi_{0}(x,t) =
	(4\pi (t-t_0))^{-2} \int_{\RR^4}
	e^{-\frac{|x-z|^2}{4(t-t_0)} }
	g_0(z) d z
\end{equation*}
which satisfies
$\pp_{t} \psi_{0} = \Delta \psi_{0} $ in $ \RR^4\times (t_0,\infty)$, $\psi(x,t_0)=g_0(x)$ in $\RR^4$.
Without loss of generality, we only consider the case $2<\ell<4$. By Lemma \ref{cauchy-est-lem}, one has
\begin{equation*}
|\psi_{0}(x,t)|
\lesssim
t_0^{-\frac{\ell}{2}} \left(
	\langle t-t_0 \rangle^{-\frac{\ell}{2}}	
\1_{\{ |x|\le \langle t-t_0 \rangle^{\frac 12} \}}
+
| x |^{-\ell}
\1_{\{ |x| > \langle t-t_0 \rangle^{\frac 12} \}} \right)
\lesssim
t^{-\frac{\ell}{2}}	
\1_{\{ |x|\le t^{\frac 12} \}}
+
| x |^{-\ell}
\1_{\{ |x| > t^{\frac 12} \}}
.
\end{equation*}

We modify the proof of Proposition \ref{psi-prop} slightly  in order to match the perturbation $g_0$. Indeed, we split $\psi= \bar{\psi} + \psi_0$ and consider
	\begin{equation*}
	\pp_t \bar{\psi}(x,t)
	=
	\Delta \bar{\psi}(x,t) + \mathcal{G} [\bar{\psi} +\psi_0,\phi,\mu_1,\xi]
	\mbox{ \ in \ } \RR^4 \times (t_{0},\infty),
	\quad
\bar{\psi}(x,t_0) = 0 \mbox{ \ in \ } \RR^4 .
\end{equation*}
When $\ell >2(\kappa+a\gamma)$, by \eqref{para-require}, one has $|\psi_0|\lesssim t_0^{-\epsilon} w_o$, and thus $\bar{\psi}$ can be solved in $B_o$ by the same method in Proposition \ref{psi-prop}. Repeating the rest procedures in the construction of Theorem \ref{thm1}, $(\mu_1,\xi,\phi, e_0)= (\mu_1[g_0],\xi[g_0],\phi[g_0],e[g_0])$ can be solved in the same topology that we have used before, and the leading order of blowup rate $\bar{\mu}_0\sim (\ln t)^{-1}$ remains the same.  The perturbed initial value is then given by
\begin{equation*}
\begin{aligned}
&
\Bigg[(\bar{\mu}_0 + \mu_{1}[g_0])^{-1} w\bigg(\frac{x-\xi[g_0] }{ \bar{\mu}_0 + \mu_{1}[g_0] }\bigg) \eta\bigg(\frac{ x-\xi[g_0]  }{\sqrt {t} }\bigg)
+
2^{\frac 32} (\bar{\mu}_0 + \mu_{1}[g_0]) |x-\xi[g_0] |^{-2} \bigg(e^{-\frac{|x-\xi[g_0] |^2}{4t }} -\eta(\frac{ x-\xi[g_0] }{\sqrt {t} })\bigg)
\\
&
+
\bar{\mu}_0^{-1} \Phi_0\bigg(\frac{ x-\xi[g_0] }{\bar{\mu}_0} , t\bigg)\eta\bigg(\frac{4(x-\xi[g_0]) }{\sqrt {t}}\bigg)
+
\eta\bigg(\frac{x-\xi[g_0]}{ \mu_0 R}\bigg)
e_0[g_0]  (\bar{\mu}_0 + \mu_{1}[g_0])^{-1} Z_{0}\bigg(\frac{x-\xi[g_0]}{ \bar{\mu}_0 + \mu_{1}[g_0] }\bigg) \Bigg] \Bigg|_{t=t_0} + g_0.
\end{aligned}
\end{equation*}

From \eqref{para-require}, $\kappa>1$ and $a \gamma>1-\alpha\nu$. So all $\ell>3$ is permitted for $\kappa$ and $\alpha \nu $ close to $1$ and $\frac 12$, respectively.

In the radial setting, the translation parameter $\xi\equiv 0$ automatically in \eqref{eqn}. Then for $2<\ell\le 3$, we put $3 u_1^2 \psi_0$ into the right hand side in the equation \eqref{Phi0-eq}. Since $|\psi_{0}(x,t)|  \lesssim t^{-\frac{\ell}{2}}$, $\ell>2$, the extra term involving $3 u_1^2 \psi_0$ will not influence the leading order $\mu_0$ and will be absorbed into $\Phi_0$. But recalling the construction of $\bar{\mu}_0$ in Section \ref{SolElliptic-sec}, $\bar{\mu}_0$  depends on $g_0$, namely, $\bar{\mu}_0 = \bar{\mu}_0[g_0] $.

We omit the tedious calculations about the Lipschitz continuity with respect to $g_0$ for $\psi,\phi,\mu_1,\xi$ here.
\end{proof}

\begin{remark}\label{rmk10.0.1}
\noindent
\begin{itemize}
	\item In general nonradial case and $\ell>2$, since $\psi_0$ is not radial about $\bar{x}=x-\xi$, the previous ODE solution about \eqref{Phi0-eq} is not allowed. Instead, we can expand \eqref{Phi0-eq} by modes similar to the manipulation in section \ref{sec-inner} and solve the leading order of $\mu$ and $\xi$. Since this involves more technicalities, given the length of this paper, we refrain from considering
such a generality here.
	
\item The borderline $\ell>2$ is also provided in \cite{filaking12}.
	
\item The stability result can be expected for $|g_0(x)|\lesssim t_{0}^{-1}(\ln t_0)^{-b_1} \langle x\rangle^{-2} (\ln (|x|+2))^{-b_2}$ for some $b_1,b_2>0$. The proof can be in fact achieved by similar computations as in the proof of Theorem \ref{stability-th}.
\end{itemize}
\end{remark}

\medskip

\section{ Linear theory for the inner problem }\label{sec-inner}
In this section, we develop a linear theory for the associated inner problem. Since the construction is independent of the spatial dimension $n$, we assume $n\ge 3$ in this section unless specifically stated otherwise. Set
\[
\mathcal{D}_{R} = \{
(y,\tau) \ | \
y\in B_{R(\tau)}, \ \tau \in (\tau_0,\infty)
\}, \quad
\pp \mathcal{D}_{R} = \{
(y,\tau) \ | \
y\in \pp B_{R(\tau)}, \ \tau \in (\tau_0,\infty)
\} .
\]
We consider the associated linear problem
\begin{equation}\label{A1}
		\pp_{\tau} \phi = \Delta \phi + pU^{p-1} \phi + f_1(y,\tau) \phi +  f_2(y,\tau) y \cdot \nabla \phi  +h(y,\tau)
		\mbox{ \ in \ } \DD_{R}
\end{equation}
where
\begin{equation*}
p=\frac{n+2}{n-2},\quad
U(y)
= (n(n-2))^{\frac{n-2}{4} }   \left( 1+|y|^2 \right)^{-\frac{n-2}{2} }
.
\end{equation*}
Throughout this section, we always assume that $f_1, f_2$ satisfy
 \begin{equation}\label{f-assump}
 f_i(y,\tau) = f_i(|y|,\tau)  \mbox{ \ are radial in space}, ~i=1,2, \ \
|f_1|, |f_2|, |y||\nabla f_2| \le C_{f} \tau^{-d}, \ d>0, \  C_{f}\ge 0.
 \end{equation}

It is easier to make mode expansion by spherical harmonic functions when $f_1$ and $f_2$ are radial. And it is very possible to generalize the linear theory without the assumption that $f_1$ and $f_2$ are radially symmetric.

Recall that the linearized operator $\Delta + pU^{p-1}$
has only one positive eigenvalue $\gamma_0>0$ such that
\begin{equation}\label{def-Z0Z0}
\Delta Z_0 + pU^{p-1} Z_0=\gamma_0 Z_0,
\end{equation}
where the corresponding eigenfunction $Z_0 \in L^{\infty}(\R^n)$ is radially symmetric with the asymptotic behavior
\begin{equation*}
	Z_0(y)\sim |y|^{-\frac{n-1}{2}}e^{-\sqrt{\gamma_0}|y|}~\mbox{ as }~|y|\to \infty.
\end{equation*}

The bounded kernels of $\Delta + pU^{p-1}$ are given by
\begin{equation*}
Z_{i}(y) = \pp_{y_i} U(y), \ i=1,2,\dots,n,\quad
Z_{n+1}(y)
 =
  y\cdot \nabla U(y) +\frac{n-2}{2} U(y).
\end{equation*}

Define the weighted $L^{\infty}$ norm
\begin{equation*}
	\|h\|_{v, a} := \sup_{(y,\tau)\in \mathcal D_{ R}} v^{-1}(\tau) \langle y \rangle ^{a} |h(y,\tau)|
\end{equation*}
where $a\ge 0$ is a constant. Throughout this section, we assume  $R(\tau)$, $v(\tau) \in C^{1}(\tau_0,\infty)$ with the form
\begin{equation*}
\begin{aligned}
&
v(\tau) =
a_0 \tau^{a_1}(\ln \tau)^{a_2} (\ln\ln \tau)^{a_3} \cdots, \quad
R(\tau) = b_0 \tau^{b_1}(\ln \tau)^{b_2} (\ln\ln \tau)^{b_3} \cdots,\quad
 v(\tau)>0, \quad
1\ll R(\tau) \ll \tau^{\frac 12} ,
\\
&
v'(\tau)= O( \tau^{-1} v(\tau) ),\quad R'(\tau)= O(\tau^{-1}R(\tau))
\end{aligned}
\end{equation*}
where $a_0,b_0>0$, $a_{i}, b_{i}\in \RR$, $i=1,2,\dots$. For brevity, we write $v=v(\tau),~R=R(\tau)$.

We impose a linear constraint on the initial value $\phi(y,\tau_0)$ to handle the instability caused by $Z_0$. Consider the associated Cauchy problem
\begin{equation}\label{A2}
	\begin{cases}
	\pp_{\tau}	\phi=\Delta \phi + pU^{p-1}(y)\phi + f_1(y,\tau)\phi + f_2(y,\tau) y \cdot \nabla \phi + h, & \mbox{ in } \mathcal D_{R}, \\
		\phi(y,\tau_0)=e_0 Z_0(y), & \mbox{ in } B_{R(\tau_0)},
	\end{cases}
\end{equation}
where $\tau_0$ is sufficiently large.
Formally speaking, when $R \ll \tau^{\frac{d}{2} -\epsilon } $ for some $\epsilon>0$, we can expect that $f_1\phi + f_2 y \cdot \nabla \phi$  is a small perturbation since $|f_i|\ll \tau^{-2\epsilon} R^{-2} \lesssim \tau^{-2\epsilon}  \langle y\rangle^{-2}$ in $\mathcal D_{R}$.

The construction of solution to \eqref{A2} is achieved by decomposing the equation into different spherical harmonic modes. Consider an orthonormal basis $\{\Upsilon_i\}_{i=0}^{\infty}$ made up of spherical harmonic functions in $L^2(S^{n-1})$,
 namely eigenvalues of the problem
\[
\Delta_{S^{n-1} } \Upsilon_j + \iota_j \Upsilon_j = 0
\mbox{ \ in \ }  S^{n-1}.
\]
where $0 = \iota_0 < \iota_1 = \iota_2 =\dots = \iota_n = n-1 < \iota_{n+1} \le \dots $ and $\int_{ S^{n-1}} \Upsilon_{i}(\theta) \Upsilon_{j}(\theta) d\theta =\delta_{ij}$.
More precisely, $\Upsilon_0(y)=a_0,~\Upsilon_i(y)=a_1y_i,~i=1,\cdots,n$ for two constants $a_0$, $a_1$ and the eigenvalue $\iota_{l} = l(n-2+l)$ has
multiplicity
\begin{equation*}
\begin{pmatrix}
n+l-1
\\
l
\end{pmatrix}
-
\begin{pmatrix}
	n+l-3
	\\
	l-2
\end{pmatrix}
\mbox{ \ for \ } l\ge 2 .
\end{equation*}

For $h(\cdot,\tau) \in L^2( B_{R(\tau)} )$, we decompose $h$ into the form
\[
h(y,\tau) = \sum\limits_{j=0}^{\infty} h_j(r,\tau) \Upsilon_j(y/r), \ r= |y|, \ h_j(r,\tau ) = \int_{ S^{n-1} } h(r\theta, \tau ) \Upsilon_j(\theta)d \theta .
\]
Write $h=h^0 + h^1 + h^{\perp} $ with
\[
h^0 = h_0(r, \tau) \Upsilon_0, \
h^1 = \sum\limits_{j=1}^{n} h_j(r,\tau) \Upsilon_j, \
h^{\perp} = \sum\limits_{j= n+1}^{\infty} h_j(r,\tau) \Upsilon_j.
\]

Also, we decompose $\phi=\phi^0+\phi^1+\phi^{\perp}$ in a similar form. Then looking for a solution to problem \eqref{A2} is equivalent to finding the pairs $(\phi^0,h^0),~(\phi^1,h^1),~(\phi^{\perp},h^{\perp})$ in each mode.

The key linear theory for the inner problem is stated as follows.
\begin{prop}\label{R0-linear}
	Consider
	\begin{equation*}
		\begin{cases}
			\pp_{\tau} \phi=\Delta \phi + pU^{p-1}\phi + f_1 \phi + f_2 y\cdot\nabla \phi + h(y,\tau) +
			\sum\limits_{i=1}^{n+1}
			c_i(\tau) \eta(y) Z_{i}(y)
			& \mbox{ in } \mathcal D_{R}
			\\
			\phi(y,\tau_0)=e_0 Z_0(y) & \mbox{ in } B_{R(\tau_0)}
		\end{cases}
	\end{equation*}
where $n\ge 4$,  $\|h\|_{v,2+a}<\infty$, $0< a < 2$. Suppose that $R^2\ll \tau^{d-}$, $R_0=C\tau^{\delta}\gg 1$, $\delta\ge 0$ and $R_0^{n+2} \ll \tau^{ \min{ \{1,d \} } - }$, then for $\tau_0$ sufficiently large, there exists  $(\phi,e_0,c_i)$ solving above equation, and $(\phi,e_0,c_i)=(\TT_{3i}[h],\TT_{3e}[h],c_i[h] )$ defines a linear mapping of $h$ with the estimates
	\begin{equation*}
		\langle y\rangle |\nabla \phi| +	|\phi|
		\lesssim  R_0^{n+1} v
		\langle y \rangle^{-a} \|h\|_{v,2+a} ,
	\quad
		|e_0| \lesssim
		v(\tau_0) R_0^{2-a}(\tau_0) \|h \|_{v,2+a }  ,
	\end{equation*}
	\begin{equation*}
		\begin{aligned}
			c_{i}[h](\tau)
			= \ &
			-\big(\int_{B_{2} } \eta(y) Z_{i}^2(y) dy\big)^{-1} \left(\int_{B_{2R_0} }
			h(y,\tau)  Z_{i}(y) dy
			+
			R_0^{-\epsilon_{0}}  O( v \|h_i \|_{v,2+a} )  \right)
			, \quad i=1,\dots,n,
			\\
			c_{n+1}[h](\tau)
			= \ &
			-\big(\int_{B_{2} } \eta(y) Z_{n+1}^2(y) dy\big)^{-1} \left(\int_{B_{2R_0} }
			h(y,\tau)  Z_{n+1}(y) dy
			+
			R_0^{-\epsilon_{0}}  O( v\|h_0 \|_{v,2+a} )  \right),
		\end{aligned}
	\end{equation*}
	where  $0<\epsilon_0< \frac{\min\{ a,1 \}}{2}$ is a small constant,
	\begin{equation*}
		h(y,\tau) = \sum\limits_{j=0}^{\infty} \Upsilon_j(\frac{y}{|y|}) h_j(|y|,\tau) ,\quad
		h_j(|y|,\tau) = \int_{ S^{n-1}} h(|y| \theta,\tau) \Upsilon_j(\theta) d\theta ,
	\end{equation*}
$O( v \|h_i \|_{v,2+a} )$ linearly depends on $h_{i}$ for $i=0,1,\dots,n$.
\end{prop}

The proof of Proposition \ref{R0-linear} is achieved by the following Proposition and by another gluing procedure (re-gluing).

\medskip

\begin{prop}\label{linear-theory}
	Consider
	\begin{equation*}
		\begin{cases}
			\pp_{\tau} \phi =
			\Delta \phi + pU^{p-1} \phi
			+
			f_1 \phi +  f_2 y \cdot \nabla \phi  + h(y,\tau)
			&
			\mbox{ \ in \ }
			\mathcal{D}_R
			\\
			\phi(y,\tau_0) = e_0 Z_{0}(y)
			&
			\mbox{ \ in \ }
			B_{R }
		\end{cases}
	\end{equation*}
	where
	$\| h\|_{v, 2+a} < \infty$,
	$a>0$ and $h$ satisfies the orthogonal condition
	\begin{equation*}
		\int_{ B_{R(\tau)} } h(y,\tau) Z_i(y) d y = 0 \mbox{ \ for all \ } \tau >\tau_0
		,\quad i=1,\dots,n+1 .
	\end{equation*}
	Assume $R^n \theta_{Ra}^{1}  \ll \tau^{ \min{ \{1,d \}  } }$.
	Then for $\tau_0$ sufficiently large, there exists a solution $(\phi,e_0) = (\TT_{2i}[h],\TT_{2e}[h] )$ which is a linear mapping of $h$ with the  estimates
	\begin{equation*}
		\begin{aligned}
		\langle y\rangle |\nabla \phi| +	|\phi|
			\lesssim \ &
			v \min\{ \tau^{\frac 12}, \lambda_R^{- \frac 12 } \}  \lambda_R^{-\frac 12}  \theta_{R\hat{a}_{0}}^0
			\left(  \langle y\rangle^{-n} + C_{f}	\tau^{-d}  \min\{ \tau^{\frac 12}, \lambda_R^{- \frac 12 } \}  \lambda_R^{-\frac 12}
			\ln R \langle y \rangle^{2-n}  \right)
			\|h^0 \|_{v,2+a}
			\\
			& +
			v \left( \Theta_{R \hat{a}_{0} }^0(|y| )  \langle y\rangle^{-2}  +
			C_{f}\tau^{-d}
			\min\{ \tau^{\frac 12}, \lambda_R^{- \frac 12 } \}  \lambda_R^{-\frac 12} \theta_{R\hat{a}_{0}}^0
			\ln R \right)
			\|h^0 \|_{v,2+a}
			\\
			&
			+
			v \theta_{R \hat{a}_{1} }^1  R^n  \left(	
			\langle y\rangle^{-1-n}
			+
		C_{f}	\tau^{-d} R^{n}
			\langle y \rangle^{1-n}
			\right)
			\| h^1 \|_{v,2+a}
			\\
			& +
			v
			\left(
\Theta_{R,2+a}^0(|y|) + R \langle y\rangle^{2-n}
	\right)
			\|h^\perp \|_{v,2+a} ,
		\end{aligned}
	\end{equation*}
	\begin{equation*}
		|e_0|\lesssim
		v(\tau_0) \theta_{R(\tau_0) \hat{a}_{0}}^0  \left(
		1 +
	C_{f}	\tau_0^{-d}
		\min\{ \tau_0^{\frac 12}, \lambda_{R(\tau_0)}^{- \frac 12 } \}  \lambda_{R(\tau_0)}^{-\frac 12}  \ln R(\tau_0) \right) \| h^0\|_{v,2+a}
	\end{equation*}
	where
\begin{equation}\label{ahat-def}
\hat{a}_{0} =
		\begin{cases}
		a & \mbox{ \ if \ } a\ne n-2
			\\
		(n-2)- & \mbox{ \ if \ } a = n-2
		\end{cases},
\quad
\hat{a}_{1} =
		\begin{cases}
			a
			&
			\mbox{ \ if \ }  a\ne n-1
			\\
			(n-1)-
			&
			\mbox{ \ if \ } a=n-1
		\end{cases} ,
\end{equation}
	\begin{equation}\label{theta-def}
		\theta_{Ra}^0  =
		\begin{cases}
			R^{2-a}  & \mbox{ \ if \ }  a<2
			\\
			\ln R
			& \mbox{ \ if \ } a=2
			\\
			1 & \mbox{ \ if \ } a>2
		\end{cases},
		\quad
		\theta_{Ra}^1   =
		\begin{cases}
			R^{1-a}  & \mbox{ \ if \ } a < 1
			\\
			\ln R   &\mbox{ \ if \ }  a = 1
			\\
			1 &\mbox{ \ if \ }  a>1
		\end{cases}
,
	\end{equation}
\begin{equation}\label{lambdaR}
\lambda_{R} =
	\begin{cases}
		R^{-2}
		& \mbox{ \ if \ } n=3
		\\
		(R^2 \ln R)^{-1}
		& \mbox{ \ if \ } n=4
		\\
		R^{2-n}
		& \mbox{ \ if \ } n\ge 5
	\end{cases}
	,
	\quad
		\Theta_{R a}^0(r ) =
		\begin{cases}
			R^{2-a} &\mbox{ \ if \ } a<2
			\\
			\ln R
			&\mbox{ \ if \ } a=2
			\\
			\langle r\rangle^{2-\min\{a,n-\}}
			&  \mbox{ \ if \ } a>2
		\end{cases}
		.
\end{equation}
\end{prop}

Before we prove Proposition \ref{R0-linear}, we first use Proposition \ref{linear-theory} to prove Proposition \ref{R0-linear}.
\begin{proof}[Proof of Proposition \ref{R0-linear}]
	Set $\phi(y,\tau) = \eta_{R_0}(y) \phi_{i}(y,\tau) + \phi_{o}(y,\tau) $, where $\eta_{R_0}(y) = \eta(\frac{y}{R_0})$. In order to find a solution $\phi$, it suffices to  consider the following inner--outer gluing system for $(\phi_i,\phi_o)$
	\begin{equation}
		\label{phio-eq}
		\begin{cases}
			\pp_{\tau} \phi_{o}
			=
			\Delta \phi_{o}
			+
			J[\phi_{o},\phi_{i}]
			\mbox{ \ in \ } \DD_{R},
			\\
			\phi_{o} = 0
			\mbox{ \ on \ } \pp \DD_{R},\quad
			\phi_{o} = 0
			\mbox{ \ in \ } B_{R(\tau_0)} ,
		\end{cases}
	\end{equation}
	\begin{equation}
		\label{phii-eq}
		\begin{cases}
			\pp_{\tau} \phi_{i}
			=
			\Delta \phi_{i} + pU^{p-1} \phi_i
			+ f_1 \phi_i
			+
			f_2 y\cdot \nabla  \phi_i
			+
			pU^{p-1} \phi_o
			+ h
			+
			\sum\limits_{i=1}^{n+1}
			c_i(\tau) \eta(y) Z_{i}(y)
			&
			\mbox{ \ in \ } \DD_{2R_0},
			\\
			\phi_{i} = e_0 Z_{0}(y)
			&
			\mbox{ \ in \ } B_{2R(\tau_0)} ,
		\end{cases}
	\end{equation}
	where
	\begin{equation*}
		\begin{aligned}
			J[\phi_{o},\phi_{i}]= \ &
			f_1 \phi_o + f_2 y\cdot \nabla  \phi_o
			+
			pU^{p-1} \phi_o(
			1-\eta_{R_0}
			)
			+
			A[\phi_{i}] + h(1-\eta_{R_0}) ,
			\\
			A[\phi_{i}] = \ &
			\Delta \eta_{R_0} \phi_i + 2\nabla \eta_{R_0} \cdot \nabla \phi_i + f_2 y\cdot \nabla \eta_{R_0} \phi_i -\pp_{\tau} \eta_{R_0} \phi_i .
		\end{aligned}
	\end{equation*}
Here	$c_i(\tau)$ is given by
	\begin{equation*}
		c_{i}(\tau) = c_{i}[\phi_{o}](\tau) =
		C_{i} \int_{B_{2R_0} }
		(pU^{p-1} (z) \phi_o(z,\tau)
		+ h(z,\tau) ) Z_{i}(z) dz,\quad
		C_{i} =  -(\int_{B_{2} } \eta(y) Z_{i}^2(y) dy)^{-1}
	\end{equation*}
	such that the orthogonal conditions
	\begin{equation*}
		\int_{B_{2R_0}} \left( pU^{p-1}(z) \phi_o(z,\tau)
		+ h(z,\tau)
		+
		\sum\limits_{i=1}^{n+1}
		c_i(\tau) \eta(y) Z_{i}(z) \right) Z_{j }(z) dy = 0 \mbox{ \ for \ } j=1,\dots,n+1
	\end{equation*}
	are satisfied.

	We reformulate \eqref{phio-eq} and \eqref{phii-eq} into the following form
	\begin{equation}\label{z6}
		\begin{aligned}
			\phi_{o}(y,\tau) = \ &
			\TT_{o} [
			J[\phi_{o},\phi_{i}]  ] ,
			\quad
			\phi_{i}(y,\tau) =
			\TT_{2i} \left[
			pU^{p-1}(y) \phi_o
			+ h
			+
			\sum\limits_{i=1}^{n+1}
			c_i(\tau) \eta(y) Z_{i}(y) \right] ,
			\\
			e_0 = \ &
			\TT_{2e} \left[
			pU^{p-1}(y) \phi_o
			+ h
			+
			\sum\limits_{i=1}^{n+1}
			c_i(\tau) \eta(y) Z_{i}(y) \right]
			,
		\end{aligned}
	\end{equation}
	where $\TT_{o}$ is a linear mapping given by the standard parabolic theory, and $\TT_{2i}$, $\TT_{2e}$ are given by Proposition \ref{linear-theory}.
	We will solve the system \eqref{z6} by the contraction mapping theorem.
	
	Denote the leading term of the right hand side of \eqref{phii-eq} as $H_{1} := h+ \sum\limits_{i=1}^{n+1} C_{i} \eta(y) Z_{i}(y) \int_{B_{2R_0} }
	h(z,\tau) Z_{i}(z) dz $. It is easy to check $\| H_{1} \|_{v,2+a} \lesssim \|h\|_{v,2+a}$. If $H_1$ satisfies the orthogonal condition in $\mathcal{D}_{2R_0}$, under the assumption $R_0^{n+2} \ll \tau^{ \min{ \{1,d \} - } }$, Proposition \ref{linear-theory} gives following a priori estimates
	\begin{equation*}
		\langle y \rangle
		|\nabla \TT_{2i}[H_1]|
		+	|\TT_{2i}[H_1] | \le
		D_i   w_{i}(y,\tau)
		,
		\quad
		|\TT_{2e}[H_1] | \le D_i
		v(\tau_0) R_0^{2-a} \|h \|_{v,2+a },
	\end{equation*}
	where  $D_{i}\ge 1$ is a constant and
	\begin{equation*}
		w_{i}(y,\tau) =
		v \left( \lambda_{R_0}^{-1} R_0^{2-a} \langle y \rangle^{-n}
		+
		\theta_{R_0 a}^1 R_0^n
		\langle y \rangle^{-1-n}
		+
		\langle y \rangle^{-a}
		+ R_0 \langle y \rangle^{2-n} \right)
		\| h \|_{v,2+a}
	\end{equation*}
where $\theta_{R_0a}^{1} $ is given in \eqref{theta-def}. For this reason, we will solve the inner part in the space
	\begin{equation*}
		\B_{i} =
		\left\{
		g(y,\tau) \ : \
		\langle y \rangle |\nabla_y g(y,\tau)|  + |g(y,\tau)| \le
		2D_i  w_{i}(y,\tau)
		\right\} .
	\end{equation*}
	For any $\tilde{\phi}_{i}\in \B_{i}$, we will find a solution $\phi_{o} = \phi_{o}[\tilde{\phi}_{i}]$ of \eqref{phio-eq} by the contraction mapping theorem. Let us estimate $J[0,\tilde{\phi}_{i}]$ term by term.
	For $n\ge 4$,
	\begin{equation*}
		|A[\tilde{\phi}_{i}] |
		\lesssim
		D_i v
		(R_0^{-2} + \tau^{-d} ) (R_0^{-a} \ln R_0 + R_0^{-1})
		\1_{ \{ R_0\le |y| \le 2R_0 \} }  \| h \|_{v,2+a}
		\lesssim
		D_i v
		R_{0}^{-\epsilon_0}
		\langle y \rangle^{-2-a_1}
		\| h \|_{v,2+a}
	\end{equation*}
	where constants $0<a_1 <\min\{a,1\}$ and $\epsilon_0 = \frac{\min\{a,1\} -a_1}{2}$. Also we have
	\begin{equation*}
		| h (1-\eta_{R_0})| \lesssim
		\1_{\{ y\ge R_0 \}} v\langle y\rangle^{-2-a } \| h\|_{v,2+a}
		\lesssim
		v
		R_{0}^{-\epsilon_0}
		\langle y \rangle^{-2-a_1}
		\| h \|_{v,2+a}
		.
	\end{equation*}
	Consider \eqref{phio-eq} with the right hand side $J[0,\tilde{\phi}_{i}]$. Using $C v (-\Delta)^{-1}(\langle y\rangle^{-2-a_1}) R_{0}^{-\epsilon_0}
	\| h \|_{v,2+a}$ as the barrier function with a large constant $C$ and then scaling argument, we have
	\begin{equation*}
	\langle y\rangle |\nabla \TT_o[J[0,\tilde{\phi}_{i}]](y,\tau)| +	|\TT_o[J[0,\tilde{\phi}_{i}]] (y,\tau) |
		\le
		w_{o}(y,\tau) =
		D_o	D_i v
		R_{0}^{-\epsilon_0}
		\langle y \rangle^{-a_1}
		\| h \|_{v,2+a}
	\end{equation*}
	with a large constant $D_o\ge 1$.
	This suggests us solve $\phi_o$ in the following space:
	\begin{equation*}
		\B_{o} =
		\left\{
		f(y,\tau) \ : \
		\langle y\rangle |\nabla f(y,\tau)|  + |f(y,\tau)| \le
		2w_{o}(y,\tau)
		\right\} .
	\end{equation*}
	For any $\tilde{\phi}_{o} \in \B_{o}$, due to $|y| \le 2 R(\tau)$, we have
	\begin{equation*}
		|pU^{p-1} \tilde{\phi}_o(
		1-\eta_{R_0}
		)| \lesssim
		R_{0}^{-2} D_o	D_i v
		R_{0}^{-\epsilon_0}
		\langle y \rangle^{-2-a_1}
		\| h \|_{v,2+a} ,
	\end{equation*}
	\begin{equation*}
		| f_1\phi_o + f_2 y\cdot \nabla  \phi_o| \lesssim
		\tau^{-d} R^{2}(\tau)
		D_o	D_i v
		R_{0}^{-\epsilon_0}
		\langle y \rangle^{-2-a_1}
		\| h \|_{v,2+a}.
	\end{equation*}

	Since $\tau^{-d} R^{2}$, $R_{0}^{-2}$ provide smallness, by comparison principle, we have
	\begin{equation*}
		\TT_{o}[
		J[\tilde{\phi}_{o},\tilde{\phi}_{i}] ]  \in \B_{o} .
	\end{equation*}
	The contraction mapping property can be  deduced in the same way.
	
	Now we have found a solution $\phi_{o} = \phi_{o}[\tilde{\phi}_{i}]\in \B_{o}$. It follows that
	\begin{equation*}
		\Big\|
		pU^{p-1}(y) \phi_{o}[\tilde{\phi}_{i}]
		+
		\sum\limits_{i=1}^{n+1}
		C_{i} \int_{B_{2R_0} }
		pU^{p-1}(z) \phi_{o}[\tilde{\phi}_{i}] (z,\tau) Z_{i}(z) dz \eta(y) Z_{i}(y)
		\Big\|_{v,2+a}
		\lesssim
		D_o D_i
		R_0^{-\epsilon_0} \| h\|_{v,2+a}.
	\end{equation*}
	Due to the choice of $c_i(\tau)$, $H_2 := pU^{p-1}(y) \phi_{o}[\tilde{\phi}_{i}]
	+ h
	+
	\sum\limits_{i=1}^{n+1}
	c_i[\phi_{o}[\tilde{\phi}_{i}]](\tau) \eta(y) Z_{i}(y)$ satisfies the orthogonal condition in $\DD_{2R_0}$.
	By Proposition \ref{linear-theory}, since $R_0^{-\epsilon_{0}}  $ provides smallness, we have
	\begin{equation*}
		\TT_{2i}[h_2] \in \B_{i}
	\end{equation*}

	The contraction property can be deduced in the same way. Thus we find a solution
	\begin{equation}\label{phii-upp}
\phi_{i}=\phi_{i}[h] \in \B_{i} .
\end{equation}
Finally we obtain a solution
	$(\phi_{o},\phi_{i})$ for \eqref{phio-eq} and \eqref{phii-eq}.
	
	From the construction above and the topology of $\B_{i}$,   $\phi_{i}[h]=0$ if $h=0$, which deduces that $\phi_{i}[h]$ is a linear mapping of $h$. By the similar argument, $\phi_{o}[h]$ and $c_i[h]$ are also linear mappings of $h$, and so does $\phi$.
	
	We will regard $D_o$, $D_i$  as general constants hereafter. Then by Propostition \ref{linear-theory} and \eqref{z6}, we have
	\begin{equation*}
		|e_0| \lesssim
		v(\tau_0) R_0^{2-a} \|h \|_{v,2+a }  .
	\end{equation*}
	Since $\phi_{o}[h] \in \B_{o}$, one has
	\begin{equation*}
		c_{i}[h](\tau)
		=
		C_{i} \int_{B_{2R_0} }
		h(y,\tau)  Z_{i}(y) dy
		+
		R_0^{-\epsilon_{0}}  O( v) \|h\|_{v,2+a}
		.
	\end{equation*}
	
	Since the above operation is linear about $h$, we are able to decompose $h$ into
	\begin{equation*}
		h(y,\tau) = \sum\limits_{j=0}^{\infty} \Upsilon_j(\frac{y}{|y|}) h_j(|y|,\tau) ,\quad
		h_j(|y|,\tau) = \int_{S^{n-1}} h(|y| \theta,\tau) \Upsilon_j(\theta) d\theta
	\end{equation*}
	and repeat the construction about $\Upsilon_j(\frac{y}{|y|}) h_j(|y|,\tau)$ separately. Then
	\begin{equation*}
		c_{n+1}[h](\tau)
		=
		C_{n+1} \int_{B_{2R_0} }
		h(y,\tau)  Z_{n+1}(y) dy
		+
		R_0^{-\epsilon_{0}}  O( v) \|h_0 \|_{v,2+a},
	\end{equation*}
	\begin{equation*}
		c_{i}[h](\tau)
		=
		C_{i} \int_{B_{2R_0} }
		h(y,\tau)  Z_{i}(y) dy
		+
		R_0^{-\epsilon_{0}}  O( v) \| h_i \|_{v,2+a}
		\mbox{ \ for \ } i=1,\dots,n .
	\end{equation*}
	Reviewing the re-gluing procedure, we have
	\begin{equation*}
		|J[0,\phi_{i}]|  \lesssim
		R_0
		v\langle y\rangle^{- 2-a } \| h\|_{v,2+a} .
	\end{equation*}
	Using comparison principle to \eqref{phio-eq} several times, the upper bound of $\phi_{o}$ can be improved to
	\begin{equation}\label{phio-upp}
		|\phi_{o}| \lesssim
		R_0
		v\langle y\rangle^{-a } \| h\|_{v,2+a} .
	\end{equation}
	Combining \eqref{phii-upp}, \eqref{phio-upp} and then using scaling argument, we conclude
	\begin{equation*}
	\langle y\rangle |\nabla \phi| + |\phi|
		\lesssim  R_0^{n+1} v
		\langle y \rangle^{-a} \|h\|_{v,2+a}
		.
	\end{equation*}
\end{proof}

The rest of this section is devoted to the proof of Proposition \ref{linear-theory}. We first invoke a coercive estimate for the linearized operator
\begin{lemma}\cite[Lemma 7.2]{Green16JEMS} \label{eigenvalue-problem}
	There exists a constant $c_0 >0$ such that for all sufficiently large $R$ and all radially symmetric functions $\phi \in H_0^1(B_{R} )$
	with  $\int_{B_{R}} \phi Z_0 = 0$, we have
	\[
	c_0  \lambda_{R}  \int_{B_{R}} |\phi|^2 \le Q(\phi,\phi),
	\]
	where $\lambda_R$ is given in \eqref{lambdaR} and
$ Q(\phi,\phi):= \int_{B_{R} } ( |\nabla \phi|^2 - pU^{p-1}|\phi|^2 ) $.
	
\end{lemma}
Note that in \cite[Lemma 7.2]{Green16JEMS}, there is above coercive estimate only for higher dimensions $n\geq 5$. The proof in lower dimensions $n=3,~4$ is in fact similar and by slight modifications.

\begin{lemma}\label{chiM-eq-lem}
	Consider
	\begin{equation*}
		\begin{cases}
			\pp_{\tau} \phi = \Delta \phi + pU^{p-1} (1-\chi_{M} ) \phi +
			f_1 \phi +  f_2 y \cdot \nabla \phi  +h
		 \mbox{ \ in \ } \DD_{R}
			\\
			\phi  = 0  \mbox{ \ on \ }
			\pp \DD_{R},
			\quad
			\phi(\cdot, \tau_0)= 0
			 \mbox{ \ in \ } B_{R(\tau_0)}
		\end{cases}
	\end{equation*}
	where $\chi_{M}(y) = \eta(\frac{y}{M})$, $M>0$ is a large constant, $R^2\ln R\ll \tau^{\min\{1,d\}}$, $\|h\|_{v,a} <\infty$, $a\ge 0$.
	Then when $\tau_0$ sufficiently large,
for $\Theta_{R a}^0(|y|)$ given in \eqref{lambdaR}, the unique solution $\phi_*[h]$ has the following estimate:
	\[
	|\phi_*[h]| \lesssim C(M,a,n) v
	\Theta_{R a}^0(|y|) \|h \|_{v,a} .
	\]
\end{lemma}
\begin{proof}
	Set $\bar{a}=\min\{a,n-\}$,  $r=|y|$, $L_M \phi = \Delta \phi + p U^{p-1}(y) (1 - \chi_M)\phi $.  Set a barrier function as $\bar{\phi}(r,\tau)= C v g(r,R)$, where
	\[
	L_{M} g(r,R) = -\langle r\rangle^{-\bar{a}},\quad
	g(r,R) = g_2(r) \int_r^{R} \frac{d \rho}{g_2^2(\rho) \rho^{n-1}} \int_{0}^{\rho} g_2(s)s^{n-1} \langle s \rangle^{-\bar{a}} d s
	\]
	and $g_2(r)>0$ is the positive kernel of $L_M$ and $g_2(r)\sim 1$ for $r\in (0,\infty)$. By direct calculation, one has
	\begin{equation*}
		\begin{aligned}
			\langle r\rangle^{\bar{a}} g(r,R)  \lesssim \ &
			\langle r\rangle^{\bar{a}}
			\Theta_{R \bar{a}}^0(|y|)
			\lesssim
		R^2 \ln R
			.
		\end{aligned}
	\end{equation*}
By scaling argument, one has $\langle r\rangle^{\bar{a}} |r\pp_{r}g(r,R)| \lesssim R^2 \ln R$. Then
	\[
	\begin{aligned}
	&	P(\bar{\phi}):= L_M(C v g(r,R) ) + h(y,\tau)
		+ C  v  (  f_1 g(r,R)  + f_2 r \pp_r g(r,R) )
		- \pp_{\tau} (C v g(r,R ) ) \\
		=\ & - C v\langle r\rangle^{-\bar{a}} + h(y,\tau)
		+ C  v  ( f_1 g(r,R) + f_2 r \pp_r g(r,R)  )
		- Cv' g(r,R)
		- \frac{C v g_2(r) R' }{g_2^2(R) R^{n-1}}
		\int_0^{R} g_2(s)s^{n-1}
		\langle s\rangle^{-\bar{a}}  d s
		\\
		\le \ &
		C v \langle r\rangle^{-\bar{a}} \Bigg[ -1 +   \langle r\rangle^{\bar{a}} ( f_1g(r,R) + f_2r \pp_r g(r,R) ) - v' v^{-1} g(r,R) \langle r \rangle^{\bar{a}}
		\\
		&
		-
		\frac{\langle r \rangle^{\bar{a}}  g_2(r) R' }{g_2^2(R) R^{n-1}}
		\int_0^{R} g_2(s)s^{n-1} \langle s\rangle^{-\bar{a}} d s
		\Bigg]
		+
		v \langle r\rangle^{-a}  \|h\|_{v,a}
		\le
		-\frac 34  C v \langle r\rangle^{-\bar{a}}
		+
		v \langle r\rangle^{-\bar{a}}  \|h\|_{v,a}
	\end{aligned}
	\]
where we have used
\begin{equation*}
| \langle r\rangle^{\bar{a}} ( f_1g(r,R) + f_2r \pp_r g(r,R) ) | \lesssim (|f_1|+ |f_2|) R^2\ln R  \lesssim C_f \tau^{-d} R^2\ln R  \ll 1,
\end{equation*}
	\begin{equation*}
		|v' v^{-1} g(r,R) \langle r \rangle^{\bar{a}} | \lesssim |v'| v^{-1} R^2\ln R \lesssim  \tau^{-1} R^2\ln R \ll 1,
	\end{equation*}
	\begin{equation*}
		\begin{aligned}
			&
			\frac{ \langle r \rangle^{\bar{a}}  g_2(r) | R'| }{g_2^2(R) R^{n-1}}
			\int_0^{R} g_2(s)s^{n-1} \langle s\rangle^{-\bar{a}} d s
			\sim
			\frac{ \langle r \rangle^{\bar{a}}  | R'| }{ R^{n-1}}
			\int_0^{R} s^{n-1} \langle s\rangle^{-\bar{a}} d s
			\lesssim
			R | R' | \lesssim \tau^{-1} R^2\ll 1 .
		\end{aligned}
	\end{equation*}
Set $C= 2\|h\|_{v,a} $, then $P(\bar{\phi}) \le 0 $.

\end{proof}

\subsection{Mode $0$ without orthogonality}
\begin{lemma}\label{mode0-nonorth}
Consider
\begin{equation}\label{m0-eq-nonorth}
	\begin{cases}
		\pp_{\tau} \phi^0 =  \Delta \phi^0 + p U^{p-1} \phi^0 + f_1 \phi^0 +  f_2 y \cdot \nabla \phi^0
		+ h^0
		&
		\mbox{ \ in \ } \mathcal{D}_{R},
		\\
		\phi(\cdot,\tau_0) =  e_0 Z_0(y)
		&
		\mbox{ \ in \ } B_{R(\tau_0)}
	\end{cases}
\end{equation}
where $\| h^0\|_{v,a}<\infty$, $a\ge 0$. Assume $\lambda_R \tau^{d}\gg 1$, $R^2\ln R\ll \tau^{\min\{1,d\}}$.
Then for $\tau_0$ sufficiently large, there exists a linear mapping  $(\phi^0,e_0)= (\TT_{1i}[h^0],\TT_{1e}[h^0])$ solving \eqref{m0-eq-nonorth} with the following estimates
\begin{equation*}
\langle y\rangle |\nabla \phi^0|	+|\phi^0| \lesssim
	v \left( \min\{ \tau^{\frac 12}, \lambda_R^{- \frac 12 } \}  \lambda_R^{-\frac 12}
 \theta_{Ra}^0  \langle y \rangle^{2-n}    + 	 \Theta_{R a}^0( |y| ) \right)
	\|h^0 \|_{v,a} ,
\end{equation*}
\begin{equation*}
 |e_0|\lesssim v(\tau_0) \theta_{R(\tau_0)a}^0 \| h^0\|_{v,a} .
\end{equation*}

\end{lemma}

\begin{proof}
First, we decompose $\phi^0$ into two parts
\[
\phi^0 = \phi_*[h^0] + \tilde{\phi},
\]
where $\phi_*[h^0]$ is the solution derived from Lemma \ref{chiM-eq-lem} with the following estimate
\begin{equation}\label{phi*-h0-est}
	|\phi_*[h^0]| \lesssim
	v \Theta_{R a}^0(|y|)
	\|h^0 \|_{v,a}  .
\end{equation}
Then
\[
\begin{aligned}
	\pp_{\tau} \phi_*[h^0] + \pp_{\tau} \tilde{\phi}
	= \ & \Delta \phi_*[h^0] + \Delta \tilde{\phi}
	+ p U^{p-1}(y) \chi_M \phi_*[h^0]
	+ p U^{p-1}(y) (1- \chi_M) \phi_*[h^0]
	\\
	&  +  p U^{p-1}(y) \tilde{\phi}
	+
	f_1(\phi_*[h^0] + \tilde{\phi} )
	+
f_2 y\cdot \nabla (\phi_*[h^0] + \tilde{\phi} )
	+ h^0
	\mbox{ \ in \ } \mathcal{D}_R ,
\end{aligned}
\]
which implies that
\begin{equation*}
	\pp_{\tau} \tilde{\phi}
	= \Delta \tilde{\phi}
	+ p U^{p-1}(y) \tilde{\phi}
	+  f_1 \tilde{\phi}
	+f_2 y\cdot \nabla  \tilde{\phi}
		+ p U^{p-1}(y) \chi_M \phi_*[h^0]
	\mbox{ \ in \ }
	\mathcal{D}_R .
\end{equation*}
We will construct  a linear mapping $\tilde{\phi} = \tilde{\phi}[h^0]$.
Take $\tilde{\phi} =\tilde{\phi}_1 + e(\tau) Z_0(y)$
and
consider the following equation
\begin{equation}\label{tildephi-h0}
	\begin{cases}
		\pp_{\tau} \tilde{\phi}_1  = \Delta \tilde{\phi}_1 + pU^{p-1} \tilde{\phi}_1 + f_1\tilde{\phi}_1  +  f_2y\cdot \nabla \tilde{\phi}_1
		-  \pp_{\tau} e(\tau) Z_0 + \gamma_0 e(\tau) Z_0
		\\
		\qquad \qquad
		+ pU^{p-1} \chi_M \phi_*[h^0]
		+  e(\tau) ( f_1Z_0(y) + f_2 y\cdot \nabla Z_0(y)  )
			\mbox{ \ in \ }  \DD_R,
		\\
		\tilde{\phi}_1 = 0
			\mbox{ \ on \ } \pp \DD_R
		, \quad
		\tilde{\phi}_1(\cdot,\tau_0) = 0
			\mbox{ \ in \ }  B_{R(\tau_0)}, \quad
		\int_{B_{ R(\tau)} }
		\tilde{\phi}_1(y, \tau )  Z_0(y)  d y = 0
		 \quad \forall \tau > \tau_0.
	\end{cases}
\end{equation}
Here
$e(\tau)$ will be chosen  to make $\int_{B_{ R(\tau)} }
\tilde{\phi}_1(y, \tau )  Z_0(y)  d y = 0$ for all $\tau>\tau_0$. Indeed, multiplying  \eqref{tildephi-h0} by $Z_{0}$ and integrating by parts, one has
\begin{equation*}
	\begin{aligned}
		&
	\pp_{\tau} \int_{B_{R(\tau)}} \tilde{\phi}_1 Z_0(y) dy  = 	\int_{B_{R(\tau)}}\pp_{\tau} \tilde{\phi}_1 Z_0(y) dy
		=
		\gamma_0
		\int_{B_{R(\tau)}} \tilde{\phi}_1  Z_0(y) dy +
		\int_{\pp B_{R(\tau)}} Z_{0}(y) \pp_{n} \tilde{\phi}_1 dy
\\
&
		+
	\int_{B_{R(\tau)}}  ( f_1\tilde{\phi}_1  + f_2y\cdot \nabla \tilde{\phi}_1  ) Z_0(y) dy
		- (\pp_{\tau} e(\tau) -\gamma_0 e(\tau))
		\int_{B_{R(\tau)}}   Z_0^2(y) dy
		\\
		&
		+ \int_{B_{R(\tau)}}  pU^{p-1} \chi_M \phi_*[h^0] Z_0(y) dy
		+   e(\tau) \int_{B_{R(\tau)}}  (f_1Z_0(y) + f_2y\cdot \nabla Z_0(y)  )   Z_0(y) dy .
	\end{aligned}
\end{equation*}
By $\tilde{\phi}_1(\cdot,\tau_0) = 0 $, the orthogonality $\int_{B_{ R(\tau)} }
\tilde{\phi}_1(y, \tau )  Z_0(y)  d y = 0$ holds for all $\tau>\tau_0$ if and only if
\begin{equation*}
\begin{aligned}
\pp_{\tau} e(\tau)
-
\tilde{\gamma}_0(\tau) e(\tau)
 =  \ &
 \big(\int_{B_{R(\tau)}}   Z_0^2(y) dy \big)^{-1}
 \Bigg[
\int_{\pp B_{R(\tau)}} Z_{0}(y) \pp_{n} \tilde{\phi}_1 dy
+
\int_{B_{R(\tau)}}  ( f_1\tilde{\phi}_1  + f_2y\cdot \nabla \tilde{\phi}_1  ) Z_0(y) dy
\\
&
 + \int_{B_{R(\tau)}}  pU^{p-1} \chi_M \phi_*[h^0] Z_0(y) dy
 \Bigg] ,
\end{aligned}
\end{equation*}
where $\tilde{\gamma}_0(\tau) =\gamma_0 +
(\int_{B_{R(\tau)}}   Z_0^2(y) dy )^{-1} \int_{B_{R(\tau)}}  (f_1Z_0(y) + f_2y\cdot \nabla Z_0(y)  )   Z_0(y) dy $. By \eqref{f-assump}, $\lim\limits_{\tau\rightarrow \infty}\tilde{\gamma}_0(\tau) = \gamma_0$ as $\tau\rightarrow \infty$.  We take $e(\tau)$ as
\begin{equation*}
	\begin{aligned}
	&	e(\tau)
		=  -
		e^{ \int^\tau \tilde{\gamma}_0( u ) du }
		\int_{\tau}^\infty e^{-\int^s \tilde{\gamma}_0 ( u ) du }
 \big(\int_{B_{R(s)}}   Z_0^2(y) dy \big)^{-1}
\Bigg[
\int_{\pp B_{R(s)}} Z_{0}(y) \pp_{n} \tilde{\phi}_1(y,s) dy
\\
&
+
\int_{B_{R(s)}}  ( f_1(y,s)\tilde{\phi}_1(y,s)  + f_2(y,s) y\cdot \nabla \tilde{\phi}_1(y,s)  ) Z_0(y) dy
+ \int_{B_{R(s)}}  pU^{p-1}(y) \chi_M(y) \phi_*[h^0](y,s) Z_0(y) dy
\Bigg]  ds.
	\end{aligned}
\end{equation*}
Set
\begin{equation*}
	\|\tilde{\phi}_{1} \|_{w} = \sup\limits_{\tau>\tau_0}
	\left( \min\{ \tau^{\frac 12}, \lambda_R^{-\frac 12} \}
	\lambda_R^{-\frac 12}  v \theta_{Ra}^0  \right)^{-1}
	\left( \| \tilde{\phi}_{1}(\cdot,\tau) \|_{L^{\infty} (B_{R(\tau)})} + \| \langle \cdot \rangle \nabla \tilde{\phi}_{1}(\cdot,\tau) \|_{L^{\infty} (B_{R(\tau)})} \right) .
\end{equation*}
By \eqref{phi*-h0-est}, it is straightforward to get
\begin{equation}\label{e-t-est}
	\begin{aligned}
	&	|e(\tau)|
		\lesssim
		e^{ \int^\tau \tilde{\gamma}_0( u ) du }
		\int_{\tau}^\infty e^{-\int^s \tilde{\gamma}_0 ( u ) du }
		\Big(
		e^{-c R(s) }
		\| \nabla \tilde{\phi}_1\|_{L^{\infty} (B_{R(s)}  )}
		+
		\| |f_1 \tilde{\phi}_1| + |f_2 \nabla \tilde{\phi}_1| \|_{L^{\infty} (B_{R(s)}  )}
		+ v(s) \theta_{Ra}^0(s) \|h^0\|_{v,a}
		\Big) ds
		\\
		\lesssim \ &
e^{ \int^\tau \tilde{\gamma}_0( u ) du }
\int_{\tau}^\infty e^{-\int^s \tilde{\gamma}_0 ( u ) du }
\Big[
(s^{-d} + e^{-c R(s) } )
 \min\{ s^{\frac 12}, \lambda_R^{-\frac 12}(s) \}
\lambda_R^{-\frac 12}(s)  v(s) \theta_{Ra}^0(s)
\| \tilde{\phi}_1\|_{w}
+ v(s) \theta_{Ra}^0(s) \|h^0\|_{v,a}
\Big] ds	
\\
\lesssim \ &
(\tau_0^{-d} + e^{-cR(\tau_0)} )\min\{ \tau^{\frac 12}, \lambda_R^{-\frac 12} \}
\lambda_R^{-\frac 12}  v \theta_{Ra}^0
\| \tilde{\phi}_1\|_{w}
+ v \theta_{Ra}^0 \|h^0\|_{v,a}
	\end{aligned}
\end{equation}
for some constant $c>0$, and $\theta_{Ra}^0$ is given in \eqref{theta-def}. It follows that
\begin{equation*}
	|\pp_{\tau} e(\tau)|
	\lesssim
(\tau_0^{-d} + e^{-cR(\tau_0)} )\min\{ \tau^{\frac 12}, \lambda_R^{-\frac 12} \}
\lambda_R^{-\frac 12}  v \theta_{Ra}^0
\| \tilde{\phi}_1\|_{w}
+ v \theta_{Ra}^0 \|h^0\|_{v,a}	.
\end{equation*}

With the above choice of $e(\tau)$, the global existence of \eqref{tildephi-h0} can be deduced  by the local existence.

Multiplying equation \eqref{tildephi-h0} by $\tilde{\phi}_1$ and integrating by parts, one has
\begin{equation*}
	\begin{aligned}
		&
		\frac 12
		\pp_{\tau} \int_{B_{R(\tau)}} (\tilde{\phi}_1 )^2 dy +
		\int_{B_{R(\tau)}} (|\nabla \tilde{\phi}_1|^2 - pU^{p-1} (\tilde{\phi}_1)^2 ) dy
		\\
		= \ &  \int_{B_{R(\tau)}} ( f_1\tilde{\phi}_1 + f_2 y\cdot \nabla \tilde{\phi}_1) \tilde{\phi}_1 dy
		+ \int_{B_{R(\tau)}} pU^{p-1} \chi_M \phi_*[h^0] \tilde{\phi}_1 dy
		+ 	e(\tau) \int_{B_{R(\tau)}}
	 ( f_1Z_0(y) + f_2 y\cdot \nabla Z_0(y)  )  \tilde{\phi}_1 dy.
	\end{aligned}
\end{equation*}
Then by Lemma \ref{eigenvalue-problem} and \eqref{f-assump}, we get
\begin{equation*}
	\begin{aligned}
		& \frac 12
		\pp_{\tau} \int_{B_{R(\tau)}} (\tilde{\phi}_1 )^2 dy +
		c \lambda_R
		\int_{B_{R(\tau)}}  (\tilde{\phi}_1)^2  dy
		\le
		C \tau^{-d}\int_{B_{R(\tau)}}  (\tilde{\phi}_1)^2 dy
		+ \int_{B_{R(\tau)}} \frac{4}{c\lambda_R}(pU^{p-1} \chi_M \phi_*[h^0] )^2  dy
		\\
		&
		+ \int_{B_{R(\tau)}} \frac{c\lambda_R}{4} (\tilde{\phi}_1 )^2 dy
		+  \int_{B_{R(\tau)}}  \frac{4}{c \lambda_R}
		e^2(\tau)( f_1Z_0(y) + f_2 y\cdot \nabla Z_0(y)  ) ^2  dy
		+
		\int_{B_{R(\tau)}} \frac{c \lambda_R}{4} (\tilde{\phi}_1 )^2 dy
	\end{aligned}
\end{equation*}
for some constant $c>0$. By \eqref{f-assump}, \eqref{phi*-h0-est}, \eqref{e-t-est}  and the assumption $\lambda_R  \tau^{d} \gg 1$, we get
\begin{equation*}
\begin{aligned}
\frac 12
		\pp_{\tau} \int_{B_{R(\tau)}} (\tilde{\phi}_1 )^2 dy +
		\frac {c \lambda_R}{4}
		\int_{B_{R(\tau)}}  (\tilde{\phi}_1)^2  dy
		\lesssim \ &
		 \lambda_R^{-1} [ (v \theta_{Ra}^0 \| h^0\|_{v,a} )^2
		+   \tau^{-2d}  e^2(\tau) ]
		\\
\lesssim \ &
\lambda_R^{-1}  (v \theta_{Ra}^0  )^2
 \left[\| h^0\|_{v,a} + (\tau_0^{-d} + e^{-cR(\tau_0)} )\| \tilde{\phi}_1\|_{w} \right]^2 .
 \end{aligned}
\end{equation*}
 Since $\tilde{\phi}_{1}(\cdot,\tau_0)=0$, one has
 \begin{equation*}
 \begin{aligned}
 \int_{B_{R(\tau)}}  (\tilde{\phi}_1)^2  dy
 \lesssim \ &
 e^{-\int^{\tau} \frac{c\lambda_R(u)}{2} du } \int_{\tau_0}^{\tau}
 e^{ \int^{s} \frac{c \lambda_R(u)}{2}  du } \lambda_R^{-1}(s)  (v(s) \theta_{Ra}^0(s)  )^2
 [\| h^0\|_{v,a} + (\tau_0^{-d} + e^{-cR(\tau_0)} )\| \tilde{\phi}_1\|_{w} ]^2
  ds
 \\
 \lesssim \ &
 \min\{ \tau, \lambda_R^{-1} \}
 \lambda_R^{-1}  (v \theta_{Ra}^0  )^2
\left[\| h^0\|_{v,a} + (\tau_0^{-d} + e^{-cR(\tau_0)} )\| \tilde{\phi}_1\|_{w} \right]^2 .
\end{aligned}
 \end{equation*}

Applying parabolic estimate to \eqref{tildephi-h0}, one has
\begin{equation*}
\begin{aligned}
	&
\| \tilde{\phi}_{1}(\cdot,\tau) \|_{L^{\infty} (B_{R(\tau)})} \lesssim
 \min\{ \tau^{\frac 12}, \lambda_R^{-\frac 12} \}
\lambda_R^{-\frac 12}  v \theta_{Ra}^0
[\| h^0\|_{v,a} + (\tau_0^{-d} + e^{-cR(\tau_0)} )\| \tilde{\phi}_1\|_{w} ]
\\
&
+
\tau^{-d} ( \| \tilde{\phi}_{1}(\cdot,\tau) \|_{L^{\infty} (B_{R(\tau)})}  +  \| y\cdot \nabla \tilde{\phi}_1  \|_{L^{\infty} (B_{R(\tau)})}
)
+  |\pp_{\tau} e(\tau) | +
| e(\tau) |
+ v \theta_{R a}^0
\|h^0 \|_{v,a}
\\
\lesssim \ &
\min\{ \tau^{\frac 12}, \lambda_R^{-\frac 12} \}
\lambda_R^{-\frac 12}  v \theta_{Ra}^0
[\| h^0\|_{v,a} + (\tau_0^{-d} + e^{-cR(\tau_0)} )\| \tilde{\phi}_1\|_{w} ] .
\end{aligned}
\end{equation*}

By comparison principle, the spatial decay of $\tilde{\phi}_1$ can be improved and scaling argument will give the spatial decay about $\nabla \tilde{\phi}_1$. Then one has
\begin{equation*}
\langle y \rangle | \nabla \tilde{\phi}_1  |  +	| \tilde{\phi}_1 | \lesssim
\min\{ \tau^{\frac 12}, \lambda_R^{-\frac 12} \}
\lambda_R^{-\frac 12}  v \theta_{Ra}^0
[\| h^0\|_{v,a} + (\tau_0^{-d} + e^{-cR(\tau_0)} )\| \tilde{\phi}_1\|_{w} ]  \langle y \rangle^{2-n} ,
\end{equation*}
which implies
\begin{equation*}
\langle y \rangle 	|\nabla \tilde{\phi}_1   |  +	| \tilde{\phi}_1 |
	\lesssim
	\min\{ \tau^{\frac 12}, \lambda_R^{- \frac 12 } \}  \lambda_R^{-\frac 12} v \theta_{Ra}^0
	  \langle y \rangle^{2-n}  \| h^0\|_{v,a}  .
\end{equation*}
Reviewing the computations in \eqref{e-t-est} and using $\lambda_R  \tau^{d} \gg 1$, one has $|e(\tau)|
		\lesssim  v \theta_{Ra}^0 \|h^0\|_{v,a}$ and
		 then
\begin{equation}\label{tildephi-est}
| \tilde{\phi} | = |\tilde{\phi}_{1} + e(\tau) Z_0(y)| \lesssim
\min\{ \tau^{\frac 12}, \lambda_R^{- \frac 12 } \}  \lambda_R^{-\frac 12}
v \theta_{Ra}^0  \langle y \rangle^{2-n}  \| h^0\|_{v,a}
\end{equation}
 Finally, we take $e_0 = e(\tau_0)$.
Combining \eqref{phi*-h0-est} and \eqref{tildephi-est}, we complete the proof of this Lemma.
\end{proof}

\subsection{Modes $1$ to $n$ without orthogonality}
\begin{lemma}\label{M1-nonortho-lem}
	Consider
	\begin{equation}\label{mode-1-eq-nonorthogonal}
		\begin{cases}
			\pp_\tau \phi^{1} = \Delta \phi^{1} +
			pU^{p-1} \phi^{1}  + f_1 \phi^1 + f_2 y\cdot \nabla \phi^1 + h^1(y,\tau)
			\mbox{ \ in \ }
			\mathcal{D}_{R}
			\\
			\phi^{1} = 0
			\mbox{ \ on \ }
			\pp\mathcal{D}_{R}
	\quad
			\phi^{1}(\cdot,\tau_0) = 0
			\mbox{ \ in \ } B_{R(\tau_0) }
		\end{cases}
	\end{equation}
where
$
	h^1(y,\tau) = \sum\limits_{j=1}^n h_j(|y|,\tau) \Upsilon_j
$.
Assume $
 R^n \theta_{Ra}^1  \ll \tau^{ \min{ \{1,d \} } }$,
where $\theta_{Ra}^1$ is given in \eqref{theta-def}. Then for $\tau_0$ sufficiently large,   there exists a unique linear mapping $\phi^1 = \phi^1[h^1]$ solving \eqref{mode-1-eq-nonorthogonal} of the form
$
		\phi^{1} = \sum\limits_{j=1}^n \phi_j(|y|,\tau) \Upsilon_j
$
with the following estimate
	\begin{equation*}
	\langle y\rangle |\nabla \phi^1| +	|\phi^{1}|
		\lesssim
		v
		\theta_{Ra}^1  R^n
		\langle y \rangle^{1-n}
		\| h^1\|_{v,a}.
	\end{equation*}

\end{lemma}

\begin{proof}
Set $r=|y|$. Notice $y\cdot \nabla (\phi_j(r,\tau) \Upsilon_j) = r \pp_{r}\phi_j(r,\tau) \Upsilon_j $. It is equivalent to considering
	\begin{equation}\label{m1-eq}
		\begin{cases}
			\pp_{\tau} \phi_j  = \LLL_{1}[\phi_j] + f_1 \phi_j(r,\tau)   + f_2 r \pp_{r}\phi_j(r,\tau)   +  h_j(r,\tau)
			\mbox{ \ for \ } r\in (0,R(\tau)), \ \tau \in  (\tau_0,\infty)
			\\
			\pp_{r}\phi_{j}(0,\tau) = 0  = \phi_j(R(\tau),\tau)
			\mbox{ \ for \ }
			\tau\in (\tau_0,\infty),\quad
			\phi_j(r,\tau_0) = 0
			\mbox{ \ for \ }
			r\in (0,R(\tau_0))
		\end{cases}
	\end{equation}
	where
$\LLL_{1}[\phi_j]:=
		\pp_{rr} \phi_j + \frac{n-1}{r} \pp_{r}\phi_j
		- \frac{n-1}{r^2} \phi_j + pU (r)^{p-1} \phi_j $, $|h_j|\le  v \langle y\rangle ^{-a} \|h_j\|_{v,a}$, $\|h_j\|_{v,a} \lesssim \|h^{1}\|_{v,a}$.
		
One positive kernel of $\LLL_1$ is given by
	$Z(r) := -U_r = (n(n-2))^{\frac{n-2}{4} }  (n-2) r (1+r^2)^{-\frac n2} $.  Set a barrier function of \eqref{m1-eq} as $\phi_{s}= Cv \bar{\phi}(r,R)$, where
	\begin{equation*}
		\LLL_{1} [\bar{\phi}] = - \langle r\rangle ^{-\bar{a} }, \quad \bar{a} = \min\{a,n-1 \}
	\end{equation*}
	with
	$\bar{\phi}$ given by the variation of parameter formula
	\[
	\bar{\phi}(r,R) = Z(r) \int_r^{R} \frac {1}{\rho^{n-1} Z^2(\rho)}
	\int_0^{\rho} \langle s\rangle ^{-\bar{a}} Z(s) s^{n-1} d s d \rho.
	\]
Then
	\[
		\bar{\phi}
		\lesssim  R^n \theta_{R\bar{a}}^1 r \langle r\rangle^{-n} ,\quad
		|\pp_{R}\bar{\phi} |= \left| \frac { Z(r)  }{R^{n-1} Z^2(R)}
		\int_0^{R} \langle s\rangle ^{-\bar{a}} Z(s) s^{n-1} d s  \right|
		\lesssim
		R^{n-1} \theta_{R\bar{a}}^1 r \langle r\rangle^{-n}
	\]
for all $r>0$.	This estimate holds for all $n>2$, and $a\le 0$ is also allowed here.
	Next, we compute
	\begin{equation*}
		\begin{aligned}
		&	P(\phi_s):=
			\LLL_1 \phi_{s}  +f_1 \phi_s + f_2 r\pp_{r} \phi_s - \pp_{\tau } \phi_s + h_j
			=
			-C v \langle r \rangle^{-\bar{a}} + C v ( f_1\bar{\phi} + f_2 r\pp_{r} \bar{\phi}) - Cv' \bar{\phi}
			-Cv \pp_{R} \bar{\phi}R'
			+ h_j
			\\
			\le \ &
C v \langle r \rangle^{-\bar{a}}\left[
- 1 +  \langle r \rangle^{\bar{a}} ( f_1\bar{\phi} + f_2r\pp_{r} \bar{\phi} )  - v^{-1} v' \langle r \rangle^{\bar{a}} \bar{\phi}
-
\langle r \rangle^{\bar{a}}
\pp_{R} \bar{\phi}R'
+  C^{-1} \langle r \rangle^{\bar{a} - a }  \|h_j\|_{v,a} \right]
\\
\le \ &
C v \langle r \rangle^{-\bar a}(
- \frac 34
+  C^{-1} \|h_j\|_{v,a} )
		\end{aligned}
	\end{equation*}
where we have used
\begin{equation*}
| \langle r \rangle^{\bar{a}} ( f_1\bar{\phi} + f_2r\pp_{r} \bar{\phi} ) |
\lesssim \tau^{-d}
R^n \theta_{R\bar{a}}^1
\langle r\rangle^{\bar{a}+1-n} \lesssim
\tau^{-d}
    R^n \theta_{Ra}^1
\ll 1,
\end{equation*}
\begin{equation*}
|v^{-1} v' \langle r \rangle^{ \bar{a}} \bar{\phi}  |
\lesssim
\tau^{-1}  R^n \theta_{R\bar{a}}^1  \langle r \rangle^{\bar{a}+1-n}
\lesssim
\tau^{-1}  R^n \theta_{Ra}^1 \ll 1,
\end{equation*}
\begin{equation*}
|\langle r \rangle^{\bar{a}}
\pp_{R} \bar{\phi}R' | \lesssim
\langle r \rangle^{\bar{a}}
R^{n-1} \theta_{R\bar{a}}^1 r \langle r\rangle^{-n} |R'|
\lesssim
\tau^{-1}  R^n \theta_{Ra}^1 \ll 1
\end{equation*}
by \eqref{f-assump}, $\theta_{R\bar{a}}^1=\theta_{Ra}^1$ and the assumption $
 R^n \theta_{Ra}^1  \ll \tau^{ \min{ \{1,d \} } }$. Taking $C= 4 \|h_j\|_{v,a} $, one has $P(\phi_s) <0$.

\end{proof}

\subsection{Higher modes}

\begin{lemma}\label{phi-perp-lemma}
	Consider
	\begin{equation*}
		\begin{cases}
			\pp_\tau \phi^{ \perp }
			=
			\Delta \phi^{ \perp} + p U^{p-1} \phi^{ \perp }
			+
f_1\phi^{\perp} +	f_2y\cdot \nabla \phi^{\perp}
			+ h^{\perp}
			\mbox{ \ in \ }
			\mathcal{D}_{R}
			\\
			\phi^{ \perp } =  0
			\mbox{ \ on \ }
			\pp \mathcal{D}_{R},
	\quad
			\phi^{ \perp }(\cdot,\tau_0) = 0
			\mbox{ \ in \ } B_{R(\tau_0) }
		\end{cases}
	\end{equation*}
	where $\|h^\perp \|_{v,a} <\infty$, $a\ge 0$. Assume $R^2\ln R \ll \tau^{\min\{1,d\}}$.
 Then there exists a unique linear mapping $\phi^{ \perp } = \phi^{\perp }[h^{\perp}]$ of the form
	\begin{equation}\label{phi-s-perp-form}
		\phi^{\perp} = \sum\limits_{j=n+1}^{\infty}  \phi_j^{ \perp}(|y|,\tau) \Upsilon_j
	\end{equation}
with the following estimate
	\begin{equation*}
\langle y\rangle |\nabla \phi^{\perp}| 	+	|\phi^{\perp}|
		\lesssim
		v
		\left(
		\Theta_{R a}^0(|y|)
		+
		\theta_{R a}^0  R 	\langle  y \rangle^{2-n}
		\right)
		\|h^\perp \|_{v,a}
		.
	\end{equation*}

\end{lemma}
 First we give the following technical lemma.
\begin{lemma}
	For $f \in C^2(B_R)\cap C_0(B_R)$, by the expansion of spherical harmonic functions,
$
		f = \sum\limits_{j=0 }^{\infty}  f_j(r) \Upsilon_j  ,
$
	where $r=|y|$, $f_j(r) = \int_{ S^{n-1} } f(r\theta ) \Upsilon_j(\theta)d \theta \in C^2[0,R]$. Then
	\begin{equation*}
		Q(f, f) 	=
		\int_{B_{R}}(|\nabla f |^2
		- pU^{p-1} f ^2) d y  = |S^{n-1} | \sum\limits_{j=0 }^{\infty} Q_j(f_j,f_j),
	\end{equation*}
	where $|S^{n-1}|$ is the volume of the unit $(n-1)$-sphere and
	\[
	Q_j(f_j, f_j) = \int_0^{R } \left(  f_j'^2 + \frac{\iota_j}{r^2} f_j^2
	- pU^{p-1} f_j^2  \right) r^{n-1} d r .
	\]
	Specially, if $f_j = 0$ for $j=0,1,\dots ,n$, it holds that
	\begin{equation}\label{perp-coercive}
		\begin{aligned}
			Q(f, f )
			\ge    (n+1) \int_{B_{R  } } \frac{|f|^2}{|y|^2} d y .
		\end{aligned}
	\end{equation}
\end{lemma}

\begin{proof}
Since $\Delta_{S^{n-1} } \Upsilon_i = -\iota_i \Upsilon_i$, $\iota_i = i(n-2 + i)$ for a nonnegative integer $i$, we have
	\[
	\Delta(f_i \Upsilon_i)
	= ( f_i'' + \frac{n-1}{r} f_i' - \frac{\iota_i}{r^2} f_i ) \Upsilon_i .
	\]
 $f |_{\pp B_R} = 0$ implies
$ f_j(R)  = 0$, $ j = 0, 1,  \dots $. Then
	\[
	\begin{aligned}
		& Q(f, f)
		=
		\int_{B_{R}}|\nabla f |^2
		- pU^{p-1} f ^2 d y
		=
		- \int_{B_{R }} (  f \Delta f
		+ pU^{p-1}  f^2  ) d y
		\\
= \ &
- |S^{n-1}| \int_0^{R }
\left[  \sum\limits_{i=0}^\infty f_i
( f_i'' + \frac{n-1}{r} f_i' - \frac{\iota_i}{r^2} f_i  )
+  pU^{p-1} \sum\limits_{i=0}^\infty f_i^2\right] r^{n-1} d r
=
|S^{n-1}| \sum\limits_{i=0}^{\infty}  Q_i (f_i,f_i) .
	\end{aligned}
	\]
For $i \ge n+1$,  $\iota_i \ge 2n $, we have
	\[
	Q_i (f_i,f_i) \ge Q_1(f_i, f_i)
	+  (n+1)\int_0^{R  } \frac{f_i^2}{r^2} r^{n-1} d r \ge
	(n+1)\int_0^{R  } \frac{f_i^2}{r^2} r^{n-1} d r
	\]
since
	$\Delta u - \frac{n-1}{r^2} u + pU^{p-1}(y) u=0$ has a positive kernel $- U_r$, and by \cite[Lemma 4.2]{sun2021bubble}, one has $Q_1 (f_i, f_i) \ge 0$.
	Specially, if $f_j = 0$ for $j=0,1,\dots ,n$, we have \eqref{perp-coercive}.
\end{proof}

\begin{proof}[Proof of Lemma \ref{phi-perp-lemma}]
The existence and uniqueness of the linear mapping $ \phi^{\perp } = \phi^{ \perp } [h^{\perp}]$ are guaranteed by the classical parabolic theory. The form \eqref{phi-s-perp-form} is derived from the existence of every component $\phi_j$ with
	\begin{equation*}
		\begin{cases}
			\pp_{\tau} \phi_{j } =  \pp_{rr} \phi_j + \frac{n-1}{r}\pp_r \phi_j
			- \frac{\iota_j}{r^2} \phi_j + pU^{p-1} \phi_j
			 + f_1 \phi_j +
			 f_2 r\pp_{r} \phi_j
			+ h_j
			 \mbox{ \ for \ }
		r \in (0,R(\tau)) ,
		\tau \in (\tau_0,\infty )
		\\
			\pp_r \phi_j(0,\tau) =  \phi_j(R(\tau),\tau) = 0 \mbox{ \ for \ } \tau \in (\tau_0,\infty ),
	 \quad	\phi_j(r,\tau_0) =  0
			\mbox{ \ for \ } r\in (0, R(\tau_0) ) .
		\end{cases}
	\end{equation*}
	
By similar operation in mode $0$, we set $\phi^{\perp} = \phi_*[h^{\perp}]+ \tilde{\phi}[h^\perp]$, where $\phi_*[h^{\perp}]$ satisfies
	\begin{equation*}
		\begin{cases}
			\pp_\tau \phi_{*}
			=
			\Delta \phi_* + p U^{p-1} (1-\chi_M)\phi_*  + f_1 \phi_* +
			f_2 y\cdot \nabla \phi_*
			 + h^{\perp} \mbox{ \ in \ }
			\mathcal{D}_{R},
			\\
			\phi_* =
			0 \mbox{ \ on \ } \pp \mathcal{D}_{R},
			\quad
			\phi_{*  }(\cdot,\tau_0) = 0
			\mbox{ \ in \ } B_{R(\tau_0) },
		\end{cases}
	\end{equation*}
	and $\tilde{\phi}[h^\perp]$ satisfies
	\begin{equation}\label{phi-tilde-perpeq}
		\begin{cases}
			\pp_{\tau}
			\tilde{\phi} =
			\Delta \tilde{\phi} + p U^{p-1} \tilde{\phi}
		 + f_1 \tilde{\phi} 	+
			f_2 y\cdot \nabla \tilde{\phi}
			 + p U^{p-1} \chi_M \phi_*[h^{\perp}] \mbox{ \ in \ }
			\mathcal{D}_{R}
			, \\
			\tilde{\phi} =
			0 \mbox{ \ on \ } \pp \mathcal{D}_{R},
			\quad	\tilde{\phi}(\cdot,\tau_0) = 0
			\mbox{ \ in \ } B_{R(\tau_0) }.
		\end{cases}
	\end{equation}
	
Under the assumption $R^2\ln R \ll \tau^{\min\{1,d\}}$, by Lemma \ref{chiM-eq-lem}, we have
	\begin{equation}\label{phi*-hperp-estimate}
		|\phi_*[h^\perp]| \lesssim
		v \Theta_{R a}^0(|y| )
		\|h^\perp \|_{v,a}
		.
	\end{equation}
	$\phi_*[h^{\perp}]$ has the form
$
	\phi_*[h^{\perp}] =  \sum\limits_{j=n+1}^{\infty} \phi_{*j}(r,\tau)\Upsilon_j $
	by the same reason as \eqref{phi-s-perp-form}. By the same argument, there exists a linear mapping $\tilde{\phi} = \tilde{\phi}[h^\perp]$ and $\tilde{\phi} $ has the same form as \eqref{phi-s-perp-form}. Thus we are able to apply \eqref{perp-coercive} to $\tilde{\phi}$.
	
	Multiplying \eqref{phi-tilde-perpeq} by $\tilde{\phi}$ and integrating both sides, we have
	\[
	\frac 12 \pp_{\tau} \int_{B_{R  }} \tilde{\phi}^2 d y
	+ Q( \tilde{\phi},  \tilde{\phi})
	=
	\int_{B_{R  }}
 ( f_1 \tilde{\phi}  + f_2 y\cdot \nabla \tilde{\phi} ) \tilde{\phi} dy
	+
	\int_{B_{R  }} p U^{p-1}(y) \chi_M \phi_*[h^{\perp}] \tilde{\phi} d y .
	\]
	By \eqref{perp-coercive}, \eqref{f-assump} and H\"{o}lder inequality, one has
	\[
	\frac 12 \pp_{\tau} \int_{B_{R  }} \tilde{\phi}^2 d y
	+ (n+ \frac 12) \int_{B_{R  } } \frac{ \tilde{\phi}^2}{|y|^2} d y
	\le
	C \tau^{-d} \int_{B_R }  \tilde{\phi}^2 dy +
	\frac 12 \int_{B_{R  }} \left( p U^{p-1}(y) \chi_M \phi_*[h^{\perp}] |y|\right)^2 d y  .
	\]
	Then, by \eqref{phi*-hperp-estimate} and the assumption  $R^2\ln R \ll \tau^{\min\{1,d\}}$, we have
	\begin{equation*}
		\pp_{\tau} \int_{B_{R }} \tilde{\phi}^2 d y
		+  R^{-2}
		\int_{B_{R  } } |\tilde{\phi}|^2 d y
		\lesssim
		( v  \theta_{R a}^0  )^2
		\|h^\perp \|_{v,a}^2
		.
	\end{equation*}
	Since $\tilde{\phi}(\cdot,\tau_0) = 0 $ and the assumption $R^2\ln R \ll \tau^{\min\{1,d\}}$, we have
	\begin{equation*}
	\int_{B_{R }} \tilde{\phi}^2 d y
	\lesssim
	e^{-\int^{\tau} R^{-2}(u) du}
	\int_{\tau_0}^{\tau}
	e^{ \int^{s} R^{-2}(u) du}
	(v(s) \theta_{Ra}^{0}(s))^2 ds \|h^\perp \|_{v,a}^2
	\lesssim
	(	v\theta_{R a}^0   R
	\|h^\perp \|_{v,a})^2 .
	\end{equation*}
Using the same argument in Lemma \ref{mode0-nonorth}, one has
	\begin{equation}\label{tildephi-high}
		|\tilde{\phi}(y,\tau)|
		\lesssim
		v \theta_{R a}^0   R
		\langle y \rangle^{2-n}
		\|h^\perp \|_{v,a}
		.
	\end{equation}
Combining \eqref{phi*-hperp-estimate}, \eqref{tildephi-high} and scaling argument, we get
	\begin{equation*}
\langle y\rangle |\nabla \phi^{\perp}| 	+	|\phi^{\perp}|
		\lesssim
		v
		\left(
		\Theta_{R a}^0(|y| )
		+
		\theta_{R a}^0   R  \langle y \rangle^{2-n} \right)
		\|h^\perp \|_{v,a}
		.
	\end{equation*}
\end{proof}

\begin{proof}[Proof of Proposition \ref
{linear-theory}]
The case for higher modes  has been given in Lemma \ref{phi-perp-lemma}. Since the fast spatial decay of the right hand side $h$ cannot be recovered in non-orthogonal case in lower modes $i$, $0\le i \le n$, we transform the fast decay right hand side into slower decay function by solving the corresponding elliptic equation.

	\subsection{Mode 0 with orthogonality}
	Consider
	\[
	\Delta H^0 + p U^{p-1} H^0 = \tilde{h}^0(r,\tau) \mbox{\ \ in \ } \RR^n
	\]
	where $\tilde{h}^0$ is the extension of $h^0$ as zero outside $\DD_{R}$.
	The orthogonal condition is reformulated as
	\begin{equation*}
		\int_0^{R} h^0(r,\tau) Z_{n+1}(r) r^{n-1} d r = 0 \mbox{ \ for all \ } \tau>\tau_0.
	\end{equation*}
	Take $H^0(r,\tau)$ as in the following form
	\[
	H^0(r,\tau) =  \tilde{Z}_{n+1}(r) \int_0^{r} \tilde{h}^0(s,\tau)Z_{n+1}(s)s^{n-1} d s
	- Z_{n+1}(r) \int_0^{r} \tilde{h}^0(s,\tau) \tilde{Z}_{n+1}(s)s^{n-1} d s,
	\mbox{ \ if \ } a\le n-2,
	\]
\begin{equation*}
H^0(r,\tau) =  \tilde{Z}_{n+1}(r) \int_0^{r} \tilde{h}^0(s,\tau)Z_{n+1}(s)s^{n-1} d s
	+ Z_{n+1}(r) \int_{r}^{\infty} \tilde{h}^0(s,\tau) \tilde{Z}_{n+1}(s)s^{n-1} d s,
	\mbox{ \ if \ } a> n-2,
\end{equation*}
where  $\tilde{Z}_{n+1}(r)$ is the other linearly independent kernel of the homogeneous equation, which satisfies that the Wronskian $W[Z_{n+1},\tilde{Z}_{n+1}]=r^{1-n}$, $\tilde{Z}_{n+1}(r) \sim r^{2-n}$ if $r\rightarrow 0$ and $\tilde{Z}_{n+1}(r) \sim  1$ if $r\rightarrow \infty$.
	It is straightforward to check
	\begin{equation*}
		\|H^0 \|_{v,\hat{a}_{0}} \lesssim
		\|h^0\|_{v,2+a},
	\end{equation*}
where $\hat{a}_{0}$ is given in \eqref{ahat-def} and $a>0$ is used to ensure that the spatial decay of $ \tilde{h}^0(s,\tau) Z_{n+1}(s)s^{n-1}$ is faster than $s^{-1}$ for $s\ge 1$. Next,
consider
	\begin{equation}\label{Phi0-H-eq}
		\begin{cases}
			\pp_{ \tau} \Phi^0 = \Delta \Phi^0 + p U^{p-1} \Phi^0 +H^0
			&
			\mbox{ \ in \ }
			\mathcal{D}_{2R},
			\\
			\Phi^0(\cdot,\tau_0) = \bar{e}_0 Z_0
			&
			\mbox{ \ in \ } B_{2 R(\tau_0)},
		\end{cases}
	\end{equation}
where $(\Phi^0, \bar{e}_0)$	is given by Lemma \ref{mode0-nonorth} under the condition $f_1=f_2=0$. By scaling argument, one has
	\begin{equation*}
		\langle y \rangle^2 | \nabla^2 \Phi^0 |
		+
		\langle y \rangle |\nabla \Phi^0| + |\Phi^0|
\lesssim
 v \left( \min\{ \tau^{\frac 12}, \lambda_R^{- \frac 12 } \}  \lambda_R^{-\frac 12}
  \theta_{R \hat{a}_{0}}^0  \langle y \rangle^{2-n}    + 	\Theta_{R \hat{a}_{0}}^0( |y| ) \right)
	\|h^0\|_{v,2+a} ,
	\end{equation*}
\begin{equation*}
	|\bar{e}_0|
	\lesssim v(\tau_0) \theta_{R(\tau_0) \hat{a}_{0}}^0 \| h^0\|_{v,2+a} .
\end{equation*}	
	
Acting the operator $L:=\Delta + pU^{p-1}$ on both sides of \eqref{Phi0-H-eq} and denoting $\phi_1^0 = L \Phi^0$, we obtain
	\begin{equation*}
\begin{cases}
		\pp_{\tau} \phi_1^{0}
		=
		\Delta \phi_1^0 + pU^{p-1} \phi_1^0 + h^0  & \mbox{ \ in \ }
		\mathcal{D}_R
\\
\phi_1^{0}(\cdot,\tau_0) = \gamma_0 \bar{e}_{0} Z_0
&
\mbox{ \ in \ }
B_R
\end{cases}
	\end{equation*}
with the following estimate
	\begin{equation*}
	\langle y\rangle|\nabla \phi_1^0 |
	+ |\phi_1^0 |
		\lesssim
		v
 \left( \min\{ \tau^{\frac 12}, \lambda_R^{- \frac 12 } \}  \lambda_R^{-\frac 12}  \theta_{R\hat{a}_{0}}^0   \langle y\rangle^{-n}
		+  \Theta_{R \hat{a}_{0} }^0(|y| )  \langle y\rangle^{-2}
	\right)
		\|h^0 \|_{v,2+a } .
	\end{equation*}
	
Taking into account $f_1 \phi^0 + f_2 y\cdot \nabla \phi^0$, we consider
\begin{equation*}
\begin{cases}
	\pp_{\tau} \phi_2^0 = \Delta \phi_2^0 + pU^{p-1}\phi_2^0 + f_1 \phi_2^0  + f_2 y\cdot \nabla \phi_2^0  + f_1 \phi_1^0  +
	f_2 y\cdot \nabla \phi_1^0
&	\mbox{ \ in \ } \DD_R
	\\
\phi_2^0(\cdot,\tau_0) = e_{02} Z_0
&	\mbox{ \ in \ } B_R
\end{cases}
\end{equation*}
where
\begin{equation*}
\left|f_1 \phi_1^0  +
f_2 y\cdot \nabla \phi_1^0  \right|
\lesssim
C_{f}  \tau^{-d}
 \min\{ \tau^{\frac 12}, \lambda_R^{- \frac 12 } \}  \lambda_R^{-\frac 12}
 v \theta_{R\hat{a}_{0}}^0  \langle y\rangle^{-2} \|h^0\|_{v,2+a} .
\end{equation*}
Using Lemma \ref{mode0-nonorth} again, one can find a solution $(\phi_2^0,e_{02})$ with the following estimates
\begin{equation*}
	|\phi_2^0|  \lesssim
C_{f}	\tau^{-d}
	\min\{ \tau^{\frac 12}, \lambda_R^{- \frac 12 } \}  \lambda_R^{-\frac 12} v \theta_{R\hat{a}_{0}}^0
	\left( \min\{ \tau^{\frac 12}, \lambda_R^{- \frac 12 } \}  \lambda_R^{-\frac 12}
 \ln R \langle y \rangle^{2-n}    +  \ln R \right)
	\|h^0 \|_{v,2+a} ,
\end{equation*}
\begin{equation*}
	|e_{02}|\lesssim  C_{f}
	\tau_0^{-d}
	\min\{ \tau_0^{\frac 12}, \lambda_{R(\tau_0)}^{- \frac 12 } \}  \lambda_{R(\tau_0)}^{-\frac 12} v(\tau_0) \theta_{R(\tau_0)\hat{a}_{0}}^0 \ln R(\tau_0) \| h^0\|_{v,2+a} .
\end{equation*}
Finally, we take $(\phi^0,e_0)= (\phi_1^0 + \phi_2^0  ,\gamma_0 \bar{e}_{0} + e_{02})$ and conclude the result for mode $0$:
\begin{equation*}
\begin{aligned}
	|\phi^0|  \lesssim \ &
	v \min\{ \tau^{\frac 12}, \lambda_R^{- \frac 12 } \}  \lambda_R^{-\frac 12}  \theta_{R\hat{a}_{0}}^0
	\left(  \langle y\rangle^{-n} +
	C_{f}	\tau^{-d}  \min\{ \tau^{\frac 12}, \lambda_R^{- \frac 12 } \}  \lambda_R^{-\frac 12}
	\ln R \langle y \rangle^{2-n}  \right)
	\|h^0 \|_{v,2+a}
\\
& +
v \left( \Theta_{R \hat{a}_{0} }^0(|y| )  \langle y\rangle^{-2}  +
C_{f} \tau^{-d}
\min\{ \tau^{\frac 12}, \lambda_R^{- \frac 12 } \}  \lambda_R^{-\frac 12} \theta_{R\hat{a}_{0}}^0
  \ln R \right)
\|h^0 \|_{v,2+a} ,
\end{aligned}
\end{equation*}
\begin{equation*}
	|e_0|\lesssim
		v(\tau_0) \theta_{R(\tau_0) \hat{a}_{0}}^0  \left(
1 +
C_{f}	\tau_0^{-d}
	\min\{ \tau_0^{\frac 12}, \lambda_{R(\tau_0)}^{- \frac 12 } \}  \lambda_{R(\tau_0)}^{-\frac 12}  \ln R(\tau_0) \right) \| h^0\|_{v,2+a} .
\end{equation*}

	\subsection{Modes $1$ to $n$ with orthogonality}
Set $r=|y|$. Consider
$
		h^1(y,\tau) = \sum\limits_{j=1}^n
		h_j(r,\tau) \Upsilon_j
$ satisfying
$\int_{B_{2R} } h^1 Z_j = 0 $
 for all  $j=1,\dots,n$,
 $\tau \in (\tau_0,\infty)$. Then
	\begin{equation}\label{hj-orthogonal}
		\int_0^{2R} h_j(r,\tau) U_r(r) r^{n-1} d r = 0
		\mbox{ \ for all \ }
		\tau \in (\tau_0,\infty)
	\end{equation}
	where $U_r(r) = (n(n-2))^{\frac{n-2}{4}} (2-n)r(1+r^2)^{-\frac n2}$.
	Let $H = H_j(r,\tau) \Upsilon_j$ satisfying $\mathcal{L}_{1} H_{j} +  \tilde{h}_j  = 0 $ in $\RR^n$,
	where $\tilde{h}_j$ is the extension of $h_j$ as zero outside $\DD_{R}$. $H_j$ is given by
	\begin{equation*}
	\begin{aligned}
		H_j(r,\tau) = \ &  U_r(r)
		\int_0^r
		\frac{1}{\rho^{n-1} U_r(\rho)^2 }
		\int_\rho^\infty
		\tilde{h}_j(s,\tau) U_r(s) s^{n-1} d s
		d \rho  \mbox{ \ for \ } -1 < a\le n-1,
		\\
	H_j(r,\tau) = \ &  -U_r(r)
	\int_r^{\infty}
	\frac{1}{\rho^{n-1} U_r(\rho)^2 }
	\int_\rho^\infty
	\tilde{h}_j(s,\tau) U_r(s) s^{n-1} d s
	d \rho
	\mbox{ \ for \ } a > n-1
	\end{aligned}
\end{equation*}
where $a>-1$ is used to guarantee that the spatial decay of $\tilde{h}_j(s,\tau) U_r(s) s^{n-1}$ is faster than $s^{-1}$. Using \eqref{hj-orthogonal}, one has
 the following estimate
	\begin{equation*}
		\|H_j \|_{v, \hat{a}_{1} } \lesssim \| h_j \|_{v,2+a} ,
	\end{equation*}
where $\hat{a}_{1}$ is given in \eqref{ahat-def}.
	Next, consider
	\begin{equation*}
		\begin{cases}
			\pp_\tau \Phi = \Delta \Phi + pU^{p-1} \Phi + H_j(r,\tau) \Upsilon_j
			\mbox{ \ in \ } D_{2R},
			\\
			\Phi = 0
			\mbox{ \ on \ } \pp D_{2R}
	\quad	\Phi(\cdot,\tau_0) = 0
			\mbox{ \ in \ }  B_{2R(\tau_0)} .
		\end{cases}
	\end{equation*}
	By Lemma \ref{M1-nonortho-lem}, we find a solution $\Phi_j$ with the estimate
	\begin{equation*}
		|\Phi_j| \lesssim
		v
		\theta_{R \hat{a}_{1} }^1  R^n  \langle y\rangle^{1-n}
		\| h^1 \|_{v,2+a} .
	\end{equation*}
It follows that
\begin{equation*}
\phi_{j1} = L \Phi_j ~~\mbox{ with }~~ |\phi_{j1}| \lesssim v
\theta_{R \hat{a}_{1} }^1  R^n  \langle y\rangle^{-1-n}
\| h^1 \|_{v,2+a} .
\end{equation*}
	Consider
\begin{equation*}
	\begin{cases}
		\pp_\tau \phi_{j2} =
		\mathcal{L}_{1} \phi_{j2}   + f_1 \phi_{j2} + f_2 r \pp_{r}\phi_{j2} + f_1 \phi_{j1} + f_2 r \pp_{r}\phi_{j1}
	\mbox{ \ for \ } r\in (0,R(\tau)), \ \tau \in  (\tau_0,\infty)
		\\
\pp_{r}\phi_{j2}(0,\tau) = 0  = \phi_{j2}(R(\tau),\tau)
\mbox{ \ for \ }
\tau\in (\tau_0,\infty),
\quad
	\phi_{j2}(r,\tau_0) = 0
\mbox{ \ for \ }
r\in (0,R(\tau_0))
	\end{cases}
\end{equation*}
where
\begin{equation*}
|f_1 \phi_{j1} + f_2 r \pp_{r}\phi_{j1} |
\lesssim C_{f} \tau^{-d} v
\theta_{R \hat{a}_{1} }^1  R^n  \langle y\rangle^{-1-n}
\| h^1 \|_{v,2+a} .
\end{equation*}
Using Lemma \ref{M1-nonortho-lem} again, we get $\phi_{j2}$ with the following estimate	\begin{equation*}
	|\phi_{j2} |
	\lesssim
C_{f} \tau^{-d} v
\theta_{R \hat{a}_{1} }^1  R^{2n}
	\langle y \rangle^{1-n}
	\| h^1\|_{v,2+a},
\end{equation*}

Set $\phi_{j}[h_j] = \phi_{j1}+\phi_{j2}$. Then
$ \phi^1[h^1] = \sum\limits_{j=1}^n
	\phi_j[h_j]\Upsilon_j $
with the following estimate
	\begin{equation*}
	\langle y\rangle |\nabla \phi^1| + |\phi^1|
		\lesssim
	v \theta_{R \hat{a}_{1} }^1  R^n  \big(	
		\langle y\rangle^{-1-n}
		+
		C_{f} \tau^{-d} R^{n}
		\langle y \rangle^{1-n}
		\big)
		\| h^1 \|_{v,2+a}
	\end{equation*}
as desired.

\end{proof}

\medskip

\appendix
	
\section{Estimates for heat equations}

Recalling $\TT_n^{out}$ defined in \eqref{T-out-def}, we only require $t_0\ge 0$ in Lemma \ref{annular-forward} and Lemma \ref{far-forward}.

\subsection{Heat equation with right hand side $v(t) |x|^{-b}  \1_{\{ l_1(t) \le |x| \le l_2(t) \}}$}
\begin{lemma}\label{annular-forward}
	Assume $n>2$, $v(t)\ge 0$, $b\in \RR$, $0\le l_1(t) \le l_2(t) \le C_{*} t^{\frac{1}{2}}$, $C_l^{-1} l_i(t) \le l_i(s) \le C_l l_i(t)  $, $i=1,2$, for all $\frac{t}{2} \le s\le t$, $t\ge t_0\ge 0$, where $C_*>0,C_l\ge1$.
	 Then
	\begin{equation*}
		\begin{aligned}
			&
			\TT_n^{out}\left[ v(t) |x|^{-b} \1_{\{ l_1(t) \le |x| \le l_2(t) \}}\right]
			\lesssim
			t^{-\frac n2}
			e^
			{-\frac{|x|^2}{16 t } }
			\int_{\frac{t_0}{2}}^{\frac{t}{2} }   v(s)
			\begin{cases}
				l_2^{n-b}(s)
				&
				\mbox{ \ if \ } b<n
				\\
				\ln (\frac{l_2(s)}{l_1(s)} )
				&
				\mbox{ \ if \ } b=n
				\\
				l_1^{n-b}(s)
				&
				\mbox{ \ if \ } b>n
			\end{cases}
			d s
			\\
			&
			+
			\sup\limits_{ t_1 \in [t/2,t] } v(t_1)
			\begin{cases}
				\begin{cases}
					l_2^{2-b}(t)
					& \mbox{ \ if \ }
					b<2
					\\
				\langle	\ln (\frac{l_2(t)}{l_1(t)}) \rangle
					& \mbox{ \ if \ }
					b=2
					\\
					l_1^{2-b}(t)
					& \mbox{ \ if \ }
					b>2
				\end{cases}
				&
				\mbox{ \ for \ }  |x| \le l_1(t)
				\\
				\begin{cases}
					l_2^{2-b}(t)
					&
					\mbox{ \ if \ } b<2
					\\
				\langle\ln(\frac{l_2(t)}{|x|} )  \rangle
					&
					\mbox{ \ if \ } b=2
					\\
					|x|^{2-b}
					&
					\mbox{ \ if \ } 2<b<n
					\\
					|x|^{2-n} \langle	\ln (\frac{|x|}{l_1(t)}) \rangle
					&
					\mbox{ \ if \ } b=n
					\\
					|x|^{2-n}	l_1^{n-b}(t)
					&
					\mbox{ \ if \ } b>n
				\end{cases}
				&
				\mbox{ \ for \ }
				l_1(t) < |x| \le l_2(t)
				\\
				|x|^{2-n}
				e^{-\frac{|x|^2}{16 t}}
				\begin{cases}
					l_2^{n-b}(t)
					&
					\mbox{ \ if \ } b<n
					\\
					\langle \ln (\frac{l_2(t) }{l_1(t)} ) \rangle
					&
					\mbox{ \ if \ } b=n
					\\
					l_1^{n-b}(t)
					&
					\mbox{ \ if \ } b>n
				\end{cases}
				&
				\mbox{ \ for \ }
				|x| > l_2(t)
			\end{cases}
			.
		\end{aligned}
	\end{equation*}

\end{lemma}

\begin{proof}
	\begin{equation*}
		\begin{aligned}
			&
			\TT_n^{out}\left[ v(t) |x|^{-b} \1_{\{ l_1(t) \le |x| \le l_2(t) \}}\right]
			\lesssim
			t^{-\frac n2}
			\int_{\frac{t_{0}}{2} }^{\frac t2}
			\int_{\RR^n}
			e^{-\frac{|x- y|^2}{4 t} }
			v(s) |y|^{-b} \1_{\{ l_1(s) \le |y| \le l_2(s) \}} d y d s
			\\
			& +
			\sup\limits_{ t_1 \in [t/2,t] } v(t_1)
			\int_{\frac t2}^{t}
			\int_{\RR^n}
			(t-s)^{-\frac n2}
			e^{-\frac{|x- y|^2}{4(t-s )} }
			|y|^{-b} \1_{\{ C_l^{-1} l_1(t) \le |y| \le C_l l_2(t) \}} d y d s
			:=
			u_{1} + \sup\limits_{ t_1 \in [t/2,t] } v(t_1) u_{2} .
		\end{aligned}
	\end{equation*}
	
	For $u_{1}$, notice $|y|\le C_* t^{\frac 12}$. For $|x|\le 2C_* t^{\frac 12}$, we have
	\begin{equation*}
		u_{1}
		\lesssim
		t^{-\frac n2}
		\int_{\frac{t_{0}}{2} }^{\frac t2}
		\int_{\RR^n}
		v(s) |y|^{-b} \1_{\{ l_1(s)\le |y| \le l_2(s) \}} d y d s
		\lesssim
		t^{-\frac n2}
		\int_{\frac {t_{0}}2} ^{\frac t2}
		v(s)
		\begin{cases}
			l_2^{n-b}(s)
			&
			\mbox{ \ if \ }
			b<n
			\\
			\ln (\frac{l_2(s)}{ l_1(s) } )
			&
			\mbox{ \ if \ }
			b=n
			\\
			l_1^{n-b}(s)
			&
			\mbox{ \ if \ }
			b>n
		\end{cases}
		d s
		.
	\end{equation*}
For $|x| > 2C_* t^{\frac 12}$, one has  $|x-y|\ge \frac{|x|}{2}$. Then
	\begin{equation*}
		u_{1}
		\lesssim
		t^{-\frac n2} e^{-\frac{|x|^2}{16 t} }
		\int_{\frac {t_{0}}2} ^{\frac t2}
		v(s)
		\begin{cases}
			l_2^{n-b}(s)
			&
			\mbox{ \ if \ }
			b<n
			\\
			\ln (\frac{l_2(s)}{ l_1(s)} )
			&
			\mbox{ \ if \ }
			b=n
			\\
			l_1^{n-b}(s)
			&
			\mbox{ \ if \ }
			b>n
		\end{cases}
		d s
		.
	\end{equation*}
	
	Let us estimate $u_2$ in  different regions.
	
For $|x| \le (2 C_l)^{-1} l_1(t)$,
since $\frac{|y|}{2} \le |x-y|\le 2|y|$, then
	\begin{equation*}
		\begin{aligned}
			u_{2}
			\le \ &
			\int_{\frac t2}^{t}
			\int_{\RR^n}
			(t-s)^{-\frac n2}
			e^{-\frac{|y|^2}{16 (t-s )} }
			|y|^{-b} \1_{\{ C_l^{-1} l_1(t) \le |y| \le C_l l_2(t) \}} d y d s
			\sim
			\int_{\frac t2}^{t}
			(t-s)^{-\frac b2}
			\int_{\frac{l_1^2(t)}{16 C_l^2 (t-s)} }^{
				\frac{C_l^2 l_2^2(t)}{16(t-s)}}
			e^{-z} z^{\frac{n-b}{2} -1}  d z d s
			\\
			= \ &
			\left(
			\int_{\frac t2}^{t-l_2^2(t) }
			+
			\int_{t-l_2^2(t)}^{t-l_1^2(t)}
			+
			\int_{t- l_1^2(t) }^{t}
			\cdots
			\right)
			:=  u_{21} + u_{22} + u_{23}
			\lesssim
			\begin{cases}
				l_2^{2-b}(t)
				& \mbox{ \ if \ }
				b<2
				\\
			\langle	\ln (\frac{l_2(t)}{l_1(t)}) \rangle
				& \mbox{ \ if \ }
				b=2
				\\
			l_1^{2-b}(t)
				& \mbox{ \ if \ }
				b>2 .
			\end{cases}
		\end{aligned}
	\end{equation*}
	In order to get the last inequality above, we need the following estimates.
	For $u_{21} $, since $n>2$, we have
	\begin{equation*}
			u_{21} \sim
			\int_{\frac t2}^{t- l_2^2(t) }
			(t-s)^{-\frac b2}
			\int_{\frac{l_1^2(t) }{ 16 C_l^2 (t-s)}}^{\frac{ C_l^2 l_2^2(t) }{16(t-s)}}
			z^{\frac{n-b}{2} - 1}
			d z d s
			\lesssim
			\begin{cases}
				l_2^{2-b}(t)
				&
				\mbox{ \ if \ } b<n
				\\
				\langle \ln (\frac{l_2(t)}{l_1(t)}) \rangle
				l_2^{2-n}(t)
				&
				\mbox{ \ if \ } b=n
				\\
				l_1^{n-b}(t)
				l_2^{2-n}(t)
				&
				\mbox{ \ if \ } b>n .
			\end{cases}
	\end{equation*}
	For $u_{22}$, since $l_1^2(t) \lesssim t-s \lesssim l_2^2(t)$, then
	\begin{equation*}
		\begin{aligned}
			u_{22}	\lesssim \ &
			\int_{t-l_2^2(t)}^{t-l_1^2(t)}
			(t-s)^{-\frac b2}
			\begin{cases}
				1
				& \mbox{ \ if \ }
				b<n
				\\
			\langle \ln ( \frac{ t-s}{l_1^2(t)})  \rangle
				& \mbox{ \ if \ }
				b=n
				\\
				( \frac{l_1^2(t)}{t-s})^{\frac{n-b}{2} }
				& \mbox{ \ if \ }
				b>n
			\end{cases}
			d s
			\lesssim
			\begin{cases}
				l_2^{2-b}(t)
				& \mbox{ \ if \ }
				b<2
				\\
				\langle \ln (\frac{l_2(t)}{l_1(t)})  \rangle
				& \mbox{ \ if \ }
				b=2
				\\
				l_1^{2-b}(t)
				& \mbox{ \ if \ }
				b>2 .
			\end{cases}
		\end{aligned}
	\end{equation*}	
	For $u_{23}$, we have
	\begin{equation*}
			u_{23}
			\lesssim
			\int_{t-l_1^2(t)}^{t}
			(t-s)^{-\frac b2}
			\int_{\frac{l_1^2(t) }{ 16 C_l^2 (t-s)}}^{\frac{C_l^2 l_2^2(t) }{16(t-s) }}
			e^{-\frac z2}
			d z d s
			\lesssim
			\int_{t-l_1^2(t)}^{t}
			(t-s)^{-\frac b2}
			e^{- \frac{l_1^2(t)   }{ 32C_l^2 (t-s)} } d s
			\sim
			l_1^{2-b}(t) .
	\end{equation*}
	
	For $ (2C_l)^{-1} l_1(t) \le |x| \le 2C_l l_2(t)$,  then
	\begin{equation*}
		\begin{aligned}
			u_{2}
			\le \ &
			\int_{\frac t2}^{t}
			\int_{\RR^n}
			(t-s)^{-\frac n2}
			e^{-\frac{|x- y|^2}{4(t-s )} }
			|y|^{-b}
			\left(
			\1_{\{ (4C_l)^{-1} l_1(t) \le |y| \le \frac{|x|}{2} \}}
			+
			\1_{\{\frac{|x|}{2} \le |y| \le 2|x| \}}
			+
			\1_{\{ 2|x| \le |y| \le 4C_l l_2(t) \}}
			\right)
			d y d s
			\\
			:= \ &
			u_{21} + u_{22} + u_{23}
			\lesssim
			\begin{cases}
				l_2^{2-b}(t)
				&
				\mbox{ \ if \ } b<2
				\\
				\langle 	\ln(\frac{l_2(t)}{|x|} ) \rangle
				&
				\mbox{ \ if \ } b=2
				\\
				|x|^{2-b}
				&
				\mbox{ \ if \ } 2<b<n
				\\
				|x|^{2-n} \langle	\ln (\frac{|x|}{l_1(t)}) \rangle
				&
				\mbox{ \ if \ } b=n
				\\
				|x|^{2-n}	l_1^{n-b}(t)
				&
				\mbox{ \ if \ } b>n .
			\end{cases}
		\end{aligned}
	\end{equation*}
	For the last inequality above, we need to estimate $u_{21}$, $u_{22}$ and $u_{23}$. For $u_{21}$, since $n>2$, one has
	\begin{equation*}
			u_{21}
			\le
			\int_{\frac t2}^{t}
			\int_{\RR^n}
			(t-s)^{-\frac n2}
			e^{-\frac{|x|^2}{16(t-s )} }
			|y|^{-b}
			\1_{ \{ (4C_l)^{-1} l_1(t) \le |y| \le \frac{|x|}{2} \}}
			d y d s
			\lesssim
			\begin{cases}
				|x|^{2-b}
				&
				\mbox{ \ if \ } b<n
				\\
				|x|^{2-n}  \langle	\ln (\frac{|x|}{l_1(t)} )  \rangle
				&
				\mbox{ \ if \ } b=n
				\\
				|x|^{2-n} l_1^{n-b}(t)
				&
				\mbox{ \ if \ } b>n .
			\end{cases}
	\end{equation*}
	For $u_{22}$, we have
	\begin{equation*}
			u_{22}
		\lesssim
			|x|^{-b}
			\int_{\frac t2}^{t}
			\int_{\RR^n}
			(t-s)^{-\frac n2}
			e^{-\frac{|x- y|^2}{4(t-s )} }
			\1_{\{ |x-y| \le 3|x| \}}
			d y d s
			\sim
			|x|^{-b}
			\int_{\frac t2}^{t}
			\int_{0}^{\frac{9|x|^2}{4(t-s)}} e^{-z} z^{\frac n2 -1} d z d s
			\sim
			|x|^{2-b} .
	\end{equation*}
	For $u_{23}$, we have
	\begin{equation*}
		\begin{aligned}
			u_{23}
			\le \ &
			\int_{\frac t2}^{t}
			\int_{\RR^n}
			(t-s)^{-\frac n2}
			e^{-\frac{|y|^2}{16 (t-s )} }
			|y|^{-b}
			\1_{\{ 2|x| \le |y| \le 4C_l l_2(t) \}}
			d y d s
			\sim
			\int_{\frac t2}^{t}
			\int_{\frac{|x|^2}{4(t-s)} }^{ \frac{C_l^2 l_2^2(t)}{t-s} }
			(t-s)^{-\frac b2}
			e^{-z} z^{\frac{n-b}{2} -1 }
			d z d s
			\\
			=  \ &
			\left(
			\int_{\frac t2}^{t-\frac{l_2^2(t)}{2C_*^2}}
			+
			\int_{ t-\frac{l_2^2(t)}{2C_*^2} }^{t- \frac{|x|^2}{ 8C_l^2 C_*^2 }}
			+
			\int_{t- \frac{|x|^2}{8C_l^2 C_*^2 }}^{t}
			\cdots
			\right)
			:=
			u_{231} + u_{232} + u_{233}
			\lesssim
			\begin{cases}
				l_2^{2-b}(t)
				&
				\mbox{ \ if \ } b<2
				\\
				\langle \ln(\frac{l_2(t)}{|x|} ) \rangle
				&
				\mbox{ \ if \ } b=2
				\\
				|x|^{2-b}
				&
				\mbox{ \ if \ } b>2 .
			\end{cases}
		\end{aligned}
	\end{equation*}
	
	In order to get the last inequality above, we need the following estimates. For $u_{231}$, we have
	\begin{equation*}
			u_{231}
			\sim
			\int_{\frac t2}^{t-\frac{l_2^2(t)}{2C_*^2}}
			\int_{\frac{|x|^2}{4(t-s)} }^{ \frac{C_l^2 l_2^2(t)}{ t-s } }
			(t-s)^{-\frac b2}
			z^{\frac{n-b}{2} -1 }
			d z d s
			\lesssim
			\begin{cases}
				l_2^{2-b}(t)
				&
				\mbox{ \ if \ } b<n
				\\
				l_2^{2-n}(t)
			\langle	\ln(\frac{l_2(t)}{|x|}) \rangle
				&
				\mbox{ \ if \ } b=n
				\\
				|x|^{n-b} l_2^{2-n}(t)
				&
				\mbox{ \ if \ } b>n .
			\end{cases}
	\end{equation*}
	For $u_{232}$, since $n>2$, we estimate
	\begin{equation*}
			u_{232}
			\lesssim
	\int_{ t-\frac{l_2^2(t)}{2C_*^2} }^{t- \frac{|x|^2}{ 8C_l^2 C_*^2 }}
			(t-s)^{-\frac b2}
			\begin{cases}
				1
				&
				\mbox{ \ if \ } b<n
				\\
			\langle \ln(\frac{|x|^2}{t-s}) \rangle
				&
				\mbox{ \ if \ } b=n
				\\
				(\frac{|x|^2}{t-s})^{\frac{n-b}{2}}
				&
				\mbox{ \ if \ } b>n
			\end{cases}
			d s
			\lesssim
			\begin{cases}
				l_2^{2-b}(t)
				&
				\mbox{ \ if \ } b<2
				\\
				\langle	\ln(\frac{l_2(t)}{|x|} ) \rangle
				&
				\mbox{ \ if \ } b=2
				\\
				|x|^{2-b}
				&
				\mbox{ \ if \ } b>2 .
			\end{cases}
	\end{equation*}
	For $u_{233}$, one has
	\begin{equation*}
			u_{233}
			\lesssim
			\int_{t- \frac{|x|^2}{8C_l^2 C_*^2 }}^{t}
			(t-s)^{-\frac b2}
			e^{-\frac{|x|^2}{8(t-s)}} d s
			\sim
			|x|^{2-b}
			\int_{ C_l^2 C_{*}^2}^{\infty}
			e^{-z} z^{\frac b2 -2}
			d z
			\sim
			|x|^{2-b}.
	\end{equation*}
	
	\medskip
	
	For $|x|\ge 2C_l l_2(t)$, since $\frac{|x|}{2} \le |x-y| \le 2|x|$, then for  $n>2$, it follows that
	\begin{equation*}
			u_{2}
			\lesssim
			\int_{\frac{t}{2} }^{t}  (t-s)^{-\frac n2}
			e^
			{-\frac{|x|^2}{16(t-s )} }
			d s
			\begin{cases}
				l_2^{n-b}(t)
				&
				\mbox{ \ if \ } b<n
				\\
			\langle \ln (\frac{l_2(t)}{l_1(t)} ) \rangle
				&
				\mbox{ \ if \ } b=n
				\\
			l_1^{n-b}(t)
				&
				\mbox{ \ if \ } b>n
			\end{cases}
			\lesssim
			|x|^{2-n}
			e^{-\frac{|x|^2}{16 t}}
			\begin{cases}
				l_2^{n-b}(t)
				&
				\mbox{ \ if \ } b<n
				\\
				\langle \ln (\frac{l_2(t)}{l_1(t)} ) \rangle
				&
				\mbox{ \ if \ } b=n
				\\
				l_1^{n-b}(t)
				&
				\mbox{ \ if \ } b>n .
			\end{cases}
	\end{equation*}

\end{proof}

\subsection{Heat equation with right hand side $ v(t)|x|^{-b}\1_{\{ |x|\ge t^{1/2} \}}$}
\begin{lemma}\label{far-forward}
	Assume $n>0$, $v(t)\ge 0$, $b\in \RR$, $t_0\ge 0$, then
	\begin{equation*}
		\begin{aligned}
			&
			\TT_{n}^{out}\left[v(t) |x|^{-b} \1_{\{ |x|\ge t^{\frac 12}\}}\right]
			\\
			\lesssim  \ &
			\begin{cases}
				t^{-\frac n2}
				\int_{t_{0}/2}^{t/2}
				v(s)
				\begin{cases}
					t^{\frac{n-b}{2}}
					&
					\mbox{ \ if \ } b<n
					\\
				\langle	\ln ( t s^{-1}) \rangle
					&
					\mbox{ \ if \ } b=n
					\\
					s^{\frac{n-b}{2}}
					&
					\mbox{ \ if \ } b>n
				\end{cases}
				d s
				+ t^{1-\frac{b}{2}} \sup\limits_{ t_1 \in [t/2,t] } v(t_1)
				&
				\mbox{ \ if \ } |x|\le t^{\frac 12}
				\\
				|x|^{-b}
				\left( t \sup\limits_{ t_1 \in [t/2,t] } v(t_1) +	\int_{t_{0}/2}^{t/2}
				v(s) d s \right) +
				t^{-\frac n2} e^{-\frac{|x|^2}{16 t} }
				\int_{t_{0}/2}^{t/2}
				v(s)
				\begin{cases}
					0
					&
					\mbox{ \ if \ }
					b<n
					\\
				\langle	\ln (|x| s^{-\frac 12}) \rangle
					&
					\mbox{ \ if \ }
					b=n
					\\
					s^{\frac{n-b}{2}}
					&
					\mbox{ \ if \ }
					b>n
				\end{cases}
				d s
				&
				\mbox{ \ if \ } |x| > t^{\frac 12}
			\end{cases}
			.
		\end{aligned}
	\end{equation*}

\end{lemma}

\begin{proof}
By definition, we write
	\begin{equation*}
		\begin{aligned}
			&
			\TT_{n}^{out}\left[v(t) |x|^{-b} \1_{\{ |x|\ge t^{\frac 12}\}}\right]
			\lesssim
			t^{-\frac n2}
			\int_{t_{0}/2}^{t/2}
			\int_{\RR^n}
			e^{-\frac{|x- y|^2}{4 t} }
			v(s)|y|^{-b} \1_{\{ |y|\ge s^{\frac 12}\}} d y d s
			\\
			&
			+
			\sup\limits_{ t_1 \in [t/2,t]} v(t_1)
			\int_{t/2}^t
			\int_{\RR^n}
			(t-s)^{-\frac n2}
			e^{-\frac{|x- y|^2}{4(t-s )} }
			|y|^{-b} \1_{\{ |y|\ge 2^{-\frac 12} t^{\frac 12}\}} d y d s
			:=
		t^{-\frac n2} u_{1} + \sup\limits_{ t_1 \in [t/2,t] } v(t_1) u_{2}.
		\end{aligned}
	\end{equation*}
	For $u_{1}$, when $|x|\le 2t^{\frac 12}$, we have
	\begin{equation*}
		\begin{aligned}
		&	u_{1}
			\lesssim
			\int_{t_{0}/2}^{t/2}
			\int_{\RR^n}
			v(s)|y|^{-b}
			\1_{\{ s^{\frac 12} \le |y|\le 4 t^{\frac 12}  \}}
			d y d s
			+
			\int_{t_{0}/2}^{t/2}
			\int_{\RR^n}
			e^{-\frac{|y|^2}{16 t} }
			v(s)|y|^{-b}
			\1_{\{ 4 t^{\frac 12} \le  |y| \}}
			d y d s
			\\
\lesssim \ &
\int_{t_{0}/2}^{t/2}
v(s)
\begin{cases}
	t^{\frac{n-b}{2}}
	&
	\mbox{ \ if \ } b<n
	\\
	\langle \ln ( t s^{-1}) \rangle
	&
	\mbox{ \ if \ } b=n
	\\
	s^{\frac{n-b}{2}}
	&
	\mbox{ \ if \ } b>n
\end{cases}
d s
+
t^{\frac {n-b}2}
\int_{t_{0}/2}^{t/2} v(s) d s
\sim
\int_{t_{0}/2}^{t/2}
v(s)
\begin{cases}
	t^{\frac{n-b}{2}}
	&
	\mbox{ \ if \ } b<n
	\\
	\langle \ln ( t s^{-1}) \rangle
	&
	\mbox{ \ if \ } b=n
	\\
	s^{\frac{n-b}{2}}
	&
	\mbox{ \ if \ } b>n
\end{cases}
d s  .
		\end{aligned}
	\end{equation*}
	
	For $u_{1}$, when $|x| > 2 t^{\frac 12}$, we have
		\begin{align*}
			u_{1} = \ &
			\int_{t_{0}/2}^{t/2}
			\int_{\RR^n}
			e^{-\frac{|x- y|^2}{4 t} }
			v(s)|y|^{-b}
			\left(
			\1_{\{ s^{\frac 12} \le |y|\le \frac{|x|}{2} \}}
			+
			\1_{\{ \frac{|x|}{2} \le |y|\le 2|x|\}}
			+
			\1_{\{ 2|x| \le |y| \}}
			\right)
			d y d s
			\\
\lesssim \ &
	\int_{t_{0}/2}^{t/2}
	v(s)
\int_{\RR^n}
\left(
e^{-\frac{|x|^2}{16 t} }
|y|^{-b}
\1_{\{ s^{\frac 12} \le |y|\le \frac{|x|}{2} \}}
+
|x|^{-b}
e^{-\frac{|x- y|^2}{4 t} }
\1_{\{ |x-y|\le 3|x|\}}
+
e^{-\frac{ |y|^2}{16 t} }
|y|^{-b}
\1_{\{ 2|x| \le |y| \}}
\right)
d y d s
\\
\lesssim \ &
 e^{-\frac{|x|^2}{16 t} }
			\int_{t_{0}/2}^{t/2}
			v(s)
			\begin{cases}
				|x|^{n-b}
				&
				\mbox{ \ if \ }
				b<n
				\\
			\langle	\ln (|x| s^{-\frac 12}) \rangle
				&
				\mbox{ \ if \ }
				b=n
				\\
				s^{\frac{n-b}{2}}
				&
				\mbox{ \ if \ }
				b>n
			\end{cases}
			d s
			+ t^{\frac{n}{2}} |x|^{-b}
			\int_{t_{0}/2}^{t/2}
			v(s) d s
			+
			t^{\frac {n-b}2} e^{-\frac{|x|^2}{8t}}
			\int_{t_{0}/2}^{t/2}
			v(s) d s
			\\
			\lesssim \ &
	 e^{-\frac{|x|^2}{16 t} }
			\int_{t_{0}/2}^{t/2}
			v(s)
			\begin{cases}
				0
				&
				\mbox{ \ if \ }
				b<n
				\\
			\langle	\ln ( |x| s^{-\frac 12}) \rangle
				&
				\mbox{ \ if \ }
				b=n
				\\
				s^{\frac{n-b}{2}}
				&
				\mbox{ \ if \ }
				b>n
			\end{cases}
			d s
			+ 	t^{ \frac n2} |x|^{-b}
			\int_{t_{0}/2}^{t/2}
			v(s) d s .
		\end{align*}

	For $u_{2}$, when $|x|\le 2^{-\frac 32} t^{\frac 12}$, we have $|y|\ge 2|x|$. Then
	\begin{equation*}
			u_{2} \le
			\int_{t/2}^t
			\int_{\RR^n}
			(t-s)^{-\frac n2}
			e^{-\frac{|y|^2}{16(t-s )} }
			|y|^{-b} \1_{\{ |y|\ge 2^{-\frac{1}{2}} t^{\frac 12}\}} d y d s
			\lesssim
			\int_{t/2}^t
			(t-s)^{-\frac b2}
			e^{- \frac{t}{64(t-s) }}
			d s
			\sim
			t^{1-\frac b2} .
	\end{equation*}

	For $u_{2}$, when $|x|\ge 2^{-\frac 32} t^{\frac 12}$, one has
	\begin{equation*}
		\begin{aligned}
			& u_{2} \le
			\int_{t/2}^t
			\int_{\RR^n}
			(t-s)^{-\frac n2}
			e^{-\frac{|x- y|^2}{4(t-s )} }
			|y|^{-b}
			\left(
			\1_{\{ 9^{-1} t^{\frac 12} \le |y|\le  \frac{|x|}{2} \}}
			+
			\1_{\{ \frac{|x|}{2} \le |y|\le  4|x| \}}
			+
			\1_{\{ 4|x|\le |y| \}}
			\right) d y d s
			\\
\lesssim \ &
\int_{t/2}^t
\int_{\RR^n}
(t-s)^{-\frac n2}
\left(
e^{-\frac{|x|^2}{16(t-s )} }
|y|^{-b}
\1_{\{ 9^{-1} t^{\frac 12} \le |y|\le  \frac{|x|}{2} \}}
+
|x|^{-b} e^{-\frac{|x- y|^2}{4(t-s )} }
\1_{\{ |x-y|\le  5|x| \}}
+
	e^{-\frac{|y|^2}{16(t-s )} }
|y|^{-b}
\1_{\{ 4|x|\le |y| \}}
\right) d y d s
\\
			\lesssim \ &
			|x|^{2-n}
			e^{-\frac{|x|^2}{16 t}}
			\begin{cases}
				|x|^{n-b}
				&
				\mbox{ \ if \ } b<n
				\\
			\langle\ln( |x|t^{-\frac 12}) \rangle
				&
				\mbox{ \ if \ } b=n
				\\
				t^{\frac{n-b}{2}}
				&
				\mbox{ \ if \ } b>n
			\end{cases}
			+
			t  |x|^{-b}
			+
			|x|^{2-b}
			e^{-\frac{|x|^2}{2t}}
			\sim
			t |x|^{-b} .
		\end{aligned}
	\end{equation*}
	
\end{proof}

\medskip

\subsection{Cauchy problem with initial value $\langle y \rangle^{-b}$}
\begin{lemma}\label{cauchy-est-lem}
	For $n\ge 1$, $b\in \RR$ and $t>0$, it holds that
	\begin{equation*}
\begin{aligned}
	&
			(4\pi t)^{-\frac n2} \int_{\RR^n}
			e^{-\frac{|x-y|^2}{4t} }
			\langle y \rangle^{-b} d y
			\\
			\lesssim \ &
\begin{cases}
	\langle t \rangle^{-\frac{b}{2}}	
	\1_{\{ |x|\le \langle t \rangle^{\frac 12} \}}
	+
| x|^{-b}
	\1_{\{ |x| > \langle t \rangle^{\frac 12} \}}
	&
	\mbox{ \ if \ }
	b < n
	\\
	\langle t \rangle^{-\frac n2}	\ln ( t +2 ) \1_{\{ |x|\le \langle t \rangle^{\frac 12} \}}
	+
	\left(| x|^{-n}
	+
	t^{-\frac n2}	e^{-\frac{|x|^2}{16t} } \ln (|x| +2 )\right)
	\1_{\{ |x| > \langle t \rangle^{\frac 12} \}}
	&
	\mbox{ \ if \ }
	b = n
	\\
	\langle t \rangle^{-\frac n2}
	\1_{\{ |x|\le \langle t \rangle^{\frac 12} \}}
	+
	\left( | x|^{-b}
	+  t^{-\frac n2}
	e^{-\frac{|x|^2}{16t} }  \right)
	\1_{\{ |x| > \langle t \rangle^{\frac 12} \}}
	&
	\mbox{ \ if \ }
	b>n
\end{cases}
.
\end{aligned}
	\end{equation*}
	
\end{lemma}

\begin{proof}
	Set
	\begin{equation*}
		u(x,t) =
		(4\pi t)^{-\frac n2} \int_{\RR^n}
		e^{-\frac{|x-y|^2}{4t} }
		\langle y \rangle^{-b} d y
\sim
t^{-\frac n2}
\left(
\int_{|y|\le \frac{|x|}{2}}
+
\int_{\frac{|x|}{2} \le|y| \le 2|x|}
+
\int_{2|x| \le |y|}
\right)
e^{-\frac{|x-y|^2}{4t} }
\langle y \rangle^{-b} d y
		.
	\end{equation*}
We estimate term by term:
	\begin{equation*}
		\int_{|y|\le \frac{|x|}{2}}
		e^{-\frac{|x-y|^2}{4t} }
		\langle y\rangle^{-b} d y
		\lesssim
		e^{-\frac{|x|^2}{16t} }
		\int_{|y|\le \frac{|x|}{2}}
		\langle y\rangle^{-b}  d y
		\sim
\begin{cases}
e^{-\frac{|x|^2}{16t} } |x|^{n}
&
\mbox{ \ if \ } |x| \le  1
\\
		\begin{cases}
			e^{-\frac{|x|^2}{16t} } |x|^{n-b}
			&
			\mbox{ \ if \ }
			b<n
			\\
			e^{-\frac{|x|^2}{16t} } \ln (|x| +2 )
			&
			\mbox{ \ if \ }
			b = n
			\\
			e^{-\frac{|x|^2}{16t} }
			&
			\mbox{ \ if \ }
			b > n
		\end{cases}
&
\mbox{ \ if \ } |x|> 1
\end{cases},
	\end{equation*}
	\begin{equation*}
		\begin{aligned}
			&
			\int_{\frac{|x|}{2} \le|y| \le 2|x|}
			e^{-\frac{|x-y|^2}{4t} }
			\langle y\rangle^{-b}  d y
			\lesssim
			\langle x\rangle^{-b}
			\int_{ |x-y| \le 3|x|}
			e^{-\frac{|x-y|^2}{4t} }
			d y
			\sim
			\begin{cases}
			 \langle x\rangle^{-b} 	|x|^{n}
				&
				\mbox{ \ if \ } |x|\le t^{\frac{1}{2}}
				\\
			\langle x\rangle^{-b}
			t^{\frac{n}{2}}
				&
				\mbox{ \ if \ } |x| > t^{\frac{1}{2}}
			\end{cases}
			,
		\end{aligned}
	\end{equation*}
and
\begin{equation*}
\int_{2|x| \le |y|}
e^{-\frac{|x-y|^2}{4t} }
\langle y\rangle^{-b} d y
\le
\int_{2|x| \le |y|}
e^{-\frac{|y|^2}{16t} }
\langle y\rangle^{-b} d y .
\end{equation*}
	For $|x|\ge 1$, we have
	\begin{equation*}
			\int_{2|x| \le |y|}
			e^{-\frac{|y|^2}{16t} }
			\langle y\rangle^{-b} d y
			\sim
			t^{\frac{n-b}{2}}
			\int_{\frac{|x|^2}{4t}}^\infty
			e^{-z} z^{\frac{n-b}{2}-1} d z
			\lesssim
			\begin{cases}
				t^{\frac{n-b}{2}}	
				&
				\mbox{ \ if \ }
				|x|\le t^{\frac 12},
				b < n
				\\
				\ln(\frac{t}{|x|^2} ) +1	
				&
				\mbox{ \ if \ }
				|x|\le t^{ \frac 12 },
				b = n
				\\
				|x|^{n-b}
				&
				\mbox{ \ if \ }
				|x|\le t^{ \frac 12 },
				b>n
				\\
				t^{\frac{n-b}{2}}
				e^{-\frac{|x|^2}{8t}}
				&
				\mbox{ \ if \ }
				|x| > t^{ \frac 12 }
			\end{cases} .
	\end{equation*}
For $|x|<1$, we have
\begin{equation*}
	\begin{aligned}
		&
		\int_{2|x| \le |y|}
		e^{-\frac{|y|^2}{16t} }
		\langle y\rangle^{-b} d y
		\sim
\int_{2|x|}^{2}
	e^{-\frac{r^2}{16t} } r^{n-1} dr
+
		\int_{2}^\infty
		e^{-\frac{ r^2}{16t} }
		r^{n-1-b} d r
\\
\lesssim \ &
\begin{cases}
	t^{\frac n2} e^{ -\frac{|x|^2}{8t} }
& \mbox{ \ if \ } t\le |x|^2
\\
	t^{\frac n2}
& \mbox{ \ if \ } |x|^2<t\le 1
\\
1 & \mbox{ \ if \ } t >1
\end{cases}
+
\begin{cases}
t^{\frac{n-b}{2}}
e^{-\frac{ 1 }{8t}}
&
\mbox{ \ if \ }
t<1
\\
	t^{\frac{n-b}{2}}	
	&
	\mbox{ \ if \ }
 t\ge 1,
	b < n
	\\
 \ln (t+2)
	&
	\mbox{ \ if \ }
t\ge 1,
	b = n
	\\
1
	&
	\mbox{ \ if \ }
 t\ge 1,
	b>n
\end{cases}
\lesssim
\begin{cases}
t^{\frac n2} e^{ -\frac{|x|^2}{16 t} }
& \mbox{ \ if \ } t\le |x|^2
\\
	t^{\frac n2}
& \mbox{ \ if \ } |x|^2<t\le 1
	\\
	t^{\frac{n-b}{2}}	
	&
	\mbox{ \ if \ }
	t\ge 1,
	b < n
	\\
	\ln (t+2)
	&
	\mbox{ \ if \ }
	t\ge 1,
	b = n
	\\
	1
	&
	\mbox{ \ if \ }
	t\ge 1,
	b>n
\end{cases} .
	\end{aligned}
\end{equation*}

Combining above estimates, one has
\begin{equation*}
	\begin{aligned}
		u(x,t)
		\lesssim \ &
		\begin{cases}
			\begin{cases}
				1
				& \mbox{ \ if \ } t\le 1
				\\
				t^{ - \frac{b}{2}}	
				&
				\mbox{ \ if \ }
				t\ge 1,
				b < n
				\\
				t^{-\frac n2}\ln (t+2)
				&
				\mbox{ \ if \ }
				t\ge 1,
				b = n
				\\
				t^{-\frac n2}
				&
				\mbox{ \ if \ }
				t\ge 1,
				b>n
			\end{cases}
			&
			\mbox{ \ if \ } |x|< 1
			\\
			\begin{cases}
				t^{-\frac{b}{2}}	
				&
				\mbox{ \ if \ }
				|x|\le t^{\frac 12},
				b < n
				\\
				t^{-\frac n2}	\ln ( t +2 )
				&
				\mbox{ \ if \ }
				|x|\le t^{ \frac 12 },
				b = n
				\\
				t^{-\frac n2}
				&
				\mbox{ \ if \ }
				|x|\le t^{ \frac 12 },
				b>n
				\\
				\langle x\rangle^{-b}
				
				&
				\mbox{ \ if \ }
				|x| > t^{ \frac 12 }, b<n
				\\
				\langle x\rangle^{-n}
				+
				t^{-\frac n2}	e^{-\frac{|x|^2}{16t} } \ln (|x| +2 )
				&
				\mbox{ \ if \ }
				|x| > t^{ \frac 12 }, b=n
				\\
				\langle x\rangle^{-b}
				+  t^{-\frac n2}
				e^{-\frac{|x|^2}{16t} }
				&
				\mbox{ \ if \ }
				|x| > t^{ \frac 12 }, b>n
			\end{cases}
			&
			\mbox{ \ if \ } |x|\ge 1
		\end{cases}
		\\
		\lesssim \ &
		\begin{cases}
			\langle t \rangle^{-\frac{b}{2}}	
			&
			\mbox{ \ if \ }
			|x|\le \max\{1,t^{\frac 12} \},
			b < n
			\\
			\langle t \rangle^{-\frac n2}	\ln ( t +2 )
			&
			\mbox{ \ if \ }
			|x|\le  \max\{1,t^{\frac 12} \} ,
			b = n
			\\
			\langle t \rangle^{-\frac n2}
			&
			\mbox{ \ if \ }
			|x|\le  \max\{1,t^{\frac 12} \} ,
			b>n
			\\
			\langle x\rangle^{-b}
			
			&
			\mbox{ \ if \ }
			|x| > \max\{1,t^{\frac 12} \}, b<n
			\\
			\langle x\rangle^{-n}
			+
			t^{-\frac n2}	e^{-\frac{|x|^2}{16t} } \ln (|x| +2 )
			&
			\mbox{ \ if \ }
			|x| > \max\{1,t^{\frac 12} \}, b=n
			\\
			\langle x\rangle^{-b}
			+  t^{-\frac n2}
			e^{-\frac{|x|^2}{16t} }
			&
			\mbox{ \ if \ }
			|x| > \max\{1,t^{\frac 12} \}, b>n .
		\end{cases}
	\end{aligned}
\end{equation*}
This completes the proof of Lemma \ref{cauchy-est-lem}.
\end{proof}

\medskip

\section{Proof of Proposition \ref{psi-prop}: solving the outer problem }\label{sec-outer-proof}

\begin{proof}
	It suffices to find a fixed point for $\psi = \mathcal{T}^{out}_{4}[\mathcal{G} [\psi,\phi,\mu_1,\xi] ]$. Set
	\begin{equation*}
		w_{o}(x,t)=
		\ln t (t (\ln t)^{2})^{5\delta-\kappa }
		R^{- a}  \left(\1_{\{|x|\le t^{\frac 12} \}} + t |x|^{-2} \1_{\{|x| > t^{\frac 12} \}}\right) ,
	\end{equation*}
\begin{equation*}
\| g\|_{o} = \sup\limits_{(x,t)\in \RR^4\times (t_0,\infty) } w^{-1}_{o}(x,t) |g(x,t)|, \quad
	B_{o} = \{ g(x,t) : \ \| g\|_{o} \le D_o  \},
\end{equation*}
	where $D_o\ge 1$ will be determined later.
	For any $\psi_{1} \in B_{o}$,
	let us estimate $\mathcal{G} [\psi_1,\phi,\mu_1,\xi]$ term by term.
	In this proof, we will apply Lemma \ref{annular-forward} and Lemma \ref{far-forward} multiple times to estimate convolution $\TT_{4}^{out}$ and will not state them repetitively.
	
		By the definitions of the norms \eqref{norm-phi}, \eqref{norm-mu}, \eqref{norm-xi},  one has
	\begin{equation*}
		|\phi(y, t)| + \langle y\rangle |\nabla \phi(y, t)|
		\lesssim
		(t (\ln t)^{2})^{5\delta-\kappa }
		\langle y\rangle^{-a} \| \phi\|_{i,\kappa-5\delta,a}
		,
	\end{equation*}
\begin{equation*}
|\mu_{1}| + t|\mu_{1t}| \lesssim   t \ln t	(t (\ln t)^{2})^{5\delta-\kappa }
R^{- a} \| \mu_{1} \|_{*1} ,\quad
|\xi | + t|\xi_{t}| \lesssim
t (\ln t)^2 (t (\ln t)^{2})^{5\delta-\kappa }
R^{- a} \| \xi \|_{*2} .
\end{equation*}
Then
	\begin{equation*}
		\begin{aligned}
			\left| \Delta_x \eta_R \mu^{-1}\phi(\frac {x-\xi}{\mu},t) \right|
			\lesssim \ &
			(\mu_{0} R)^{-2} \1_{\{ \mu_{0}R \le |x-\xi| \le 2 \mu_{0}R \}}
			\mu^{-1}
			(t (\ln t)^{2})^{5\delta-\kappa }
			\langle y\rangle^{-a} \| \phi\|_{i,\kappa-5\delta,a}
			\\
			\sim \ &
			\Lambda_1
			(\mu_{0} R)^{-2}
			\ln t
			(t (\ln t)^{2})^{5\delta-\kappa }
			R^{-a}   \1_{\{ \mu_{0}R \le |x-\xi| \le 2 \mu_{0}R \}}
			,
		\end{aligned}
	\end{equation*}
	\begin{equation*}
		\begin{aligned}
			\left|2 \nabla_x \eta_R \cdot \mu^{-2} \nabla_y \phi(\frac {x-\xi}{\mu},t) \right|
			\lesssim \ &
			(\mu_{0} R)^{-1} \1_{\{ \mu_{0}R \le |x-\xi| \le 2 \mu_{0}R \}}
			\mu^{-2}
			(t (\ln t)^{2})^{5\delta-\kappa }
			\langle y\rangle^{-1-a} \| \phi\|_{i,\kappa-5\delta,a}
			\\
			\lesssim \ &
			\Lambda_1
			(\mu_{0} R)^{-2}
			\ln t
			(t (\ln t)^{2})^{5\delta-\kappa }
			R^{-a}  \1_{\{ \mu_{0}R \le |x-\xi| \le 2 \mu_{0}R \}}
			,
		\end{aligned}
	\end{equation*}
	\begin{equation*}
		\begin{aligned}
			\left| \pp_t \eta_R \mu^{-1}\phi(\frac {x-\xi}{\mu},t) \right|
			= \ &
			\left|
			\nabla \eta(\frac{x-\xi}{\mu_0 R}) \cdot
			(\frac{-\xi_t}{\mu_0 R} - \frac{x-\xi}{\mu_0 R} \frac{(\mu_0 R)_{t}}{\mu_0 R} ) \mu^{-1}\phi(\frac {x-\xi}{\mu},t) \right|
			\\
			\lesssim \ & \Lambda_1^2
			(\mu_{0} R)^{-2}
			\ln t
			(t (\ln t)^{2})^{5\delta-\kappa }
			R^{-a}  \1_{\{ \mu_{0}R \le |x-\xi| \le 2 \mu_{0}R \}},
		\end{aligned}
	\end{equation*}
where we have used $\gamma<\frac 12$ and $5\delta-\kappa<-1$ in the last inequality.	Then one has
	\begin{equation*}
		\begin{aligned}
			&
			\mathcal{T}^{out}_{4}\left[  (
			\mu_{0} R)^{-2}
			\ln t (t (\ln t)^{2})^{5\delta-\kappa }
			R^{-a}   \1_{\{ \mu_{0}R \le |x-\xi| \le 2 \mu_{0}R \}}
			\right]
			\lesssim
			\mathcal{T}^{out}_{4}\left[  (\mu_{0} R)^{-2}
			\ln t (t (\ln t)^{2})^{5\delta-\kappa }
			R^{-a}   \1_{\{ \mu_{0}R/2 \le |x| \le 4 \mu_{0}R \}}
			\right]
			\\
			\lesssim \ &
	 t^{-2} e^{-\frac{|x|^2}{16 t }} \int_{\frac{t_0}{2}}^{\frac t2}
			(\ln s)^{-1} (s (\ln s)^2)^{5\delta-\kappa} R^{2-a}(s) ds
			+
			\begin{cases}
				\ln t (t (\ln t)^{2})^{5\delta-\kappa } R^{-a}
				&
				\mbox{ \ if \ }
				|x| \le \mu_{0} R
				\\
				\ln t (t (\ln t)^{2})^{5\delta-\kappa } R^{-a}
				(\mu_{0}R)^2
				|x|^{-2} e^{-\frac{|x|^2}{16 t} }
				&
				\mbox{ \ if \ }
				|x| > \mu_{0} R
			\end{cases}
			\\
			\lesssim \ &
		 w_{o}
		\end{aligned}
	\end{equation*}
provided $5\delta -\kappa -a\gamma >-2$. Also,
\begin{equation*}
	\begin{aligned}
		\left|\eta_R \mu^{-2} \xi_{t} \cdot \nabla_y \phi(\frac{x-\xi}{\mu},t) \right|
		\lesssim \ & \Lambda_1^2
		\1_{\{ |x|\le 3\mu_0R\}}
		(\ln t)^2
		(\ln t)^2 (t (\ln t)^{2})^{5\delta-\kappa }
		R^{- a}
		(t (\ln t)^{2})^{5\delta-\kappa }
		\langle y\rangle^{-1-a}
		\\
		\sim \ & \Lambda_1^2
		(\ln t)^4
		(t (\ln t)^{2})^{10\delta-2\kappa }
		R^{- a}
		\left(\1_{\{ |x|\le \mu_0 \}} + (\ln t)^{-1-a} |x|^{-1-a} \1_{\{ \mu_0 < |x|\le 3\mu_0R\}} \right),
	\end{aligned}
\end{equation*}
and
\begin{equation*}
	\begin{aligned}
		&
		\TT_{4}^{out} \left[	(\ln t)^4
		(t (\ln t)^{2})^{10\delta-2\kappa }
		R^{- a}
		(\ln t)^{-1-a} |x|^{-1-a} \1_{\{ \mu_0 < |x|\le 3\mu_0R\}} \right]
		\\
		\le \ &
		\TT_{4}^{out} \left[	(\ln t)^3
		(t (\ln t)^{2})^{10\delta-2\kappa }
		R^{- a}
		|x|^{-1} \1_{\{ \mu_0 < |x|\le 3\mu_0R\}} \right]
		\\
		\lesssim \ &
		t^{-2} e^{-\frac{|x|^2}{16  t}}\int_{\frac{t_0}{2}}^{\frac{t}{2}}
		(\ln s)^3
		(s (\ln s)^{2})^{10\delta-2\kappa }
		R^{- a}(s)
		(\mu_0R)^{3}(s) ds
		\\
		&
		+
		\begin{cases}
			\mu_0 R (\ln t)^3
			(t (\ln t)^{2})^{10\delta-2\kappa }
			R^{- a}
			&
			\mbox{ \ if \ } |x|\le \mu_0R
			\\
			(\ln t)^3
			(t (\ln t)^{2})^{10\delta-2\kappa }
			R^{- a}
			|x|^{-2} e^{-\frac{|x|^2}{16 t }} (\mu_0R)^3
			&
			\mbox{ \ if \ } |x| > \mu_0R
		\end{cases}
		\lesssim  t_{0}^{-\epsilon} w_{o} ,
	\end{aligned}
\end{equation*}
\begin{equation*}
	\TT_{4}^{out} [ (\ln t)^4
	(t (\ln t)^{2})^{10\delta-2\kappa }
	R^{- a}
	\1_{\{ |x|\le \mu_0 \}}  ]
	\lesssim  t_{0}^{-\epsilon} w_{o}
\end{equation*}
provided $5\delta -\kappa-a\gamma >-2$  and $\epsilon>0$ is sufficiently small.
	
Using \eqref{S2-est}, one has
\begin{equation*}
	\begin{aligned}
		&
		\left|(1-\eta_R) S\Big[u_1 + \varphi[\mu] +  \bar{\mu}_0^{-1} \Phi_0(\frac{x-\xi}{\bar{\mu}_0} , t)
		\eta(\frac{4(x-\xi) }{\sqrt t}) \Big]  \right|
		\\
		\lesssim \ &
\Bigg[ 	t^{-2}(\ln t)^{-1}   |x|^{-2}  +  t^{-1}(\ln t)^{-2} |\mu_1| |x|^{-4}
\\
&  +    (\ln t)^{-2}    \left(|\tilde{g}[\bar{\mu}_0,\mu_1] | +
|\bar{\mu}_{0t}| \ln t
\sup\limits_{t_1\in[t/2,t]} (\frac{|\mu_1(t_1)|}{\bar{\mu}_{0}(t) } + \frac{|\mu_{1t}(t_1)|}{|\bar{\mu}_{0t}(t)|}  )
 \right) |x|^{-4}
+
|\xi_t|
(\ln t)^{-1} |x|^{-3}
\Bigg]
\1_{\{  \frac{\mu_0R}{2} \le |x|\le 9t^{\frac 12} \}}
\\
&  +
|\xi_t| t^{\frac 32} (\ln t)^{-1} |x|^{-6}
\1_{\{ |x| > 2t^{\frac 12}\}}  + ( t^{2} (\ln t)^{-1}
|x|^{-6}
)^3 \1_{\{ |x| > 2t^{\frac 12}  \}}
\\
\lesssim \ &
\left( 	t^{-2}(\ln t)^{-1}   |x|^{-2}
+
\Lambda_1^2 \ln t (t (\ln t)^{2})^{5\delta-\kappa }
R^{- a}
|x|^{-3}
\right)
\1_{\{  \frac{\mu_0R}{2} \le |x|\le 9t^{\frac 12} \}}
\\
&  +
\Lambda_1 t^{\frac 32} \ln t (t (\ln t)^{2})^{5\delta-\kappa }
R^{- a}  |x|^{-6}
\1_{\{ |x| > 2t^{\frac 12}\}}
+ ( t^{2} (\ln t)^{-1}
|x|^{-6}
)^3 \1_{\{ |x| > 2t^{\frac 12}  \}}
	\end{aligned}
\end{equation*}
since
\begin{equation*}
|\bar{\mu}_{0t}| \ln t
\sup\limits_{t_1\in[t/2,t]} \left(\frac{|\mu_1(t_1)|}{\bar{\mu}_{0}(t) } + \frac{|\mu_{1t}(t_1)|}{|\bar{\mu}_{0t}(t)|}  \right)
\lesssim
\Lambda_{1}
(\ln t)^2 (t(\ln t)^2)^{5\delta-\kappa} R^{-a} ,
\end{equation*}
\begin{equation}\label{tilde-g-mu1-est}
	\begin{aligned}
		\tilde{g}[\bar{\mu}_0, \mu_{1}]
		\lesssim \ &
		\Lambda_{1}^2
		t^{-2}
		\int_{t_0/2}^{t }
		(
		\ln s	(s (\ln s)^{2})^{5\delta-\kappa }
		R^{- a}(s) (\ln s)^{-2}   +
		s \ln s (s(\ln s)^2)^{5\delta-\kappa} R^{-a}(s)
		)
		d s
		\\
		&
		+
		\Lambda_{1}^2 (t\ln t)^{-1} [t (\ln t)^3 (t(\ln t)^2)^{5\delta-\kappa} R^{-a}]^2
		\lesssim
		\Lambda_{1}^2
		\ln t (t(\ln t)^2)^{5\delta-\kappa} R^{-a}
	\end{aligned}
\end{equation}
when $5\delta-\kappa -a\gamma>-2$.

Then we estimate
\begin{equation*}
\begin{aligned}
&
\TT_4^{out}\left[t^{-2} (\ln t)^{-1} |x|^{-2}
\right] \1_{\{ \frac{\mu_0 R }{2} \le |x|\le 9t^{\frac 12} \}}
\\
\lesssim \ &
t^{-2} \ln \ln t e^{-\frac{|x|^2}{16 t}}
+
\begin{cases}
t^{-2}
&
\mbox{ \ if \ } |x|\le \mu_0 R
\\
t^{-2} (\ln t)^{-1} (\ln (|x|^{-1}t^{\frac 12}) + 1 )
&
\mbox{ \ if \ }  \mu_0 R < |x|\le t^{\frac 12}
\\
(t \ln t)^{-1} |x|^{-2} e^{-\frac{|x|^2}{16 t}}
&
\mbox{ \ if \ }   |x| > t^{\frac 12}
\end{cases}
\lesssim
t_0^{-\epsilon_0 }w_{o}
\end{aligned}
\end{equation*}
since $5\delta -\kappa -a\gamma>-2$.
\begin{equation*}
	\begin{aligned}
		&
		\TT_{4}^{out} \left[ \ln t (t (\ln t)^{2})^{5\delta-\kappa }
		R^{- a}
		|x|^{-3}
		\1_{\{  \frac{\mu_0R}{2} \le |x|\le 9t^{\frac 12} \}}\right]
		\\
		\lesssim \ &
		t^{-2} \int_{\frac{t_0}{2}}^{\frac t2}
		\ln s (s(\ln s)^2)^{5\delta-\kappa} R^{-a}(s) s^{\frac 12} ds
		+
		\begin{cases}
			\ln t (t(\ln t)^2)^{5\delta-\kappa} R^{-a} (\mu_0R)^{-1}
			&
			\mbox{ \ if \ } |x|\le \mu_0R
			\\
			\ln t (t(\ln t)^2)^{5\delta-\kappa} R^{-a} |x|^{-1}
			&
			\mbox{ \ if \ }  \mu_0R < |x|\le t^{\frac 12}
			\\
			\ln t (t(\ln t)^2)^{5\delta-\kappa} R^{-a}t^{\frac 12} |x|^{-2} e^{-\frac{|x|^2}{16  t}}
			&
			\mbox{ \ if \ }   |x| > t^{\frac 12}
		\end{cases}
		\\
		\lesssim \ &
		t_0^{-\epsilon} w_{o}
	\end{aligned}
\end{equation*}
since $5\delta -\kappa -a\gamma>-2$.
\begin{equation*}
	\begin{aligned}
		&
		\TT_{4}^{out}\left[t^{\frac 32} \ln t (t (\ln t)^{2})^{5\delta-\kappa }
		R^{- a}  |x|^{-6}
		\1_{\{ |x| > 2t^{\frac 12}\}} \right]
		\\
		\lesssim \ &
		\begin{cases}
			t^{-2} \int_{\frac{t_0}{2}}^{t}
			s^{\frac 12} \ln s (s(\ln s)^{2})^{5\delta-\kappa }
			R^{- a}(s) ds
			&
			\mbox{ \ if \ } |x|\le t^{\frac 12}
			\\
			|x|^{-6} \int_{\frac{t_0}{2}}^{t}
			s^{\frac 32} \ln s (s (\ln s)^{2})^{5\delta-\kappa }
			R^{- a}(s) ds +
			t^{-2} e^{-\frac{|x|^2}{16  t }} \int_{\frac{t_0}{2}}^{\frac t2}
			s^{\frac 12} \ln s (s(\ln s)^{2})^{5\delta-\kappa }
			R^{- a}(s) ds
			&
			\mbox{ \ if \ } |x| > t^{\frac 12}
		\end{cases}
		\\
		\lesssim \ &
		t_0^{-\epsilon} w_{o} .
	\end{aligned}
\end{equation*}
\begin{equation*}
\TT_{4}^{out} \left[( t^{2} (\ln t)^{-1}
|x|^{-6}
)^3 \1_{\{ |x| > 2t^{\frac 12}  \}}\right]
\lesssim
t^{-2} \1_{\{ |x|\le t^{\frac 12} \}}
+
\left(t^{7} (\ln t)^{-3} |x|^{-18}+ t^{-2} e^{-\frac{|x|^2}{16 t }}\right)
\1_{\{ |x| > t^{\frac 12} \}}
\lesssim
t_0^{-\epsilon} w_o
\end{equation*}
when $5\delta -\kappa -a\gamma>-2$.

	\begin{align*}
&
\left(u_1 + \varphi[\mu] +  \bar{\mu}_0^{-1} \Phi_0(\frac{x-\xi}{\bar{\mu}_0} , t) \eta(\frac{4(x-\xi) }{\sqrt t}) + \psi_{1} + \eta_R  \mu^{-1} \phi(\frac{x-\xi}{\mu},t) \right)^3
\\
&
-
\left(u_1 + \varphi[\mu] +  \bar{\mu}_0^{-1} \Phi_0(\frac{x-\xi}{\bar{\mu}_0} , t) \eta(\frac{4(x-\xi) }{\sqrt t}) \right)^3
- 3 \left(\mu^{-1} w(\frac{x-\xi}{\mu})\right)^2 \left( \psi_{1} + \eta_R  \mu^{-1} \phi(\frac{x-\xi}{\mu},t)  \right)
\\
& -
\Bigg[ 3 \left(u_1 + \varphi[\mu] -\mu^{-1} w(\frac{x-\xi}{\mu}) \right)
\left(u_1 + \varphi[\mu] +\mu^{-1} w(\frac{x-\xi}{\mu}) \right)
\\
& +
6 (u_1 + \varphi[\mu] )   \bar{\mu}_0^{-1} \Phi_0(\frac{x-\xi}{\bar{\mu}_0} , t) \eta(\frac{4(x-\xi)}{\sqrt t})
\Bigg] \eta_R  \mu^{-1} \phi(\frac{x-\xi}{\mu},t)
		\\
= \ &
3 \left(u_1 + \varphi[\mu] -\mu^{-1} w(\frac{x-\xi}{\mu}) \right)
\left(u_1 + \varphi[\mu] +\mu^{-1} w(\frac{x-\xi}{\mu}) \right) \psi_{1}
\\
& +
6 (u_1 + \varphi[\mu] )   \bar{\mu}_0^{-1} \Phi_0(\frac{x-\xi}{\bar{\mu}_0} , t) \eta(\frac{4(x-\xi)}{\sqrt t})   \psi_{1}
\\
& +
3  \left(  \bar{\mu}_0^{-1} \Phi_0(\frac{x-\xi}{\bar{\mu}_0} , t) \eta(\frac{4(x-\xi)}{\sqrt t})  \right)^2 \left(\psi_{1} + \eta_R  \mu^{-1} \phi(\frac{x-\xi}{\mu},t)\right)
\\
&
+
3 \left(u_1 + \varphi[\mu] +  \bar{\mu}_0^{-1} \Phi_0(\frac{x-\xi}{\bar{\mu}_0} , t) \eta(\frac{4(x-\xi)}{\sqrt t})  \right)\left(\psi_{1} + \eta_R  \mu^{-1} \phi(\frac{x-\xi}{\mu},t)\right)^2
+
\left(\psi_{1} + \eta_R  \mu^{-1} \phi(\frac{x-\xi}{\mu},t)\right)^3
\\
\lesssim  \ &
 \Big( (t\ln t)^{-1} \1_{\{ |\bar{x}| \le 2t^{\frac 12} \}}
 +
  t^{2} (\ln t)^{-1}
 |\bar{x}|^{-6}
  \1_{\{ |\bar{x}| > 2t^{\frac 12}  \}}  + \ln t \langle y\rangle^{-2}  \1_{\{ |\bar{x}|\ge t^{\frac 12} \}}\Big)
 \\
 &\times
\Big( (t\ln t)^{-1} \1_{\{ |\bar{x}| \le 2t^{\frac 12} \}}
+
 t^{2} (\ln t)^{-1}
|\bar{x}|^{-6}
 \1_{\{ |\bar{x}| > 2t^{\frac 12}  \}} +  \ln t \langle y\rangle^{-2}  \Big) |\psi_{1}|
\\
& +
( \ln t \langle y\rangle^{-2} \1_{\{ |\bar{x}|\le 2t^{\frac 12} \}} +  (t\ln t)^{-1} \1_{\{ |\bar{x}| \le 2t^{\frac 12} \}}
+
 t^{2} (\ln t)^{-1}
|\bar{x}|^{-6}
 \1_{\{ |\bar{x}| > 2t^{\frac 12}  \}}  )
    (t \ln t)^{-1} \langle \bar{y} \rangle^{-2}
\ln (2+ |\bar{y}|)
\1_{\{ |\bar{x}|\le 8t^{\frac 12} \}}  |\psi_{1}|
\\
& +
(t \ln t)^{-2} \langle \bar{y} \rangle^{-4}
\ln^2 (2+ |\bar{y}|)
\1_{\{ |\bar{x}|\le 8t^{\frac 12} \}} \Big(|\psi_{1}| + |\eta_R  \mu^{-1} \phi(\frac{x-\xi}{\mu},t) |\Big)
\\
&
+
 \Big( \ln t \langle y\rangle^{-2} \1_{\{ |\bar{x}| \le 2t^{\frac 12} \}} +
 (t\ln t)^{-1} \1_{\{ |\bar{x}| \le 2t^{\frac 12} \}}
 +
 O( t^{2} (\ln t)^{-1}
 |\bar{x}|^{-6}
 ) \1_{\{ |\bar{x}| > 2t^{\frac 12}  \}}
 \\
&
  + (t \ln t)^{-1} \langle \bar{y} \rangle^{-2}
 \ln (2+ |\bar{y}|) \1_{\{ |\bar{x}|\le 8t^{\frac 12} \}}
    \Big)\Big|\psi_{1} + \eta_R  \mu^{-1} \phi(\frac{x-\xi}{\mu},t)\Big|^2
+
\Big|\psi_{1} + \eta_R  \mu^{-1} \phi(\frac{x-\xi}{\mu},t)\Big|^3
\\
\lesssim  \ &
\left( t^{-1} \langle y\rangle^{-2}   \1_{\{ |x| \le t^{\frac 12} \}}
  +  (\ln t)^{-2} |x|^{-4}    \1_{\{ |x| > t^{\frac 12} \}}\right)
   |\psi_{1}|
\\
& +
(t \ln t)^{-2} \langle \bar{y} \rangle^{-4}
\ln^2 (2+ |\bar{y}|)
\1_{\{ |\bar{x}|\le 8t^{\frac 12} \}}  \Big|\eta_R  \mu^{-1} \phi(\frac{x-\xi}{\mu},t) \Big|
\\
&
+
\left( \ln t \langle y\rangle^{-2} \1_{\{ |x| \le t^{\frac 12} \}}
+
 t^{2} (\ln t)^{-1}
|x|^{-6}
 \1_{\{ |x| > t^{\frac 12}  \}}
\right)\Big|\psi_{1} + \eta_R  \mu^{-1} \phi(\frac{x-\xi}{\mu},t)\Big|^2
+
\Big|\psi_{1} + \eta_R  \mu^{-1} \phi(\frac{x-\xi}{\mu},t)\Big|^3,
	\end{align*}
where we have used Corollary \ref{varphi-coro} and \eqref{Phi0-est}.

Consider the terms involving $\phi$:
	\begin{equation*}
	\begin{aligned}
		&
		\TT_4^{out}\left[(t \ln t)^{-2} \langle \bar{y} \rangle^{-4}
		\ln^2 (2+ |\bar{y}|)
		\1_{\{ |\bar{x}|\le 8t^{\frac 12} \}}  \Big|\eta_R  \mu^{-1} \phi(\frac{x-\xi}{\mu},t) \Big| \right]
		\\
		\lesssim \ &
		\Lambda_1
		\TT_4^{out}[t^{-2} \langle \bar{y} \rangle^{-4}
		\1_{\{ |\bar{x}|\le 2\mu_0 R \}}  \ln t (t (\ln t)^{2})^{5\delta-\kappa }
		\langle y\rangle^{-a}   ]
		\\
		\sim \ &
		\Lambda_1
		\TT_4^{out}[t^{-2} \ln t (t (\ln t)^{2})^{5\delta-\kappa }
		\langle y\rangle^{-4-a}  \1_{\{ |\bar{x}|\le 2\mu_0 R \}} ]
		\lesssim	\Lambda_1
		t^{-2} e^{-\frac{|x|^2}{16 t }}
	\end{aligned}
\end{equation*}
since
\begin{equation*}
	\begin{aligned}
		&
		\TT_4^{out}[t^{-2} \ln t (t (\ln t)^{2})^{5\delta-\kappa }
		\langle y\rangle^{-4-a}  \1_{\{\mu_0 \le  |\bar{x}|\le 2\mu_0 R \}} ]
		\\
		\lesssim \ &
		\TT_4^{out}[t^{-2} (\ln t)^{-3-a} (t (\ln t)^{2})^{5\delta-\kappa }
		|x|^{-4-a}  \1_{\{\frac{\mu_0}{2} \le  |x|\le 4\mu_0 R \}} ]
		\\
		\lesssim \ &
		t^{-2} e^{-\frac{|x|^2}{16 t}} +
		\begin{cases}
			t^{-2} (\ln t)^{-1} (t (\ln t)^{2})^{5\delta-\kappa }
			&
			\mbox{ \ if \ } |x|\le \mu_0
			\\
			t^{-2} (\ln t)^{-3} (t (\ln t)^{2})^{5\delta-\kappa }
			|x|^{-2} e^{-\frac{|x|^2}{16 t }}
			&
			\mbox{ \ if \ }   |x| >\mu_0
		\end{cases}
		\lesssim
		t^{-2} e^{-\frac{|x|^2}{16 t}}  \lesssim t_0^{-\epsilon} w_{o} ,
	\end{aligned}
\end{equation*}
\begin{equation*}
	\begin{aligned}
		&
		\TT_4^{out}[t^{-2} \ln t (t (\ln t)^{2})^{5\delta-\kappa }
		\langle y\rangle^{-4-a}  \1_{\{ |\bar{x}| < \mu_0  \}} ]
		\lesssim
		\TT_4^{out}[ t^{-2} \ln t (t (\ln t)^{2})^{5\delta-\kappa }
		\1_{\{ |x| <  2 \mu_0  \}} ]
		\\
		\lesssim \ &
		t^{-2} e^{-\frac{|x|^2}{16 t }}
		+
		\begin{cases}
			t^{-2} (\ln t)^{-1} (t (\ln t)^{2})^{5\delta-\kappa }
			&
			\mbox{ \ if \ } |x|\le \mu_0
			\\
			t^{-2} (\ln t)^{-3} (t (\ln t)^{2})^{5\delta-\kappa }
			|x|^{-2} e^{-\frac{|x|^2}{16 t }}
			&
			\mbox{ \ if \ } |x| > \mu_0
		\end{cases}
		\lesssim
		t^{-2} e^{-\frac{|x|^2}{16 t }}
		.
	\end{aligned}
\end{equation*}

Next, we have
\begin{equation*}
	\begin{aligned}
		&
		\TT_4^{out} \left[\left( \ln t \langle y\rangle^{-2} \1_{\{ |x| \le t^{\frac 12} \}}
		+
		t^{2} (\ln t)^{-1}
		|x|^{-6}
		\1_{\{ |x| > t^{\frac 12}  \}}
		\right)\Big| \eta_R  \mu^{-1} \phi(\frac{x-\xi}{\mu},t)\Big|^2 \right]
		\\
		\lesssim \ & \Lambda_1^2
		\TT_4^{out} [( \ln t \1_{\{ |x| \le \mu_0 \}}
		+
		(\ln t)^{-1} |x|^{-2} \1_{\{ \mu_0 <|x| \le 4\mu_0 R \}}
		)
		(\ln t)^2 (t (\ln t)^{2})^{10\delta-2\kappa }
		\langle y\rangle^{-2a}
		]
		\\
		\sim \ & \Lambda_1^2
		\TT_4^{out} \left[ (\ln t)^3 (t (\ln t)^{2})^{10\delta-2\kappa }  \1_{\{ |x| \le \mu_0 \}}
		+
		(t (\ln t)^{2})^{10\delta-2\kappa }
		(\ln t)^{1-2a} |x|^{-2-2a} \1_{\{ \mu_0 <|x| \le 4\mu_0 R \}}
		\right] ,
	\end{aligned}
\end{equation*}
and
\begin{equation*}
	\begin{aligned}
		&
		\TT_4^{out} \left[
		(t (\ln t)^{2})^{10\delta-2\kappa }
		(\ln t)^{1-2a} |x|^{-2-2a} \1_{\{ \mu_0 <|x| \le 4\mu_0 R \}}
		\right]
\lesssim
\TT_4^{out} \left[
(t (\ln t)^{2})^{10\delta-2\kappa }
\ln t |x|^{-2} \1_{\{ \mu_0 <|x| \le 4\mu_0 R \}}
\right]
\\
		\lesssim \ &
t^{-2} e^{-\frac{|x|^2}{16 t }}\int_{\frac{t_0}{2}}^{\frac t2}
(s (\ln s)^{2})^{10\delta-2\kappa }
\ln s (\mu_0R)^2(s) ds
\\
&
+
\begin{cases}
(t (\ln t)^{2})^{10\delta-2\kappa }
(\ln t)^2
 &
\mbox{ \ if \ } |x|\le \mu_0
\\
(t (\ln t)^{2})^{10\delta-2\kappa }
\ln t (\ln(\frac{\mu_0R}{|x|}) + 1)
 &
\mbox{ \ if \ }  \mu_0 < |x|\le \mu_0R
\\
(t (\ln t)^{2})^{10\delta-2\kappa }
\ln t (\mu_0R)^2 |x|^{-2} e^{-\frac{|x|^2}{16 t }}
 &
\mbox{ \ if \ } |x| > \mu_0 R
\end{cases}
\lesssim
t_0^{-\epsilon} w_{o}
	\end{aligned}
\end{equation*}
when $5\delta-\kappa < -1$ and $5\delta -\kappa-a \gamma >-2$. Also,
\begin{equation*}
	\TT_4^{out} \left[ (\ln t)^3 (t (\ln t)^{2})^{10\delta-2\kappa }  \1_{\{ |x| \le \mu_0 \}}
	\right]  \lesssim
	t_0^{-\epsilon} w_{o} .
\end{equation*}

When $5\delta-\kappa+(2-a)\gamma<0$, one has
\begin{equation*}	\left| \eta_R  \mu^{-1} \phi(\frac{x-\xi}{\mu},t)\right|^3
	\lesssim \Lambda_1 \ln t (t (\ln t)^{2})^{5\delta-\kappa }
	\langle y\rangle^{-a} \left| \eta_R  \mu^{-1} \phi(\frac{x-\xi}{\mu},t)\right|^2
\lesssim
 \ln t \langle y\rangle^{-2}  \left| \eta_R  \mu^{-1} \phi(\frac{x-\xi}{\mu},t)\right|^2 .
\end{equation*}

Let us now estimate terms involving $\psi_1$.
	\begin{align*}
		&
		\left| \mu^{-2} w^2(\frac{x-\xi}{\mu})  \psi_{1} (1-\eta_R)
		\right|
		\lesssim
		(\ln t)^{-2} |x|^{-4} |\psi_{1} | \1_{\{ |x|\ge \frac{\mu_0R}{2} \}}
		\\
		\lesssim \ & D_{o}
		(\ln t)^{-2} |x|^{-4}
		\ln t (t (\ln t)^{2})^{5\delta-\kappa }
		R^{- a}  \left(\1_{\{ \frac{\mu_0R}{2} \le |x|\le t^{\frac 12} \}} + t |x|^{-2} \1_{\{|x| > t^{\frac 12} \}}\right)
		\\
		= \ &
		D_{o}  (\ln t)^{-1} (t (\ln t)^{2})^{5\delta-\kappa }
		R^{- a} |x|^{-4} \1_{\{ \frac{\mu_0R}{2} \le |x|\le t^{\frac 12} \}}
		+
		D_{o} t (\ln t)^{-1} (t (\ln t)^{2})^{5\delta-\kappa }
		R^{- a}   |x|^{-6}
		\1_{\{|x| > t^{\frac 12} \}}  .
	\end{align*}
For the first term, we have
\begin{equation*}
	\begin{aligned}
		& \TT_{4}^{out} \left[
		(\ln t)^{-1} (t (\ln t)^{2})^{5\delta-\kappa }
		R^{- a}
		|x|^{-4} \1_{\{ \frac{\mu_0R}{2} \le |x|\le t^{\frac 12} \}}
		\right]
		\\
		\lesssim  \ &
		t^{-2} e^{-\frac{|x|^2}{16 t }}
		+
		\begin{cases}
			(\ln t)^{-1} (t (\ln t)^{2})^{5\delta-\kappa }
			R^{- a} (\mu_0 R)^{-2}
			&
			\mbox{ \ if \ } |x|\le \mu_0 R
			\\
			(\ln t)^{-1} (t (\ln t)^{2})^{5\delta-\kappa }
			R^{- a}  |x|^{-2}
			(\ln (\frac{|x|}{\mu_0 R}) +1)
			&
			\mbox{ \ if \ }  \mu_0 R < |x|\le t^{\frac 12}
			\\
			(t (\ln t)^{2})^{5\delta-\kappa }
			R^{- a}
			|x|^{-2} e^{-\frac{|x|^2}{16 t }}
			&
			\mbox{ \ if \ }   |x| > t^{\frac 12}
		\end{cases}
		\lesssim
		t_{0}^{-\epsilon} w_{o}
	\end{aligned}
\end{equation*}
when $5\delta -\kappa-a \gamma >-2$. For the second term, one has
\begin{equation*}
	\begin{aligned}
		\TT_{4}^{out} \left[ t (\ln t)^{-1} (t (\ln t)^{2})^{5\delta-\kappa }
		R^{- a}   |x|^{-6}
		\1_{\{|x| > t^{\frac 12} \}}\right]
		\lesssim
		t^{-2} \1_{\{ |x|\le t^{\frac 12} \}}
		+
		\left(|x|^{-6} + t^{-2 } e^{-\frac{|x|^2}{16 t }}\right)
		\1_{\{ |x| > t^{\frac 12} \}}
		\lesssim
		t_{0}^{-\epsilon} w_{o} .
	\end{aligned}
\end{equation*}
Thus
\begin{equation*}
	\left|\TT_{4}^{out}\left[  \mu^{-2} w^2(\frac{x-\xi}{\mu})  \psi_{1} (1-\eta_R)
	\right] \right| \lesssim
	t_{0}^{-\epsilon} w_{o} .
\end{equation*}

Notice
\begin{equation*}
	\left( \ln t \langle y\rangle^{-2} \1_{\{ |x| \le t^{\frac 12} \}}
	+
	t^{2} (\ln t)^{-1}
	|x|^{-6}
	\1_{\{ |x| > t^{\frac 12}  \}}
	\right)|\psi_{1}|^2  + |\psi_{1}|^3
	\lesssim
	D_{o}^2
	\left( t^{-1} \langle y\rangle^{-2}   \1_{\{ |x| \le t^{\frac 12} \}}
	+  (\ln t)^{-2} |x|^{-4}    \1_{\{ |x| > t^{\frac 12} \}}\right)
	|\psi_{1}|
\end{equation*}
when $5\delta-\kappa-a\gamma <-1$.
And
\begin{equation*}
	(\ln t)^{-2} |x|^{-4}    \1_{\{ |x| > t^{\frac 12} \}}
	|\psi_{1}| \lesssim
	(\ln t)^{-2} |x|^{-4} |\psi_{1} | \1_{\{ |x|\ge \frac{\mu_0R}{2} \}} ,
\end{equation*}
where the last term has been estimated above. So we only need to estimate the following term
\begin{equation*}
	\begin{aligned}
		&
		t^{-1} \langle y\rangle^{-2}   \1_{\{ |x| \le t^{\frac 12} \}}
		|\psi_{1}|
		\lesssim  D_{o}
		\ln t (t (\ln t)^{2})^{5\delta-\kappa }
		R^{- a}  t^{-1} \langle y\rangle^{-2}   \1_{\{ |x| \le t^{\frac 12} \}}
		\\
		\lesssim \ & D_{o}
		\ln t (t (\ln t)^{2})^{5\delta-\kappa }
		R^{- a}  t^{-1}
		\left(\1_{\{ |x|\le \mu_0 \}} + (\ln t)^{-2} |x|^{-2}  \1_{\{  \mu_0  < |x| \le t^{\frac 12} \}}
		\right) ,
	\end{aligned}
\end{equation*}
and
\begin{equation*}
	\begin{aligned}
		&
		\TT_{4}^{out} \left[ (\ln t)^{-1} (t (\ln t)^{2})^{5\delta-\kappa }
		R^{- a}  t^{-1}
		|x|^{-2}  \1_{\{  \mu_0  < |x| \le t^{\frac 12} \}}
		\right]
		\\
		\lesssim \ &
		t^{-2} e^{-\frac{|x|^2}{16  t}}
		+
		\begin{cases}
			(t (\ln t)^{2})^{5\delta-\kappa }
			R^{- a}  t^{-1}
			&
			\mbox{ \ if \ } |x|\le \mu_0
			\\
			(\ln t)^{-1} (t (\ln t)^{2})^{5\delta-\kappa }
			R^{- a}  t^{-1}
			(\ln (|x|^{-1} t^{\frac 12}) + 1)
			&
			\mbox{ \ if \ }  \mu_0 < |x|\le t^{\frac 12}
			\\
			(\ln t)^{-1} (t (\ln t)^{2})^{5\delta-\kappa }
			R^{- a}
			|x|^{-2} e^{-\frac{|x|^2}{16  t }}
			&
			\mbox{ \ if \ }  |x| > t^{\frac 12}
		\end{cases}
		\lesssim
		t_0^{-\epsilon} w_{o} ,
	\end{aligned}
\end{equation*}
\begin{equation*}
	\TT_{4}^{out} [\ln t (t (\ln t)^{2})^{5\delta-\kappa }
	R^{- a}  t^{-1}
	\1_{\{ |x|\le \mu_0 \}} ]
	\lesssim
	t_0^{-\epsilon} w_{o} .
\end{equation*}
These imply
\begin{equation*}
	\TT_{4}^{out} \left[  t^{-1} \langle y\rangle^{-2}   \1_{\{ |x| \le t^{\frac 12} \}}
	|\psi_{1}|  \right] \lesssim
	t_0^{-\epsilon} w_{o} .
\end{equation*}

	Taking $D_{o}= D_o(\Lambda_1)$ large depending on $\Lambda_1$ and then choosing $t_0$ large enough, we have
	\begin{equation*}
		\mathcal{T}^{out}_{4}[\mathcal{G} [\psi_{1},\phi,\mu_{1},\xi] ] \in B_{o} .
	\end{equation*}

The contraction property is given by the similar method which is used in dealing with terms including $\psi_1$.
Then the unique solution $\psi$ is found in $B_{o}$ for \eqref{o-eq} by the contraction mapping theorem.
	
	From now on, we also regard $D_o(\Lambda_1)$ as a constant depending on $\Lambda_1$. Reviewing the estimates above and utilizing $\1_{\{ |x|\ge \frac{\mu_0R}{2}\}}$ to transform the spatial decay to time decay, one has
\begin{equation*}
	|\mathcal{G}[\psi,\phi,\mu_1,\xi]|
	\lesssim
	C(\Lambda_1)
	[   (
	\mu_{0} R)^{-2}
	\ln t (t (\ln t)^{2})^{5\delta-\kappa }
	R^{-a}
	+
	(\ln t)^3 (t (\ln t)^{2})^{10\delta-2\kappa }
	]
\end{equation*}
where $C(\Lambda_1)$ is a constant depending on $\Lambda_1$ which changes from line to line.

	By gradient estimate, we have
	\begin{equation*}
		|\nabla \psi | \lesssim
		C(\Lambda_1) \ln t (t (\ln t)^{2})^{5\delta-\kappa }
		R^{- a}.
	\end{equation*}
	
	Next, we will use scaling argument to deduce the H\"older estimate of $\psi(x,t)$ in time variable $t$. 	
	For $x_1\in \RR^4$, $t_1>4t_0$, set
	\begin{equation*}
		\tilde{\psi}(z,s) = \psi(x_1+ \lambda(t_1) z , t_1 + \lambda^2 (t_1) s)
	\end{equation*}
	where $0<\lambda(t_1) \le t_{1}^{\frac 12}$.
	Then
	\begin{equation*}
		\pp_{s} \tilde{\psi} = \Delta_z \tilde{\psi} + \tilde{\mathcal{G} }(z,s)
	\end{equation*}
	where $\tilde{ \mathcal{G} }(z,s)= \lambda^2(t_1) \mathcal{G}[\psi,\phi,\mu_{1},\xi] (x_1+ \lambda(t_1) z , t_1 + \lambda^2(t_1) s)$, and standard parabolic regularity theory implies
	\begin{equation*}
		\| \tilde{\psi } \|_{C^{2\alpha, \alpha } (B(0,\frac 1{16}) \times (-\frac 14,0))  }
		\le C(\alpha) \left(
		\| \tilde{\psi } \|_{L^{\infty} (B(0,\frac 1{4}) \times (-\frac 12,0))  }
		+
		\| \tilde{\mathcal
			G } \|_{L^{\infty} (B(0,\frac 1{4}) \times (-\frac 12,0))  } \right)
	\end{equation*}
	where $\alpha$ can be chosen as any constant in $(0,1)$ and $C(\alpha)$ is a constant depending on $\alpha$.
Moreover, one has
	\begin{equation*}
		\begin{aligned}
			\| \tilde{\psi } \|_{L^{\infty} (B(0,\frac 1{4}) \times (-\frac 12,0))  } \lesssim \ & C(\Lambda_1)
	 \ln t_1 (t_1 (\ln t_1)^{2})^{5\delta-\kappa }
	R^{- a}(t_1)
			,
			\\
			\| \tilde{\mathcal
				G } \|_{L^{\infty} (B(0,\frac 1{4}) \times (-\frac 12,0))  }
			\lesssim \ & C(\Lambda_1) \lambda^2 (t_1)
\left[   (
\mu_{0} R)^{-2}(t_1)
\ln t_1 (t_1 (\ln t_1)^{2})^{5\delta-\kappa }
R^{-a}(t_1)
+
(\ln t_1)^3 (t_1 (\ln t_1)^{2})^{10\delta-2\kappa }
\right]  ,
		\end{aligned}
	\end{equation*}
	and
	\begin{equation*}
		\begin{aligned}
			&
			\| \tilde{\psi } \|_{C^{2\alpha, \alpha } (B(0,\frac 1{16}) \times (-\frac 14,0))  }
			\ge
			\sup\limits_{s_1, s_2\in (-1/4,0)}
			\frac{|\psi(x_1 , t_1 + \lambda(t_1)^2 s_1) -\psi(x_1 , t_1 + \lambda(t_1)^2 s_2) |}{|s_1 -s_2|^{\alpha}}
			\\
			= \ &
			\lambda^{2\alpha}(t_1)
			\sup\limits_{s_1, s_2\in (-1/4,0)}
			\frac{|\psi(x_1 , t_1 + \lambda(t_1)^2 s_1) -\psi(x_1 , t_1 + \lambda(t_1)^2 s_2) |}{| (t_1 + \lambda(t_1)^2 s_1) - (t_1 + \lambda(t_1)^2 s_2)|^{\alpha}} \\
			= \ &
			\lambda^{2\alpha}(t_1)
			\sup\limits_{s_1, s_2\in ( t_1 -\frac{ \lambda(t_1)^2}{4} ,t_1)}
			\frac{|\psi(x_1 , s_1) -\psi(x_1 ,  s_2) |}{| s_1 -  s_2|^{\alpha}}
			.
		\end{aligned}
	\end{equation*}
	Thus
	\begin{equation*}
		\begin{aligned}
			&
			\sup\limits_{s_1, s_2\in ( t_1 -\frac{ \lambda(t_1)^2}{4} ,t_1)}
			\frac{|\psi(x_1 , s_1) -\psi(x_1 ,  s_2) |}{| s_1 -  s_2|^{\alpha}}
			\lesssim
			C(\Lambda_1,\alpha)
			\{
			\lambda^{-2\alpha}(t_1)
	 \ln t_1 (t_1 (\ln t_1)^{2})^{5\delta-\kappa }
	R^{- a}(t_1)
\\
&
			+
			\lambda^{2-2\alpha}(t_1)
[   (
\mu_{0} R)^{-2}(t_1)
\ln t_1 (t_1 (\ln t_1)^{2})^{5\delta-\kappa }
R^{-a}(t_1)
+
(\ln t_1)^3 (t_1 (\ln t_1)^{2})^{10\delta-2\kappa }
]  \} .
		\end{aligned}
	\end{equation*}
	
\end{proof}

\medskip

\section{Estimates for $\nabla_{\bar{x}} \varphi[\bar{\mu}_{0} ] $ and $\pp_{t} \varphi[\bar{\mu}_{0} ] $}
In this section, we will revisit the calculations in Section \ref{sec-cut-slowdecay} and derive the following estimates
\begin{equation}\label{varphi-bar-mu0}
		|\pp_{t} \varphi[\bar{\mu}_{0} ]  | \lesssim t^{-2} (\ln t)^{-1} ,\quad
	| \nabla_{\bar{x}}\varphi[\bar{\mu}_0 ]  |
\lesssim
\begin{cases}
	(t\ln t)^{-1}
	&
	\mbox{ \ if \ } |\bar{x}|\le \mu_0
	\\
	t^{-1} (\ln t)^{-2} |\bar{x}|^{-1}
	+
	t^{-\frac 32} (\ln t)^{-1}
	&
	\mbox{ \ if \ } \mu_0 <  |\bar{x}|\le t^{\frac 12}
	\\
	t^{-\frac 32} (\ln t)^{-1}
	e^{- \frac{| \bar{x} |^2}{16 t}}
	+ t(\ln t)^{-2} |\bar{x}|^{-5}
	&
	\mbox{ \ if \ }  |\bar{x}| > t^{\frac 12}
\end{cases}
	.
\end{equation}
\begin{proof}
Notice $\bar{\mu}_0\sim (\ln t)^{-1}$ and $|\bar{\mu}_{0t}| \sim t^{-1}(\ln t)^{-2}$.
By \eqref{z1}, we have
\begin{equation*}
	|\nabla_{\bar{x}} \tilde{\varphi}_1 [\bar{\mu}_{0}] |
	\lesssim
	|\bar{x} |  t^{-2} (\ln t)^{-1}  \1_{\{ |\bar{x}| \le 2t^{\frac 12} \}}
	+
	| \bar{x} |^{-1}    (t \ln t)^{-1}
	e^{-\frac{| \bar{x} |^2}{4t}}
	\1_{\{ |\bar{x}| > 2t^{\frac 12} \}}
	.
\end{equation*}

For $\nabla_{\bar{x}} \tilde{\varphi}_{1b}[\bar{\mu}_{0} ] $, we abbreviate $\nabla_{\bar{x}} \tilde{\varphi}_{1b}[\bar{\mu}_{0} ]$ as $\nabla_{\bar{x}} \tilde{\varphi}_{1b}$. 	
By \eqref{til-varphi1b}, then
	\begin{equation*}
	\nabla_{\bar{x}} \tilde{\varphi}_{1b} (\bar x,t) = \TT_4^{out}[ \nabla_{\bar{x}}(-\bar{\mu}_{0t} \hat{\varphi}_{1} + (E - \tilde{E})[\bar{\mu}_0] ) ] (\bar{x},t) .
\end{equation*}

Notice by \eqref{z1}, we have
\begin{equation*}
	|\bar{\mu}_{0t} \nabla_{\bar{x}} \hat{\varphi}_1(\bar{x},t)  |
	\lesssim
	|\bar{x} |  t^{-3} (\ln t)^{-2} \1_{\{ |\bar{x}| \le t^{\frac 12} \}}
	+
	| \bar{x} |^{-1} t^{-2}  (\ln t)^{-2}
	e^{-\frac{| \bar{x} |^2}{4t}}
	\1_{\{ |\bar{x}| > t^{\frac 12} \}}
	,
\end{equation*}
\begin{equation*}
	|\nabla_{\bar{x}}(E - \tilde{E} )[\bar{\mu}_0]| \lesssim
 t^{-\frac 72} ( \ln t)^{-3}
	\1_{\{ t^{\frac 12} \le |\bar{x}| \le 2  t^{\frac 12}  \} },
\end{equation*}
and
\begin{equation*}
	\TT_{4}^{out}\left[ t^{-\frac 72} ( \ln t)^{-3}
	\1_{\{ t^{\frac 12} \le |\bar{x}| \le 2  t^{\frac 12}  \} } + |\bar{x} |  t^{-3} (\ln t)^{-2} \1_{\{ |\bar{x}| \le t^{\frac 12} \}}\right]
	\lesssim
\TT_{4}^{out}\left[ t^{-\frac 52 }  ( \ln t)^{-2} \1_{\{ |\bar{x}| \le 2 t^{\frac 12} \}}\right]
\lesssim
	t^{-\frac 32} (\ln t)^{-2} e^{-\frac{|\bar{x}|^2}{16 t}}
	,
\end{equation*}
\begin{equation*}
	\TT_{4}^{out}\left[| \bar{x} |^{-1} t^{-2}  (\ln t)^{-2}
	e^{-\frac{| \bar{x} |^2}{4t}}
	\1_{\{ |\bar{x}| > t^{\frac 12} \}}\right]
\lesssim
\TT_{4}^{out}\left[ (\ln t)^{-2} | \bar{x} |^{-5}
\1_{\{ |\bar{x}| > t^{\frac 12} \}}\right]
\lesssim
\begin{cases}
t^{-\frac 32} (\ln t)^{-2}
&
\mbox{ \ if \ } |\bar{x}|\le t^{\frac 12}
\\
t(\ln t)^{-2} |\bar{x}|^{-5}
&
\mbox{ \ if \ } |\bar{x}| > t^{\frac 12}
\end{cases}
.
\end{equation*}
Thus
\begin{equation*}
	|\nabla_{\bar{x}} \tilde{\varphi}_{1b}|
	\lesssim
	t^{-\frac 32} (\ln t)^{-2}
	\1_{\{ |\bar{x}|\le t^{\frac 12} \} }
	+
t(\ln t)^{-2} |\bar{x}|^{-5}
\1_{\{ |\bar{x}| > t^{\frac 12} \} }.
\end{equation*}

Next, let us consider $\nabla_{\bar{x}} \varphi_2[\bar{\mu}_0]$. Recall the definition of $\varphi_2$ in Lemma \ref{varphi2-lem}, then
\begin{equation*}
\nabla_{\bar{x}} \varphi_{2}
=
\TT^{out}_4
\left[ \bar{\mu}_{0}^{-2} \bar{\mu}_{0t}
\nabla_{ \bar{x} } \left( Z_5(\frac{\bar{x}}{ \bar{\mu}_{0} }) \eta(\frac{\bar{x} }{\sqrt t}) \right) \right] .
\end{equation*}
Notice
\begin{equation*}
	\begin{aligned}
		&
		\left|
		\bar{\mu}_{0}^{-2} \bar{\mu}_{0t}
		\nabla_{ \bar{x} } ( Z_5(\frac{\bar{x}}{ \bar{\mu}_{0} }) \eta(\frac{\bar{x} }{\sqrt t}) )  \right|
		=
		\left| \bar{\mu}_{0}^{-2} \bar{\mu}_{0t}
		(
		\bar{\mu}_{0}^{-1}
		\nabla Z_5(\frac{\bar{x}}{ \bar{\mu}_{0} }) \eta(\frac{\bar{x} }{\sqrt t}) + t^{-\frac 12}  Z_5(\frac{\bar{x}}{ \bar{\mu}_{0} })
		\nabla \eta(\frac{\bar{x} }{\sqrt t} )
		) \right|
		\\
		\lesssim \ &
		t^{-1}
		\left[\ln t (1+|\frac{\bar{x}}{ \bar{\mu}_{0} }|)^{-3} \1_{ \{ |\bar{x}|\le 2t^{\frac 12} \} } + t^{-\frac 12} (1+|\frac{\bar{x}}{ \bar{\mu}_{0} }|)^{-2}
		\1_{ \{  t^{\frac 12} \le |\bar{x}|\le 2t^{\frac 12} \} }
		\right]
		\sim
		 t^{-1}
		\ln t (1+|\frac{\bar{x}}{ \bar{\mu}_{0} }|)^{-3} \1_{ \{ |\bar{x}|\le 2t^{\frac 12} \} }
		\\
		\sim \ &
		t^{-1}
		\ln t  \1_{ \{ |\bar{x}|\le \mu_0 \} }
		+
		t^{-1}
		(\ln t)^{-2} |\bar{x}|^{-3} \1_{ \{ \mu_0 < |\bar{x}|\le 2t^{\frac 12} \} } .
	\end{aligned}
\end{equation*}
Therefore, we obtain
\begin{equation*}
| \nabla_{\bar{x}}\varphi_{2}  | \lesssim
	\TT_{4}^{out}\left[ t^{-1}
	\ln t  \1_{ \{ |\bar{x}|\le \mu_0 \} }  + t^{-1}
	(\ln t)^{-2} |\bar{x}|^{-3} \1_{ \{ \mu_0 < |\bar{x}|\le 2t^{\frac 12} \} }\right]
	\lesssim
	\begin{cases}
		(t\ln t)^{-1}
		&
		\mbox{ \ if \ } |\bar{x}|\le \mu_0
		\\
		t^{-1} (\ln t)^{-2} |\bar{x}|^{-1}
		&
		\mbox{ \ if \ } \mu_0 <  |\bar{x}|\le t^{\frac 12}
		\\
		t^{-\frac 32} (\ln t)^{-2} e^{-\frac{|\bar{x}|^2}{16 t}}
		&
		\mbox{ \ if \ }  |\bar{x}| > t^{\frac 12}
	\end{cases}
	.
\end{equation*}
Since $\varphi=\tilde{\varphi}_1 + \tilde{\varphi}_{1b}+ \varphi_2$, one concludes the upper bound of $| \nabla_{\bar{x}}\varphi[\bar{\mu}_0 ]  |$ in \eqref{varphi-bar-mu0}.

The left part is devoted to estimating $\pp_{t} \varphi[\bar{\mu}_{0} ] $. For $\pp_{t} \tilde{\varphi}_1[\bar{\mu}_{0} ] $, by \eqref{tilde-varphi1},
\begin{equation*}
		|\pp_{t}\tilde{\varphi}_1[\bar{\mu}_0] |
		= \left| 2^{\frac 32 } \bar{\mu}_{0t} | \bar  x|^{-2}
		\left(
		e^{-\frac{| \bar  x|^2}{4t}} -\eta(\frac{|\bar  x|}{\sqrt t})
		\right)
		+
		2^{\frac 32 } \bar{\mu}_0 | \bar  x|^{-2}
		\left(
		\frac{| \bar  x|^2}{4t^2}  e^{-\frac{| \bar  x|^2}{4t} } +\frac{|\bar{x}|}{2t^{\frac 32} } \eta'(\frac{|\bar  x|}{\sqrt t})
		\right)  \right|
		\lesssim  t^{-2} (\ln t)^{-1} e^{-\frac{|\bar{x}|^2}{4t}}.
\end{equation*}

Next, we estimate $\pp_{t}  \tilde{\varphi}_{1b}[\bar{\mu}_{0}]$.
For any integer $n\ge1$ and $f(x,t)\in C^1(\RR^n\times (t_0,\infty))$,
\begin{equation}\label{pptu-eq}
	\begin{aligned}
		&
		\pp_{t} \left( 	\int_{t_0}^{t} \int_{\RR^n} \left[4\pi(t-s) \right]^{-\frac{n}{2}} e^{-\frac{|x-y|^2}{4(t-s)}} f(y,s) dyds  \right)
		\\
		= \ &
		\pp_{t} \left( 	\int_{t_0}^{\frac{t}{2}} \int_{\RR^n} \left[4\pi(t-s) \right]^{-\frac{n}{2}} e^{-\frac{|x-y|^2}{4(t-s)}} f(y,s) dyds
		+
		\int_{0}^{\frac{t}{2}}
		\int_{\RR^n}
		\left(4\pi a \right)^{-\frac{n}{2}}
		e^{-\frac{|x-y|^2}{4a}}
		f(y,t-a) dy da
		\right)
		\\
		= \ &
		\frac12 \int_{\RR^n} ( 2\pi t )^{-\frac{n}{2}} e^{-\frac{|x-y|^2}{2t}} f(y,\frac{t}{2}) dy
		+
		\int_{t_0}^{\frac{t}{2}}
		\int_{\RR^n} \pp_{t} \left\{ \left[4\pi(t-s) \right]^{-\frac{n}{2}} e^{-\frac{|x-y|^2}{4(t-s)}} \right\} f(y,s) dyds
		\\
		& +
		\int_{
			\frac{t}{2}}^{t} \int_{\RR^n} \left[4\pi(t-s) \right]^{-\frac{n}{2}} e^{-\frac{|x-y|^2}{4(t-s)}} (\pp_{t}f)(y,s) dyds .
	\end{aligned}
\end{equation}
As a consequence of \eqref{til-varphi1b} and \eqref{pptu-eq}, we have
\begin{equation*}
	\begin{aligned}
		&
		\pp_{t} \tilde{\varphi}_{1b}
		=
		\frac12\int_{\RR^4} ( 2\pi t )^{-2} e^{-\frac{|x-y|^2}{2t}} \left(-\bar{\mu}_{0t} \hat{\varphi}_{1} + (E - \tilde{E})[\bar{\mu}_0] \right) (y,\frac{t}{2}) dy
		\\
		&
		+
		\int_{t_0}^{\frac{t}{2}}
		\int_{\RR^4} \pp_{t} \left\{ \left[4\pi(t-s) \right]^{-2} e^{-\frac{|x-y|^2}{4(t-s)}} \right\} \left(-\bar{\mu}_{0t} \hat{\varphi}_{1} + (E - \tilde{E})[\bar{\mu}_0] \right)(y,s) dyds
		\\
		& +
		\int_{
			\frac{t}{2}}^{t} \int_{\RR^4} \left[4\pi(t-s) \right]^{-2} e^{-\frac{|x-y|^2}{4(t-s)}} \left[ \pp_{t} \left(-\bar{\mu}_{0t} \hat{\varphi}_{1} + (E - \tilde{E})[\bar{\mu}_0] \right) \right](y,s) dyds .
	\end{aligned}
\end{equation*}
where by \eqref{ZE-1} and \eqref{ZE-2}, it follows that
\begin{equation*}
	\begin{aligned}
		&
\left|
\left(
-\bar{\mu}_{0t} \hat{\varphi}_{1} + (E - \tilde{E})[\bar{\mu}_0 ] \right)(\bar{x},\frac{t}{2})
\right| +
\left|
		\left(
		-\bar{\mu}_{0t} \hat{\varphi}_{1} + (E - \tilde{E})[\bar{\mu}_0 ] \right)(\bar{x},t)
		\right|
		\\
		\lesssim \ &
		(t \ln t)^{-2} \1_{\{ |\bar x|\le t^{\frac 12} \} } +
		t^{-1} (\ln t)^{-2} |\bar x|^{-2}
		e^{-\frac{|\bar x|^2}{4t }}
		\1_{ \{ |\bar x| >  t^{\frac 12} \} }
		+
		(t \ln t)^{-3}
		\1_{\{ t^{\frac{1}{2}} \le |\bar{x}| \le 2 t^{\frac{1}{2}} \} }
		\lesssim
		(t \ln t)^{-2}
		e^{-\frac{|\bar x|^2}{4t }} ,
	\end{aligned}
\end{equation*}
and
\begin{equation*}
	\begin{aligned}
		&
		|\pp_{t} (\bar{\mu}_{0t} \hat{\varphi}_{1} ) |
		=
		|\bar{\mu}_{0tt} \hat{\varphi}_{1}
		+
		\bar{\mu}_{0t} \pp_{t} \hat{\varphi}_{1}|
		\\
		\lesssim \ &
		\left| (t\ln t)^{-2}
		\left( t^{-1} \1_{\{ |\bar{x}|  \le 2 t^{\frac 12} \}}
		+
		|\bar{x}|^{-2} e^{-\frac{| \bar  x|^2}{4t}}
		\1_{\{ |\bar{x}|  > 2 t^{\frac 12} \}}  \right)
		+
		t^{-3} (\ln t)^{-2}   e^{-\frac{|\bar{x}|^2}{4t}}
		\right|
		\sim
		t^{-3} (\ln t)^{-2}   e^{-\frac{|\bar{x}|^2}{4t}} ,
	\end{aligned}
\end{equation*}
and
\begin{equation*}
	| \pp_{t} (E - \tilde{E})[\bar{\mu}_0]
	|
	\lesssim
	t^{-4} (\ln t)^{-3} \1_{\{ \sqrt t \le |\bar{x}| \le 2\sqrt t \} } .
\end{equation*}
Thus by Lemma \ref{cauchy-est-lem} and same calculation for deducing \eqref{til1b-upp}, we have
\begin{equation*}
		|\pp_{t} \tilde{\varphi}_{1b} |
		\lesssim
		(t \ln t)^{-2}
		\int_{\RR^4} t^{-2} e^{-\frac{|x-y|^2}{2t}}
		  dy
		+t^{-1}
		\int_{t_0}^{ t }
		\int_{\RR^4}  (t-s)^{-2} e^{-\frac{|x-y|^2}{8(t-s)}}   (s \ln s)^{-2}
		e^{-\frac{|y|^2}{4s }}   dyds
	\lesssim
	(t \ln t)^{-2}  .
\end{equation*}

Finally, we consider $	\pp_{t} \varphi_{2}[\bar{\mu}_0] $.
By \eqref{pptu-eq} and  the definition of $\varphi_2$ in Lemma \ref{varphi2-lem}, we have
\begin{equation*}
	\begin{aligned}
		&
		\pp_{t}	\varphi_{2}
		= \frac12\int_{\RR^4} ( 2\pi t )^{-2} e^{-\frac{|x-y|^2}{2t}} \bar{\mu}_{0}^{-2}(\frac{t}{2}) \bar{\mu}_{0t}( \frac{t}{2} )
		Z_5(\frac{y}{ \bar{\mu}_{0}( \frac{t}{2} ) }) \eta(\frac{ \sqrt{2} y }{\sqrt{t } })  dy
		\\
&
+ \int_{t_0}^{\frac{t}{2}} \int_{\RR^4}
\pp_{t} \left\{
\left[ 4\pi(t-s)\right]^{-2}
e^{-\frac{|x-y|^2}{4(t-s)}}
\right\}
\bar{\mu}_{0}^{-2}(s) \bar{\mu}_{0t}(s)
Z_5(\frac{y}{ \bar{\mu}_{0}(s) }) \eta(\frac{y}{\sqrt s})      dy ds
\\
		&
		+
	\int_{\frac{t}{2}}^{t} \int_{\RR^4}
		\left[ 4\pi(t-s)\right]^{-2}
		e^{-\frac{|x-y|^2}{4(t-s)}}
		\left[ \pp_{s}  \left(\bar{\mu}_{0}^{-2}(s) \bar{\mu}_{0t}(s)
		Z_5(\frac{y}{ \bar{\mu}_{0}(s) }) \eta(\frac{y}{\sqrt s})  \right)  \right] dy ds ,
	\end{aligned}
\end{equation*}
and by \eqref{ZA-3-2},
\begin{equation*}
\begin{aligned}
&
\left|
\bar{\mu}_{0}^{-2}(\frac{t}{2}) \bar{\mu}_{0t}(\frac{t}{2})
Z_5(\frac{\bar{x}}{ \bar{\mu}_{0}(\frac{t}{2}) }) \eta(\frac{ \sqrt{2}\bar{x} }{\sqrt{t}})
\right|
+
\left|
\bar{\mu}_{0}^{-2}(t) \bar{\mu}_{0t}(t)
Z_5(\frac{\bar{x}}{ \bar{\mu}_{0}(t) }) \eta(\frac{\bar{x}}{\sqrt t})
\right|
\\
	\lesssim \ &
	 t^{-1} \1_{\{ |\bar{x}| \le (\ln t)^{-1} \}} +   t^{-1} (\ln t)^{-2} |\bar x|^{-2} \1_{\{ (\ln t)^{-1}  < |\bar{x}| \le 2 t^{\frac 12} \}} ,
\end{aligned}
\end{equation*}
and
\begin{equation*}
	\begin{aligned}
		&
		\left|\pp_{t}  (
		\bar{\mu}_{0}^{-2} \bar{\mu}_{0t}
		Z_5(\frac{\bar{x}}{ \bar{\mu}_{0} }) \eta(\frac{\bar{x} }{\sqrt t})  )
		\right|
		\\
		= \ &
		\left|
		\pp_{t}  (\bar{\mu}_{0}^{-2} \bar{\mu}_{0t} )
		Z_5(\frac{\bar{x}}{ \bar{\mu}_{0} }) \eta(\frac{\bar{x} }{\sqrt t})
		-
		\bar{\mu}_{0}^{-2} \bar{\mu}_{0t}
		\frac{\bar{x}}{\bar{\mu}_0}
		\cdot
		\nabla Z_5(\frac{\bar{x}}{ \bar{\mu}_{0} })
		\frac{\bar{\mu}_{0t}}{\bar{\mu}_0} \eta(\frac{\bar{x} }{\sqrt t})
		-
		\bar{\mu}_{0}^{-2} \bar{\mu}_{0t}
		Z_5(\frac{\bar{x}}{ \bar{\mu}_{0} }) \frac{\bar{x}}{2t^{\frac 32}} \cdot \nabla \eta(\frac{\bar{x} }{\sqrt t})
		\right|
		\\
		\lesssim \ &
		t^{-2} \langle \frac{\bar{x}}{\bar{\mu}_0} \rangle^{-2} \1_{\{ |\bar{x}|\le 2t^{\frac 12} \}} .
	\end{aligned}
\end{equation*}
Thus, by similar calculation for Lemma \ref{cauchy-est-lem} and the upper bound of $\varphi_2$ in Lemma \ref{varphi2-lem}, we have
\begin{equation*}
	\begin{aligned}
		&
		|\pp_{t}	\varphi_{2} |
\lesssim
\int_{\RR^4}  t^{-2} e^{-\frac{|x-y|^2}{2t}}
\left(
t^{-1} \1_{\{ |y| \le (\ln t)^{-1} \}} +   t^{-1} (\ln t)^{-2} |y|^{-2} \1_{\{ (\ln t)^{-1}  < |y| \le 2 t^{\frac 12} \}}
\right)  dy
		\\
		&
		+ t^{-1} \int_{t_0}^{t} \int_{\RR^4}
	(t-s)^{-2}
		e^{-\frac{|x-y|^2}{8(t-s)}}
\left(
 s^{-1} \1_{\{ |y| \le (\ln s)^{-1} \}} +   s^{-1} (\ln s)^{-2} |y|^{-2} \1_{\{ (\ln s)^{-1}  < |y| \le 2 s^{\frac 12} \}}
\right)     dy ds
\lesssim t^{-2} (\ln t)^{-1} .
	\end{aligned}
\end{equation*}
Collecting above estimates, we obtain $
	|\pp_{t} \varphi[\bar{\mu}_{0} ]  | \lesssim t^{-2} (\ln t)^{-1} $.
	
\end{proof}

\medskip

\section*{Acknowledgements}

J. Wei is partially supported by
NSERC of Canada. We thank Professors Manuel del Pino and Monica Musso for their interests and suggestions.



\end{document}